%% file: 2verma.tex
\documentclass[12pt]{amsart}

\input{preamble/packages.tex}
\input{preamble/macros.tex}

\input{preamble/paperdata.tex}


\begin{document}

\input{preamble/abstract.tex}

\maketitle

\tableofcontents

%%%%%%%%%%%%%%%%%%%%%%
% Sections						  %
%%%%%%%%%%%%%%%%%%%%%%

\input{sections/intro.tex}

\input{sections/qgroups.tex}
\input{sections/dgKLR.tex}
\input{sections/dgenhancement.tex}
\input{sections/cataction.tex}
\input{sections/catthm.tex}
\input{sections/2verma.tex}

\input{sections/appendix_dgcat.tex}

%%%%%%%%%%%%%%%%%%%%%%
% Bibliography					  %
%%%%%%%%%%%%%%%%%%%%%%
\input{bibliography/bibliography.tex}

\end{document}

%% file: preamble/packages.tex
\usepackage{
	amsmath,
 	amsfonts,
  	amssymb,
 	amsthm,
	}

\usepackage{tikz-cd}
\usepackage{tikz}

\usepackage{rotating}

\usetikzlibrary{decorations.pathmorphing, decorations.pathreplacing}
 \usetikzlibrary{decorations.pathreplacing,backgrounds,decorations.markings}

\usepackage{color}

\usepackage{stmaryrd}
\usepackage[cal=boondoxo]{mathalfa}
\usepackage{mathtools}
\usepackage{bbm}

\usepackage{hyperref}
\usepackage[capitalise]{cleveref}

\usepackage{a4wide}

\usepackage{mathabx}

%% file: preamble/macros.tex
\theoremstyle{plain}
\newtheorem{thm}{Theorem}[section]
\newtheorem{cor}[thm]{Corollary}
\newtheorem{lem}[thm]{Lemma}
\newtheorem{prop}[thm]{Proposition}
\newtheorem{conj}[thm]{Conjecture}

\theoremstyle{definition}
\newtheorem{rem}[thm]{Remark}

\newtheorem{exe}[thm]{Example}

\theoremstyle{definition}
\newtheorem{defn}[thm]{Definition}

\newenvironment{citethm}[1]{%
	\renewcommand\thethm{\ref{#1}}\thm}{\endthm\addtocounter{thm}{-1}}
	
\newenvironment{citecor}[1]{%
	\renewcommand\thethm{\ref{#1}}\cor}{\endthm\addtocounter{thm}{-1}}

\def\makeautorefname#1#2{\expandafter\def\csname#1autorefname\endcsname{#2}}
\makeautorefname{lem}{Lemma}%
\makeautorefname{prop}{Proposition}%
\makeautorefname{rem}{Remark}%
\makeautorefname{section}{Section}%

\newcommand{\cA}{\mathcal{A}}
\newcommand{\cB}{\mathcal{B}}
\newcommand{\cC}{\mathcal{C}}
\newcommand{\cD}{\mathcal{D}}

\newcommand{\cM}{\mathcal{M}}

\newcommand{\cT}{\mathcal{T}}

\newcommand{\cV}{\mathcal{V}}

\newcommand{\bN}{\mathbb{N}}

\newcommand{\bQ}{\mathbb{Q}}
\newcommand{\bR}{\mathbb{R}}

\newcommand{\bZ}{\mathbb{Z}}

\newcommand{\be}{{\boldsymbol{e}}}
\newcommand{\bg}{{\boldsymbol{g}}}
\newcommand{\bi}{{\boldsymbol{i}}}
\newcommand{\bj}{{\boldsymbol{j}}}

\newcommand{\vsimeq}{\rotatebox{-90}{\(\simeq\)}}

\newcommand{\brak}[1]{\langle #1\rangle}
\newcommand{\pp}[1]{(\!( #1 )\!)}
\newcommand{\opalg}[1]{{#1}^{\opp}}
\newcommand{\und}[1]{{\underline{#1}}}

\newcommand{\sssum}[1]{{\sum\limits_{\substack{#1}}}}

\newcommand{\bV}{\raisebox{0.03cm}{\mbox{\footnotesize$\textstyle{\bigwedge}$}}}

\DeclareMathOperator{\id}{Id}
\DeclareMathOperator{\Hom}{Hom}
\DeclareMathOperator{\End}{End}

\DeclareMathOperator{\Rep}{Rep}

\DeclareMathOperator{\Ind}{Ind}
\DeclareMathOperator{\Res}{Res}

\DeclareMathOperator{\HOM}{HOM}
\DeclareMathOperator{\END}{END}

\DeclareMathOperator{\cHom}{\mathcal{H}\mathit{om}}
\DeclareMathOperator{\cRHom}{\mathcal{RH}\mathit{om}}

\DeclareMathOperator{\cEnd}{\mathcal{E}\mathit{nd}}

\DeclareMathOperator{\Lderiv}{L}
\newcommand{\Lotimes}{\otimes^{\Lderiv}}

\DeclareMathOperator{\opp}{{op}}

\DeclareMathOperator{\gdim}{gdim}

\DeclareMathOperator{\Image}{im}
\DeclareMathOperator{\cok}{cok}

\DeclareMathOperator{\mcolim}{MColim}
\DeclareMathOperator{\mlim}{MLim}

\DeclareMathOperator{\cone}{Cone}

\newcommand{\bKO}{\boldsymbol{K}_0}

\newcommand{\slt}{\mathfrak{sl_2}}
\newcommand{\g}{{\mathfrak{g}}}
\newcommand{\p}{{\mathfrak{p}}}
\newcommand{\bo}{{\mathfrak{b}}}

\newcommand\E{{\sf{E}}}
\newcommand\F{{\sf{F}}}
\newcommand\K{{\sf{K}}}
\newcommand\Q{{\sf{Q}}}

\newcommand{\ch}{\mathcal{h}}

\newcommand{\inlmin}[1]{\mbox{$\min_{#1}$}}

\DeclareMathOperator{\amod}{\mathrm{-}mod}

\DeclareMathOperator{\dgcat}{dg\mathrm{-}cat}

\DeclareMathOperator{\Hqe}{Hqe}

\DeclareMathOperator{\nh}{NH}
\DeclareMathOperator{\Supp}{Supp}
\DeclareMathOperator{\Seq}{Seq}
\DeclareMathOperator{\Seqd}{Seqd}
\DeclareMathOperator{\Sym}{Sym}

\newcommand{\tikzdiagh}[2][]{\tikz[#1,very thick,baseline={([yshift=1ex+#2]current bounding box.center)}]}
\newcommand{\tikzdiag}[1][]{\tikzdiagh[#1]{-1.5ex}}\tikzstyle{tikzdot}=[fill, circle, inner sep=2pt]
\newcommand{\fdot}[3][]{ \node  [anchor = center, fill=white, draw=black,circle,inner sep=2pt] at (#3) {} ; \node[xshift=-.065cm, yshift=-.035cm,anchor = south west] at (#3){\small $#1$}; \node[xshift=-.065cm, yshift=.065cm,anchor = north west] at (#3){\small $#2$}; }
\newcommand{\plusspacing}{\llap{\phantom{\small ${+}1$}}}

\definecolor{myblue}{rgb}{0,.5,1}
\definecolor{mygreen}{rgb}{.2,.8,0}

\newcommand{\remacyclicproj}{Remark 9.5}
\newcommand{\thmasympKO}{Theorem 9.15}
\newcommand{\propcblfbim}{Proposition 9.18}

\hypersetup{colorlinks=true, pdfstartview=FitV, linkcolor=blue, citecolor=blue, urlcolor=blue}

%% file: preamble/paperdata.tex
\title{2-Verma modules}

\author{Gr\'egoire Naisse}
\address{Max-Planck Institute for Mathematics\\
 Vivatsgasse 7 \\ 
53111 Bonn\\ 
Germany}
\email{gregoire.naisse@gmail.com}
\author{Pedro Vaz}
\address{Institut de Recherche en Math\'ematique et Physique\\
Universit\'e Catholique de Louvain\\ 
Chemin du Cyclotron 2\\ 
1348 Louvain-la-Neuve\\ 
Belgium}
\email{pedro.vaz@uclouvain.be}

%% file: preamble/abstract.tex
%%%%%%%%%%%%%%%%%%%%%%%%%%%%%%%%%%%%
%												%
%	Abstract										%
%												%
%%%%%%%%%%%%%%%%%%%%%%%%%%%%%%%%%%%%

\begin{abstract}
We construct a categorification of parabolic Verma modules for symmetrizable Kac--Moody algebras using KLR-like diagrammatic algebras. 
We show that our construction arises naturally from a dg-enhancement of the cyclotomic quotients of the KLR-algebras. 
As a consequence, we are able to recover the usual categorification of integrable modules. 
We also introduce a notion of dg-2-representation for quantum  Kac--Moody algebras, and in particular of parabolic 2-Verma modules. 
\end{abstract}

%%%%%%%%%%%%%%%%	End of file	%%%%%%%%%%%%%

%% file: sections/intro.tex
%%%%%%%%%%%%%%%%%%%%%%%%%%%%%%%%%%%%
%                 		   %
%	Introduction   		   %
%                 		   %
%%%%%%%%%%%%%%%%%%%%%%%%%%%%%%%%%%%%

\section{Introduction}\label{sec:intro}

The study of categorical actions of (quantum enveloping algebras of) Kac--Moody algebras leads to many interesting
results.
An impressive example is due to Chuang and Rouquier~\cite{CR}, who introduced categorical actions of $\slt$ to prove the Brou\'e abelian defect group conjecture for symmetric groups.
Another interesting result is Webster's construction of homological versions of quantum invariants of links obtained by the Reshetikhin--Turaev machinery~\cite{webster}. 

\smallskip

Until recently, only categorifications of integrable representations of quantum Kac--Moody algebras were known.
These are given by additive (or abelian) categories, on which the quantum group acts by (exact) endofunctors respecting certain direct sum decompositions, corresponding to the defining relations of the algebra (see for example~\cite{fks,kashiwara, L1, L2, rouquier}).
In~\cite{naissevaz1}, the authors followed a slightly different approach to construct a categorification of the universal Verma module $M(\lambda)$ for quantum $\slt$. The construction of~\cite{naissevaz1} is given in the form of an abelian, bigraded (super)category, where the commutator relation takes the form of a (non-split) natural short exact sequence 
\[
0 \rightarrow \F\E \rightarrow \E\F \rightarrow \Q\K \oplus \Pi\Q\K^{-1} \rightarrow 0,
\]
where $\Pi$ is the parity shift functor, and $\Q$ a categorification of $\frac{1}{q-q^{-1}}$ in the form of an infinite direct sum. 
This category is obtained as a certain category of modules over cohomology rings of infinite Grassmannianns and their Koszul duals.
Categorification of Verma modules appeared independently in the litterature with a strongly different flavor in \cite{coxim} and in \cite{ACNJTY}.

 \smallskip

 Studying the endomorphism ring of $\F^k := \F \circ \cdots \circ \F$ yields a (super)algebra $A_k$ that extends the ubiquitous nilHecke algebra $\nh_k$.
This superalgebra was studied by the authors in the follow up~\cite{naissevaz2}, where it was used to construct an equivalent categorification of Verma modules for quantum $\slt$.
The supercenter of $A_k$ was also studied in~\cite{AEHL}. 
The definition of the superalgebra $A_k$ and is supercenter were extended in~\cite{reeks} to the case of a Weyl group of type $B$.

 \smallskip

The superalgebra $A_k$ comes equipped with a family of differentials $d_n$ for $n \ge 0$. The corresponding dg-algebras are formal, with homology being isomorphic to the $n$-cyclotomic quotients of the nilHecke algebra. These quotients are known to categorify the irreducible integrable $U_q(\slt)$-representations $V(n)$ of highest weight $n$. 
We interpret this as a categorification of the universal property of the Verma module $M(\lambda)$, that is there is a surjection $M(\lambda)\twoheadrightarrow V(n)$ for all $n$.
This also means the dg-algebra $(A_k, d_n)$ can be seen as a dg-enhancement of the cyclotomic nilHecke algebra $\nh_k^n$, and in particular, of categorified $V(n)$. 

\smallskip

In~\cite{KL1, KL2} and~\cite{rouquier}, Khovanov--Lauda and Rouquier introduced generalizations of the nilHecke algebra for any Cartan datum.
These algebras are presented in the form of braid-like diagrams in~\cite{KL1, KL2}, with strands labeled by simple roots and decorated with dots.
It is proven in~\cite{KL1, KL2,rouquier} that KLR algebras categorify the half quantum group associated with the input Cartan datum.
Khovanov and Lauda conjectured that certain quotients of these algebras categorify irreducible, integrable representations of the quantum group.
Due to the isomorphism between these quotient algebras and cyclotomic Hecke algebras in type $A$ (see~\cite{BK,rouquier}), these quotients have become known as cyclotomic KLR algebras. The corresponding  
cyclotomic conjecture was first proven in~\cite{brundankleshchev,brundanstroppel,laudavazirani} for some special cases,
and then for all symmetrizable Kac--Moody algebras by Kang--Kashiwara in~\cite{kashiwara}, and independently by Webster in~\cite{webster}.

\smallskip

In this paper, we introduce a version of KLR algebra associated to a pair $(\p,\g)$, where $\p$ is a (standard) parabolic subalgebra of a quantum Kac-Moody algebra $\g$. This construction generalizes the algebra $A_k$ from~\cite{naissevaz1},
which we view as associated to the (standard) Borel subalgebra of $\slt$. The usual KLR algebra is recovered by taking $\p = \g$.  
We prove that certain `cyclotomic quotients' of these $\p$-KLR algebras categorify parabolic Verma modules induced over the parabolic subalgebra $\p$, with the cyclotomic quotient depending on the highest weight. 
The proof goes by showing first that if $\p=\bo$ is the (standard) Borel subalgebra of $\g$, then the $\bo$-KLR algebra is equipped with a categorical $\g$-action similar to the one constructed in~\cite{naissevaz2}. In particular, it categorifies the universal Verma module of $\g$. 
Next, we show that the $\bo$-KLR algebra can be equipped with a family of differentials, turning it into a dg-enhancement of the cyclotomic $\p$-KLR algebras. This induces a categorical $\g$-action on the cyclotomic $\p$-KLR algebra. In particular, we recover the usual categorical action on cyclotomic KLR algebras, and we can reinterpret  Kang--Kashiwara's proof of Khovanov--Lauda's cyclotomic conjecture in terms of dg-enhanced KLR algebras. 
The world of dg-categories also allows to reinterpret the usual categorical $\slt$-commutator relation in terms of mapping cones. 
More precisely, the derived  category of dg-modules over the dg-enhanced KLR algebra comes equipped with functors $\E_i, \F_i$ and an autoequivalence $\K_i$ for all simple root $\alpha_i$, that categorifies  the action of the Chevalley generators $E_i, F_i$ and of the Cartan element $K_i = q_i^{H_i}$. 
Then, the $\slt$-commutator relation of the categorical action takes the form of a quasi-isomorphism of mapping cones
\[
\cone(\F_i\E_i \rightarrow \E_i \F_i) \xrightarrow{\simeq} \cone(\Q_i\K_i \rightarrow \Q_i\K_i^{-1}),
\]
where $\Q_i$ is a direct sum of grading shift copies of the identity functor that categories $\frac{1}{q_i^{-1}-q_i}$. 
Whenever $F_i$ is locally nilpotent,  $\cone(\Q_i\K_i \rightarrow \Q_i\K_i^{-1})$ is quasi-isomorphic to a finite direct sum of shifted copies of the identity functor, corresponding to the usual notion of an integrable categorical $\g$-action (as in \cite{kashiwara} for example).

\smallskip

Categorification of parabolic Verma modules have found connections with topology in the work of the authors in~\cite{naissevaz3}. 
In particular, they have constructed Khovanov--Rozansky's triply graded link homology using parabolic 2-Verma modules of $\mathfrak{gl}_{2k}$. 
On the decategorified level, the connection between the HOMFPY-PT link polynomial and Verma modules was not known before. 
We expect to find in the future more connections between categorified Verma modules and low-dimensional topology. 

%%%%%%%%%%%%%%%%%%%%%%%%%%%%%%%%%%%%%%%%%%%%%%%%%%%%%%
\subsection*{Outline of the paper}

In~\cref{sec:verma}, we recall the basics about quantum groups and their parabolic Verma modules. 

In~\cref{sec:dgKLR}, we introduce the $\bo$-KLR algebra $R_\bo$ (\cref{def:Rbo}) as a diagrammatic algebra over a unital commutative ring $\Bbbk$, in the same spirit as Khovanov--Lauda's~\cite{KL1}. We construct a faithful action on a polynomial ring and exhibit a basis, proving $R_\bo$ is a free $\Bbbk$-module.  

In~\cref{sec:dgenhancement}, we introduce the $\p$-KLR algebra $R_\p$ for any (standard) parabolic subalgebra $\p$ of $\g$. We also introduce the corresponding $N$-cyclotomic quotient $R_\p^N$. We introduce a differential $d_N$ on $R_\bo$, turning it into a dg-enhancement of $R_\p^N$. In particular, we prove the following theorem:
\begin{citethm}{thm:RbodNformal}
The dg-algebra $(R_\bo(m), d_N)$ is formal with homology
\[
H(R_\bo(m), d_N) \cong R_\p^N(m).
\]
\end{citethm}

In  \cref{sec:cataction}, we construct a categorical action of $U_q(\g)$ on $R_\bo$, where the action of the Chevalley generators $F_i$ and $E_i$ is given by functors $\F_i$ and $\E_i$ which are defined in terms of induction and restriction functors for the map that adds a strand labeled $i$.  The  $\slt$-commutator relation takes the form of a non-split natural short exact sequence. 
Let $\oplus_{[\beta_i - \alpha_i^\vee(\nu)]_{q_i}} \id_\nu$
be an infinite direct sum of degree shifts of the identity functor that categorifies the power series $({\lambda_i q_i^{- \alpha_i^\vee(\nu)} - \lambda_i^{-1}q_i^{\alpha_i^\vee(\nu)}})/({q_i-q_i^{-1}})$ (see \cref{eq:directsumbetaquantum} in the beginning of ~\cref{sec:cataction}). 
\begin{citecor}{cor:catsltactionRbo}
There is a natural short exact sequence
\[
0 \rightarrow \F_i\E_i \id_\nu \rightarrow \E_i\F_i \id_\nu \rightarrow \oplus_{[\beta_i - \alpha_i^\vee(\nu)]_{q_i}} \id_\nu \rightarrow 0,
\]
for all $i \in I$, and there is a natural isomorphism
\[
\F_i\E_j \cong \E_j\F_i,
\]
for all $i \neq j \in I$.
\end{citecor}
Fix $\p \subset \g$, and let $I_f$ be the set of simple roots for which $F_i \in \p$. 
Let $\oplus_{[n]_{q_i}} \id_\nu$ be a finite direct sum of degree shifts of the identity functor that categorifies the quantum integer $[n]_{q_i}$. 
The categorical $\g$-action on $R_\bo$ lifts to the dg-algebra $(R_\bo,d_N)$, and thus to $R_\p^N$ by \cref{thm:RbodNformal}. 
 The short exact sequence of \cref{cor:catsltactionRbo} lifts to a short of exact sequence of complexes, inducing a long exact sequence in homology. This allows us to compute the action of the functors of induction $\F_i^N$ and restriction $\E_i^N$ on $R_\p^N$: 
\begin{citethm}{thm:sl2commutRpN}
For $i \notin I_f$ there is a natural short exact sequence
\begin{equation*}
0 \rightarrow \F^N_i\E^N_i \id_\nu \rightarrow \E^N_i\F^N_i \id_\nu \rightarrow \oplus_{[\beta_i - \alpha_i^\vee(\nu)]_{q_i}} \id_\nu \rightarrow 0,
\end{equation*}
and for $i \in I_f$ there are natural isomorphisms
\begin{equation*}
  \begin{aligned}
 \E^N_i\F^N_i \id_\nu &\cong   \F^N_i\E^N_i \id_\nu \oplus_{[n_i - \alpha_i^\vee(\nu)]_{q_i}} \id_\nu, & \text{ if $n_i - \alpha_i^\vee(\nu) \geq 0$}, \\
 \F^N_i\E^N_i \id_\nu &\cong   \E^N_i\F^N_i \id_\nu \oplus_{[\alpha_i^\vee(\nu)-n_i]_{q_i}} \id_\nu, & \text{ if $n_i - \alpha_i^\vee(\nu) \leq 0$}.
\end{aligned}
\end{equation*}
Moreover, there is a natural isomorphism
\begin{equation*}
\F^N_i\E^N_j \cong \E^N_j\F^N_i,
\end{equation*}
 for $i \neq j \in I$. 
\end{citethm}

In \cref{sec:catthm}, we compute the asymptotic Grothendieck group of $(R_\bo,d_N)$.
The asymptotic Grothendieck group is a refined version of Grothendieck group, that was introduced by the first author in~\cite{asympK0}. It allows taking in consideration infinite iterated extensions of objects, such as infinite projective resolutions and infinite composition series (see \cref{def:toptriangulatedK0}). 
Let $M^\p(\Lambda,N)$ be the parabolic Verma module of  highest weight $(\Lambda,N)$, and  $\cM^\p(\Lambda,N)$ be the c.b.l.f. derived category of  $(R_\bo,d_N)$ (see \cref{sec:cblfderived}).
\begin{citethm}{thm:catallverma}
The asymptotic Grothendieck group ${}_\bQ\bKO^\Delta(\cM^\p(\Lambda,N))$ is a $U_q(\g)$-weight module, with action of $E_i, F_i$ given by $[\E_i], [\F_i]$. Moreover, there is an isomorphism of $U_q(\g)$-modules
\[
{}_\bQ\bKO^\Delta(\cM^\p(\Lambda,N)) \otimes_\bZ \bQ \cong M^\p(\Lambda,N).
\]
\end{citethm}

In~\cref{sec:2Verma}, we introduce a notion of categorical dg-action of $\g$ on a pretriangulated dg-category (\cref{def:dgcat}), and of (parabolic) 2-Verma module (\cref{def:2verma}). In particular,  we show that $\cM^\p(\Lambda,N)$ admits a dg-enhancement $\cM_{dg}^\p(\Lambda,N)$ in the form of a dg-category. It yields an example of parabolic 2-Verma module, for which \cref{thm:catallverma} takes the following form:

\begin{citecor}{cor:quasiisocones}
For all $i \in I$ there is a quasi-isomorphism of cones
\[
\cone\bigl(\F_i^N\E_i^N \id_\nu \rightarrow \E_i^N\F_i^N \id_\nu \bigr) \xrightarrow{\simeq}  \cone\bigl(\Q_i \lambda_i q_i^{-\alpha_i^\vee(\nu)}  \id_\nu \rightarrow \Q_i \lambda_i^{-1} q^{\alpha_i^\vee(\nu)} \id_\nu\bigr), 
\]
in $\cEnd_{\Hqe}(\cD_{dg}(R_\bo,d_N))$. 
\end{citecor}

%\smallskip

Finally, in \cref{sec:appendixA} we recall the construction of the homotopy category of dg-categories up to quasi-equivalence, based on Toen \cite{toen}. We also recall how to compute the (derived) dg-hom-spaces between pretriangulated dg-categories.

%%%%%%%%%%%%%%%%%%%%%%%%%%%%%%%%%%%%%%%%%%%%%%%%%%%%%%

\subsection*{Acknowledgments}
G.N. is a Research Fellow of the Fonds de la Recherche Scientifique - FNRS, under Grant no.~1.A310.16. 
G.N. is grateful to the Max Planck Institute for Mathematics in Bonn for its hospitality and financial support. 
P.V. was supported by the Fonds de la Recherche Scientifique - FNRS under Grant no.~J.0135.16.

%%%%%%%%%%%%%%%%	End of file	%%%%%%%%%%%%%

%% file: sections/qgroups.tex
%%%%%%%%%%%%%%%%%%%%%%%%%%%%%%%%%%%%%
%
%	Quantum groups and Verma modules	 					 %
%                 					  						 %
%%%%%%%%%%%%%%%%%%%%%%%%%%%%%%%%%%%%

\section{Quantum groups and Verma modules}\label{sec:verma}

We recall the basics about quantum groups and their (parabolic) Verma modules. Our presentation is close to~\cite{jantzen96} and~\cite{lusztig}, where the proofs can be found. References for classical results about  Verma modules are~\cite{mazorchukverma} and~\cite{humphreys} (and~\cite{andersenmazorchuk} for the quantum case).

\subsection{Quantum groups}

A \emph{generalized Cartan matrix} is a finite dimensional square matrix $A = \{a_{ij}\}_{i,j \in I} \in \bZ^{|I| \times |I|}$ such that
\begin{itemize}
\item $a_{ii} = 2$ and $a_{ij} \leq 0$ for all $i \neq j \in I$;
\item $a_{ij} = 0 \Leftrightarrow a_{ji} = 0$.
\end{itemize}
One says that $A$ is symmetrizable if there exists a diagonal matrix $D$ with positive entries $d_i \in \bZ_{>0}$ for all $i \in I$, such that $DA$ is symmetric.
A \emph{Cartan datum} consists of
\begin{itemize}
\item a symmetrizable generalized Cartan matrix $A$;
\item a free abelian group $Y$ called the \emph{weight lattice};
\item a set of linearly independent elements $\Pi = \{\alpha_i\}_{i \in I} \subset Y$ called \emph{simple roots};
\item a \emph{dual weight lattice} $Y^\vee := \Hom(Y, \bZ)$;
\item a set of \emph{simple coroots} $\Pi^\vee = \{\alpha_i^\vee\}_{i \in I} \subset Y^\vee$;
\end{itemize}
such that
\begin{itemize}
\item $\alpha_i^\vee(\alpha_j) = a_{ij}$;
\item for each $i\in I$ there is a \emph{fundamental weight} $\Lambda_i \in Y$ such that $\alpha_j^\vee(\Lambda_i) = \delta_{ij}$ for all $j \in I$.
\end{itemize}
The abelian subgroup $X  := \bigoplus_i \bZ \alpha_i \subset Y$ is called the \emph{root lattice}. We also write $X^+ := \bigoplus_{i} \bN \alpha_i \subset X$ for the \emph{positive roots}.
Given a Cartan datum, since $A$ is symmetrizable with $d_i a_{ij} = d_j a_{ji}$, one can construct a symmetric bilinear form 
\[
(-|-) : Y \times Y \rightarrow \bZ,
\]
 respecting
\begin{itemize}
\item $(\alpha_i|\alpha_i) = 2d_i \in \{2,4, \dots\}$;
\item $(\alpha_i|\alpha_j) = d_ia_{ij} \in \{0,-1,-2, \dots \}$ for all $i \neq j$;
\item $\alpha_i^\vee(y) = 2 \frac{(\alpha_i | y)}{(\alpha_i | \alpha_i)}$ for all $y \in Y$.
\end{itemize}
In the end, a Cartan datum is completely determined by $\left(I, X, Y, (-|-)\right)$.

\begin{defn}
The \emph{quantum Kac--Moody algebra $U_q(\g)$} associated to a Cartan datum $\left(I, X, Y, (-|-)\right)$ is the associative, unital $\bQ(q)$-algebra generated by the set of elements $E_i, F_i$ and $K_{\gamma}$ for all $i \in I$ and $\gamma \in Y^\vee$, with relations for all $i \in I$ and $\gamma,\gamma' \in Y^\vee$:
\begin{align*}
K_0 &= 1,
& K_{\gamma} K_{\gamma'} &= K_{\gamma+\gamma'}, \\
K_{\gamma} E_i  &= q^{\gamma(\alpha_i)} E_i K_{\gamma},  &
 K_{\gamma} F_i &= q^{-\gamma(\alpha_i)} F_i K_{\gamma},
\end{align*}
One also imposes the \emph{$\mathfrak{sl_2}$-commutator relation} for all $i,j \in I$:
\begin{equation}\label{eq:sl2com}
E_iF_j - F_jE_i = \delta_{ij} \frac{K_i - K_i^{-1}}{q_i - q_i^{-1}},
\end{equation}
where $q_i := q^{d_i}$ and $K_i := K_{\alpha_i^\vee}$. \\
Finally, there are the \emph{Serre relations} for $i \neq j \in I$:
\begin{align}
\sum_{r+s = 1- a_{ij}} (-1)^r  \begin{bmatrix}1-a_{ij} \\ r \end{bmatrix}_{q_i} E_i^r E_j E_i^s &= 0, \label{eq:serreE} \\
 \sum_{r+s = 1- a_{ij}} (-1)^r \begin{bmatrix}1-a_{ij} \\ r \end{bmatrix}_{q_i} F_i^r F_j F_i^s &= 0. \label{eq:serreF}
\end{align}
This ends the definition of $U_q(\g)$. 
\end{defn}

Given a sequence $\bi = i_1\cdots i_m$ of elements in $I$, we write $F_\bi := F_{i_1}\cdots F_{i_m}$  and $E_\bi := E_{i_1}\cdots E_{i_m}$. We write $\Seq(I)$ for the set of such sequences. Any element of $U_q(\g)$ decomposes as a sum of elements $F_\bi K_\gamma E_\bj$ with $\bi, \bj \in \Seq(I)$. 

\smallskip

The \emph{half quantum group $U^-_q(\g)$} of $U_q(\g)$ is the subalgebra generated by the elements $\{F_i\}_{i \in I}$. As a $\bQ(q)$-vector space, it admits a basis given by a subset of $\{F_\bi\}_{\bi \in \Seq(I)}$.

\subsection{Weight modules}\label{sec:qgweightmodules}

 Let $M$ be an $U_q(\g)$-module with ground ring $R \supset \bQ(q)$. Consider a $\bZ$-linear functional
\[
\lambda : Y^\vee \rightarrow R^\times,
\]
where the group structure on $R^\times$ is the product. For each such $\lambda$ and $y \in Y$, we call \emph{$(\lambda,y)$-weight space} the set
\[
M_{\lambda,y} := \{ v \in M | K_\gamma v = \lambda(\gamma) q^{\gamma(y)} v \text{ for all $\gamma \in Y^\vee$} \}.
\]
Note that $E_i M_{\lambda, y} \subset M_{\lambda, y+\alpha_i}$ and $F_i M_{\lambda, y} \subset M_{\lambda, y - \alpha_i}$.
A weight module is a module that decomposes as a direct sum of weight spaces. 
A highest weight module is a module $M$ such that $M= U_q(\g) v_\lambda$ for some $v_\lambda \in M_{\lambda,0}$ with $E_i v_\lambda = 0$ for all $i \in I$. 
In that case, we call $\lambda$ the \emph{highest weight} and we have
\[
M \cong \bigoplus_{y \in X^+} M_{\lambda, -y}.
\]
as $R$-module.

\smallskip

One says that a $U_q(\g)$-module $M$ is \emph{integrable} if for each $v \in M$ there exists $k \gg 0$ such that $E_i^k v = 0$ and $F_i^k v = 0$ for all $i \in I$. Any finite dimensional module is integrable, and any integrable module is a weight module with $\lambda(\Pi^\vee) \subset \bZ[q]$. We consider only type 1 modules, that is $\lambda(\Pi^\vee) \subset \bN[q]$.

\smallskip

Let $M$ be a highest weight module with highest weight vector $v_\lambda \in M_{\lambda,0}$. Then we set $\lambda_i := \lambda(\alpha_i^\vee)$ for each $i \in I$. We are interested in $\lambda$ such that each $\lambda_i$ is either $\lambda_i = q^{n_i}$ for some $n_i \in \bZ$ or $\lambda_i$ is formal. In that case, we write it $\lambda_i = q^{\beta_i}$ where we interpret $\beta_i$ as a formal parameter.

\subsubsection{Parabolic Verma modules}
The (standard) \emph{Borel subalgebra} $U_q(\bo)$ of $U_q(\g)$ is generated by $K_{\gamma}$ and $E_i$ for all $\gamma \in Y^\vee$ and $i \in I$.
A (standard) \emph{parabolic subalgebra} of $U_q(\g)$ is a subalgebra containing $U_q(\bo)$. It is generated by $K_\gamma, E_i$ and $F_j$ for all $\gamma \in Y^\vee, i \in I$ and $j \in I_f$ for some fixed subset $I_f \subset I$. The part given by $K_\gamma, E_j$ and $F_j$ for $j \in I_f$ is called the \emph{Levi factor} and written $U_q(\mathfrak l)$. The \emph{nilpotent radical} $U_q(\mathfrak n)$ is generated by $E_i$ for all $i \in I_r := I \setminus I_f$. 
Note that parabolic subalgebras are in bijection with partitions $I = I_f \sqcup I_r$.

\smallskip

Let $U_q(\p)$ be a parabolic subalgebra determined by $I = I_f \sqcup I_r$. For each $i \in I_f$, we choose a weight $n_i \in \bN$. For each $j \in I_r$ we choose a weight $\lambda_j \in \{q^{\beta_j}, q^{n_j}\}$. We write $N = \{n_i\}_{i \in I_f}$ and $\Lambda = \{\lambda_j\}_{j \in I_r}$. Let $V(\Lambda,N)$ be the unique (type 1) integrable, irreducible representation of $U_q(\mathfrak l)$ on the ground ring $R = \bQ(q, \Lambda)$, and with highest weight $\lambda$ determined by 
\[
\lambda(\alpha^\vee_k) = 
\begin{cases}
q^{n_i}, & \text{if $k = i \in I_f$,} \\
\lambda_j, & \text{if $k = j \in I_r$.}  
\end{cases}
\]
We extend it to a representation of $U_q(\p)$ by setting $U_q(\mathfrak{n}) V(\Lambda, N) = 0$. 

\begin{defn}
The \emph{parabolic Verma module} of highest weight $(\Lambda, N)$ associated to $U_q(\mathfrak p) \subset  U_q(\g)$ is the induced module
\[
M^\p(\Lambda, N) := U_q(\g) \otimes_{U_q(\p)} V(\Lambda, N).
\]
\end{defn}
Whenever $U_q(\p) \subsetneq U_q(\g)$, we have that $M^\p(\Lambda, N)$ is an infinite dimensional module. Moreover, for all parabolic Verma modules, there is a $\bQ(q)$-linear surjection
\[
U_q^-(\g) \otimes_{\bQ(q)} R \twoheadrightarrow M^\p(\Lambda, N).
\]

\begin{exe}
  If $U_q(\p) = U_q(\bo)$, then $N = \emptyset$, and $V(\Lambda, N) \cong \bQ(q, \Lambda)v_{\Lambda}$
  is 1-dimensional, and such that
\begin{align*}
E_i v_\Lambda &= 0,  & 
K_\gamma v_\Lambda = \prod_{j \in I} \lambda_j^{\gamma(\Lambda_j)} v_\Lambda. 
\end{align*}
In this case, we simply call it \emph{Verma module}, and denote it $M^\bo(\Lambda)$. If $\lambda_j = q^\beta$ is formal for all $j \in I_r$, then we call it the \emph{universal Verma module}.
\end{exe}
\begin{exe}
 If $U_q(\p) = U_q(\g)$,  then $\Lambda = \emptyset$ and $M^\p(\Lambda,N) \cong V(N)$ is an integrable, irreducible $U_q(\g)$ representation.
\end{exe}

Since $q$ is a generic parameter we can apply Jantzen's criterion~\cite[Theorem~9.12]{humphreys}, thanks to the results in~\cite{andersenmazorchuk}. We obtain that $M^\p(\Lambda,N)$ is irreducible whenever $\lambda_j \notin \{ q^n | n \in \bN \}$ for all $j \in I_r$. 
If $\lambda_j = q^{n_j}$ for $n_j \in \bN$, then $M^\p(\Lambda, N)$ contains a non-trivial, proper submodule, which is isomorphic to 
$M^\p(\Lambda_{-n_j-2}^{n_j},N)$ for $\Lambda_{-n_j-2}^{n_j}$ given by exchanging $q^{n_j}$ with $q^{-n_j-2}$ in $\Lambda$. Moreover, the quotient
\[
\frac{M^\p(\Lambda, N)}{M^\p(\Lambda_{-n_j-2}^{n_j},N)} \cong M^{\p+j}(\Lambda \setminus \{q^{n_j}\}, N \sqcup \{n_j\}),
\]
is isomorphic to the parabolic Verma module associated to the parabolic subalgebra $\p+j$ given by adding $j$ to $I_f$, that is generated by $\p$ and $F_j$.

\smallskip

Furthermore, whenever $\lambda_j = q^{\beta_j}$ is formal, there is a surjective map
\[
ev_{n_j} : M^\p(\Lambda, N) \twoheadrightarrow M^\p(\Lambda_{\beta_j}^{n_j}, N),
\]
for all $n_j \in \bZ$, given by evaluating $\beta_j = n_j$. 

\smallskip

These two facts together allow us to define a partial order on parabolic Verma modules. For this, we say that there is an arrow from $M^\p(\Lambda,N)$ to $M^{\p'}(\Lambda', N')$ if we have an evaluation map $ev_{n_j}$ such that
\[
ev_{n_j}(M^\p(\Lambda,N)) \cong M^{\p'}(\Lambda', N'),
\]
or if there is a short exact sequence 
\[
0 \rightarrow
M^\p(\Lambda_{-n_j-2}^{n_j}, N) \rightarrow M^\p(\Lambda, N) \rightarrow M^{\p'}(\Lambda', N')
\rightarrow 0.
\]
For parabolic Verma modules $M$ and $M'$ we say that $M$ is bigger than $M'$ if there is a chain of arrows from $M$ to $M'$. In that case, there is an $M''$, which is either trivial or a parabolic Verma module, and a short exact sequence
\[
0 \rightarrow M'' \rightarrow ev(M) \rightarrow M' \rightarrow 0,
\]
where $ev$ is a composition of evaluation maps $ev_{n_j}$.
With this partial order, the universal Verma module is a maximal element and each integrable, irreducible module is a minimum. This also means that we can recover any parabolic Verma module from the universal one. 

\subsubsection{The Shapovalov form}

Let $\rho : U_q(\g) \rightarrow \opalg{U_q(\g)}$ be the $\bQ(q)$-linear algebra anti-involution given by
\begin{align}\label{eq:rhog}
\rho(E_i) &:= q_i^{-1} K_i^{-1} F_i, & \rho(F_i) &:= q_i^{-1}K_i E_i, & \rho(K_\gamma) :=  K_\gamma,
\end{align}
for all $i \in I$ and $\gamma \in Y^\vee$.

\begin{defn}\label{def:shap}
The \emph{Shapovalov form}
\[
(-,-) : M^\p(\Lambda, N) \times M^\p(\Lambda,N) \rightarrow \bQ(q, \Lambda),
\]
is the unique bilinear form respecting
\begin{itemize}
\item $(v_{\Lambda,N}, v_{\Lambda,N}) = 1$, for $v_{\Lambda,N}$ the highest weight vector;
\item $(uv,v') = (v,\rho(u) v')$ where $\rho$ is defined in~\eqref{eq:rhog}; 
\item $f(v,v') = (fv,v') = (v,fv')$,
\end{itemize}
 for all $v,v' \in M^\p(\Lambda, N), u \in U_q(\g)$ and $f\in \bQ(q, \Lambda)$.
\end{defn}

\subsubsection{Basis}

Since parabolic Verma modules are highest weight modules, they admit at least one basis given in terms of elements of the form $F_\bi v_{{\Lambda,N}}$ for $\bi \in \Seq(I)$, where $v_{\Lambda,N}$ is a highest weight vector. In particular, as $R$-modules they are all submodules of $U_q^-(\g) \otimes_{\bQ(q)} R$, meaning that these basis lives in a subset of $\{F_\bi v_{\Lambda,N} | \bi \in \Seq(I)\}$ modded out by the Serre relations.
 We call such a basis an \emph{induced basis} and write it $\{v_{\Lambda,N} = m_0, m_1, \dots \}$. 
Any element in such basis takes the form $F_\bi = F_{i_r}^{b_r} \cdots F_{i_1}^{b_1}$ for some $i_1, \dots, i_r \in I$ and $b_1, \dots, b_r \in \bN$, with $i_\ell \neq i_{\ell+1}$. 
Replacing each $F_i^b$ by the \emph{divided power} $F_i^{(b)} := F_i^b/([b]_{q_i} !)$ yields another basis $\{v_{\Lambda,N} = m_0', m_1', \dots \}$. 
Lusztig's \emph{canonical basis} \cite{lusztig} is given by a certain choice of such a divided power basis characterized by
\[
 (m_i', m_i') - 1 \in  \bZ^+_{\prec}\llbracket q, \Lambda \rrbracket,
\]
for any order such that $0 \prec q \prec \lambda_i$ (see \cref{sec:catthm} for a definition of $ \bZ^+_{\prec}\llbracket q, \Lambda \rrbracket$). 
Whenever $M^\p(\Lambda,N)$ is irreducible, the Shapovalov form is non-degenerate. 
Therefore, in this case, there is a \emph{dual canonical basis} uniquely determined by  
\[
(m_i', m^j) = \delta_{ij}.
\]

%%%%%%%%%%%%%%%%	End of file	%%%%%%%%%%%%%

%% file: sections/dgKLR.tex
%%%%%%%%%%%%%%%%%%%%%%%%%%%%%%%%%%%%
%                 					  				  		 %
%	Dg-enhanced KLR algebras		 					 %
%                 					  						 %
%%%%%%%%%%%%%%%%%%%%%%%%%%%%%%%%%%%%

\section{The $\bo$-KLR algebras}\label{sec:dgKLR}

Fix once and for all a Cartan datum $\left(I, X, Y, (-|-)\right)$, and let 
\[
d_{ij} := -\alpha_i^\vee(\alpha_j) \in \bN.
\] 
For $\nu \in X^+$ we write
\[
\nu = \sum_{i \in I} \nu_i \cdot \alpha_i, \quad \nu_i \in \bN,
\]
and we set $|\nu| := \sum_{i} \nu_i$, and $\Supp(\nu) := \{ i | \nu_i \neq 0\}$. 

\smallskip

We also fix a choice of scalars in a commutative, unital ring $\Bbbk$ as introduced in~\cite{rouquierquiv}. Following the conventions in~\cite{laudaimplicit}, it consists of:
\begin{itemize}
\item $t_{ij} \in \Bbbk^\times$ for all $i,j \in I$;
\item $s_{ij}^{tv} \in \Bbbk$ for $i \ne j$, $0 \le t < d_{ij}$ and $0 \le v <d_{ji}$;
\item $r_i \in \Bbbk^\times$ for all $i \in I$,
\end{itemize}
respecting
\begin{itemize}
\item $t_{ii} = 1$;
\item $t_{ij} = t_{ji}$ whenever $d_{ij} = 0$;
\item $s_{ij}^{tv} = s_{ji}^{vt}$;
\item $s_{ij}^{tv} = 0$ whenever $t(\alpha_i | \alpha_i) + v (\alpha_j | \alpha_j) \neq -2  (\alpha_i | \alpha_j)$.
\end{itemize}
In addition, whenever $t < 0$ or $v < 0$,  we put $s_{ij}^{tv} := 0$.
Thus we have $s_{ij}^{pq} = 0$ for $p > d_{ij}$ or $q > d_{ji}$. 
We will also write $s_{ij}^{d_{ij}0} := t_{ij}$ and $s_{ij}^{0d_{ji}} := t_{ji}$. Hence if $(\alpha_i|\alpha_j) = 0$ we get $s_{ij}^{00} = s_{ji}^{00} = t_{ij} = t_{ji}$.

\begin{defn}[{\cite{KL1,rouquier}}]\label{def:KLRalgebra}
 For $m \in \bN$, the \emph{Khovanov--Lauda--Rouquier (KLR) algebra} $R(m)$ is the $\Bbbk$-algebra generated by braid-like diagrams on $m$ strands, read from bottom to top, such that
  \begin{itemize}
  \item two strands can intersect transversally, but no triple intersections are allowed;
  \item strands can be decorated by dots (we use a dot with a label $k$ to denote $k$ consecutive dots on a strand);
  \item each strand is labeled by a simple root, written $i \in I$, that we (usually) write at the bottom;
  \item multiplication is given by concatenation of diagrams, which preserves the labeling (i.e. connecting two strands with different labels gives zero);
  \item diagrams are taken modulo planar isotopies and the following local relations:
  \end{itemize}
  \begin{align}\label{eq:KLRR2}
\tikzdiagh{0}{
	      	\draw  (0,-.75) node[below] {\small $i$} .. controls (0,-.375) and (1,-.375) .. (1,0) .. controls (1,.375) and (0, .375) .. (0,.75);
 	 	\draw[myblue]  (1,-.75) node[below] {\small $j$} .. controls (1,-.375) and (0,-.375) .. (0,0) .. controls (0,.375) and (1, .375) .. (1,.75);
}
\ = \ 
\begin{cases}
	\hfil 0 &\text{ if } i = j, \\ \\
	\sum\limits_{t,v} s_{ij}^{tv} \  
	\tikzdiagh{0}{
      		\draw  (0,-.5) node[below] {\small $i$} -- (0,.5)node [midway,tikzdot]{} node[midway,xshift=1.5ex,yshift=.75ex] {\small $t$};
		\draw[myblue]  (1,-.5) node[below] {\small $j$} -- (1,.5)node [midway,tikzdot]{} node[midway,xshift=1.5ex,yshift=.75ex] {\small $v$};
	} &\text{ if } i \ne j,\\
\end{cases} 
\end{align}
for all $i,j \in I$,
\begin{align}\label{eq:KLRdotslide}
	\tikzdiagh{0}{
	          \draw (0,-.5) node[below] {\small $i$} .. controls (0,0) and (1,0) .. (1,.5);
	          \draw[myblue] (1,-.5) node[below] {\small $j$} .. controls (1,0) and (0,0) .. (0,.5)  node [near end,tikzdot]{};
	} 
	&\  = \ 
	\tikzdiagh{0}{
	          \draw (0,-.5) node[below] {\small $i$} .. controls (0,0) and (1,0) ..  (1,.5);
	          \draw[myblue] (1,-.5) node[below] {\small $j$}  .. controls (1,0) and (0,0) ..  (0,.5) node [near start,tikzdot]{};
	} 
&
	\tikzdiagh{0}{
	          \draw (0,-.5) node[below] {\small $i$} .. controls (0,0) and (1,0) ..  (1,.5) node [near start,tikzdot]{};
	          \draw[myblue] (1,-.5) node[below] {\small $j$} .. controls (1,0) and (0,0) ..  (0,.5);
	} 
	&\  = \ 
	\tikzdiagh{0}{
	          \draw (0,-.5) node[below] {\small $i$} .. controls (0,0) and (1,0) ..  (1,.5)node [near end,tikzdot]{};
	          \draw[myblue] (1,-.5) node[below] {\small $j$} .. controls (1,0) and (0,0) ..  (0,.5);
	} 
	% & \text{ if }i&\ne j,
%\end{align}
%\begin{align}
\\
\label{eq:KLRnh}
%\begin{aligned}
	\tikzdiagh{0}{
	          \draw (0,-.5) node[below] {\small $i$} .. controls (0,0) and (1,0) ..  (1,.5);
	          \draw (1,-.5) node[below] {\small $i$} .. controls (1,0) and (0,0) ..  (0,.5)  node [near end,tikzdot]{};
	}  
	&\  = \ 
	\tikzdiagh{0}{
	          \draw (0,-.5) node[below] {\small $i$} .. controls (0,0) and (1,0) ..  (1,.5); 
	          \draw (1,-.5) node[below] {\small $i$} .. controls (1,0) and (0,0) ..  (0,.5) node [near start,tikzdot]{};
	} \  + \  r_i
	\tikzdiagh{0}{
	          \draw (0,-.5) node[below] {\small $i$} -- (0,.5);
	          \draw (1,-.5) node[below] {\small $i$} -- (1,.5);
	} ,
&
	\tikzdiagh{0}{
	          \draw (0,-.5) node[below] {\small $i$} .. controls (0,0) and (1,0) ..  (1,.5) node [near start,tikzdot]{};
	          \draw (1,-.5) node[below] {\small $i$}  .. controls (1,0) and (0,0) ..  (0,.5);
	} 
	&\  = \ 
	\tikzdiagh{0}{
	          \draw (0,-.5) node[below] {\small $i$} .. controls (0,0) and (1,0) ..  (1,.5)node [near end,tikzdot]{};
	          \draw (1,-.5) node[below] {\small $i$}  .. controls (1,0) and (0,0) ..  (0,.5);
	} \  + \  r_i
	\tikzdiagh{0}{
	          \draw (0,-.5) node[below] {\small $i$} -- (0,.5);
	          \draw (1,-.5) node[below] {\small $i$} -- (1,.5);
	} 
%\end{aligned}
\end{align}
for all $i \neq j \in I$, 
\begin{align}
 \label{eq:KLRR3} 
	\tikzdiagh[scale=.75]{0}{
		\draw  (0,0)node[below] {\small $i$} .. controls (0,0.5) and (2, 1) ..  (2,2);
		\draw[mygreen]  (2,0)node[below] {\small $k$} .. controls (2,1) and (0, 1.5) ..  (0,2);
		\draw[myblue]  (1,0)node[below] {\small $j$} .. controls (1,0.5) and (0, 0.5) ..  (0,1) .. controls (0,1.5) and (1, 1.5) ..  (1,2);
	 }  \  - \  
	\tikzdiagh[scale=.75]{0}{
		\draw  (0,0)node[below] {\small $i$} .. controls (0,1) and (2, 1.5) ..  (2,2);
		\draw[mygreen]  (2,0)node[below] {\small $k$} .. controls (2,.5) and (0, 1) ..  (0,2);
		\draw[myblue]  (1,0)node[below] {\small $j$} .. controls (1,0.5) and (2, 0.5) ..  (2,1) .. controls (2,1.5) and (1, 1.5) ..  (1,2);
	 }
\  &= 
\begin{cases}
\hfil 0 &\text{ if } i \neq k, \\ \\
r_i \sum\limits_{t,v} s_{ij}^{tv} \sssum{u+\ell=\\t-1} \ 
\tikzdiagh[scale=.75]{0}{
	     	 \draw[myblue]  (1,-1) node[below] {\small $j$}  --(1,1)node [midway,tikzdot]{} node[midway,xshift=1.5ex,yshift=.75ex]{\small $v$}; 
	      	\draw  (0,-1) node[below] {\small $i$} -- (0,1)node [midway,tikzdot]{} node[midway,xshift=1.5ex,yshift=.75ex]{\small $u$}; 
	      	\draw  (2,-1) node[below] {\small $i$} -- (2,1)node [midway,tikzdot]{} node[midway,xshift=1.5ex,yshift=.75ex]{\small $\ell$}; 	
	 }
\quad& \text{otherwise,} 
\end{cases}
\end{align}
for all $i, j, k \in I$. 
In addition, $R(m)$ is $\bZ$-graded by setting 
\begin{align*}
\deg_q \left(
\ 
\tikzdiag{
	\draw (0,0) node[below] {\small $i$}  ..controls (0,.5) and (1,.5) .. (1,1);
	\draw[myblue] (1,0) node[below] {\small $j$}  ..controls (1,.5) and (0,.5) .. (0,1);
}
\ 
\right)
&:= -(\alpha_i | \alpha_j),
&
\deg_q \left(
\ 
\tikzdiag{
	\draw (0,0) node[below] {\small $i$}  -- (0,1) node [midway,tikzdot]{};
}
\ 
\right)
& := (\alpha_i | \alpha_i).
\end{align*}
\end{defn}

\begin{rem}
Note that in \cref{eq:KLRR2} and \cref{eq:KLRR3}, the sum $\sum\limits_{t,v} s_{ij}^{tv}$ can be restricted to the finite number of pairs $t,v \in \bN$ such that $t(\alpha_i | \alpha_i) + v (\alpha_j | \alpha_j) = -2 (\alpha_i | \alpha_j)$. Moreover, it contains at least two non-zero elements with invertible coefficients, given by $t=d_{ij}, v=0$ and $t=0, v = d_{ji}$.
\end{rem}

As proven in~\cite{KL1,KL2} (see also~\cite{rouquier}), these algebras categorify the half quantum group $U_q^-(\g)$ associated to $\left(I, X, Y, (-|-)\right)$, as a (twisted) bialgebra. The multiplication and comultiplication are categorified using respectively induction and restriction functors, obtained by putting diagrams side by side.

\smallskip

For each non-negative integral highest weight $N := \{n_i \in \bN | i \in I\}$, there is a $N$-cyclotomic quotient $R^N(m)$ of $R(m)$ given by modding out the two-sided ideal generated by all diagrams of the form
\[
\tikzdiagh{0}{
	     	 \draw  (0,-.5) node[below] {\small $i$}  --(0,.5)node [midway,tikzdot]{} node[midway,xshift=1.75ex,yshift=.75ex]{\small $n_i$}; 
	      	\draw  (1,-.5) node[below] {\small $j$} -- (1,.5);
	      	\node at(2,0) {\small $\dots$};
	      	\draw  (3,-.5) node[below] {\small $k$} -- (3,.5);
	 }
	\ =\  0.
\]
As first conjectured in~\cite{KL1} and proven in~\cite{kashiwara} and independently in~\cite{webster}, these cyclotomic quotients categorify the irreducible integrable $U_q(\g)$-module of highest weight $N$, where the action of $F_i$ (resp. $E_i$) is given by induction (resp. restriction) along the map $R(m) \hookrightarrow R(m+1)$ that adds a vertical strand with label $i$, at the right.

\subsection{$\bo$-KLR algebra}

Our first goal is to construct a dg-enhancement of the cyclotomic KLR algebras $R^N(m)$, in the same spirit as in \cite{naissevaz2}. We introduce the following algebra:

\begin{defn}\label{def:Rbo}
For $m \in \bN$, the \emph{$\bo$-KLR algebra} $R_\bo(m)$ is the $\Bbbk$-algebra generated by braid-like diagrams on $m$ strands, read from bottom to top, such that
\begin{itemize}
\item two strands can intersect transversally, but no triple intersections are allowed;
 \item strands can be decorated by dots;
  \item regions in-between strands can be decorated by \emph{floating dots}, which are labeled by a subscript in $I$ and a superscript in $\bN$;
  \item each strand is labeled by a simple root, written $i \in I$;
  \item multiplication is given by concatenation of diagrams, which preserves the labeling;
  \item diagrams are taken modulo planar isotopies  that preserve the relative height of the floating dots, and modulo the KLR relations Eq. (\ref{eq:KLRR2} -- \ref{eq:KLRR3}) and the following local relations:
  \end{itemize}
\begin{align}\label{eq:fdmoves}
	\tikzdiag {
      	 	   \node at (.5,0){$\cdots$};
		\fdot[a]{i}{-.5,.25};
		\fdot[b]{j}{1.25,-.25};
  	}
	&\ =\  -\quad
	\tikzdiag {
      	 	   \node at (.5,0){$\cdots$};
		\fdot[a]{i}{-.5,-.25};
		\fdot[b]{j}{1.25,.25};
  	} 
&
	\tikzdiag {
		\fdot[a]{i}{-.5,.5};
		\fdot[a]{i}{-.5,-.5};
	} \ = \ 0,
\end{align} 
meaning floating dots anti-commute with each other for all $i,j \in I$ and $a,b \in \bN$, 
\begin{align}
\label{eq:fdots}
	\tikzdiagh{0}{
	          \draw[myblue] (0,-.5) node[below] {\small $i$} -- (0,.5);
		  \fdot[a]{j}{.5,0};
	} 
	& \ = \ 
\begin{cases}
	\tikzdiagh{0}{
	          \draw (0,-.5) node[below] {\small $i$} -- (0,.5);
		 \fdot[a-1]{i}{-1.1,0};
	}
	\quad  - \quad
	\tikzdiagh{0}{
	          \draw (0,-.5) node[below] {\small $i$} -- (0,.5) node [midway,tikzdot]{};
		 \fdot[a-1]{i}{.5,0};
	}
&\text{ if $i=j$ and $a > 0$, } \\
	\sum\limits_{t,v} (-1)^v
	s_{ij}^{tv}\quad 
	\tikzdiagh{0}{
	          \draw[myblue] (0,-.5) node[below] {\small $i$} -- (0,.5) node [midway,tikzdot]{} node [midway, xshift=1.5ex, yshift=.75ex] {\small $t$};
		 \fdot[a+v]{j}{-1.25,0};
	}  
	 & \text{ if $ i \neq j$, }
\end{cases}
\end{align}
\begin{align}\label{eq:ExtR2}
	\tikzdiagh[scale=1]{0}{
	      	\draw[myblue]  (0,-.75) node[below] {\small $i$} .. controls (0,-.5) and (1,-.5) .. (1,0) .. controls (1,.5) and (0, .5) .. (0,.75);
 	 	\draw  (1,-.75) node[below] {\small $j$} .. controls (1,-.5) and (0,-.5) .. (0,0) .. controls (0,.5) and (1, .5) .. (1,.75);
		\fdot[a]{j}{.4,.05};
	} 
	\  &= \  
	 \tikzdiagh{0}{
	      \draw[myblue]  (0,-.75) node[below] {\small $i$} -- (0,.75);
	      \draw  (1,-.75) node[below] {\small $j$} -- (1,.75);
		\fdot[a]{j}{1.5,0}
	}
	 \  +  \sum_{t,v} s_{ij}^{tv} \sum_{\substack{u+\ell=\\v-1}} (-1)^u\ 
	\tikzdiagh{0}{
	      \draw[myblue]  (0,-.75) node[below] {\small $i$} -- (0,.75)  node [midway,tikzdot]{} node[midway,xshift=1.5ex,yshift=.75ex] {\small $t$};
	      \draw  (1,-.75) node[below] {\small $j$} -- (1,.75)  node [midway,tikzdot]{} node[midway,xshift=1.5ex,yshift=.75ex] {\small $\ell$};	
		\fdot[a+u]{j}{-1.2,0};
	}& 
	 \quad \text{if } i \ne j,
\end{align}
Moreover, a floating dot in the left-most region is zero
\begin{align*}
\tikzdiagh[xscale=.75]{0}{
	\fdot[a]{i}{-0.5,0.5};
	\draw (0,0) node[below] {\plusspacing \small $j$} -- (0,1);
	\draw (1,0) node[below] {\plusspacing \small $k$} -- (1,1);
	\node at(2,.5) {$\dots$};
	\draw (3,0) node[below] {\plusspacing \small $\ell$} -- (3,1);
}
\ = \ 
0.
\end{align*}
Given a diagram, it is sometimes useful to decorate some of its regions with an element $K := \sum_{i \in I} k_i \cdot  \alpha_i \in X^+$, where $k_i$ denotes the number of strands with label $i$ to the left of the region.
The algebra $R_\bo$ is $\bZ^{1+|I|}$-graded (a $q$-grading and a $\lambda_k$-grading for each $k \in I$) with
\begin{align*}
\deg_q \left(
\ 
\tikzdiag{
	\draw (0,0) node[below] {\small $i$}  ..controls (0,.5) and (1,.5) .. (1,1);
	\draw[myblue] (1,0) node[below] {\small $j$}  ..controls (1,.5) and (0,.5) .. (0,1);
}
\ 
\right)
&:= -(\alpha_i | \alpha_j),
&
\deg_q \left(
\ 
\tikzdiag{
	\draw (0,0) node[below] {\small $i$}  -- (0,1) node [midway,tikzdot]{};
}
\ 
\right)
& := (\alpha_i | \alpha_i), \\
\deg_{\lambda_k} \left(
\ 
\tikzdiag{
	\draw (0,0) node[below] {\small $i$}  ..controls (0,.5) and (1,.5) .. (1,1);
	\draw[myblue] (1,0) node[below] {\small $j$}  ..controls (1,.5) and (0,.5) .. (0,1);
}
\ 
\right)
&:= 0,
&
\deg_{\lambda_k} \left(
\ 
\tikzdiag{
	\draw (0,0) node[below] {\small $i$}  -- (0,1) node [midway,tikzdot]{};
}
\ 
\right)
& := 0,
\end{align*}
and
\begin{align*}
\deg_q \left(
\ 
\tikzdiag{
	\fdot[a]{i}{0,0};
	\node at(-.5,-.5) {\small $K$};
}
\ 
\right)
&:= (1+a-\alpha_i^\vee(K)+k_i) (\alpha_i | \alpha_i),
\\
\deg_{\lambda_k} \left(
\ 
\tikzdiag{
	\fdot[a]{i}{0,0};
	\node at(-.5,-.5) {\small $K$};
}
\ 
\right)
&:= 2\delta_{ik}.
\end{align*}
This ends the definition of $R_\bo(m)$.
\end{defn}

\subsection{Tightened basis}\label{sec:bKLRbasis}

Before going any further, let us introduce some useful notations borrowed from~\cite{KL1}. 
First, let $R_\bo(\nu)$ be the subalgebra of $R_\bo(m)$ given by diagrams where there are exactly $\nu_i$ strands labeled $i$, for each $i \in I$.
We also denote $\Seq(\nu)$ the set of all ordered sequences $\bi = i_1i_2 \cdots i_m$ with $i_k \in I$ and $i$ appearing $\nu_i$ times in the sequence.
The symmetric group $S_m$ acts on $\Seq(\nu)$ with the simple transposition $\sigma_k \in S_m$ acting on $\bi = i_1i_2\cdots i_m \in \Seq(\nu)$ by permuting $i_k$ and $i_{k+1}$. Sometimes, for $K = \sum_{i \in I} k_i \cdot  \alpha_i \in X^+$, we abuse notation by writing $\sigma_K$ instead of $\sigma_{|K|}$. 

\smallskip

For  $\bi = i_1i_2\cdots i_m \in \Seq(\nu)$, let $1_\bi \in R_\bo(\nu)$ be the idempotent given by $m$ vertical strands with labels $i_1, i_2, \dots, i_m$, that is
\begin{align*}%\label{eq:idemei}
1_\bi &:= 
        \tikzdiagh[xscale=1.5]{0}{
          \draw (-.5,-.5) node[below] {\small $i_1$} -- (-.5,.5); 
          \draw (0,-.5) node[below] {\small $i_2$} -- (0,.5); 
          \draw (1.5,-.5) node[below] {\small $i_m$} -- (1.5,.5); 
          \node at (.75,0){$\cdots$};
        }
\end{align*}
We have $1_{\bi}1_{\bj} = \delta_{\bi\bj}$ for all $\bi,\bj \in \Seq(\nu)$, and so there is a decomposition of $\Bbbk$-modules 
\[
 R_\bo(\nu) \cong \bigoplus_{\bi,\bj \in \Seq(\nu)} 1_\bj  R_\bo(\nu) 1_\bi .
\]
Our goal is to construct a basis of $1_\bj  R_\bo(\nu) 1_\bi$ as $\Bbbk$-module.

\subsubsection{An action of $R_\bo(\nu)$ on a polynomial space}\label{sec:KLRpolaction}

We construct a polynomial representation of $R_\bo(\nu)$ with a similar flavor as in~\cite[\S2.3]{KL1}.
 We fix $\nu  \in X^+$ with $|\nu| = m$. For each $i \in I$ we define
\[
Q_i := 
 \Bbbk[x_{1,i}, \dots, x_{\nu_i, i}] \otimes \bV^\bullet \brak{\omega_{1,i}, \dots, \omega_{\nu_i,i}}.
\]
We write $Q_I := \bigotimes_{i \in I} Q_i$, where $\otimes$ means the supertensor product in the sense that $\omega_{\ell,i}\omega_{\ell',j} = -\omega_{\ell',j}\omega_{\ell,i}$  for all $i,j \in I$ and $x_{i,\ell}$ commutes with everything. 
Thus, $Q_I$ is a supercommutative superring. 
Then, we construct the ring
\[
Q_\nu := \bigoplus_{\bi \in \Seq(\nu)} Q_I 1_\bi,
\]
where the elements $1_\bi$ are central idempotents. It is $\bZ^{1+|I|}$-graded by setting 
\begin{align*}
\deg_q(x_{\ell,i}) &=  (\alpha_i | \alpha_i), & \deg_q(\omega_{\ell,i}) &= (1-\ell) (\alpha_i | \alpha_i), \\
\deg_{\lambda_j}(x_{\ell,i})  &= 0, & \deg_{\lambda_j}(\omega_{\ell,i}) &= 2\delta_{ij}.
\end{align*}

%\smallskip

We first construct an action of the symmetric group $S_m$ on $Q_\nu$ by letting the simple transposition
\[
\sigma_k : Q_I 1_\bi \rightarrow Q_I 1_{\sigma_k\bi},
\]
to act by sending 
\begin{align*}
x_{p,i} 1_\bi &\mapsto 
\begin{cases}
 x_{p+1,i} 1_{\sigma_k \bi}, &\text{ if $i_k = i_{k+1} = i$ and $p = \# \{ s \le k | i_s = i\}$, } \\[1ex]
 x_{p-1,i} 1_{\sigma_k \bi}, &\text{ if $i_k = i_{k+1} = i$ and $p = 1+ \# \{ s \le k | i_s = i\}$, } \\[1ex]
 x_{p,i} 1_{\sigma_k \bi}, &\text{ otherwise,} 
\end{cases}
\intertext{for $i \in I, p \in \{1, \dots, \nu_i\}$ and $\bi = i_1 \dots i_m$, and by sending}
\omega_{p,i} 1_\bi &\mapsto  
\begin{cases}
\left( \omega_{p,i} + (x_{p,i} - x_{p+1,i}) \omega_{p+1,i} \right)   1_{\sigma_k \bi}, 
& \text{if $i_k = i_{k+1} = i$ and 
$p = \# \{ s \le k | i_s = i\}$},  \\[1ex]
\omega_{p,i} 1_{\sigma_k\bi}, &\text{ otherwise,} 
\end{cases}
\end{align*} 
which we extend to $Q_\nu$ by setting $\sigma_k(fg) := \sigma_k(f)\sigma_k(g)$ for all $f,g \in Q_\nu$.

\begin{prop}
The procedure described above yields a well-defined action of $S_m$ on $Q_\nu$.
\end{prop}

\begin{proof}
The proof is a straightfoward computation. 
We leave the details to the reader.
\end{proof}

Then, we define inductively the element $\omega_{p,j}^a \in Q_I$ for $a \in \bN$ as
\begin{align*}
\omega_{p,j}^0 &:= \omega_{p,j}, & \omega_{p,j}^{a+1} := \omega_{p-1,j}^a - x_{p,j} \omega_{p,j}^{a}.
\end{align*}
For $K = \sum_{i\in I} k_i \cdot i \in X^+$ such that $k_i \leq \nu_i$, we define $\omega_j^a(K) \in Q_I$ inductively as
\[
\omega_j^a(K) := 
\begin{cases}
0, & \text{if $k_j$ = 0,} \\
\omega_{k_j,j}^a, & \text{if $k_i = 0$ for all $i \neq j$,} \\
\sum\limits_{t,v} (-1)^t s_{ij}^{tv} x_{k_i, i}^t \omega_j^{a+v} (K- i), &\text{otherwise,}
\end{cases}
\]
where $K-i$ is a shorthand for $K - 1 \cdot \alpha_i$. 

\begin{lem}
The element $\omega_j^a(K)$ is well-defined.
\end{lem}

\begin{proof}
Take $i \neq i' \neq j \in I$ such that $k_i > 0$ and $k_{i'} > 0$. We can suppose by induction that $\omega_j^b(K-i-i')$ is well-defined for all $ b\geq 0$.
Then we have
\begin{align*}
\sum_{t,v} (-1)^t s_{ij}^{tv} x_{k_i,i}^t 
\sum_{t',v'} (-1)^{t'} s_{i'j}^{t'v'} x_{k_{i'}, i'}^{t'}
\omega_j^{a+v}(K-i-i') \\
=
\sum_{t',v'} (-1)^{t'} s_{i'j}^{t'v'} x_{k_{i'}, i'}^{t'} 
\sum_{t,v} (-1)^t s_{ij}^{tv} x_{k_i,i} \omega_j^{a+v}(K-i'-i),
\end{align*}
for all $i \neq i' \neq j \in I$.
\end{proof}

It will be useful to give $\omega_j^a(K)$ a non-inductive expression. We write $K^{\setminus j} := \sum_{i \neq j} k_i \cdot \alpha_i$. For a given non-negative integer $n_i \in \bN$ we define
\begin{equation}\label{eq:epsnotation}
\varepsilon_{n_i,i}^j(\und x_{k_i,i}) := \sum_{|V_i| = n_i}
\left(\prod_{\ell=1}^{k_i} s_{ji}^{v_{\ell}t_{\ell}} x_{\ell,i}^{t_{\ell}}\right) \in P_i,
\end{equation}
with the sum being over all partitions $V_i :  v_{1} + \dots + v_{k_i} = v_i$ such that $(\alpha_i | \alpha_i) | v_{\ell}(\alpha_j | \alpha_j)$ for each~$\ell \in \{1,\dots, k_i\}$,  and with $t_{\ell} := \frac{-2(\alpha_i | \alpha_j) - v_{\ell} (\alpha_j | \alpha_j)}{(\alpha_i | \alpha_i)}$. 
This is a symmetric polynomial of $q$-degree $-2k_i(\alpha_i | \alpha_j) - v_i(\alpha_j | \alpha_j) $ whenever it is non zero.
 Clearly, we can suppose  $v_{\ell} \le d_{ji}$, and
therefore we can also suppose that $n_i \le d_{ji}k_i$. 
For $n \in \bN$ we define
\begin{equation}\label{eq:epsnotation2}
 \varepsilon_v^j(\und x_K) := \sum_{|V| = n} \left(\prod_{i \ne j} \varepsilon_{n_i,i}^j(\und x_{k_i,i}) \right) \in P_I,
\end{equation}
with the sum being over all partitions $V : \sum_{i \ne j} n_i = n$.
Notice that $\varepsilon_{v}^j(\und x_K)$ is a polynomial of $q$-degree $(- \alpha_j^\vee(K^{\setminus j})-n)(\alpha_j | \alpha_j)$.

\begin{lem}
We have
\begin{equation} \label{eq:fdaction}
\omega^a_j(K) = \sum_{n=0}^{-\alpha_j^\vee(K^{\setminus j})} (-1)^n \omega_{k_j,j}^{a + n} \varepsilon_{n}^j(\und x_K) \in P_I.
\end{equation}
\end{lem}

\begin{proof}
A straightforward computation shows that the RHS of \cref{eq:fdaction} respects the recursive definition of $\omega^a_j(K)$, which proves the equality. 
\end{proof}

We now have all the tools we need to define an action of $R_\p(\nu)$ on $P_\nu$. First, we choose an arbitrary orientation $i \leftarrow j$ or $i \rightarrow j$ for each pair of distinct $i,j \in I$. Then, we let $a \in R_\p(\nu) 1_\bj$ act as zero on $P_I 1_\bi$ whenever $\bj \neq \bi$. Otherwise, we declare that
\begin{itemize}
\item the dot 
\begin{align*}
\tikzdiag{
	          \draw (0,-.5) node[below]{$i$}-- (0,.5)  node [midway,tikzdot]{};
	    	  \node at (-0.5,0){$K$};
  	}
\end{align*}
acts as multiplication by $x_{k_i+1,i}1_\bi$;
\item the floating dot 
\begin{align*}
\tikzdiag{
	\fdot[a]{j}{0,0};
	\node at(-.5,-.5) {\small $K$};
  	}
\end{align*}
acts as multiplication by $\omega_j^a(K)1_\bi$;
\item the crossing 
\begin{align*}
	\tikzdiag{
	          \draw   (-.5,-.5) node[below]{$i$}  .. controls (-.5,0) and (.5,0) .. (.5,.5);
	          \draw   (.5,-.5) node[below]{$j$}  .. controls (.5,0) and (-.5,0) .. (-.5,.5);
	    	  \node at (0,-.4){$K$};
  	}
\end{align*}
acts as 
\begin{align*}
f 1_\bi &\mapsto  r_i  \frac{f1_\bi - \sigma_{K}(f1_\bi)}{x_{k_i,i} - x_{k_i+1,i}},  && \text{if $i = j$},\\
f 1_\bi &\mapsto 
\left(\sum_{t,v} s_{ij}^{tv} x_{k_i,i}^{t} x_{k_j+1,j}^{v}\right)
\sigma_{K} (f  1_{\bi}), && \text{if $i \rightarrow j$},\\
f 1_\bi &\mapsto \sigma_{K} (f  1_{\bi}), && \text{if $i \leftarrow j$}.
\end{align*}
\end{itemize}

\begin{prop}\label{prop:Rboaction}
The rules above define an action of $R_\bo(\nu)$ on $Q_\nu$.
\end{prop}

\begin{proof}
We have to check the validity of the KLR relations Eq. (\ref{eq:KLRR2} -- \ref{eq:KLRR3}) and of the relations involving floating dots Eq. (\ref{eq:fdmoves} -- \ref{eq:ExtR2}), as well as the relations coming from regular isotopies.

We start by proving the KLR relations. Clearly \cref{eq:KLRR2}, \cref{eq:KLRdotslide} and \cref{eq:KLRnh} are satisfied. The case $i \neq k$ of~\cref{eq:KLRR3} is also straightforward.  For $i \leftarrow j$ and $k = i$ we compute the action of the LHS of~\cref{eq:KLRR3} on $f \in Q_\nu$ as
\begin{align*}
f &\mapsto  \bigl(\sum_{t,v} s_{ji}^{tv} y^t x_1^v\bigr) \sigma_1 \partial_2 \sigma_1 (f) - \sigma_2 \partial_1 \bigl(\sum_{t,v} s_{ji}^{tv} y^t x_2^v\sigma_2(f) \bigr) \\
&=   \bigl(\sum_{t,v} s_{ji}^{tv} y^t x_1^v\bigr) \frac{f - \sigma_1\sigma_2\sigma_1(f)}{x_1 - x_2} - \frac{\bigl(\sum_{t,v} s_{ji}^{tv} y^t x_2^v\bigr)f - \bigl(\sum_{t,v} s_{ji}^{tv} y^t x_1^v\bigr) \sigma_2\sigma_1\sigma_2(f)}{x_1  -  x_2} \\
&=  \sum_{t,v} s_{ji}^{tv} y^t \frac{ x_1^v f -  x_2^v}{x_1 - x_2} = \sum_{t,v} s_{ji}^{tv} y^t \sssum{r+s=\\v-1} x_1^r x_2^s,
\end{align*}
where $x_1,x_2$ correspond with the $x_{k_i,i}, x_{k_i+1,i}$ and  $y$ with $x_{k_j,j}$. What remains coincides with the RHS of~\cref{eq:KLRR3}. A similar computation applies for the case $i \rightarrow j$.

For the relations involving floating dots, we remark that~\cref{eq:fdmoves} follows from the supercommutativity of $Q_\nu$, and $\omega_j^a(K)$ respects~\cref{eq:fdots} by construction. For relation \cref{eq:ExtR2},  we apply the action of the LHS on some $f \in Q_\nu$ and we obtain
\begin{align*}
f &\mapsto  \bigl(\sum_{t,v} s_{ji}^{tv} y^t x^v\bigr) \omega_j^a(K+ j) f,
\end{align*}
and for the RHS we obtain 
\begin{align*}
f &\mapsto \bigl(\omega_j^a(K+i+j) + \sum_{t,v} s_{ij}^{tv} \sssum{u+\ell=\\v-1} (-1)^u \omega_j^{a+u}(K) x^t y^\ell\bigr)f \\
&= \bigl( \sum_{t,v}(-1)^{v} s_{ij}^{tv}  x^t \omega_j^{a+v}(K+j) + \sum_{t,v} s_{ij}^{tv} \sssum{u+\ell=\\v-1} (-1)^u \omega_j^{a+u}(K) x^t y^\ell\bigr) f.
\end{align*}
Then we compute
\[
 \omega_j^{a+v}(K+j) = \bigl( \sssum{u+\ell=\\v-1} (-1)^{v-1-u} y^\ell \omega_j^{a+u}(K) \bigr) + (-1)^{v} y^v \omega_j^a(K+j),
\]
which implies that the action of the RHS of \cref{eq:ExtR2} coincides with the one of the LHS.

The only non trivial relation coming from regular isotopies we need to verify is given by the commutation of a floating dot and a crossing at its left. This is a consequence of the fact that~\cref{eq:epsnotation} is a symmetric polynomial, which commutes with divided difference operators.
\end{proof}

\subsubsection{Left-adjusted expressions}

Recall from \cite[\S 2.2.1]{naissevaz2} that a reduced expression $\sigma_{i_r}\cdots \sigma_{i_1}$ of $w \in S_n$ is \emph{left-adjusted} if $i_r + \cdots + i_1$ is minimal. Equivalently, it is left-adjusted if and only if
\[
\min_{t \in \{0, \dots, r\}} \sigma_{i_t} \cdots \sigma_{i_1}(k) \leq \min_{t \in \{0, \dots, r\}} \sigma_{j_t}\cdots \sigma_{j_1}(k),
\]
for all $k \in \{0, \dots, n\}$ and all other reduced expression $\sigma_{j_r} \cdots \sigma_{j_1} = w$. In this condition, we write
\[
\inlmin{w}(k) := \min_{t \in \{0, \dots, r\}} \sigma_{i_t} \cdots \sigma_{i_1}(k).
\]
Note that a left adjusted expression always exists and is unique up to distant permutation ( $\sigma_i\sigma_j \leftrightarrow \sigma_j\sigma_i$ for $|i-j| > 1$), so that $\inlmin{w}(k)$ is well-defined. In particular, one can obtain a left-adjusted reduced expression for any permutation by taking its representative in the coset decomposition
\begin{equation}\label{eq:Sncosetdecomp}
S_n = \bigsqcup_{a=1}^n S_{n-1} \sigma_{n-1} \cdots \sigma_a,
\end{equation}
applied recursively. If we think of a reduced expression in terms of string diagrams, then it is left-adjusted if all strings are pulled as far as possible to the left. 

\begin{exe}
The permutation $(1\ 3\ 2\  4) \in S_4$ admits as left-adjusted reduced expression the word $\sigma_1\sigma_2\sigma_1\sigma_3\sigma_2$ which comes from the summand $S_2 \sigma_3 \sigma_2$ in the first step of the recursive decomposition~\eqref{eq:Sncosetdecomp}. Note that $\sigma_1\sigma_2\sigma_3\sigma_1\sigma_2$ is also left-adjusted while  $\sigma_2\sigma_1\sigma_2\sigma_3\sigma_2$ and $\sigma_2\sigma_1\sigma_3\sigma_2\sigma_3$ are not. 
In terms of string diagrams, we consider as example the following reduced expression of the permutation $w = (1\ 4\ 3\ 5\ 2) \in S_5:$
\[
\tikzdiag{
        	 \draw +(0,-.75)   .. controls (0,-.25) and (2.25, .25) .. +(2.25,.75);
        	 \draw +(.75,-.75)  ..controls (.75,-.25) and (0,.25) ..  +(0,.75);
        	 \draw +(1.5,-.75)  .. controls (1.5, -.25) and (3,.25) .. +(3,.75);
    	 	\draw  +(2.25,-.75)   .. controls (2.25,-.5) and (3,-.25) ..  +(3,0);
    	 	\draw  +(3, 0)   ..  controls (3,.35) and (1.5,.4) ..  +(1.5,.75);
        	 \draw +(3,-.75) .. controls (3, -.25) and (.75, .25) .. +(.75,.75);
	}
\]
It is not left-adjusted since the $4$th strand (read at the bottom) can be pulled to the left. Hence we obtain the following left-adjusted minimal presentation:
\[
\tikzdiag{
        	 \draw +(0,-.75)   .. controls (0,-.25) and (2.25, .25) .. +(2.25,.75);
        	 \draw +(.75,-.75)  ..controls (.75,-.25) and (0,.25) ..  +(0,.75);
        	 \draw +(1.5,-.75)  .. controls (1.5, -.25) and (3,.25) .. +(3,.75);
    	 	\draw  +(2.25,-.75)   .. controls (2.25,-.5) and (.75,-.25) ..  +(.75,.15);
    	 	\draw  +(.75, .15)   ..  controls (.75,.5) and (1.5,.5) ..  +(1.5,.75);
        	 \draw +(3,-.75) .. controls (3, -.25) and (.75, .25) .. +(.75,.75);
	}
\]
\end{exe}

Suppose $\sigma_{i_r} \dotsm \sigma_{i_1}$ is a left-adjusted reduced expression of $w$. Then we can choose for each $k \in \{1, \dots, n\}$ an index $t_k \in \{ 1, \dots, r\}$ such that
\[
\sigma_{i_{t_k}}\dotsm \sigma_{i_1}(k) =  \inlmin{w}(k).
\]
Clearly this choice is not necessarily unique and we can have $t_{k} = t_{k'}$ for $k \neq k'$. 
However, it defines a partial order $\prec$ on the set $\{1, \dots, n\}$ where $k \prec k'$ whenever $t_{k} \leq t_{k'}$. We extend this order arbitrarily and we write $<_t$ for it. There is a bijective map $\zeta : \{1, \dots, n\} \rightarrow \{1, \dots, n\}$ which sends $k < k'$ to $\zeta(k) <_t \zeta(k')$, so that $t_{\zeta(k)} \le t_{\zeta(k')}$. 
In terms of string diagrams, the map $\zeta$ tells us in which order the strands attain their (chosen) leftmost position while reading from bottom to top. In particular, $\zeta(k)$ gives the starting point of the strand that attains its leftmost position in $k$th position.

\begin{exe}\label{exe:zeta}
Consider again the following left-adjusted string diagram:
\[
\tikzdiag{
        	 \draw +(0,-.75)   .. controls (0,-.25) and (2.25, .25) .. +(2.25,.75);
        	 \draw +(.75,-.75)  ..controls (.75,-.25) and (0,.25) ..  +(0,.75);
        	 \draw +(1.5,-.75)  .. controls (1.5, -.25) and (3,.25) .. +(3,.75);
    	 	\draw  +(2.25,-.75)   .. controls (2.25,-.5) and (.75,-.25) ..  +(.75,.15);
    	 	\draw  +(.75, .15)   ..  controls (.75,.5) and (1.5,.5) ..  +(1.5,.75);
        	 \draw +(3,-.75) .. controls (3, -.25) and (.75, .25) .. +(.75,.75);
	}
\]
Both the $1$st and $3$rd strand attain their leftmost position at the bottom of the diagram, thus we can choose $\zeta(1) = 1$ and $\zeta(2) = 3$. Then both the $2$nd and $4$th strand attain their leftmost position, thus we can take $\zeta(3) = 4$ and $\zeta(4) = 2$. Finally, the $5$th strand attains its leftmost position and we put $\zeta(5) = 5$.
\end{exe}

For $k \in\{1, \dots, n+1\}$, we put 
\[
w^k := \sigma_{i_{t_{\zeta(k)}}}\dotsm \sigma_{i_{t_{\zeta(k-1)}}},
\]
where it is understood that $t_{\zeta(0)} := 0$ and $t_{\zeta(n+1)} := r$. It defines a partition of the reduced expression of $\sigma_{i_r}\cdots \sigma_{i_1}  = w$. Moreover, it is constructed so that
\begin{align*}
w^k \cdots w^{1}(\zeta(k)) &=  \inlmin{w}(\zeta(k)),
\end{align*}
for all $1 \leq k \leq n$.

\begin{exe}Consider again $w = \sigma_1\sigma_2\sigma_1\sigma_3\sigma_2$ with $i_1=2, i_2 =3, i_3 = 1, i_4=2, i_5=1$. We can choose for example $t_1 = 0, t_2 = 0,t_3 = 3$ and $t_4 = 5$. 
Then we can put $\zeta(1) = 1$ (or $2$), $\zeta(2) = 2$ (or $1$), $\zeta(3) = 3$ and $\zeta(4) = 4$, with $w^1 = 1, w^2 = 1$, $w^3 = \sigma_1\sigma_3\sigma_2$ and $w^4 = \sigma_1\sigma_2$. 
\end{exe}

\subsubsection{A generating set}

We say that a floating dot is \emph{tight} if it is placed immediately to the right of the left-most strand, and has superscript $0$. We can also suppose it has the same subscript as the label of the strand at its left (otherwise it would slide to the left and be zero).

\begin{lem}\label{lem:dgKLRtightgen}
The algebra $R_\bo(\nu)$ is generated by KLR elements (i.e. dots and crossings) and tight floating dots.
\end{lem}

\begin{proof}
We first compute
\begin{align}\label{eq:fdtotheleft}
\tikzdiagh{0}{
	\draw (0,0) node[below]{\small $i$} -- (0,1);
	\draw (1,0) node[below]{\small $i$} -- (1,1);
	\fdot[a]{i}{1.35,.5};
}
\ = \ 
\tikzdiagh[yscale=1.5]{0}{
	\draw (0,0) node[below]{\small $i$} ..controls (0,.15) and (1,.15) .. (1,.5)
		 ..controls (1,.85) and (0,.85) .. (0,1) ;
	\draw (1,0) node[below]{\small $i$} ..controls (1,.15) and (0,.15) .. (0,.5)
		..controls (0,.85) and (1,.85) .. (1,1)  node [near end,tikzdot]{};
	\fdot[a]{i}{0.4,0.5};
} 
\ - \ 
\tikzdiagh[yscale=1.5]{0}{
	\draw (0,0) node[below]{\small $i$} ..controls (0,.15) and (1,.15) .. (1,.5) node [near start,tikzdot]{}
		 ..controls (1,.85) and (0,.85) .. (0,1);
	\draw (1,0) node[below]{\small $i$} ..controls (1,.15) and (0,.15) .. (0,.5)
		..controls (0,.85) and (1,.85) .. (1,1);
	\fdot[a]{i}{0.4,0.5};
} 
\end{align}
for all $a \geq 0, i \in I$. 
\cref{eq:fdtotheleft}, together with~\cref{eq:ExtR2} and~\cref{eq:fdots} allows to bring all floating dots to the left. Then applying~\cref{eq:fdots} recursively allows to transform all floating dots with superscript bigger than zero into dots and tight floating dots.
\end{proof}

We write $\omega$ for a tight floating dot, $\tau_a$ for a crossing between the $a$th and $(a+1)$th strands (counting from left), and $x_a$ for a dot on the $a$th strand, where we suppose the label of the strands given by the context, in the form of an idempotent $1_\bi$. We also define the \emph{tightened floating dot} in $R_\bo(\nu)$ as $\theta_a := \tau_{a-1}\cdots\tau_1 \omega \tau_1 \cdots \tau_{a-1}$, or diagrammatically
 \[
\theta_a 
\ := \ 
\tikzdiagh[xscale=1.25]{0}{
	\draw (2,0)  -- (2,0.7);
	\draw (3,0)  -- (3,0.7);
	\draw (1.5,0) % node[below] { \small $i$}  
		.. controls (1.5,0.35) and (0, 0.35) ..  (0,0.7)
		.. controls (0,1.05) and (1.5, 1.05) ..  (1.5,1.4);
	\draw (0,0)    .. controls (0,0.35) and (0.5, 0.35) ..  (0.5,0.7)
		 .. controls (0.5,1.05) and (0, 1.05) ..  (0,1.4);
	\draw (1,0)  .. controls (1,0.35) and (1.5, 0.35) ..  (1.5,0.7)
		.. controls (1.5,1.05) and (1, 1.05) ..  (1,1.4);
	\draw (2,0.7)-- (2,1.4);
	\draw (3,0.7)-- (3,1.4);
	\node at(1,.7) {\dots}; \node at(.5,.05) {\dots}; \node at(.5,1.35) {\dots};
	\node at(2.5,.7) {\dots};
	 \fdot{}{.2,.7}; 
	\draw[decoration={brace,mirror,raise=-8pt},decorate]  (-.1,-.5) -- node {$a$} (1.55,-.5);
 }
\]
We also write $\theta_a^0 := \theta_a$ and $\theta_a^{-1} := 1$.

\begin{lem}\label{lem:thightenedfdanticommutes}
Tigthened floating dots anticommute with each others, up to adding terms with a smaller number of crossings, that is
\begin{align*}
\theta_a \theta_b &= - \theta_b \theta_a + R, & 
(\theta_a)^2 = 0 + R',
\end{align*}
where $R$ (resp. $R'$) possesses strictly less crossings than $\theta_a \theta_b$ (resp. $(\theta_a)^2 $), for all $1 \leq a,b \leq m$. 
\end{lem}

\begin{proof}
We first compute that
\begin{align}\label{eq:tightfdcommutes}
\tikzdiagh[yscale=1.5]{0}{
	\draw (0,0) node[below]{\small $k$} ..controls (0,.15) and (1,.15) .. (1,.5)
		 ..controls (1,.85) and (0,.85) .. (0,1) -- (0,1.5);
	\draw (1,0) node[below]{\small $\ell$} ..controls (1,.15) and (0,.15) .. (0,.5)
		..controls (0,.85) and (1,.85) .. (1,1) -- (1,1.5);
	\fdot[a]{i}{.4,.5};
	\fdot[b]{j}{.5,1.25};
}
\ + \ 
\tikzdiagh[yscale=1.5]{0}{
	\draw (0,-.5) node[below]{\small $k$} -- (0,0)  ..controls (0,.15) and (1,.15) .. (1,.5)
		 ..controls (1,.85) and (0,.85) .. (0,1);
	\draw (1,-.5) node[below]{\small $\ell$}  -- (1,0)..controls (1,.15) and (0,.15) .. (0,.5)
		..controls (0,.85) and (1,.85) .. (1,1);
	\fdot[a]{i}{.4,.5};
	\fdot[b]{j}{.5,-.25};
}
\ = \ 0,
\end{align}
for all $i,j,k, \ell \in I$ and $a,b \in \bN$. Then we obtain
\[
\tikzdiag[scale=.6]{
      	\draw (0,-2) .. controls (0,-1.5) and (1,-1.5) .. (1,0)
      		.. controls (1,1.5) and (0,1.5) .. (0,2);
      	\draw (1,-2) .. controls (1,-1.5) and (1.5,-1.5) .. (1.5,0)
      		.. controls (1.5,1.5) and (1,1.5) .. (1,2);
	\draw (5,2)   -- (5,.5);
	\draw (5,-2)% node[below]{\plusspacing \small $b$} 
		.. controls (5,-1.375) and (0, -1.375) ..  (0,-.75)
	 	.. controls (0, -.125) and (5,-.125) ..(5,.5);
	\draw (2,-2)  %node[below]{\plusspacing \small$a$}   
		-- (2,0);
	\draw (2,0) .. controls (2,.5) and (0, .5) ..  (0,1)
		.. controls (0, 1.5) and (2,1.5) ..  (2,2);
	\draw (3,2)   -- (3,-2);
	\draw (4,2)   -- (4,-2);
	 \fdot{}{.35,-.75};   \fdot{}{.35,1}; 
 } 
\  \overset{\eqref{eq:KLRR3}}{=} \  
\tikzdiag[scale=.6]{
      	\draw (0,-2) .. controls (0,-1.5) and (1.5,-1.5) .. (1.5,0)
      		.. controls (1.5,1.5) and (0,1.5) .. (0,2);
      	\draw (1,-2) .. controls (1,-1.5) and (2,-1.5) .. (2,0)
      		.. controls (2,1.5) and (1,1.5) .. (1,2);
	\draw (5,-2)  %node[below]{\plusspacing \small$b$} 
		.. controls (5,-1) and (0, -1) ..  (0,0)
		.. controls (0, 1) and (5,1) .. (5,2);
	\draw (2,-2)  %node[below]{\plusspacing \small$a$} 
		.. controls (2,-1.5) and (0, -1.5) ..  (0,-1)
		.. controls (0,-.5) and (1, -.5) ..  (1,0)
		.. controls (1, .5) and  (0,.5) ..  (0,1)
		.. controls (0, 1.5) and (2,1.5) .. (2,2);
	\draw (3,2)   -- (3,-2);
	\draw (4,2)   -- (4,-2);
	 \fdot{}{.45,0};   \fdot{}{.45,1}; 
 } 
 \ + \ R_0
\  \overset{\eqref{eq:tightfdcommutes}}{=} \  - \ 
\tikzdiag[xscale=.6,yscale=-.6]{
      	\draw (0,-2) .. controls (0,-1.5) and (1,-1.5) .. (1,0)
      		.. controls (1,1.5) and (0,1.5) .. (0,2);
      	\draw (1,-2) .. controls (1,-1.5) and (1.5,-1.5) .. (1.5,0)
      		.. controls (1.5,1.5) and (1,1.5) .. (1,2);
	\draw (5,2) % node[below]{\plusspacing \small$b$}   
		-- (5,.5);
	\draw (5,-2)  .. controls (5,-1.375) and (0, -1.375) ..  (0,-.75)
	 	.. controls (0, -.125) and (5,-.125) ..(5,.5);
	\draw (2,-2)   -- (2,0);
	\draw (2,0) .. controls (2,.5) and (0, .5) ..  (0,1)
		.. controls (0, 1.5) and (2,1.5) ..  (2,2);  %node[below]{\plusspacing \small$a$};
	\draw (3,2)   -- (3,-2);
	\draw (4,2)   -- (4,-2);
	 \fdot{}{.35,-.75};   \fdot{}{.35,1}; 
 } 
 + R_0 + R_1,
 \]
 where both $R_0$ and $R_1$ have less crossings.
\end{proof}

Fix $\bi, \bj \in \Seq(\nu)$. Since they are both sequences of the same elements, there is a subset ${_\bj}S{_\bi} \subset S_m$ of permutations $w \in S_m$ such that $i_k = j_{w(k)}$ for all $k \in \{1, \dots, m\}$. 
Given such a permutation $w \in {_\bj}S{_\bi}$, we can choose a left-adjusted reduced expression. It comes with a partition $w^{m+1} \cdots w^{2} w^1 = w$ and a bijection $\zeta :\{1, \dots, m\} \rightarrow \{1, \dots, m\} $, such that
\begin{align*}
w^k \cdots w^{1}(\zeta(k)) &=  \inlmin{w}(\zeta(k)),
\end{align*}
for all $1 \leq k \leq m$. Then, consider the collection of elements
\begin{gather}\begin{aligned}
%\begin{align}
\label{eq:Antightbasis}
{_\bj}B{_\bi} :=
\bigl\{
&x_m^{a_m} \dotsm x_1^{a_1} \tau_{ w^{m+1}}   \theta_{\min_w(\zeta(m))}^{\ell_{m}}
\tau_{ w^{m}} 
\dotsm   \theta_{\min_w(\zeta(2))}^{\ell_{2}}\tau_{ w^{2}}  \theta_{\min_w(\zeta(1))}^{\ell_{1}}
 \tau_{w^1} | \\
 & a_i\in\bN, \ell_i\in \{0,-1\}, w\in {_\bj}S{_\bi}
\bigr\}
\end{aligned}\end{gather}
in $1_\bj R_\bo(\nu) 1_\bi$. 
Diagrammatically, elements in ${_\bj}B{_\bi}$ can be constructed using the following algorithm: 
\begin{enumerate}
\item choose a permutation $w \in  {_\bj}S_\bi$, consider its corresponding string diagram and make it left-adjusted by bringing all strands to the left;
\item for each strand, choose whether we want to add a floating dot. If so, add a tightened floating dot where the strand attains its left-most position by pulling the strand to the far left and adding the floating dot immediately at its right;
\item for each strand, choose a number of dots to add at the top of the diagram.
\end{enumerate}

\begin{exe}
Take $\bi = i_1 i_2 i_3 i_4 i_5$ and $\bj = i_2i_5i_4i_1i_3$, and consider the following left-adjusted permutation $w \in {}_\bj S _\bi$:
\[
w = 
\tikzdiag{
        	 \draw (0,-.75) node[below]{\small $i_1$}   .. controls (0,-.25) and (2.25, .25) .. (2.25,.75);
        	 \draw (.75,-.75) node[below]{\small $i_2$}  ..controls (.75,-.25) and (0,.25) ..  (0,.75);
        	 \draw (1.5,-.75) node[below]{\small $i_3$}  .. controls (1.5, -.25) and (3,.25) .. (3,.75);
    	 \draw  (2.25,-.75) node[below]{\small $i_4$}   .. controls (2.25,-.5) and (.75,-.25) ..  
    	 	  (.75, .15)   ..  controls (.75,.5) and (1.5,.5) ..  (1.5,.75);
        	 \draw (3,-.75)  node[below]{\small $i_5$} .. controls (3, -.25) and (.75, .25) .. (.75,.75);
	}
\]
Take $\ell_1 = 1, \ell_2 = 0, \ell_3 = 1, \ell_4 = 0$ and $\ell_5 = 0$ (for the same $\zeta$ as in \cref{exe:zeta}). Then we obtain the following element in ${}_\bj B _\bi$:
\[
\tikzdiag{
        	 \draw (0,-.75) node[below]{\small $i_1$}   .. controls (0,-.25) and (2.25, .25) .. (2.25,.75) 
        	 	-- (2.25,1)
        	 	node[midway, tikzdot]{} node[midway, xshift=1.5ex, yshift=.75ex]{\small $a_4$};
	\fdot{}{.3,.-.65}; 
        	% \draw (.75,-.75) node[below]{\small $i_2$}  ..controls (.75,-.25) and (0,.5) ..  (0,.75);
        	\draw (.75,-.75)node[below]{\small $i_2$} .. controls (.75,-.5) and (0,-.5) .. (0,-.25)
        		.. controls (0,0) and (.75,0) .. (.75,.25)
        		.. controls (.75,.5) and (0,.5) .. (0,.75)
        	 	-- (0,1)
        	 	node[midway, tikzdot]{} node[midway, xshift=1.5ex, yshift=.75ex]{\small $a_1$};
        	 \draw (1.5,-.75) node[below]{\small $i_3$}  .. controls (1.5, -.25) and (3,.25) .. (3,.75)
        	 	-- (3,1)
        	 	node[midway, tikzdot]{} node[midway, xshift=1.5ex, yshift=.75ex]{\small $a_5$};
    	 \draw  (2.25,-.75) node[below]{\small $i_4$}   .. controls (2.25,-.25) and (0,0) ..  
    	 	  (0, .25)   ..  controls (0,.5) and (1.5,.5) ..  (1.5,.75)
        	 	-- (1.5,1)
        	 	node[midway, tikzdot]{} node[midway, xshift=1.5ex, yshift=.75ex]{\small $a_3$};
        	 \fdot{}{.35,.25}; 
        	 \draw (3,-.75)  node[below]{\small $i_5$} .. controls (3, -.25) and (.75, .25) .. (.75,.75)
        	 	-- (.75,1)
        	 	node[midway, tikzdot]{} node[midway, xshift=1.5ex, yshift=.75ex]{\small $a_2$};
}
\]
\end{exe}

\begin{prop}\label{prop:Rbogen}
Elements in ${_\bj}B{_\bi}$ generate $1_\bj R_\bo(\nu) 1_\bi$ as a $\Bbbk$-vector space.
\end{prop}

\begin{proof}
The proof is an induction on the number of crossings. 
By \cref{lem:dgKLRtightgen}, we can assume that all floating dots are tight. 
 By \cref{eq:KLRdotslide} and \cref{eq:KLRnh} we can bring all the dots to the top of any diagram, at the cost of adding diagrams with fewer crossings. 
 Moreover, all braid isotopies hold up to adding terms with a lower amount of crossings thanks to~\cref{eq:KLRR2} and~\cref{eq:KLRR3}.

We claim that we can also assume that there is at most one floating dot at the immediate right of each strand. Indeed, suppose there are two tight floating dots at the right of the same strand. Then we can apply a braid-isotopy to bring it as most as possible to the left, until it is possibly blocked by other tight floating dots. We depict it by the following picture:
\[
\tikzdiag[yscale=.8,xscale=.9]{
	\node at(2,1.8) {\dots}; \node at(2,-1.8) {\dots};
	\draw  (1.5,2) .. controls (1,1) ..  (1,0) .. controls (1,-1) ..  (1.5,-2); 
	\draw (2.5,2)   -- (2.5,-2);
	%\node (rect) at (.5,0) [anchor = west, draw,fill=white,minimum width=2.5cm,minimum height=1cm] {};
	\filldraw [fill=white, draw=black] (-.25,.5) rectangle (2.75,-.5);
 	\fdot{}{.75,1.5}; 
	\draw[red,dashed]  (1,2) .. controls (0,1.7) and (0,1.3)  ..(1,1) 
				.. controls (2.5,.5) and (2.5,-.5) ..  (1,-1)
		 		.. controls (0,-1.3) and (0,-1.7) ..  (1,-2);
	\fdot{}{.75,-1.5}; 
} 
\ = \ 
\tikzdiag[yscale=.8,xscale=.9]{
	\draw  plot [smooth, tension=1] coordinates {(1.5,1.35)  (.25,.75) (1.5,.25) };
	\fdot{}{.75,.75}; 
	\draw  plot [smooth, tension=1] coordinates {(1.5,-1.35)  (.25,-.75) (1.5,-.25) };
	\fdot{}{.75,-.75}; 
	\node at(2.5,1.8) {\dots}; \node at(2.5,-1.8) {\dots};
	\draw (2,2)   -- (2,-2);
	\draw (3,2)   -- (3,-2);
	%\node (rect) at (1.5,0) [anchor = west, draw,fill=white,minimum width=2cm,minimum height=3cm] {};
	\filldraw [fill=white, draw=black] (1.5,1.5) rectangle (3.25,-1.5);
 	\fdot{}{.75,1.5}; 
	\fdot{}{.75,-1.5}; 
	\draw [red,dashed] plot [smooth, tension=1] coordinates { (1,2) (.25,1.5) (1.25, .75) (.5,0) (1.25,-.75) (.25,-1.5) (1,-2) };
} \ + \ 
\parbox{.25\textwidth}{ terms with  \\ fewer crossings,}
\]
where the dashed strand in red represents the one we want to pull, and the boxes represent other elements in $R_\bo(\nu)$. If there is no floating dot in-between, then it is zero by~\cref{eq:fdmoves}. Otherwise, we apply~\cref{eq:tightfdcommutes} to jump the bottom floating dot over all the floating dots in-between, until we have two floating dots in the same region at the top, which is zero by~\cref{eq:fdmoves}. This proves the claim. 

Finally, we observe that given a strand with a single tight floating dot, we can tighten it by braid isotopy, until we end up with a tightened floating dot. Since by \cref{lem:thightenedfdanticommutes} tightened floating dots anticommute, this concludes the proof.
\end{proof}

\subsubsection{The basis theorem}

\begin{prop}\label{prop:faithfulaction}
The action in \cref{prop:Rboaction} is faithful.
\end{prop}

\begin{proof}
The proof is inspired by~\cite[Proposition~3.8]{rouquierquiv} (see also~\cite[Theorem~2.5]{KL1} for a different approach).  
We claim that elements of ${_\bj} B_\bi$ act as linearly independent endomorphisms on $P_\nu$. The action yields morphisms
\[
P_I 1_\bi \rightarrow P_I 1_\bj,
\] 
that we will consider as endomorphisms of $P_I$. 

First we extend the scalars to $\Bbbk(x_{1,i},\dots, x_{\nu_i,i})$ in $P_i$ for all $i\in I$. We claim that different choices of $w \in {_\bj}S{_\bi}$ and $\ell_i \in \{-1,0\}$ give linearly independent operators. Notice that since $\bi, \bj$ is fixed, $w$ is given by choices of permutations between strands of the same label.
Since crossings between strands with different labels act as multiplication by a polynomial, we can ignore them as being multiplication by a scalar. By \cite[Corollary 3.9]{naissevaz2}, we know that different choices of permutations and tightened floating dots for strands with label $i$ act as linearly independent operators on $P_i$, hence proving our claim. Finally, taking into account the multiplication by the polynomial given by the choice of the $a_i \in \bN$ as in \cref{eq:Antightbasis} concludes the proof. 
\end{proof}

\begin{thm}\label{thm:Rbobasis}
The $\Bbbk$-module $1_\bj R_\bo(\nu) 1_\bi$ is free with basis ${_\bj}B{_\bi}$.
\end{thm}

\begin{proof}
It follows from \cref{prop:Rbogen} and \cref{prop:faithfulaction}.
\end{proof}

From this, we also deduce the following:

\begin{cor}\label{cor:Rbmindesc}
The $\bo$-KLR algebra admits a presentation given by the KLR-generators and tight floating dots, subjected to the KLR-relations Eq. (\ref{eq:KLRR2}--\ref{eq:KLRR3}) together with
\begin{align*}
\tikzdiagh[yscale=1.5]{0}{
	\draw (0,0) node[below]{\small $j$} ..controls (0,.15) and (1,.15) .. (1,.5)
		 ..controls (1,.85) and (0,.85) .. (0,1) -- (0,1.5);
	\draw (1,0) node[below]{\small $i$} ..controls (1,.15) and (0,.15) .. (0,.5)
		..controls (0,.85) and (1,.85) .. (1,1) -- (1,1.5);
	\fdot{i}{.4,.5};
	\fdot{j}{.5,1.25};
}
\ + \ 
\tikzdiagh[yscale=1.5]{0}{
	\draw (0,-.5) node[below]{\small $j$} -- (0,0)  ..controls (0,.15) and (1,.15) .. (1,.5)
		 ..controls (1,.85) and (0,.85) .. (0,1);
	\draw (1,-.5) node[below]{\small $i$}  -- (1,0)..controls (1,.15) and (0,.15) .. (0,.5)
		..controls (0,.85) and (1,.85) .. (1,1);
	\fdot{i}{.4,.5};
	\fdot{j}{.5,-.25};
}
\ &= \ 0,
&
\tikzdiagh[yscale=1.5]{0}{
	\draw (0,0)  node[below]{\small $i$} -- (0,1);
	\fdot{i}{.4,.25};
	\fdot{i}{.4,.75};
}
\ &= 0,
\end{align*}
for all $i,j \in I_r$. 
\end{cor}

%%%%%%%%%%%%%%%%	End of file	%%%%%%%%%%%%%

%% file: sections/dgenhancement.tex
%%%%%%%%%%%%%%%%%%%%%%%%%%%%%%%%%%%%
%                 					  				  		 %
%	Dg-enhancement				 					 %
%                 					  						 %
%%%%%%%%%%%%%%%%%%%%%%%%%%%%%%%%%%%%

\section{Dg-enhancement}\label{sec:dgenhancement}

We fix a subset $I_f \subset I$ and consider the associated parabolic subalgebra $U_q(\p) \subset U_q(\g)$ as defined in~\cref{sec:qgweightmodules}. For each $j \in I_f$, we also choose a weight $n_j \in \bN$, and write $N := \{n_j\}_{j \in I_f}$.

\begin{defn}\label{def:pKLR}
  The $\p$-KLR algebra $R_\p(m)$ is given by forgetting the $\lambda_j$-degree for each $j \in I_f$ in $R_\bo(m)$ and modding out by 
\[
\tikzdiagh[xscale=.75]{0}{
	\fdot{j}{0.5,0.5};
	\draw (0,0) node[below] {\plusspacing \small $j$} -- (0,1);
	\draw (1,0) node[below] {\plusspacing \small $i_1$} -- (1,1);
	\node at(2,.5) {$\dots$};
	\draw (3,0) node[below] {\plusspacing \small $i_{m-1}$} -- (3,1);
}
\ = \ 
0,
\]
for all $j \in I_f$. The \emph{$N$-cyclotomic quotient} $R_\p^N(m)$ of $R_\p(m)$ is given by modding out by
\[
\tikzdiagh{0}{
	     	 \draw  (0,-.5) node[below] {\small $j$}  --(0,.5)node [midway,tikzdot]{} node[midway,xshift=1.75ex,yshift=.75ex]{\small $n_j$}; 
	      	\draw  (1,-.5) node[below] {\small $i_1$} -- (1,.5);
	      	\node at(2,0) {\small $\dots$};
	      	\draw  (3,-.5) node[below] {\small $i_{m-1}$} -- (3,.5);
	 }
	\ =\  0,
\]
for all $j \in I_f$.
\end{defn}

In particular, $R_{\g}(m)$ is the usual KLR algebra $R(m)$ (see \cref{def:KLRalgebra}). 
Its $N$-cyclotomic quotient $R_\g^N(m)$ is also the usual cyclotomic quotient of the KLR algebra. Taking $I_f = \emptyset$ gives $\p = \bo$ and we recover \cref{def:Rbo}.

\smallskip

We equip $R_\bo(m)$ with a homological $\bZ$-grading, denoted $h$, by setting 
\begin{align*}
\deg_h \left(
\ 
\tikzdiag{
	\draw (0,0) node[below] {\small $i$}  ..controls (0,.5) and (1,.5) .. (1,1);
	\draw[myblue] (1,0) node[below] {\small $j$}  ..controls (1,.5) and (0,.5) .. (0,1);
}
\ 
\right)
& := 0,
&
\deg_h \left(
\ 
\tikzdiag{
	\draw (0,0) node[below] {\small $i$}  -- (0,1) node [midway,tikzdot]{};
}
\ 
\right)
& := 0, 
&
\deg_{h} \left(
\ 
\tikzdiag{
	\fdot[a]{i}{0,0};
	\node at(-.5,-.5) {\small $K$};
}
\ 
\right)
&= 1,
\end{align*}
for all $i,j \in I$. Then, we equip $R_\bo(m)$ with a differential $d_N$ by setting
\[
d_N \left(
\ 
\tikzdiag{
	\draw (0,0) node[below] {\small $i$}  ..controls (0,.5) and (1,.5) .. (1,1);
	\draw[myblue] (1,0) node[below] {\small $j$}  ..controls (1,.5) and (0,.5) .. (0,1);
}
\ 
\right)
:= 
d_N \left(
\ 
\tikzdiag{
	\draw (0,0) node[below] {\small $i$}  -- (0,1) node [midway,tikzdot]{};
}
\ 
\right)
:= 0,
\]
and
\[
d_N
 \left(
\ 
\tikzdiagh[xscale=.75]{0}{
	\fdot{j}{0.5,0.5};
	\draw (0,0) node[below] {\plusspacing \small $j$} -- (0,1);
	\draw (1,0) node[below] {\plusspacing \small $i_1$} -- (1,1);
	\node at(2,.5) {$\dots$};
	\draw (3,0) node[below] {\plusspacing \small $i_{m-1}$} -- (3,1);
}
\ 
\right)
:= \begin{cases}
0, & \text{ if $j \notin I_f$,}\\
(-1)^{n_j}\tikzdiagh{0}{
	     	 \draw  (0,-.5) node[below] {\small $j$}  --(0,.5)node [midway,tikzdot]{} node[midway,xshift=1.75ex,yshift=.75ex]{\small $n_j$}; 
	      	\draw  (1,-.5) node[below] {\small $i_1$} -- (1,.5);
	      	\node at(2,0) {\small $\dots$};
	      	\draw  (3,-.5) node[below] {\small $i_{m-1}$} -- (3,.5);
	 }, & \text{ if $j \in I_f$.}\\
\end{cases}
\] 
We extend the definition of $d_N$ to the whole algebra using the graded Leibniz rule $d_N(xy) = d_N(x)y + (-1)^{\deg_h(x)}x d_N(y)$ and  \cref{lem:dgKLRtightgen}. Checking that $d_N$ is well-defined is straightforward using \cref{cor:Rbmindesc}. 
From this, we derive that for $j \in I_r$ we have
\[
d_N \left(
\ 
\tikzdiag{
	\fdot[a]{j}{0,0};
	\node at(-.5,-.5) {\small $K$};
}
\ 
\right)
= 
(-1)^{n_j-k_j+1+a} \sum_{r=0}^{-\alpha_j^\vee(K^{\setminus j})} \ch_{n_j+a-k_j+1+r}(x_{k_j,j}) \varepsilon_r^j(\und x_K),
\]
where $x_{\ell,i}$ is a dot on the $\ell$th strand with label $i$, $\ch_n$ is the $n$th complete homogeneous polynomial, and $\varepsilon_r^j(\und x_K)$ is defined in~\cref{eq:epsnotation2}.

\begin{defn}
We refer to the dg-algebra $(R_\bo(m), d_N)$ as the \emph{dg-enhanced KLR algebra}. 
\end{defn}

\begin{prop}\label{prop:acyclic}
If $n_j-\nu_j-\alpha_j^\vee(\nu^{\setminus j}) < 0$, then $(R_\bo(\nu),d_N)$ is acyclic.
\end{prop}

\begin{proof}
Taking $a := -(n_j-\nu_j-\alpha_j^\vee(\nu^{\setminus j}) + 1)$ and considering the floating dot placed on the far right with subscript $j$ and superscript $a$ yields
\[
d_N \left(
\ 
\tikzdiag{
	\fdot[a]{j}{0,0};
	\node at(-.5,-.5) {\small $\nu$};
}
\ 
\right)
= 
 (-1)^{\alpha_j^\vee(\nu^{\setminus j})}.
\]
Thus, $H(R_\bo(\nu),d_N) \cong 0$.
\end{proof}

Our goal for the rest of the section will be to prove the following:

\begin{thm}\label{thm:RbodNformal}
The dg-algebra $(R_\bo(m), d_N)$ is formal with homology
\[
H(R_\bo(m), d_N) \cong R_\p^N(m).
\]
\end{thm}

\subsection{Proof of \cref{thm:RbodNformal}}

Denote $1_{(m,i)} := \sum_{\bj \in I^m} 1_{\bj i}$, or diagrammatically
\begin{align*}
1_{(m,i)} &:= \sum_{(j_1, \dots, j_m) \in I^m}  
\tikzdiagh{.5ex}{
          \draw (-.5,-.5) node[below] {\small $j_1$} -- (-.5,.5); 
          \draw (1,-.5) node[below] {\small $j_m$} -- (1,.5); 
          \node at (.25,0){$\cdots$};
          \draw (1.5,-.5) node[below] {\small $i$} -- (1.5,.5); 
        }.
\end{align*}
It is an idempotent of $R_\bo(m+1)$. 
We also define $1_{(\nu, i)} := \sum_{\bj \in \Seq(\nu)} 1_{\bj i}$ for $\nu \in X^+$.
The algebra $R_\bo(m)$ acts on $1_{(m,i)}R_\bo(m+1)$ by first adding a vertical strand labeled $i$ at the right of $D \in R_\bo(m)$ and then multiplying on the left in $R_\bo(m+1)$.

\smallskip

We now introduce some other diagrammatic notations as in~\cite[\S3.1]{naissevaz2}. We draw $R_\bo(m)$ (viewed as $R_\bo(m)$-$R_\bo(m)$-bimodule) as a box labeled by $m$
\[
R_\bo(m)
\ = \  
\tikzdiag[xscale=.75]{
	\draw (0,-.5) -- (0,.5);
	\draw (.5,-.5) -- (.5,.5);
	\draw (1.5,-.5) -- (1.5,.5);
	\draw (2,-.5) -- (2,.5);
	\node at(1,.4) {\small $\dots$};
	\node at(1,-.4) {\small $\dots$};
	\filldraw [fill=white, draw=black] (-.25,-.25) rectangle (2.25,0.25) node[midway] {\small $m$};
}
\]
and $\otimes_m := \otimes_{R_\bo(m)}$ becomes stacking boxes on top of each other. Moreover, when $R_\bo(m+1)$ is viewed as a left $R_\bo(m)$-module, as a right $R_\bo(m)$-module or as an $R_\bo(m)$-$R_\bo(m)$-bimodule, we draw respectively
\begin{align*}
\tikzdiag[xscale=.75]{
	\draw (0,0) -- (0,.5);
	\draw (.5,0) -- (.5,.5);
	\draw (1.5,0) -- (1.5,.5);
	\draw (2,0) -- (2,.5);
	\draw (2.5, 0) -- (2.5, .25) .. controls (2.5,.5) .. (2.75,.5);
	\node at(1,.4) {\small $\dots$};
	\filldraw [fill=white, draw=black] (-.25,-.25) rectangle (2.75,0.25) node[midway] {\small $m+1$};
} &&
\tikzdiag[xscale=.75,yscale=-1]{
	\draw (0,0) -- (0,.5);
	\draw (.5,0) -- (.5,.5);
	\draw (1.5,0) -- (1.5,.5);
	\draw (2,0) -- (2,.5);
	\draw (2.5, 0) -- (2.5, .25) .. controls (2.5,.5) .. (2.75,.5);
	\node at(1,.4) {\small $\dots$};
	\filldraw [fill=white, draw=black] (-.25,-.25) rectangle (2.75,0.25) node[midway] {\small $m+1$};
}&&
\tikzdiag[xscale=.75]{
	\draw (0,-.5) -- (0,.5);
	\draw (.5,-.5) -- (.5,.5);
	\draw (1.5,-.5) -- (1.5,.5);
	\draw (2,-.5) -- (2,.5);
	\draw (2.75,-.5) .. controls (2.5,-.5) .. (2.5, -.25) -- (2.5, .25) .. controls (2.5,.5) .. (2.75,.5);
	\node at(1,-.4) {\small $\dots$};
	\node at(1,.4) {\small $\dots$};
	\filldraw [fill=white, draw=black] (-.25,-.25) rectangle (2.75,0.25) node[midway] {\small $m+1$};
}
\end{align*}

\begin{lem} \label{lem:Rboleftdecomp}
As a left $R_\bo(m)$-module, $1_{(m,i)}R_\bo(m+1)$ is free with decomposition
\[
\bigoplus_{a=1}^{m+1} \bigoplus_{\ell \geq 0} \bigl( R_\bo(m) 1_{(m,i)} \tau_m \cdots \tau_a x_a^\ell \oplus  R_\bo(m) 1_{(m,i)} \tau_m \cdots \tau_a  \theta_a^\ell \bigr),
\]
where $\theta_a^{\ell} := \tau_{a-1} \cdots \tau_1 \omega x_1^\ell \tau_1 \cdots \tau_{a-1}$. 
\end{lem}

We draw this as
\[
\tikzdiag[xscale=.75]{
	\draw (0,0) -- (0,.5);
	\draw (.5,0) -- (.5,.5);
	\draw (1.5,0) -- (1.5,.5);
	\draw (2,0) -- (2,.5);
	\draw (2.5, 0) -- (2.5, .25) .. controls (2.5,.5) .. (2.75,.5) node[right] {$i$};
	\node at(1,.4) {\small $\dots$};
	\filldraw [fill=white, draw=black] (-.25,-.25) rectangle (2.75,0.25) node[midway] {\small $m+1$};
}
\ \cong \ 
\bigoplus_{a = 1}^{m+1} \bigoplus_{\ell \geq 0}\ 
\left(
\tikzdiag[xscale=.75]{
	\draw (0,-1.5) -- (0,.5);
	\draw (.5,-1.5)   -- (.5,.5);
	\draw (2,-1.5)  .. controls (2,-1) and (1.5,-1) .. (1.5,-.5) -- (1.5,.5);
	\draw (2.5, -1.5)     .. controls (2.5,-1) and (2,-1) .. (2,-.5) -- (2,.5);
	\node at(1,.4) {\small$\dots$};
	\node at(1,-.4) {\small$\dots$};
	\filldraw [fill=white, draw=black] (-.25,-.25) rectangle (2.25,0.25) node[midway] {\small $m$};
	\draw (1.25, -1.5) 
	 	.. controls (1.25,-1) and (2.5,-1) .. (2.5, -.5) node[pos=.15, tikzdot]{} node[pos=.15, xshift=-1.5ex, yshift=.5ex]{\small $\ell$}
		-- (2.5,.25)
		.. controls (2.5,.5) .. (2.75,.5)   node[right] {$i$};
	\draw[decoration={brace,mirror,raise=-8pt},decorate]  (-.15,-1.85) -- node {$a$} (1.35,-1.85);
}
\ \oplus \ 
\tikzdiag[xscale=.75]{
	\draw (0,-1.5) 
		.. controls (0,-1.25) and (.75,-1.25) ..(.75,-1)
		.. controls (.75,-.75) and (0,-.75) ..(0,-.5)
		 -- (0,.5);
	\draw (.5,-1.5) 
		.. controls (.5,-1.25) and (1.25,-1.25) ..(1.25,-1)
		.. controls (1.25,-.75) and (.5,-.75) ..(.5,-.5)
		 -- (.5,.5);
	\draw (2,-1.5) .. controls (2,-1) and (1.5,-1) .. (1.5,-.5) -- (1.5,.5);
	\draw (2.5, -1.5) .. controls (2.5,-1) and (2,-1) .. (2,-.5) -- (2,.5);
	\node at(1,.4) {\small$\dots$};
	\node at(1,-.4) {\small$\dots$};
	\filldraw [fill=white, draw=black] (-.25,-.25) rectangle (2.25,0.25) node[midway] {\small $m$};
	\draw (1.25, -1.5)  
		.. controls (1.25,-1.25) and (0,-1.25) .. (0, -1)  node[pos=1, tikzdot]{} node[pos=1, xshift=-1.5ex, yshift=.5ex]{\small $\ell$}
		.. controls (0,-.75) and (2.5,-.75) .. (2.5, -.25) 
		-- (2.5,.25)
		.. controls (2.5,.5) .. (2.75,.5)  node[right] {$i$};
	\fdot{}{.4,-1.015};
	\draw[decoration={brace,mirror,raise=-8pt},decorate]  (-.15,-1.85) -- node {$a$} (1.35,-1.85);
}
\right)
\]

\begin{proof}
By \cref{thm:Rbobasis} we obtain
\[
\tikzdiag[xscale=.75]{
	\draw (0,0) -- (0,.5);
	\draw (.5,0) -- (.5,.5);
	\draw (1.5,0) -- (1.5,.5);
	\draw (2,0) -- (2,.5);
	\draw (2.5, 0) -- (2.5, .25) .. controls (2.5,.5) .. (2.75,.5) node[right] {$i$};
	\node at(1,.4) {\small $\dots$};
	\filldraw [fill=white, draw=black] (-.25,-.25) rectangle (2.75,0.25) node[midway] {\small $m+1$};
}
\ \cong \ 
\bigoplus_{a = 1}^{m+1} \bigoplus_{\ell \geq 0}\ 
\left(
\tikzdiag[xscale=.75]{
	\draw (0,-1.5) -- (0,.5);
	\draw (.5,-1.5)   -- (.5,.5);
	\draw (2,-1.5)  .. controls (2,-1) and (1.5,-1) .. (1.5,-.5) -- (1.5,.5);
	\draw (2.5, -1.5)     .. controls (2.5,-1) and (2,-1) .. (2,-.5) -- (2,.5);
	\node at(1,.4) {\small$\dots$};
	\node at(1,-.4) {\small$\dots$};
	\filldraw [fill=white, draw=black] (-.25,-.25) rectangle (2.25,0.25) node[midway] {\small $m$};
	\draw (1.25, -1.5) 
	 	.. controls (1.25,-1) and (2.5,-1) .. (2.5, -.5) 
		-- (2.5,.25) node[midway, tikzdot]{} node[midway, xshift=1.5ex, yshift=.5ex]{\small $\ell$}
		.. controls (2.5,.5) .. (2.75,.5)   node[right] {$i$};
	\draw[decoration={brace,mirror,raise=-8pt},decorate]  (-.15,-1.85) -- node {$a$} (1.35,-1.85);
}
\ \oplus \ 
\tikzdiag[xscale=.75]{
	\draw (0,-1.5) 
		.. controls (0,-1.25) and (.75,-1.25) ..(.75,-1)
		.. controls (.75,-.75) and (0,-.75) ..(0,-.5)
		 -- (0,.5);
	\draw (.5,-1.5) 
		.. controls (.5,-1.25) and (1.25,-1.25) ..(1.25,-1)
		.. controls (1.25,-.75) and (.5,-.75) ..(.5,-.5)
		 -- (.5,.5);
	\draw (2,-1.5) .. controls (2,-1) and (1.5,-1) .. (1.5,-.5) -- (1.5,.5);
	\draw (2.5, -1.5) .. controls (2.5,-1) and (2,-1) .. (2,-.5) -- (2,.5);
	\node at(1,.4) {\small$\dots$};
	\node at(1,-.4) {\small$\dots$};
	\filldraw [fill=white, draw=black] (-.25,-.25) rectangle (2.25,0.25) node[midway] {\small $m$};
	\draw (1.25, -1.5)  
		.. controls (1.25,-1.25) and (0,-1.25) .. (0, -1)  
		.. controls (0,-.75) and (2.5,-.75) .. (2.5, -.25) 
		-- (2.5,.25)  node[midway, tikzdot]{} node[midway, xshift=1.5ex, yshift=.5ex]{\small $\ell$}
		.. controls (2.5,.5) .. (2.75,.5)  node[right] {$i$};
	\fdot{}{.4,-1.015};
	\draw[decoration={brace,mirror,raise=-8pt},decorate]  (-.15,-1.85) -- node {$a$} (1.35,-1.85);
}
\right) 
\]
We then apply \cref{eq:KLRdotslide} and \cref{eq:KLRnh} to bring all the dots to the desired position. It is a triangular change of basis, concluding the proof.
\end{proof}

From now on, we will draw boxes with label `$m,d_N$' to denote the dg-algebra $(R_\bo(m),d_N)$. Similarly, a box with label $H(m)$ denotes its homology $H(R_\bo(m), d_N)$. 
 Then, the decomposition in \cref{lem:Rboleftdecomp} lifts directly to a direct sum decomposition of dg-modules whenever $i \notin I_f$. 
 Otherwise, for $i \in I_f$, it lifts to the mapping cone
\begin{align}\label{eq:coneleftdecomp}
&\tikzdiag[xscale=.75]{
	\draw (0,0) -- (0,.5);
	\draw (.5,0) -- (.5,.5);
	\draw (1.5,0) -- (1.5,.5);
	\draw (2,0) -- (2,.5);
	\draw (2.5, 0) -- (2.5, .25) .. controls (2.5,.5) .. (2.75,.5) node[right] {$i$};
	\node at(1,.4) {\small$\dots$};
	\filldraw [fill=white, draw=black] (-.25,-.25) rectangle (2.75,0.25) node[midway] {\small $m+1,d_N$};
}
\\
&
\qquad
\ \cong \ 
 \cone
\left(
\bigoplus_{a = 1}^{m+1} \bigoplus_{\ell \geq 0}\ 
\tikzdiag[xscale=.75]{
	\draw (0,-1.5) 
		.. controls (0,-1.25) and (.75,-1.25) ..(.75,-1)
		.. controls (.75,-.75) and (0,-.75) ..(0,-.5)
		 -- (0,.5);
	\draw (.5,-1.5) 
		.. controls (.5,-1.25) and (1.25,-1.25) ..(1.25,-1)
		.. controls (1.25,-.75) and (.5,-.75) ..(.5,-.5)
		 -- (.5,.5);
	\draw (2,-1.5) .. controls (2,-1) and (1.5,-1) .. (1.5,-.5) -- (1.5,.5);
	\draw (2.5, -1.5) .. controls (2.5,-1) and (2,-1) .. (2,-.5) -- (2,.5);
	\node at(1,.4) {\small$\dots$};
	\node at(1,-.4) {\small$\dots$};
	\filldraw [fill=white, draw=black] (-.25,-.25) rectangle (2.25,0.25) node[midway] {\small $m,d_N$};
	\draw (1.25, -1.5)  
		.. controls (1.25,-1.25) and (0,-1.25) .. (0, -1)  node[pos=1, tikzdot]{} node[pos=1, xshift=-1.5ex, yshift=.5ex]{\small $\ell$}
		.. controls (0,-.75) and (2.5,-.75) .. (2.5, -.25) 
		-- (2.5,.25)
		.. controls (2.5,.5) .. (2.75,.5)  node[right] {$i$};
	\fdot{}{.4,-1.015};
	\draw[decoration={brace,mirror,raise=-8pt},decorate]  (-.15,-1.85) -- node {$a$} (1.35,-1.85);
}
\xrightarrow{\bar d_N}
\bigoplus_{a = 1}^{m+1} \bigoplus_{\ell \geq 0}\ 
\tikzdiag[xscale=.75]{
	\draw (0,-1.5) -- (0,.5);
	\draw (.5,-1.5)   -- (.5,.5);
	\draw (2,-1.5)  .. controls (2,-1) and (1.5,-1) .. (1.5,-.5) -- (1.5,.5);
	\draw (2.5, -1.5)     .. controls (2.5,-1) and (2,-1) .. (2,-.5) -- (2,.5);
	\node at(1,.4) {\small$\dots$};
	\node at(1,-.4) {\small$\dots$};
	\filldraw [fill=white, draw=black] (-.25,-.25) rectangle (2.25,0.25) node[midway] {\small $m,d_N$};
	\draw (1.25, -1.5) 
	 	.. controls (1.25,-1) and (2.5,-1) .. (2.5, -.5) node[pos=.15, tikzdot]{} node[pos=.15, xshift=-1.5ex, yshift=.5ex]{\small $\ell$}
		-- (2.5,.25)
		.. controls (2.5,.5) .. (2.75,.5)   node[right] {$i$};
	\draw[decoration={brace,mirror,raise=-8pt},decorate]  (-.15,-1.85) -- node {$a$} (1.35,-1.85);
}
\right)
\end{align}
where the map $\bar d_N$ is induced by the differential of $(R_\bo(m+1),d_N)$. 

We will prove \cref{thm:RbodNformal} using induction on the number of strands $m$. Therefore, we can assume $(R_\bo(m),d_N)$ to be formal.

Following \cite{keller}, recall that for a dg-algebra $(A, d_A)$, we say that a dg-module is a \emph{relatively projective module} if it is a direct summand of a free module in $(A,d_A)\amod$.   Moreover, an $(A,d_A)$-module $Y$ \emph{satisfies property (P)} if  there is an exhaustive filtration of $(A,d_A)$-modules
\[
0 = F_0 \subset F_1 \subset F_2 \subset  \cdots \subset F_r \subset F_{r+1} \subset \cdots \subset Y,
\]
such that each $F_{r+1}/F_r$ is isomorphic in  $(A,d_A)\amod$ to a relatively projective module. 
An $(A,d_A)$-direct summand of a property (P) module is called \emph{cofibrant}.
 Also recall the following result of homological algebra:

\begin{lem}\label{prop:hhtensorhomology}
Let $(A, d_A)$ be a dg-algebra, $(M,d_M)$ be a right $(A,d_A)$-module, and $(N,d_N)$ a left one. If $(M,d_M)$ is formal and $(N,d_N)$ is cofibrant, then we have
\[
H\bigl((M,d_M) \otimes_{(A,d_A)} (N,d_N)\bigr) \cong H\bigl( H(M,d_M) \otimes_{(A,d_A)} (N, d_N)\bigr).
\]
\end{lem}

\begin{proof}
Tensoring with a cofibrant dg-module preserves quasi-isomorphisms.
\end{proof}

Therefore we obtain an exact sequence
\begin{align}\label{eq:exactseqleftdecomp}
\bigoplus_{a = 1}^{m+1} \bigoplus_{\ell \geq 0}\ 
\tikzdiag[xscale=.55]{
	\draw (0,-1.5) 
		.. controls (0,-1.25) and (.75,-1.25) ..(.75,-1)
		.. controls (.75,-.75) and (0,-.75) ..(0,-.5)
		 -- (0,.5);
	\draw (.5,-1.5) 
		.. controls (.5,-1.25) and (1.25,-1.25) ..(1.25,-1)
		.. controls (1.25,-.75) and (.5,-.75) ..(.5,-.5)
		 -- (.5,.5);
	\draw (2,-1.5) .. controls (2,-1) and (1.5,-1) .. (1.5,-.5) -- (1.5,.5);
	\draw (2.5, -1.5) .. controls (2.5,-1) and (2,-1) .. (2,-.5) -- (2,.5);
	\node at(1,.4) {\small$\dots$};
	\node at(1,-.4) {\small$\dots$};
	\filldraw [fill=white, draw=black] (-.25,-.25) rectangle (2.25,0.25) node[midway] {\small  $H(m)$};
	\draw (1.25, -1.5)  
		.. controls (1.25,-1.25) and (0,-1.25) .. (0, -1)  node[pos=1, tikzdot]{} node[pos=1, xshift=-1.5ex, yshift=.5ex]{\small $\ell$}
		.. controls (0,-.75) and (2.5,-.75) .. (2.5, -.25) 
		-- (2.5,.25)
		.. controls (2.5,.5) .. (2.75,.5)  node[right] {$i$};
	\fdot{}{.4,-1.015};
	\draw[decoration={brace,mirror,raise=-8pt},decorate]  (-.15,-1.85) -- node {$a$} (1.35,-1.85);
}
&\xrightarrow{\bar d_N}
\bigoplus_{a = 1}^{m+1} \bigoplus_{\ell \geq 0}\ 
\tikzdiag[xscale=.55]{
	\draw (0,-1.5) -- (0,.5);
	\draw (.5,-1.5)   -- (.5,.5);
	\draw (2,-1.5)  .. controls (2,-1) and (1.5,-1) .. (1.5,-.5) -- (1.5,.5);
	\draw (2.5, -1.5)     .. controls (2.5,-1) and (2,-1) .. (2,-.5) -- (2,.5);
	\node at(1,.4) {\small$\dots$};
	\node at(1,-.4) {\small$\dots$};
	\filldraw [fill=white, draw=black] (-.25,-.25) rectangle (2.25,0.25) node[midway] {\small  $H(m)$};
	\draw (1.25, -1.5) 
	 	.. controls (1.25,-1) and (2.5,-1) .. (2.5, -.5) node[pos=.15, tikzdot]{} node[pos=.15, xshift=-1.5ex, yshift=.5ex]{\small $\ell$}
		-- (2.5,.25)
		.. controls (2.5,.5) .. (2.75,.5)   node[right] {$i$};
	\draw[decoration={brace,mirror,raise=-8pt},decorate]  (-.15,-1.85) -- node {$a$} (1.35,-1.85);
} 
\longrightarrow\ 
\tikzdiag[xscale=.55]{
	\draw (0,0) -- (0,.5);
	\draw (.5,0) -- (.5,.5);
	\draw (1.5,0) -- (1.5,.5);
	\draw (2,0) -- (2,.5);
	\draw (2.5, 0) -- (2.5, .25) .. controls (2.5,.5) .. (2.75,.5) node[right] {$i$};
	\node at(1,.4) {\small $\dots$};
	\filldraw [fill=white, draw=black] (-.25,-.25) rectangle (2.75,0.25) node[midway] {\small $H(m{+}1)$};
}
\rightarrow 0,
\end{align}
thanks to \cref{prop:hhtensorhomology} and \cref{eq:coneleftdecomp}.

\begin{prop}\label{prop:sesleftdecomp}
The exact sequence~\cref{eq:exactseqleftdecomp} is a short exact sequence, with $\bar d_N$ being injective. 
\end{prop}

\cref{thm:RbodNformal} above is a direct consequence of \cref{prop:sesleftdecomp}. Therefore, we now focus on proving this proposition.
This is in fact similar to Kang--Kashiwara's~\cite[Eq. (4.13)]{kashiwara}, with basically only a change of basis, and thus we will follow the same ideas. We introduce the equivalent of `$g_a$' from the reference and draw it as an \emph{undercrossing}:
\begin{align*}
\tikzdiagh[scale=.8]{0}{
	          \draw   (-.5,-.5) node[below]{$i$} .. controls (-.5,0) and (.5,0) .. (.5,.5);
		\draw[fill=white, color=white] (0,0) circle (.15cm);
	          \draw[myblue]   (.5,-.5) node[below]{$j$} .. controls (.5,0) and (-.5,0) .. (-.5,.5);
  	}
:=
\begin{cases}
\tikzdiagh[scale=.8]{0}{
	          \draw   (-.5,-.5)  node[below]{$i$}.. controls (-.5,0) and (.5,0) .. (.5,.5);
	          \draw[myblue]   (.5,-.5) node[below]{$j$} .. controls (.5,0) and (-.5,0) .. (-.5,.5);
  	} & \text{ if } i\neq j, \\
  r_i\ 
\tikzdiagh[scale=.8]{0}{
	          \draw   (-.5,-.5) node[below]{$i$} -- (-.5,.5)node [midway,tikzdot]{};
	          \draw   (.5,-.5) node[below]{$i$} -- (.5,.5);
  	} 
- r_i\ 
\tikzdiagh[scale=.8]{0}{
	          \draw   (-.5,-.5) node[below]{$i$} -- (-.5,.5);
	          \draw   (.5,-.5) node[below]{$i$} -- (.5,.5)node [midway,tikzdot]{};
  	} 
-
\tikzdiagh[scale=.8]{0}{
	          \draw   (-.5,-.5) node[below]{$i$} .. controls (-.5,0) and (.5,0) .. (.5,.5) ;
	          \draw   (.5,-.5) node[below]{$i$} .. controls (.5,0) and (-.5,0) .. (-.5,.5) node [pos = .8,tikzdot]{} node[pos=.8,xshift=1.5ex,yshift=.75ex] {\small $2$}; 
  	} 
-
\tikzdiagh[scale=.8]{0}{
	          \draw   (-.5,-.5) node[below]{$i$} .. controls (-.5,0) and (.5,0) .. (.5,.5)  node [pos = .8,tikzdot]{} node[pos=.8,xshift=1.5ex,yshift=.75ex] {\small $2$}; 
	          \draw   (.5,-.5) node[below]{$i$} .. controls (.5,0) and (-.5,0) .. (-.5,.5);
  	} 
+
2
\tikzdiagh[scale=.8]{0}{
	          \draw   (-.5,-.5) node[below]{$i$} .. controls (-.5,0) and (.5,0) .. (.5,.5) node [pos = .8,tikzdot]{};
	          \draw   (.5,-.5) node[below]{$i$} .. controls (.5,0) and (-.5,0) .. (-.5,.5) node [pos = .8,tikzdot]{};
  	} 
& \text{ if } i = j.
\end{cases}
\end{align*}

In order to shorten out our diagrams, we introduce the convenient notation
\[
\tikzdiagh[scale=.85]{0}{
	          \draw   (-.5,-.5) node[below]{$i$} -- (-.5,.5) node [midway,tikzdot]{};
	          \draw[dashed] (-.5,0) -- (.5,0);
	          \draw   (.5,-.5) node[below]{$i$} -- (.5,.5) node [midway,tikzdot]{};
  	} 
\ := \ 
\tikzdiagh[scale=.85]{0}{
	          \draw   (-.5,-.5) node[below]{$i$} -- (-.5,.5)node [midway,tikzdot]{};
	          \draw   (.5,-.5) node[below]{$i$} -- (.5,.5);
  	} 
-
\tikzdiagh[scale=.85]{0}{
	          \draw   (-.5,-.5) node[below]{$i$} -- (-.5,.5);
	          \draw   (.5,-.5) node[below]{$i$} -- (.5,.5)node [midway,tikzdot]{};
  	} 
\]
It respects the relation
\begin{equation}\label{eq:doubledashedcommutescross}
\tikzdiagh[scale=.75,yscale=1.25]{0}{
	\draw (-1,-.5) -- (-1,.5) node [pos = .5, tikzdot]{}  node [pos = .8, tikzdot]{};
          \draw   (-.5,-.5) node[below]{$i$} .. controls (-.5,-.2) and (.5,-.2) .. (.5,.5) node [pos = .9,tikzdot]{};
          \draw[dashed] (-1,.32) -- (.35,.32);
          \draw[dashed] (-1,0) -- (-.5,0);
          \draw   (.5,-.5) node[below]{$i$} .. controls (.5,-.2) and (-.5,-.2) .. (-.5,.5) node [pos = .7,tikzdot]{};
  	} 
 \ = \ 
\tikzdiagh[scale=.75,yscale=1.25]{0}{
	\draw (-1,-.5) -- (-1,.5) node [pos = .5, tikzdot]{}  node [pos = .2, tikzdot]{};
          \draw   (-.5,-.5) node[below]{$i$} .. controls (-.5,.2) and (.5,.2) .. (.5,.5) node [pos = .3,tikzdot]{};
          \draw[dashed] (-1,-.32) -- (.35,-.32);
          \draw[dashed] (-1,0) -- (-.5,0);
          \draw   (.5,-.5) node[below]{$i$} .. controls (.5,.2) and (-.5,.2) .. (-.5,.5) node [pos = .1,tikzdot]{};
  	} 
\end{equation}
We also have that
\begin{equation}\label{eq:dashedcirclcross}
\tikzdiagh[scale=.85]{0}{
	          \draw   (-.5,-.5) node[below]{$i$} .. controls (-.5,0) and (.5,0) .. (.5,.5);
		\draw[fill=white, color=white] (0,0) circle (.15cm);
	          \draw  (.5,-.5) node[below]{$i$} .. controls (.5,0) and (-.5,0) .. (-.5,.5);
  	}
 \ = 
 r_i\ 
 \tikzdiagh[scale=.85]{0}{
	          \draw   (-.5,-.5) node[below]{$i$} -- (-.5,.5) node [midway,tikzdot]{};
	          \draw[dashed] (-.5,0) -- (.5,0);
	          \draw   (.5,-.5) node[below]{$i$} -- (.5,.5) node [midway,tikzdot]{};
  	} 
  \ - \ 
	\tikzdiagh[scale=.85]{0}{
	          \draw   (-.5,-.5) node[below]{$i$} .. controls (-.5,0) and (.5,0) .. (.5,.5) node [pos = .7,tikzdot]{} node [pos = .9,tikzdot]{};
	          \draw[dashed] (-.45,.37) -- (.45,.37);
	          \draw[dashed] (-.35,.17) -- (.35,.17);
	          \draw   (.5,-.5) node[below]{$i$} .. controls (.5,0) and (-.5,0) .. (-.5,.5) node [pos = .7,tikzdot]{} node [pos = .9,tikzdot]{};
  	} 
\end{equation}

\begin{lem}\emph{(\cite[Lemma~4.12]{kashiwara})} \label{lem:vcrossings}
Undercrossings respect the following relations:
\begin{align*}
\tikzdiagh[scale=.85]{0}{
	          \draw   (-.5,-.5) node[below]{\small $i$} .. controls (-.5,0) and (.5,0) .. (.5,.5) node [pos = .2,tikzdot]{};
		\draw[fill=white, color=white] (0,0) circle (.15cm);
	          \draw   (.5,-.5) node[below]{\small $j$} .. controls (.5,0) and (-.5,0) .. (-.5,.5);
  	}
&= 
\tikzdiagh[scale=.85]{0}{
	          \draw   (-.5,-.5) node[below]{\small $i$}  .. controls (-.5,0) and (.5,0) .. (.5,.5) node [pos = .8,tikzdot]{};
		\draw[fill=white, color=white] (0,0) circle (.15cm);
	          \draw   (.5,-.5)node[below]{\small $j$} .. controls (.5,0) and (-.5,0) .. (-.5,.5);
  	}
&
\tikzdiagh[scale=.85]{0}{
	          \draw   (-.5,-.5) node[below]{\small $i$} .. controls (-.5,0) and (.5,0) .. (.5,.5); .. controls (-.5,0) and (.5,0) .. (.5,.5);
		\draw[fill=white, color=white] (0,0) circle (.15cm);
	          \draw   (.5,-.5) node[below]{\small $j$} .. controls (.5,0) and (-.5,0) .. (-.5,.5) node [pos = .8,tikzdot]{};
  	}
&= 
\tikzdiagh[scale=.85]{0}{
	          \draw   (-.5,-.5) node[below]{\small $i$} .. controls (-.5,0) and (.5,0) .. (.5,.5);
		\draw[fill=white, color=white] (0,0) circle (.15cm);
	          \draw   (.5,-.5) node[below]{\small $j$} .. controls (.5,0) and (-.5,0) .. (-.5,.5)node [pos = .2,tikzdot]{};
  	}
\end{align*}
\[
\tikzdiagh[scale=.75]{0}{
		\draw  (0,0)node[below] {\small $i$} .. controls (0,0.5) and (2, 1) ..  (2,2);
		\draw  (1,0)node[below] {\small $j$} .. controls (1,0.5) and (0, 0.5) ..  (0,1) .. controls (0,1.5) and (1, 1.5) ..  (1,2);
		\draw[fill=white, color=white] (.5,1.5) circle (.17cm);
		\draw[fill=white, color=white] (1.28,1) circle (.17cm);
		\draw  (2,0)node[below] {\small $k$} .. controls (2,1) and (0, 1.5) ..  (0,2);
	 }
  \ = \  
\tikzdiagh[scale=.75]{0}{
		\draw  (0,0)node[below] {\small $i$} .. controls (0,1) and (2, 1.5) ..  (2,2);
		\draw  (1,0)node[below] {\small $j$} .. controls (1,0.5) and (2, 0.5) ..  (2,1) .. controls (2,1.5) and (1, 1.5) ..  (1,2);
		\draw[fill=white, color=white] (.72,1) circle (.17cm);
		\draw[fill=white, color=white] (1.5,.5) circle (.17cm);
		\draw  (2,0)node[below] {\small $k$} .. controls (2,.5) and (0, 1) ..  (0,2);
	 }
\]
for all $i,j,k \in I$.
\end{lem}

Still as in~\cite{kashiwara}, in order to construct a ``nearly inverse'' for $\bar d_N$, we define the map
\[
P :
\bigoplus_{a = 1}^{m+1} \bigoplus_{\ell \geq 0}\ 
\tikzdiag[xscale=.55]{
	\draw (0,-1.5) -- (0,.5);
	\draw (.5,-1.5)   -- (.5,.5);
	\draw (2,-1.5)  .. controls (2,-1) and (1.5,-1) .. (1.5,-.5) -- (1.5,.5);
	\draw (2.5, -1.5)     .. controls (2.5,-1) and (2,-1) .. (2,-.5) -- (2,.5);
	\node at(1,.4) {\small$\dots$};
	\node at(1,-.4) {\small$\dots$};
	\filldraw [fill=white, draw=black] (-.25,-.25) rectangle (2.25,0.25) node[midway] {\small $H(m)$};
	\draw (1.25, -1.5) 
	 	.. controls (1.25,-1) and (2.5,-1) .. (2.5, -.5) node[pos=.15, tikzdot]{} node[pos=.15, xshift=-1.5ex, yshift=.5ex]{\small $\ell$}
		-- (2.5,.25)
		.. controls (2.5,.5) .. (2.75,.5)   node[right] {$i$};
	\draw[decoration={brace,mirror,raise=-8pt},decorate]  (-.15,-1.85) -- node {$a$} (1.35,-1.85);
} 
\longrightarrow
\bigoplus_{a = 1}^{m+1} \bigoplus_{\ell \geq 0}\ 
\tikzdiag[xscale=.55]{
	\draw (0,-1.5) 
		.. controls (0,-1.25) and (.75,-1.25) ..(.75,-1)
		.. controls (.75,-.75) and (0,-.75) ..(0,-.5)
		 -- (0,.5);
	\draw (.5,-1.5) 
		.. controls (.5,-1.25) and (1.25,-1.25) ..(1.25,-1)
		.. controls (1.25,-.75) and (.5,-.75) ..(.5,-.5)
		 -- (.5,.5);
	\draw (2,-1.5) .. controls (2,-1) and (1.5,-1) .. (1.5,-.5) -- (1.5,.5);
	\draw (2.5, -1.5) .. controls (2.5,-1) and (2,-1) .. (2,-.5) -- (2,.5);
	\node at(1,.4) {\small$\dots$};
	\node at(1,-.4) {\small$\dots$};
	\filldraw [fill=white, draw=black] (-.25,-.25) rectangle (2.25,0.25) node[midway] {\small $H(m)$};
	\draw (1.25, -1.5)  
		.. controls (1.25,-1.25) and (0,-1.25) .. (0, -1)  node[pos=1, tikzdot]{} node[pos=1, xshift=-1.5ex, yshift=.5ex]{\small $\ell$}
		.. controls (0,-.75) and (2.5,-.75) .. (2.5, -.25) 
		-- (2.5,.25)
		.. controls (2.5,.5) .. (2.75,.5)  node[right] {$i$};
	\fdot{}{.4,-1.015};
	\draw[decoration={brace,mirror,raise=-8pt},decorate]  (-.15,-1.85) -- node {$a$} (1.35,-1.85);
}
\]
as multiplication on the left (or diagrammatically stacking above) with the element
\[
\widetilde \theta_{m+1}
\ := \ 
\tikzdiagh[scale = 1.2]{0}{
		\draw  (.25,1) .. controls (.25,.5) and (.75,.5) ..   (.75,0)   .. controls (.75,-.5) and (.25,-.5)  .. (.25,-1); 
		\draw  (.5,1) .. controls (.5,.5) and (1,.5) ..   (1,0)   .. controls (1,-.5) and (.5,-.5)  .. (.5,-1); 
		\draw  (1,1) .. controls (1,.5) and (1.5,.5) ..   (1.5,0)   .. controls (1.5,-.5) and (1,-.5)  .. (1,-1); 
		\draw  (1.25,1) .. controls (1.25,.5) and (1.75,.5) ..   (1.75,0)   .. controls (1.75,-.5) and (1.25,-.5)  .. (1.25,-1); 
	\draw[fill=white, color=white] (.62,-.37) circle (.085cm);
	\draw[fill=white, color=white] (.82,-.44) circle (.085cm);
	\draw[fill=white, color=white] (1.19,-.56) circle (.085cm);
	\draw[fill=white, color=white] (1.38,-.64) circle (.085cm);
	 \draw  (1.75,-1)  node[below]{$i$}.. controls (1.75,-.5) and (.25,-.5) ..    (.25,0)
		 .. controls (.25,.5)  and (1.75,.5).. (1.75,1);
	 \fdot{i}{.45,.05}; 
		 \node at (1.3,0) {\small $\dots$};  \node at (.8,.85) {\small $\dots$};  \node at (.8,-.85) {\small $\dots$};
	 }
\]

\begin{lem}\label{lem:RboMapP}
The map $P$ defined above is a map of $H(R_\bo(m), d_N)$-modules.
\end{lem}

\begin{proof}
We need to verify that $\widetilde \theta_{m+1}$ commutes with the elements in $H(R_\bo(m), d_N)$. 
Crossings and dots slide over the upper part of the $(m+1)$th strand in $\widetilde \theta_{m+1}$ at the cost of adding diagrams with fewer crossings. Because there are fewer crossings, we can slide the floating dot coming from $\widetilde \theta_{m+1}$ to the part $H(R_\bo(m), d_N)$ of the diagram, which gives zero.  
The crossings and dots in the remaining terms then slide over the lower part thanks to \cref{lem:vcrossings}.
Tight floating dots with subscript $j \notin I_f$ also slide over $\widetilde \theta_{m+1}$ thanks to \cref{eq:ExtR2}.
\end{proof}

\begin{lem}\label{lem:computePdN}
The composition $P \circ \bar d_N$ is given on $H(R_\bo(m),d_N)  \otimes_{m} 1_{(\nu,i)}R_\bo(m+1)$ by multiplication by 
\begin{equation}\label{eq:vcpol}
 r_i^{2 \nu_i} \sum_{p = 0}^{2\nu_i- \alpha_i^\vee(\nu)} x_{m+1}^{n_i+p} \varepsilon_p^i(\und x_\nu),
\end{equation}
where $\varepsilon_p^i(\und x_\nu)$ is as in~\cref{eq:epsnotation2}.
\end{lem}

\begin{proof}
The proof is similar to~\cite[Theorem 4.15]{kashiwara}. We have
\begin{equation}~\label{eq:PodN}
\tikzdiag[xscale=.55]{
	\draw (0,-1.5) 
		.. controls (0,-1.25) and (.75,-1.25) ..(.75,-1)
		.. controls (.75,-.75) and (0,-.75) ..(0,-.5)
		 -- (0,.5);
	\draw (.5,-1.5) 
		.. controls (.5,-1.25) and (1.25,-1.25) ..(1.25,-1)
		.. controls (1.25,-.75) and (.5,-.75) ..(.5,-.5)
		 -- (.5,.5);
	\draw (2,-1.5) .. controls (2,-1) and (1.5,-1) .. (1.5,-.5) -- (1.5,.5);
	\draw (2.5, -1.5) .. controls (2.5,-1) and (2,-1) .. (2,-.5) -- (2,.5);
	\node at(1,.4) {\small$\dots$};
	\node at(1,-.4) {\small$\dots$};
	\filldraw [fill=white, draw=black] (-.25,-.25) rectangle (2.25,0.25) node[midway] {\small $H(m)$};
	\draw (1.25, -1.5)  
		.. controls (1.25,-1.25) and (0,-1.25) .. (0, -1)  node[pos=1, tikzdot]{} node[pos=1, xshift=-1.5ex, yshift=.5ex]{\small $\ell$}
		.. controls (0,-.75) and (2.5,-.75) .. (2.5, -.25) 
		-- (2.5,.25)
		.. controls (2.5,.5) .. (2.75,.5)  node[right] {$i$};
	\fdot{}{.4,-1.015};
}
\xmapsto{\ P \circ \bar d_N \ }
\tikzdiag[xscale=.55]{
	\draw (0,-3)
		.. controls (0,-2.75) and (.75,-2.75) .. (.75,-2.5)
		.. controls (.75,-2.125) and (0,-2.125) .. (0,-1.75) 
		.. controls (0,-1.25) and (.75,-1.25) ..(.75,-1)
		.. controls (.75,-.75) and (0,-.75) ..(0,-.5)
		 -- (0,.5);
	\draw (.5,-3)
		.. controls (.5,-2.75) and (1.25,-2.75) .. (1.25,-2.5)
		.. controls (1.25,-2.125) and (.5,-2.125) .. (.5,-1.75) 
		.. controls (.5,-1.25) and (1.25,-1.25) ..(1.25,-1)
		.. controls (1.25,-.75) and (.5,-.75) ..(.5,-.5)
		 -- (.5,.5);
	\draw (2,-3)
		.. controls (2,-2.125) and (1.5,-2.125) .. (1.5,-1.75)
		.. controls (1.5,-1.25) and (2,-1.25) .. (2,-1)
		.. controls (2, -.75) and (1.5, -.75) .. (1.5,-.5)
		-- (1.5,.5);
	\draw (2.5,-3)
		.. controls (2.5,-2.125) and (2,-2.125) .. (2,-1.75)
		.. controls (2,-1.25) and (2.5,-1.25) .. (2.5,-1)
		.. controls (2.5, -.75) and (2, -.75) .. (2,-.5)
		-- (2,.5);
	\node at(1,.4) {\small$\dots$};
	\node at(1,-.4) {\small$\dots$};
	\filldraw [fill=white, draw=black] (-.25,-.25) rectangle (2.25,0.25) node[midway] {\small  $H(m)$};
	\draw[shift only,fill=white, color=white] (.25,-1.24) circle (.1cm);
	\draw[shift only,fill=white, color=white] (.46,-1.31) circle (.1cm);
	\draw[shift only,fill=white, color=white] (.88,-1.44) circle (.1cm);
	\draw[shift only,fill=white, color=white] (1.13,-1.51) circle (.1cm);
	\draw (1.25, -3)  
		.. controls (1.25,-2.75) and (0,-2.75) .. (0,-2.5)  node[pos=1, tikzdot]{} node[pos=1, xshift=-3.5ex, yshift=.5ex]{\small $\ell+n_i$}
		.. controls (0,-2.125) and (2.5,-2.125) .. (2.5,-1.75)
		.. controls (2.5,-1.375) and (0,-1.375) .. (0, -1) 
		.. controls (0,-.625) and (2.5,-.625) .. (2.5, -.25) 
		-- (2.5,.25)
		.. controls (2.5,.5) .. (2.75,.5)  node[right] {$i$};
	\fdot{}{.4,-1};
}
\end{equation}
We prove by induction on the number of strands $m$ that 
\[
\tikzdiag[xscale=.55]{
	\draw (.75,-2.5)
		.. controls (.75,-2.125) and (0,-2.125) .. (0,-1.75) 
		.. controls (0,-1.375) and (.75,-1.375) ..(.75,-1);
	\draw (1.25,-2.5)
		.. controls (1.25,-2.125) and (.5,-2.125) .. (.5,-1.75) 
		.. controls (.5,-1.375) and (1.25,-1.375) ..(1.25,-1);
	\draw (2,-2.5)
		.. controls (2,-2.125) and (1.5,-2.125) .. (1.5,-1.75)
		.. controls (1.5,-1.375) and (2,-1.375) .. (2,-1);
		% .. controls (2,-1) and (1.5,-1) .. (1.5,-.5) 
		%-- (1.5,.5);
	%
	\draw (2.5,-2.5)
		.. controls (2.5,-2.125) and (2,-2.125) .. (2,-1.75)
		.. controls (2,-1.375) and (2.5,-1.375) .. (2.5,-1);
	%\draw (2.5, -1.5) .. controls (2.5,-1) and (2,-1) .. (2,-.5) -- (2,.5);
	\node at(1,-1.75) {\tiny$\dots$};
	\draw[shift only,fill=white, color=white] (.30,-1.26) circle (.1cm);
	\draw[shift only,fill=white, color=white] (.52,-1.33) circle (.1cm);
	\draw[shift only,fill=white, color=white] (.91,-1.45) circle (.1cm);
	\draw[shift only,fill=white, color=white] (1.15,-1.52) circle (.1cm);
	\draw  (0,-2.5)  
		.. controls (0,-2.125) and (2.5,-2.125) .. (2.5,-1.75) node[pos=.2, tikzdot]{} node[pos=.2, xshift=-2ex, yshift=.5ex]{\small $n_i$}
		.. controls (2.5,-1.375) and (0,-1.375) .. (0, -1);
}
\ \equiv r_i^{2 \nu_i}
 \sum_{p = 0}^{2\nu_i- \alpha_i^\vee(\nu)} x_{m+1}^{n_i+p} \varepsilon_p^i(\und x_\nu), 
\]
where $\equiv$ means equality up to adding elements killed in the quotient $H(R_\bo(m), d_N) \cong R_\p^N(m)$. 
If $m = 0$, then it is trivial. Thus we suppose by induction that it holds for $m-1$.
We fix the label of the strands on the diagram above as $i \bj$ with $\bj = j_1\cdots j_m \in \Seq(\nu)$, and we consider the different possible cases.

If $j_m \neq i$, then the result follows by applying \cref{eq:KLRR2} with \cref{lem:vcrossings}, and using the induction hypothesis.

If $j_m = i$, we first observe that
\begin{equation}\label{eq:2croscircl}
\tikzdiagh[scale=.75]{0}{;
 	\draw (1,-.75) node[below] {\small $i$} .. controls (1,-.375) and (0,-.375) .. (0,0) .. controls (0,.375) and (1, .375) .. (1,.75);
	\draw[fill=white, color=white] (.5,.375) circle (.15cm);
	\draw  (0,-.75) node[below] {\small $i$} .. controls (0,-.375) and (1,-.375) .. (1,0) .. controls (1,.375) and (0, .375) .. (0,.75)
}
\ = r_i\ 
\tikzdiagh[scale=.75]{0}{
	          \draw   (-.5,-.5) node[below]{$i$} .. controls (-.5,0) and (.5,0) .. (.5,.5) node [pos = .8,tikzdot]{};
	          \draw[dashed] (-.45,.3) -- (.45,.3);
	          \draw   (.5,-.5) node[below]{$i$} .. controls (.5,0) and (-.5,0) .. (-.5,.5) node [pos = .8,tikzdot]{};
  	} 
\end{equation}
Then, we need to consider $j_{m-1}$. If $m = 1$, we have that 
\begin{align*}
\tikzdiagh[scale=.75]{0}{
 	\draw (1,-.75) node[below] {\small $i$} .. controls (1,-.375) and (0,-.375) .. (0,0) .. controls (0,.375) and (1, .375) .. (1,.75);
	\draw[fill=white, color=white] (.5,.375) circle (.15cm);
	\draw  (0,-.75) node[below] {\small $i$} .. controls (0,-.375) and (1,-.375) .. (1,0) node [pos = .2,tikzdot]{} node [pos = .2, xshift=-2ex, yshift=.5ex] {\small $n_i$} 
	 .. controls (1,.375) and (0, .375) .. (0,.75);
}
\ &= r_i \ 
\tikzdiagh[scale=.75]{0}{
	          \draw   (-.5,-.5) node[below]{$i$} .. controls (-.5,0) and (.5,0) .. (.5,.5) 
		node [pos = .2,tikzdot]{} node [pos = .2, xshift=-2ex, yshift=.5ex] {\small $n_i$};
	          \draw   (.5,-.5) node[below]{$i$} .. controls (.5,0) and (-.5,0) .. (-.5,.5) node [pos = .8,tikzdot]{};
  	} 
\ - r_i \ 
\tikzdiagh[scale=.75]{0}{
	          \draw   (-.5,-.5) node[below]{$i$} .. controls (-.5,0) and (.5,0) .. (.5,.5) node [pos = .8,tikzdot]{} 
		node [pos = .2,tikzdot]{} node [pos = .2, xshift=-2ex, yshift=.5ex] {\small $n_i$};
	          \draw   (.5,-.5) node[below]{$i$} .. controls (.5,0) and (-.5,0) .. (-.5,.5);
  	} 
  \\
\ &= r_i\ 
\tikzdiagh[scale=.75]{0}{
	          \draw   (-.5,-.5) node[below]{$i$} .. controls (-.5,0) and (.5,0) .. (.5,.5) 
		node [pos = .8,tikzdot]{} node [pos = .8, xshift=2ex, yshift=.5ex] {\small $n_i$};
	          \draw   (.5,-.5) node[below]{$i$} .. controls (.5,0) and (-.5,0) .. (-.5,.5) node [pos = .8,tikzdot]{};
  	} 
\ -r_i \ 
\tikzdiagh[scale=.75]{0}{
	          \draw   (-.5,-.5) node[below]{$i$} .. controls (-.5,0) and (.5,0) .. (.5,.5) node [pos = .8,tikzdot]{}  node [pos = .8, xshift=3.5ex, yshift=.5ex] {\small $n_i{+}1$};
	          \draw   (.5,-.5) node[below]{$i$} .. controls (.5,0) and (-.5,0) .. (-.5,.5);
  	} 
\ + r_i^2 \ 
\tikzdiagh[scale=.85]{0}{
	          \draw   (-.5,-.5) node[below]{$i$} -- (-.5,.5) node [midway,tikzdot]{}  node [midway, xshift=2ex, yshift=.5ex] {\small $n_i$};
	          \draw   (.5,-.5) node[below]{$i$} -- (.5,.5);
  	} 
\ - r_i^2 \ 
\tikzdiagh[scale=.85]{0}{
	          \draw   (-.5,-.5) node[below]{$i$} -- (-.5,.5);
	          \draw   (.5,-.5) node[below]{$i$} -- (.5,.5)node [midway,tikzdot]{} node [midway, xshift=2ex, yshift=.5ex] {\small $n_i$};
  	} 
\end{align*}
Moreover, we observe that
\[
\tikzdiagh[scale=.75]{0}{
	%\filldraw [fill=white, draw=black] (-.75,.5) rectangle (.75,1.25) node[midway] {$H(m)$};
          \draw   (-.5,-.5) node[below]{$i$} .. controls (-.5,0) and (.5,0) .. (.5,.5);
          \draw   (.5,-.5) node[below]{$i$} .. controls (.5,0) and (-.5,0) .. (-.5,.5) 
	node [pos = .2,tikzdot]{} node [pos = .2, xshift=2ex, yshift=.5ex] {\small $n_i$};
	\fdot{}{0,-.3};
  	} 
  \equiv
  0,
\]
which finishes the case $m=1$. 
For $j_{m-1} = i$, 
we have
\begin{align*}
\tikzdiagh[xscale=.55,yscale=1.25]{0}{
	\draw (.75,-2.5)
		.. controls (.75,-2.125) and (0,-2.125) .. (0,-1.75) 
		.. controls (0,-1.375) and (.75,-1.375) ..(.75,-1);
	\draw (1.25,-2.5)
		.. controls (1.25,-2.125) and (.5,-2.125) .. (.5,-1.75) 
		.. controls (.5,-1.375) and (1.25,-1.375) ..(1.25,-1);
	\node at (1,-1.75) {\tiny$\dots$};
	\draw (2,-2.5) node[below]{\small$i$}
		.. controls (2,-2.125) and (1.5,-2.125) .. (1.5,-1.75)
		.. controls (1.5,-1.375) and (2,-1.375) .. (2,-1);
	\draw (2.5,-2.5)  node[below]{\small$i$}
		.. controls (2.5,-2.125) and (2,-2.125) .. (2,-1.75)
		.. controls (2,-1.375) and (2.5,-1.375) .. (2.5,-1);
	%\node at(1,-1.75) {\tiny$\dots$};
%
	\draw[shift only,fill=white, color=white] (.31,-1.58) circle (.1cm);
	\draw[shift only,fill=white, color=white] (.53,-1.67) circle (.1cm);
	\draw[shift only,fill=white, color=white] (.92,-1.8) circle (.1cm);
	\draw[shift only,fill=white, color=white] (1.15,-1.9) circle (.1cm);
	\draw  (0,-2.5)   node[below]{\small$i$}
		.. controls (0,-2.125) and (2.5,-2.125) .. (2.5,-1.75) node[pos=.2, tikzdot]{} node[pos=.2, xshift=-2ex, yshift=.5ex]{\small $n_i$}
		.. controls (2.5,-1.375) and (0,-1.375) .. (0, -1);
}
\ = r_i^2 \ 
\tikzdiagh[xscale=.55,yscale=1.25]{0}{
	\draw (.75,-2.5)
		.. controls (.75,-2.25) and (0,-2.25) .. (0,-2) 
		.. controls (0,-1.75) and (.75,-1.75) ..(.75,-1.5)
		-- (.75,-1);
	\draw (1.25,-2.5)
		.. controls (1.25,-2.25) and (.55,-2.25) .. (.55,-2) 
		.. controls (.55,-1.75) and (1.25,-1.75) ..(1.25,-1.5)
		-- (1.25,-1);
	\draw[shift only,fill=white, color=white] (.15,-2.23) circle (.1cm);
	\draw[shift only,fill=white, color=white] (.34,-2.38) circle (.1cm);
	\draw (2,-2.5) node[below]{\small$i$}
		.. controls (2,-2) and (0,-2) .. (0,-1.5) 
		-- (0,-1) node[pos = .3,tikzdot]{}  node[pos = .7,tikzdot]{};
	\draw (2.5,-2.5)  node[below]{\small$i$}
		.. controls (2.5,-1.75) and (2,-1.75) .. (2,-1.5)
		-- (2,-1)  node[pos = .3,tikzdot]{};
	\draw  (0,-2.5)  node[below]{\small$i$}
		.. controls (0,-2) and (2.5,-2) .. (2.5,-1.5) node[pos=.1, tikzdot]{} node[pos=.1, xshift=-2ex, yshift=.5ex]{\small $n_i$}
		-- (2.5,-1)   node[pos = .7,tikzdot]{};
	 \draw[dashed] (0,-1.15) -- (2.5,-1.15);
	 \draw[dashed] (0,-1.35) -- (2,-1.35);
}
\ - r_i \ 
\tikzdiagh[xscale=.55,yscale=1.25]{0}{
	\draw (.75,-2.5)
		.. controls (.75,-2.25) and (0,-2.25) .. (0,-2) 
		.. controls (0,-1.75) and (.75,-1.75) ..(.75,-1.5)
		-- (.75,-1);
	\draw (1.25,-2.5)
		.. controls (1.25,-2.25) and (.55,-2.25) .. (.55,-2) 
		.. controls (.55,-1.75) and (1.25,-1.75) ..(1.25,-1.5)
		-- (1.25,-1);
	\draw[shift only,fill=white, color=white] (.29,-2.11) circle (.1cm);
	\draw[shift only,fill=white, color=white] (.5,-2.2) circle (.1cm);
	\draw (2,-2.5) node[below]{\small$i$}
		.. controls (2,-2.25) and (1.25,-2.25) .. (1.25,-2) 
		.. controls (1.25,-1.75) and (2,-1.75) .. (2,-1.5) 
		-- (2,-1)  node[pos = .0,tikzdot]{}   node[pos = .4,tikzdot]{} ;
	\draw (2.5,-2.5)  node[below]{\small$i$}
		.. controls (2.5,-1.75) and (0,-1.75) .. (0,-1.5)
		-- (0,-1)  node[pos = .0,tikzdot]{}  node[pos = .4,tikzdot]{}   node[pos = .8,tikzdot]{};
	\draw  (0,-2.5)  node[below]{\small$i$}
		.. controls (0,-2.25) and (2.5,-2.25) .. (2.5,-1.5) node[pos=.2, tikzdot]{} node[pos=.2, xshift=-2ex, yshift=.5ex]{\small $n_i$}
		-- (2.5,-1)   node[pos = .8,tikzdot]{};
	 \draw[dashed] (0,-1.1) -- (2.5,-1.1);
	 \draw[dashed] (0,-1.3) -- (2,-1.3);
	 \draw[dashed] (0,-1.5) -- (2,-1.5);
}
\end{align*}
using~\cref{eq:2croscircl}, \cref{eq:dashedcirclcross} and \cref{lem:vcrossings}. 
Using~\cref{eq:doubledashedcommutescross} followed by \cref{lem:vcrossings} and \cref{eq:2croscircl} we obtain
\[
r_i^2\ 
\tikzdiagh[xscale=.55,yscale=1.25]{0}{
	\draw (.75,-2.5)
		.. controls (.75,-2.25) and (0,-2.25) .. (0,-2) 
		.. controls (0,-1.75) and (.75,-1.75) ..(.75,-1.5)
		-- (.75,-1);
	\draw (1.25,-2.5)
		.. controls (1.25,-2.25) and (.55,-2.25) .. (.55,-2) 
		.. controls (.55,-1.75) and (1.25,-1.75) ..(1.25,-1.5)
		-- (1.25,-1);
	\draw[shift only,fill=white, color=white] (.15,-2.23) circle (.1cm);
	\draw[shift only,fill=white, color=white] (.34,-2.38) circle (.1cm);
	\draw (2,-2.5) node[below]{\small$i$}
		.. controls (2,-2) and (0,-2) .. (0,-1.5) 
		-- (0,-1) node[pos = .3,tikzdot]{}  node[pos = .7,tikzdot]{};
	\draw (2.5,-2.5)  node[below]{\small$i$}
		.. controls (2.5,-1.75) and (2,-1.75) .. (2,-1.5)
		-- (2,-1)  node[pos = .3,tikzdot]{};
	\draw  (0,-2.5)  node[below]{\small$i$}
		.. controls (0,-2) and (2.5,-2) .. (2.5,-1.5) node[pos=.1, tikzdot]{} node[pos=.1, xshift=-2ex, yshift=.5ex]{\small $n_i$}
		-- (2.5,-1)   node[pos = .7,tikzdot]{};
	 \draw[dashed] (0,-1.15) -- (2.5,-1.15);
	 \draw[dashed] (0,-1.35) -- (2,-1.35);
}
\ = r_i \ 
\tikzdiagh[xscale=.55,yscale=1.25]{0}{
	\draw (.75,-2.5)
		.. controls (.75,-2.25) and (0,-2.25) .. (0,-2) 
		.. controls (0,-1.75) and (.75,-1.75) ..(.75,-1.5)
		-- (.75,-1);
	\draw (1.25,-2.5)
		.. controls (1.25,-2.25) and (.55,-2.25) .. (.55,-2) 
		.. controls (.55,-1.75) and (1.25,-1.75) ..(1.25,-1.5)
		-- (1.25,-1);
	\draw (2,-2.5) node[below]{\small$i$}
		.. controls (2,-2.25) and (1.25,-2.25) .. (1.25,-2) 
		.. controls (1.25,-1.75) and (2,-1.75) .. (2,-1.5)
		.. controls (2,-1.25) and (2.5,-1.25) .. (2.5,-1);
	\draw[shift only,fill=white, color=white] (.29,-2.11) circle (.1cm);
	\draw[shift only,fill=white, color=white] (.52,-2.19) circle (.1cm);
	\draw[shift only,fill=white, color=white] (.8,-2.27) circle (.1cm);
	\draw (2.5,-2.5)  node[below]{\small$i$}
		-- (2.5,-1.5)
		.. controls (2.5,-1.25) and (2,-1.25) .. (2,-1)  node[pos = 0,tikzdot]{};
	\draw  (0,-2.5)  node[below]{\small$i$}
		.. controls (0,-2.25) and (2,-2.25) .. (2,-2) node[pos=.2, tikzdot]{} node[pos=.2, xshift=-2ex, yshift=.5ex]{\small $n_i$}
		.. controls (2,-1.75) and (0,-1.75) .. (0,-1.5)
		-- (0,-1)   node[pos = 0,tikzdot]{};
	 \draw[dashed] (0,-1.5) -- (2.5,-1.5);
}
\]
Keeping in mind~\cref{eq:PodN}, we have
\[
\tikzdiagh[xscale=.55,yscale=1.25]{0}{
	\draw (0,-1.5)   node[below]{\small$i$} -- (0,-1) 
		.. controls (0,-.625) and (2.5,-.625) .. (2.5,-.25);
	\draw (.75,-1.5) -- (.75,-1)
		.. controls (.75,-.625) and (0,-.625) .. (0, -.25);
	\draw (1.25,-1.5) -- (1.25,-1)
		.. controls (1.25,-.625) and (.5,-.625) .. (.5, -.25);
	\draw (2,-1.5)   node[below]{\small$i$} .. controls (2,-1.25) and (2.5,-1.25) .. (2.5,-1)
		.. controls (2.5,-.625) and (1.75,-.625) .. (1.75, -.25);
	\draw (2.5,-1.5)  node[below]{\small$i$} .. controls (2.5,-1.25) and (2,-1.25) .. (2,-1) node[pos=.15, tikzdot]{}
		.. controls (2,-.625) and (1.25,-.625) .. (1.25, -.25);
	\fdot{}{.4,-1};
}
\ \equiv \ 
\tikzdiagh[xscale=.55,yscale=1.25]{0}{
	\draw (0,-1.5)    node[below]{\small$i$}
		.. controls (0,-1.125) and (2.5,-1.125) .. (2.5,-.75)
		-- (2.5,-.25);
	\draw  (.75,-1.5)
		.. controls (.75,-1.125) and (0,-1.125) .. (0, -.75)
		-- (0,-.25);
	\draw (1.25,-1.5)
		.. controls (1.25,-1.125) and (.5,-1.125) .. (.5, -.75)
		-- (.5,-.25);
	\draw (2.5,-1.5)   node[below]{\small$i$}
		.. controls (2.5,-1.125) and (1.75,-1.125) .. (1.75, -.75)
		.. controls (1.75,-.5) and (1.25,-.5) .. (1.25,-.25) node[pos=.15, tikzdot]{};
	\draw (2,-1.5)   node[below]{\small$i$}
		.. controls (2,-1.125) and (1.25,-1.125) .. (1.25, -.75)
		.. controls (1.25,-.5) and (1.75,-.5) .. (1.75,-.25);
	\fdot{}{.4,-1.4};
}
\ \in \ 
\tikzdiagh[xscale=.55,yscale=1.25]{-1.8ex}{
	\draw (0,-1) 
		.. controls (0,-.625) and (2.5,-.625) .. (2.5,-.25)
		-- (2.5,0)
		.. controls (2.5,.25) .. (2.75,.25)  node[right] {$i$};
	\draw (.75,-1)
		.. controls (.75,-.625) and (0,-.625) .. (0, -.25)
		-- (0,0);
	\draw (1.25,-1)
		.. controls (1.25,-.625) and (.5,-.625) .. (.5, -.25)
		-- (.5,0);
	\draw (2.5,-1)
		.. controls (2.5,-.625) and (1.75,-.625) .. (1.75, -.25)
		-- (1.75,0);
	\draw  (2,-1) 
		.. controls (2,-.625) and (1.25,-.625) .. (1.25, -.25)
		-- (1.25,0);
	\filldraw [fill=white, draw=black] (-.25,-.25) rectangle (2.25,0.25) node[midway] {\small $H(m)$};
	\fdot{}{.4,-.9};
}
\]
 by~\cref{eq:KLRR3} and \cref{eq:KLRnh}. 
This means we can apply the induction hypothesis to get
\[
r_i \ 
\tikzdiagh[xscale=.55,yscale=1.25]{0}{
	\draw (.75,-2.5)
		.. controls (.75,-2.25) and (0,-2.25) .. (0,-2) 
		.. controls (0,-1.75) and (.75,-1.75) ..(.75,-1.5)
		-- (.75,-1);
	\draw (1.25,-2.5)
		.. controls (1.25,-2.25) and (.55,-2.25) .. (.55,-2) 
		.. controls (.55,-1.75) and (1.25,-1.75) ..(1.25,-1.5)
		-- (1.25,-1);
	\draw (2,-2.5) node[below]{\small$i$}
		.. controls (2,-2.25) and (1.25,-2.25) .. (1.25,-2) 
		.. controls (1.25,-1.75) and (2,-1.75) .. (2,-1.5)
		.. controls (2,-1.25) and (2.5,-1.25) .. (2.5,-1);
	\draw[shift only,fill=white, color=white] (.29,-2.11) circle (.1cm);
	\draw[shift only,fill=white, color=white] (.52,-2.19) circle (.1cm);
	\draw[shift only,fill=white, color=white] (.8,-2.27) circle (.1cm);
	\draw (2.5,-2.5)  node[below]{\small$i$}
		-- (2.5,-1.5)
		.. controls (2.5,-1.25) and (2,-1.25) .. (2,-1)  node[pos = 0,tikzdot]{};
	\draw  (0,-2.5)  node[below]{\small$i$}
		.. controls (0,-2.25) and (2,-2.25) .. (2,-2) node[pos=.2, tikzdot]{} node[pos=.2, xshift=-2ex, yshift=.5ex]{\small $n_i$}
		.. controls (2,-1.75) and (0,-1.75) .. (0,-1.5)
		-- (0,-1)   node[pos = 0,tikzdot]{};
	 \draw[dashed] (0,-1.5) -- (2.5,-1.5);
}
\ \equiv \ 
r_i^3
\tikzdiagh[xscale=.55,yscale=1.25]{0}{
	\draw (.75,-2.5)
		.. controls (.75,-2.25) and (0,-2.25) .. (0,-2) 
		.. controls (0,-1.75) and (.75,-1.75) ..(.75,-1.5)
		-- (.75,-1);
	\draw (1.25,-2.5)
		.. controls (1.25,-2.25) and (.55,-2.25) .. (.55,-2) 
		.. controls (.55,-1.75) and (1.25,-1.75) ..(1.25,-1.5)
		-- (1.25,-1);
	\draw (2,-2.5) node[below]{\small$i$}
		-- (2,-1.5)
		.. controls (2,-1.25) and (2.5,-1.25) .. (2.5,-1);
	\draw[shift only,fill=white, color=white] (.26,-2.15) circle (.1cm);
	\draw[shift only,fill=white, color=white] (.46,-2.24) circle (.1cm);
	%\draw[shift only,fill=white, color=white] (.8,-2.27) circle (.1cm);
	%
	\draw (2.5,-2.5)  node[below]{\small$i$}
		-- (2.5,-1.5)
		.. controls (2.5,-1.25) and (2,-1.25) .. (2,-1)  node[pos = 0,tikzdot]{};
	\draw  (0,-2.5)  node[below]{\small$i$}
		.. controls (0,-2.25) and (1.25,-2.25) .. (1.25,-2) node[pos=.2, tikzdot]{} node[pos=.2, xshift=-2ex, yshift=.5ex]{\small $n_i$}
		.. controls (1.25,-1.75) and (0,-1.75) .. (0,-1.5)
		-- (0,-1)   node[pos = 0,tikzdot]{};
	 \draw[dashed] (0,-1.5) -- (2.5,-1.5);
}
\]

Similarly, we have
\begin{align*}
r_i \ 
\tikzdiagh[xscale=.55,yscale=1.25]{0}{
	\draw (.75,-2.5)
		.. controls (.75,-2.25) and (0,-2.25) .. (0,-2) 
		.. controls (0,-1.75) and (.75,-1.75) ..(.75,-1.5)
		-- (.75,-1);
	\draw (1.25,-2.5)
		.. controls (1.25,-2.25) and (.55,-2.25) .. (.55,-2) 
		.. controls (.55,-1.75) and (1.25,-1.75) ..(1.25,-1.5)
		-- (1.25,-1);
	\draw[shift only,fill=white, color=white] (.29,-2.11) circle (.1cm);
	\draw[shift only,fill=white, color=white] (.5,-2.2) circle (.1cm);
	\draw (2,-2.5) node[below]{\small$i$}
		.. controls (2,-2.25) and (1.25,-2.25) .. (1.25,-2) 
		.. controls (1.25,-1.75) and (2,-1.75) .. (2,-1.5) 
		-- (2,-1)  node[pos = .0,tikzdot]{}   node[pos = .4,tikzdot]{} ;
	\draw (2.5,-2.5)  node[below]{\small$i$}
		.. controls (2.5,-1.75) and (0,-1.75) .. (0,-1.5)
		-- (0,-1)  node[pos = .0,tikzdot]{}  node[pos = .4,tikzdot]{}   node[pos = .8,tikzdot]{};
	\draw  (0,-2.5)  node[below]{\small$i$}
		.. controls (0,-2.25) and (2.5,-2.25) .. (2.5,-1.5) node[pos=.2, tikzdot]{} node[pos=.2, xshift=-2ex, yshift=.5ex]{\small $n_i$}
		-- (2.5,-1)   node[pos = .8,tikzdot]{};
	 \draw[dashed] (0,-1.1) -- (2.5,-1.1);
	 \draw[dashed] (0,-1.3) -- (2,-1.3);
	 \draw[dashed] (0,-1.5) -- (2,-1.5);
}
\ &= r_i \ 
\tikzdiagh[xscale=.55,yscale=1.25]{0}{
	\draw (.75,-2.5)
		.. controls (.75,-2.25) and (0,-2.25) .. (0,-2) 
		.. controls (0,-1.75) and (.75,-1.75) ..(.75,-1.5)
		-- (.75,-1);
	\draw (1.25,-2.5)
		.. controls (1.25,-2.25) and (.55,-2.25) .. (.55,-2) 
		.. controls (.55,-1.75) and (1.25,-1.75) ..(1.25,-1.5)
		-- (1.25,-1);
	\draw[shift only,fill=white, color=white] (.23,-2.17) circle (.1cm);
	\draw[shift only,fill=white, color=white] (.41,-2.27) circle (.1cm);
	\draw (2,-2.5) node[below]{\small$i$}
		.. controls (2,-2.25) and (2.5,-2.25) .. (2.5,-2) 
		.. controls (2.5,-1.75) and (2,-1.75) .. (2,-1.5) 
		-- (2,-1)  node[pos = .0,tikzdot]{}   node[pos = .8,tikzdot]{} ;
	\draw (2.5,-2.5)  node[below]{\small$i$}
		.. controls (2.5,-1.875) and (0,-1.875) .. (0,-1.5)
		-- (0,-1)  node[pos = .0,tikzdot]{}  node[pos = .4,tikzdot]{}   node[pos = .8,tikzdot]{};
	\draw  (0,-2.5)  node[below]{\small$i$}
		.. controls (0,-2.125) and (2.5,-2.125) .. (2.5,-1.5) node[pos=.15, tikzdot]{} node[pos=.15, xshift=-2ex, yshift=.5ex]{\small $n_i$}
		-- (2.5,-1)   node[pos = .4,tikzdot]{};
	 \draw[dashed] (0,-1.1) -- (2,-1.1);
	 \draw[dashed] (0,-1.3) -- (2.5,-1.3);
	 \draw[dashed] (0,-1.5) -- (2,-1.5);
}
\ = \ 
\tikzdiagh[xscale=.55,yscale=1.25]{0}{
	\draw (.75,-2.5)
		.. controls (.75,-2.25) and (0,-2.25) .. (0,-2) 
		.. controls (0,-1.75) and (.75,-1.75) ..(.75,-1.5)
		-- (.75,-1);
	\draw (1.25,-2.5)
		.. controls (1.25,-2.25) and (.55,-2.25) .. (.55,-2) 
		.. controls (.55,-1.75) and (1.25,-1.75) ..(1.25,-1.5)
		-- (1.25,-1);
	\draw (2.5,-2.5) node[below]{\small$i$}
		.. controls (2.5,-2.25) and (1.25,-2.25) .. (1.25,-2) 
		.. controls (1.25,-1.75) and (2,-1.75) .. (2,-1.5)
		.. controls (2,-1.25) and (2.5,-1.25) .. (2.5,-1);
	\draw[shift only,fill=white, color=white] (.29,-2.11) circle (.1cm);
	\draw[shift only,fill=white, color=white] (.52,-2.19) circle (.1cm);
	\draw[shift only,fill=white, color=white] (.8,-2.27) circle (.1cm);
	\draw (2,-2.5)  node[below]{\small$i$}
		.. controls (2,-2.25) and (2.5,-2.25) .. (2.5,-2) 
		-- (2.5,-1.5)
		.. controls (2.5,-1.25) and (2,-1.25) .. (2,-1)  node[pos = 0,tikzdot]{}  node[pos = .8,tikzdot]{};
	\draw  (0,-2.5)  node[below]{\small$i$}
		.. controls (0,-2.25) and (2,-2.25) .. (2,-2) node[pos=.2, tikzdot]{} node[pos=.2, xshift=-2ex, yshift=.5ex]{\small $n_i$}
		.. controls (2,-1.75) and (0,-1.75) .. (0,-1.5)
		-- (0,-1)   node[pos = 0,tikzdot]{}  node[pos = .75,tikzdot]{};
	 \draw[dashed] (0,-1.5) -- (2.5,-1.5);
	 \draw[dashed] (0,-1.125) -- (2,-1.125);
} 
\\
\ &\equiv r_i^2\ 
\tikzdiagh[xscale=.55,yscale=1.25]{0}{
	\draw (.75,-2.5)
		.. controls (.75,-2.25) and (0,-2.25) .. (0,-2) 
		.. controls (0,-1.75) and (.75,-1.75) ..(.75,-1.5)
		-- (.75,-1);
	\draw (1.25,-2.5)
		.. controls (1.25,-2.25) and (.55,-2.25) .. (.55,-2) 
		.. controls (.55,-1.75) and (1.25,-1.75) ..(1.25,-1.5)
		-- (1.25,-1);
	\draw (2.5,-2.5) node[below]{\small$i$}
		.. controls (2.5,-2.25) and (2,-2.25) .. (2,-2) 
		-- (2,-1.5)
		.. controls (2,-1.25) and (2.5,-1.25) .. (2.5,-1);
	%,fill=white, color=white
	\draw[shift only,fill=white, color=white] (.26,-2.15) circle (.1cm);
	\draw[shift only,fill=white, color=white] (.46,-2.24) circle (.1cm);
	\draw (2,-2.5)  node[below]{\small$i$}
		.. controls (2,-2.25) and (2.5,-2.25) .. (2.5,-2) 
		-- (2.5,-1.5)
		.. controls (2.5,-1.25) and (2,-1.25) .. (2,-1)  node[pos = 0,tikzdot]{}  node[pos = .8,tikzdot]{};
	\draw  (0,-2.5)  node[below]{\small$i$}
		.. controls (0,-2.25) and (1.25,-2.25) .. (1.25,-2) node[pos=.2, tikzdot]{} node[pos=.2, xshift=-2ex, yshift=.5ex]{\small $n_i$}
		.. controls (1.25,-1.75) and (0,-1.75) .. (0,-1.5)
		-- (0,-1)   node[pos = 0,tikzdot]{}  node[pos = .75,tikzdot]{};
	 \draw[dashed] (0,-1.5) -- (2.5,-1.5);
	 \draw[dashed] (0,-1.125) -- (2,-1.125);
}
\ = r_i^3 \ 
\tikzdiagh[xscale=.55,yscale=1.25]{0}{
	\draw (.75,-2.5)
		.. controls (.75,-2.25) and (0,-2.25) .. (0,-2) 
		.. controls (0,-1.75) and (.75,-1.75) ..(.75,-1.5)
		-- (.75,-1);
	\draw (1.25,-2.5)
		.. controls (1.25,-2.25) and (.55,-2.25) .. (.55,-2) 
		.. controls (.55,-1.75) and (1.25,-1.75) ..(1.25,-1.5)
		-- (1.25,-1);
	\draw (2,-2.5) node[below]{\small$i$}
		-- (2,-1.5)
		.. controls (2,-1.25) and (2.5,-1.25) .. (2.5,-1);
	%,fill=white, color=white
	\draw[shift only,fill=white, color=white] (.26,-2.15) circle (.1cm);
	\draw[shift only,fill=white, color=white] (.46,-2.24) circle (.1cm);
	\draw (2.5,-2.5)  node[below]{\small$i$}
		-- (2.5,-1.5)
		.. controls (2.5,-1.25) and (2,-1.25) .. (2,-1)  node[pos = .8,tikzdot]{};
	\draw  (0,-2.5)  node[below]{\small$i$}
		.. controls (0,-2.25) and (1.25,-2.25) .. (1.25,-2) node[pos=.2, tikzdot]{} node[pos=.2, xshift=-2ex, yshift=.5ex]{\small $n_i$}
		.. controls (1.25,-1.75) and (0,-1.75) .. (0,-1.5)
		-- (0,-1)  node[pos = .75,tikzdot]{};
	 \draw[dashed] (0,-1.125) -- (2,-1.125);
}
\end{align*}
Putting these two results together and using \cref{eq:KLRnh}, we obtain
\[
\tikzdiagh[xscale=.55,yscale=1.25]{0}{
	\draw (.75,-2.5)
		.. controls (.75,-2.125) and (0,-2.125) .. (0,-1.75) 
		.. controls (0,-1.375) and (.75,-1.375) ..(.75,-1);
	\draw (1.25,-2.5)
		.. controls (1.25,-2.125) and (.5,-2.125) .. (.5,-1.75) 
		.. controls (.5,-1.375) and (1.25,-1.375) ..(1.25,-1);
	\draw (2,-2.5) node[below]{\small$i$}
		.. controls (2,-2.125) and (1.5,-2.125) .. (1.5,-1.75)
		.. controls (1.5,-1.375) and (2,-1.375) .. (2,-1);
	\draw (2.5,-2.5)  node[below]{\small$i$}
		.. controls (2.5,-2.125) and (2,-2.125) .. (2,-1.75)
		.. controls (2,-1.375) and (2.5,-1.375) .. (2.5,-1);
	\draw[shift only,fill=white, color=white] (.31,-1.58) circle (.1cm);
	\draw[shift only,fill=white, color=white] (.53,-1.67) circle (.1cm);
	\draw[shift only,fill=white, color=white] (.92,-1.8) circle (.1cm);
	\draw[shift only,fill=white, color=white] (1.15,-1.9) circle (.1cm);
	\draw  (0,-2.5)  node[below]{\small$i$}
		.. controls (0,-2.125) and (2.5,-2.125) .. (2.5,-1.75) node[pos=.2, tikzdot]{} node[pos=.2, xshift=-2ex, yshift=.5ex]{\small $n_i$}
		.. controls (2.5,-1.375) and (0,-1.375) .. (0, -1);
	\node at (1,-1.75) {\tiny$\dots$};
}
\ \equiv r_i^2 \ 
\tikzdiagh[xscale=.55,yscale=1.25]{0}{
	\draw (.75,-2.5)
		.. controls (.75,-2.125) and (0,-2.125) .. (0,-1.75) 
		.. controls (0,-1.375) and (.75,-1.375) ..(.75,-1);
	\draw (1.25,-2.5)
		.. controls (1.25,-2.125) and (.5,-2.125) .. (.5,-1.75) 
		.. controls (.5,-1.375) and (1.25,-1.375) ..(1.25,-1);
	\draw (2,-2.5) node[below]{\small$i$}
		.. controls (2,-2.125) and (1.5,-2.125) .. (1.5,-1.75)
		.. controls (1.5,-1.375) and (2,-1.375) .. (2,-1);
	\draw (2.5,-2.5)  node[below]{\small$i$}
		-- (2.5,-1);
	%\node at(1,-1.75) {\tiny$\dots$};
%
	\draw[shift only,fill=white, color=white] (.29,-1.6) circle (.1cm);
	\draw[shift only,fill=white, color=white] (.5,-1.7) circle (.1cm);
	\draw[shift only,fill=white, color=white] (.88,-1.89) circle (.1cm);
	\draw  (0,-2.5)  node[below]{\small$i$}
		.. controls (0,-2.125) and (2,-2.125) .. (2,-1.75) node[pos=.2, tikzdot]{} node[pos=.2, xshift=-2ex, yshift=.5ex]{\small $n_i$}
		.. controls (2,-1.375) and (0,-1.375) .. (0, -1);
	\node at (1,-1.75) {\tiny$\dots$};
}
\]
which concludes this case.

For the final case $j_{m-1} = j \neq i$, we compute
\[
\tikzdiagh[scale=.65]{0}{
		\draw[myblue]  (1,0) node[below] {\small $j$} 
			.. controls (1,0.5) and (0, 0.5) ..  (0,1) .. controls (0,1.5) and (1, 1.5) ..  (1,2);
		\draw  (2,0) node[below] {\small $i$}
			.. controls (2,.5) and (1,.5) .. (1,1)
		 	.. controls (1,1.5) and (2, 1.5) ..  (2,2);
		\draw[fill=white, color=white] (.66,1.58) circle (.17cm);
		\draw[fill=white, color=white] (1.35,1.4) circle (.17cm);
		\draw  (0,0) node[below] {\small $i$}
			 .. controls (0,0.5) and (2, .5) ..  (2,1)
			 .. controls (2, 1.5) and (0,1.5) .. (0,2);
	 }
\ = r_i \ 
\tikzdiagh[scale=.65]{0}{
	\draw  (0,0)node[below] {\small $i$} .. controls (0,0.5) and (2, 1) ..  (2,2) node[pos=.925,tikzdot]{};
	\draw[myblue]  (1,0)node[below] {\small $j$} .. controls (1,0.5) and (0, 0.5) ..  (0,1) .. controls (0,1.5) and (1, 1.5) ..  (1,2);
	\draw  (2,0)node[below] {\small $i$} .. controls (2,1) and (0, 1.5) ..  (0,2) node[pos=.85,tikzdot]{};
	 \draw[dashed] (0,1.8) -- (2,1.8);
 }
 \ = r_i \ 
\tikzdiagh[yscale=.65,xscale=-.65]{0}{
	\draw  (0,0)node[below] {\small $i$} .. controls (0,0.5) and (2, 1) ..  (2,2) node[pos=.925,tikzdot]{};
	\draw[myblue]  (1,0)node[below] {\small $j$} .. controls (1,0.5) and (0, 0.5) ..  (0,1) .. controls (0,1.5) and (1, 1.5) ..  (1,2);
	\draw  (2,0)node[below] {\small $i$} .. controls (2,1) and (0, 1.5) ..  (0,2) node[pos=.85,tikzdot]{};
	 \draw[dashed] (0,1.8) -- (2,1.8);
 }
 \ + 
 r_i^2 \sum\limits_{t,v} s_{ij}^{tv} \sssum{u+\ell=\\t-1}\ 
 \tikzdiagh[scale=.65]{0}{
	\draw  (0,0)node[below] {\small $i$} -- (0,2)
		 node[pos=.8,tikzdot]{}  node[pos=.3,tikzdot]{} node[pos=.3,xshift=1.5ex, yshift=.5ex] {\small $u$};
	\draw[myblue]  (1,0)node[below] {\small $j$} --  (1,2) 
		 node[pos=.3,tikzdot]{} node[pos=.3,xshift=1.5ex, yshift=.5ex] {\small $v$};
	\draw  (2,0)node[below] {\small $i$} -- (2,2) 
		node[pos=.8,tikzdot]{}  node[pos=.3,tikzdot]{} node[pos=.3,xshift=1.5ex, yshift=.5ex] {\small $\ell$};
	 \draw[dashed] (0,1.6) -- (2,1.6);
 }
\]
using \cref{eq:KLRR3}. 
Then we obtain for the first term on the RHS of the second equality, using the induction hypothesis together with~\cref{eq:KLRR2} 
\[
r_i \ 
\tikzdiagh[yscale=.65,xscale=-.65]{0}{
	\draw  (0,0)node[below] {\small $i$} .. controls (0,0.5) and (2, 1) ..  (2,2) node[pos=.925,tikzdot]{};
	\draw[myblue]  (1,0)node[below] {\small $j$} .. controls (1,0.5) and (0, 0.5) ..  (0,1) .. controls (0,1.5) and (1, 1.5) ..  (1,2);
	\draw  (2,0)node[below] {\small $i$} .. controls (2,1) and (0, 1.5) ..  (0,2) node[pos=.85,tikzdot]{};
	 \draw[dashed] (0,1.8) -- (2,1.8);
 }
 \ = \ 
 \tikzdiagh[yscale=.65, xscale=-.65]{0}{
		\draw[myblue]  (1,0) node[below] {\small $j$} 
			.. controls (1,0.5) and (0, 0.5) ..  (0,1) .. controls (0,1.5) and (1, 1.5) ..  (1,2);
		\draw  (0,0) node[below] {\small $i$}
			 .. controls (0,0.5) and (2, .5) ..  (2,1)
			 .. controls (2, 1.5) and (0,1.5) .. (0,2);
		%\draw[fill=white, color=white] (.66,1.58) circle (.17cm);
		\draw[fill=white, color=white] (1.35,1.4) circle (.17cm);
		\draw  (2,0) node[below] {\small $i$}
			.. controls (2,.5) and (1,.5) .. (1,1)
		 	.. controls (1,1.5) and (2, 1.5) ..  (2,2);
	 }
\ \equiv r_i^2 \  
 \tikzdiagh[scale=.65]{0}{
	\draw  (0,0)node[below] {\small $i$} 
		-- (0,2);
	\draw[myblue]  (1,0)node[below] {\small $j$}
		.. controls (1,.5) and (2,.5) .. (2,1)
		.. controls (2,1.5) and (1,1.5) .. (1,2);
	\draw  (2,0)node[below] {\small $i$} 
		.. controls (2,.5) and (1,.5) .. (1,1)
		.. controls (1,1.5) and (2,1.5) .. (2,2);
 }
 \ = \ 
 r_i^2 \sum_{t,v} s^{tv}_{ij}\ 
 \tikzdiagh[scale=.65]{0}{
	\draw  (0,0)node[below] {\small $i$} -- (0,2);
	\draw[myblue]  (1,0)node[below] {\small $j$} --  (1,2) 
		 node[pos=.5,tikzdot]{} node[pos=.5,xshift=1.5ex, yshift=.5ex] {\small $v$};
	\draw  (2,0)node[below] {\small $i$} -- (2,2) 
		node[pos=.5,tikzdot]{} node[pos=.5,xshift=1.5ex, yshift=.5ex] {\small $t$};
 }
\]
On the other hand, we have for all $t,v$ that
\[
\sssum{u+\ell=\\t-1}\ 
 \tikzdiagh[scale=.65]{0}{
	\draw  (0,0)node[below] {\small $i$} -- (0,2)
		 node[pos=.8,tikzdot]{}  node[pos=.3,tikzdot]{} node[pos=.3,xshift=1.5ex, yshift=.5ex] {\small $u$};
	\draw[myblue]  (1,0)node[below] {\small $j$} --  (1,2) 
		 node[pos=.3,tikzdot]{} node[pos=.3,xshift=1.5ex, yshift=.5ex] {\small $v$};
	\draw  (2,0)node[below] {\small $i$} -- (2,2) 
		node[pos=.8,tikzdot]{}  node[pos=.3,tikzdot]{} node[pos=.3,xshift=1.5ex, yshift=.5ex] {\small $\ell$};
	 \draw[dashed] (0,1.6) -- (2,1.6);
 }
 \ = \ 
 \tikzdiagh[scale=.65]{0}{
	\draw  (0,0)node[below] {\small $i$} -- (0,2)
		 node[pos=.5,tikzdot]{} node[pos=.5,xshift=1.5ex, yshift=.5ex] {\small $t$};
	\draw[myblue]  (1,0)node[below] {\small $j$} --  (1,2) 
		 node[pos=.5,tikzdot]{} node[pos=.5,xshift=1.5ex, yshift=.5ex] {\small $v$};
	\draw  (2,0)node[below] {\small $i$} -- (2,2); 
 }
 \ - \
 \tikzdiagh[scale=.65]{0}{
	\draw  (0,0)node[below] {\small $i$} -- (0,2);
	\draw[myblue]  (1,0)node[below] {\small $j$} --  (1,2) 
		 node[pos=.5,tikzdot]{} node[pos=.5,xshift=1.5ex, yshift=.5ex] {\small $v$};
	\draw  (2,0)node[below] {\small $i$} -- (2,2) 
		 node[pos=.5,tikzdot]{} node[pos=.5,xshift=1.5ex, yshift=.5ex] {\small $t$};
 } 
\]
Putting these results together with the case $j_m \neq i$ yields
\[
\tikzdiagh[xscale=.55,yscale=1.25]{0}{
	\draw (.75,-2.5)
		.. controls (.75,-2.125) and (0,-2.125) .. (0,-1.75) 
		.. controls (0,-1.375) and (.75,-1.375) ..(.75,-1);
	\draw (1.25,-2.5)
		.. controls (1.25,-2.125) and (.5,-2.125) .. (.5,-1.75) 
		.. controls (.5,-1.375) and (1.25,-1.375) ..(1.25,-1);
	\draw[myblue] (2,-2.5) node[below]{\small$j$}
		.. controls (2,-2.125) and (1.5,-2.125) .. (1.5,-1.75)
		.. controls (1.5,-1.375) and (2,-1.375) .. (2,-1);
	\draw (2.5,-2.5)  node[below]{\small$i$}
		.. controls (2.5,-2.125) and (2,-2.125) .. (2,-1.75)
		.. controls (2,-1.375) and (2.5,-1.375) .. (2.5,-1);
	\draw[shift only,fill=white, color=white] (.31,-1.58) circle (.1cm);
	\draw[shift only,fill=white, color=white] (.53,-1.67) circle (.1cm);
	\draw[shift only,fill=white, color=white] (.92,-1.8) circle (.1cm);
	\draw[shift only,fill=white, color=white] (1.15,-1.9) circle (.1cm);
	\draw  (0,-2.5)  node[below]{\small$i$}
		.. controls (0,-2.125) and (2.5,-2.125) .. (2.5,-1.75) node[pos=.2, tikzdot]{} node[pos=.2, xshift=-2ex, yshift=.5ex]{\small $n_i$}
		.. controls (2.5,-1.375) and (0,-1.375) .. (0, -1);
}
\ \equiv r_i^2 \ 
\tikzdiagh[xscale=.55,yscale=1.25]{0}{
	\draw (.75,-2.5)
		.. controls (.75,-2.125) and (0,-2.125) .. (0,-1.75) 
		.. controls (0,-1.375) and (.75,-1.375) ..(.75,-1);
	\draw (1.25,-2.5)
		.. controls (1.25,-2.125) and (.5,-2.125) .. (.5,-1.75) 
		.. controls (.5,-1.375) and (1.25,-1.375) ..(1.25,-1);
	\draw[myblue] (2,-2.5) node[below]{\small$j$}
		.. controls (2,-2.125) and (1.5,-2.125) .. (1.5,-1.75)
		.. controls (1.5,-1.375) and (2,-1.375) .. (2,-1);
	\draw (2.5,-2.5)  node[below]{\small$i$}
		-- (2.5,-1);
	%\node at(1,-1.75) {\tiny$\dots$};
%
	\draw[shift only,fill=white, color=white] (.29,-1.6) circle (.1cm);
	\draw[shift only,fill=white, color=white] (.5,-1.7) circle (.1cm);
	\draw[shift only,fill=white, color=white] (.88,-1.89) circle (.1cm);
	\draw  (0,-2.5)  node[below]{\small$i$}
		.. controls (0,-2.125) and (2,-2.125) .. (2,-1.75) node[pos=.2, tikzdot]{} node[pos=.2, xshift=-2ex, yshift=.5ex]{\small $n_i$}
		.. controls (2,-1.375) and (0,-1.375) .. (0, -1);
}
\]
which concludes the proof.
\end{proof}

\begin{proof}[Proof of \cref{prop:sesleftdecomp}]
The polynomial~\cref{eq:vcpol} is monic (up to invertible scalar) with leading terms $x_{m+1}^{n_i+2\nu_i-\alpha_i^\vee(\nu)}$. Therefore, multiplication by \cref{eq:vcpol} yields an injective map. Thus, \cref{lem:computePdN} tells us that $P \circ \bar d_N$ is injective, and so is $\bar d_N$.
\end{proof}

As a consequence, this also ends the proof of \cref{thm:RbodNformal}.

%%%%%%%%%%%%%%%%	End of file	%%%%%%%%%%%%%

%% file: sections/cataction.tex
%%%%%%%%%%%%%%%%%%%%%%%%%%%%%%%%%%%%
%                 					  				  		 %
%	Categorical action				 					 %
%                 					  						 %
%%%%%%%%%%%%%%%%%%%%%%%%%%%%%%%%%%%%

\section{Categorical action}\label{sec:cataction}

For each $i \in I$ there is a (non-unital) inclusion $R_\bo(m) \hookrightarrow R_\bo(m+1) 1_{(m,i)}$,  given by adding a vertical strand with label $i$ to the right of a diagram $D \in R_\bo(m)$:
\[
\tikzdiagh[xscale=.75]{0}{
	\draw (0,0) node[below] { \small $j_1$} -- (0,1);
	\draw (1,0) node[below] { \small $j_2$} -- (1,1);
	\draw (2,0) -- (2,1);
	\node at(2,-.25) {$\dots$};
	\draw (3,0) node[below] { \small $j_m$} -- (3,1);
	\filldraw [fill=white, draw=black] (-.25,.8) rectangle (3.25,0.2) node[midway] {$D$};
}
\ \mapsto \ 
\tikzdiagh[xscale=.75]{0}{
	\draw (0,0) node[below] { \small $j_1$} -- (0,1);
	\draw (1,0) node[below] { \small $j_2$} -- (1,1);
	\draw (2,0) -- (2,1);
	\node at(2,-.25) {$\dots$};
	\draw (3,0) node[below] { \small $j_m$} -- (3,1);
	\filldraw [fill=white, draw=black] (-.25,.8) rectangle (3.25,0.2) node[midway] {$D$};
	\draw (4,0) node[below] { \small $i$} -- (4,1);
}
\]
This gives rise to induction and restriction functors
\begin{align*}
\Ind_m^{m+i} &: R_\bo(m)\amod \rightarrow R_\bo(m+1)\amod, \\ 
 &\Ind_m^{m+i}(-) \cong R_\bo(m+1) 1_{(m,i)} \otimes_m -,\\
\Res_m^{m+i} &: R_\bo(m+1)\amod \rightarrow R_\bo(m)\amod, \\
 &\Res_m^{m+i}(-) \cong 1_{(m,i)} R_\bo(m+1) \otimes_{m+1} -.
\end{align*}
which are adjoint.

\smallskip

We  write 
\[
R^{\xi_i}_\bo(\nu) := R_\bo(\nu) \otimes \Bbbk[\xi_i] \cong \bigoplus_{\ell \geq 0} q_i^{2\ell}  R_\bo(\nu),
\]
with $\deg_q(\xi_i) = (\alpha_i | \alpha_i)$. 
We will prove the following theorem in the next subsection:

\begin{thm} \label{thm:SESRbo}
There is a short exact sequence
\begin{align*}
0 
&\rightarrow q_i^{-2} R_\bo(\nu) 1_{(\nu-i,i)} \otimes_{m-1} 1_{(\nu-i,i)} R_\bo(\nu) 
\rightarrow 1_{(\nu,i)} R_\bo(\nu+i) 1_{(\nu,i)} \\
&\rightarrow R_\bo^{\xi_i}(\nu) \oplus \lambda_i^2 q_i^{ - 2\alpha_i^\vee(\nu)} R_\bo^{\xi_i}(\nu)[1]
\rightarrow 0,
\end{align*}
of $R_\bo(\nu)$-$R_\bo(\nu)$-bimodules for all $i \in I$. Moreover, there is an isomorphism
\begin{align*}
q^{-(\alpha_i | \alpha_j)} R_\bo(\nu) 1_{(\nu-i,i)} \otimes_{m-1} 1_{(\nu'-j,j)} R_\bo(\nu') 
\cong 1_{(\nu',j)} R_\bo(\nu+i) 1_{(\nu,i)},
\end{align*}
for all $i \neq j \in I$ and $\nu'+j = \nu +i$.
\end{thm}

As we will see in the proof of \cref{thm:SESRbo}, we can picture these facts as a short exact sequence of diagrams
\[
\tikzdiag[xscale=.65]{
	\draw (0,-1) -- (0,1);
	\draw (.5,-1) -- (.5,1);
	\draw (1.5,-1) -- (1.5,1);
	\draw (2,-1) 
		-- (2,-.25) 
		.. controls (2,0) and (2.5,0) .. (2.5,.25) 
		-- (2.5,.75)
		.. controls (2.5,1) .. (2.75,1) node[right] {\small $j$};
	\draw (2,1) 
		-- (2,.25) 
		.. controls (2,0) and (2.5,0) .. (2.5,-.25) 
		-- (2.5,-.75)
		.. controls (2.5,-1) .. (2.75,-1) node[right] {\small $i$};
	\node at(1,.9) {\small $\dots$};
	\filldraw [fill=white, draw=black] (-.25,-.75) rectangle (2.25,-.25) node[midway] {\small $m$};
	\node at(1,0) {\small $\dots$};
	\filldraw [fill=white, draw=black] (-.25,.25) rectangle (2.25,.75) node[midway] {\small $m$};
	\node at(1,-.9) {\small $\dots$};
}
\hookrightarrow
\tikzdiag[xscale=.65]{
	\draw (0,-1) -- (0,1);
	\draw (.5,-1) -- (.5,1);
	\draw (1.5,-1) -- (1.5,1);
	\draw (2,-1) -- (2,1);
	\draw (2.75,-1)  node [right] {\small $i$} .. controls (2.5,-1) ..
		(2.5,-.75) -- (2.5,.75)
		.. controls (2.5,1) .. (2.75,1) node [right] {\small $j$};
	\filldraw [fill=white, draw=black] (-.25,-.25) rectangle (2.75,.25) node[midway] {\small $m+1$};
	\node at(1,.65) {\small $\dots$};
	\node at(1,-.65) {\small $\dots$};
}
\twoheadrightarrow
\bigoplus_{\ell \geq 0}
\left(
\tikzdiag[xscale=.65]{
	\draw (0,-1) -- (0,1);
	\draw (.5,-1) -- (.5,1);
	\draw (1.5,-1) -- (1.5,1);
	\draw (2,-1) -- (2,1);
	\draw (3,-1)  node [right] {\small $i$} .. controls (2.75,-1) ..
		(2.75,-.75) -- (2.75,.75) node[midway, tikzdot]{} node[midway,xshift=1.5ex, yshift=.5ex]{\small $\ell$}
		.. controls (2.75,1) .. (3,1) node [right] {\small $j$};
	\filldraw [fill=white, draw=black] (-.25,-.25) rectangle (2.25,.25) node[midway] {\small $m$};
	\node at(1,.65) {\small $\dots$};
	\node at(1,-.65) {\small $\dots$};
}
\oplus
\tikzdiag[xscale=.65]{
	\draw (0,-1) 
		-- (0,-.75)
		.. controls (0,-.5) and (1,-.5) .. (1,-.25)
		-- (1,.25)
		.. controls (1,.5) and (0,.5) .. (0,.75)
		-- (0,1);
	\draw (.5,-1) 
		-- (.5,-.75)
		.. controls (.5,-.5) and (1.5,-.5) .. (1.5,-.25)
		-- (1.5,.25)
		.. controls (1.5,.5) and (.5,.5) .. (.5,.75)
		-- (.5,1);
	\draw (1.5,-1) 
		-- (1.5,-.75)
		.. controls (1.5,-.5) and (2.5,-.5) .. (2.5,-.25)
		-- (2.5,.25)
		.. controls (2.5,.5) and (1.5,.5) .. (1.5,.75)
		-- (1.5,1);
	\draw (2,-1) 
		-- (2,-.75)
		.. controls (2,-.5) and (3,-.5) .. (3,-.25)
		-- (3,.25)
		.. controls (3,.5) and (2,.5) .. (2,.75)
		-- (2,1);
	\draw (3.25,-1) node[right]{\small $i$} .. controls (3,-1) .. 
		(3,-.75) .. controls (3,-.5) and (0,-.5) .. 
	 	(0,-.125) -- (0,.125) node[midway,tikzdot]{} node [midway, xshift=-1.5ex, yshift=.5ex] {\small $\ell$}
		.. controls (0,.5) and (3,.5) .. (3,.75)
		.. controls (3,1) .. (3.25,1) node[right]{\small $j$};
 	\fdot{}{.375,0};
	\filldraw [fill=white, draw=black] (.75,-.25) rectangle (3.25,.25) node[midway] {\small $m$};
	\node at(1,.9) {\small $\dots$};
	\node at(1,-.9) {\small $\dots$};
}
\right)
\]
where the cokernel vanishes whenever $i \neq j$. We write $\pi$ for the projection
\[
\pi : \tikzdiag[xscale=.65]{
	\draw (0,-1) -- (0,1);
	\draw (.5,-1) -- (.5,1);
	\draw (1.5,-1) -- (1.5,1);
	\draw (2,-1) -- (2,1);
	\draw (2.75,-1)  node [right] {\small $i$} .. controls (2.5,-1) ..
		(2.5,-.75) -- (2.5,.75)
		.. controls (2.5,1) .. (2.75,1) node [right] {\small $i$};
	\filldraw [fill=white, draw=black] (-.25,-.25) rectangle (2.75,.25) node[midway] {\small $m+1$};
	\node at(1,.65) {\small $\dots$};
	\node at(1,-.65) {\small $\dots$};
}
\twoheadrightarrow
\bigoplus_{\ell \geq 0}
\tikzdiag[xscale=.65]{
	\draw (0,-1) -- (0,1);
	\draw (.5,-1) -- (.5,1);
	\draw (1.5,-1) -- (1.5,1);
	\draw (2,-1) -- (2,1);
	\draw (2.75,-1)  node [right] {\small $i$} .. controls (2.5,-1) ..
		(2.5,-.75) -- (2.5,.75) node[midway, tikzdot]{} node[midway,xshift=1.5ex, yshift=.5ex]{\small $\ell$}
		.. controls (2.5,1) .. (2.75,1) node [right] {\small $i$};
	\filldraw [fill=white, draw=black] (-.25,-.25) rectangle (2.25,.25) node[midway] {\small $m$};
	\node at(1,.65) {\small $\dots$};
	\node at(1,-.65) {\small $\dots$};
}
\]

We write $\id_\nu := R_\bo(\nu) \otimes_m (-)$ and we define
\begin{align*}
\F_i &:= \bigoplus_{m \geq 0} \Ind_m^{m+i}, &
\E_i &:=  \bigoplus_{m \geq 0} \ \bigoplus_{|\nu| = m}   \lambda_i^{-1} q_i^{1+\alpha_i^\vee(\nu)} \Res_m^{m+i} \id_{\nu+i}.
\end{align*}
These are exact functors 
thanks to \cref{lem:Rboleftdecomp}. 
Define 
\begin{equation}\label{eq:directsumbetaquantum}
\oplus_{[\beta_i - \alpha_i^\vee(\nu)]_{q_i}} \id_\nu := \bigoplus_{\ell \geq 0} q_i^{1+2\ell} \bigl( \lambda_i^{-1} q_i^{ \alpha_i^\vee(\nu)}  \id_\nu \oplus \lambda_i q_i^{- \alpha_i^\vee(\nu)}  \id_\nu[1] \bigr).
\end{equation}
It is a categorification of the fraction $\frac{\lambda_i q_i^{- \alpha_i^\vee(\nu)} - \lambda_i^{-1}q_i^{\alpha_i^\vee(\nu)}}{q_i-q_i^{-1}}$. 
We obtain:

\begin{cor}\label{cor:catsltactionRbo}
There is a natural short exact sequence
\[
0 \rightarrow \F_i\E_i \id_\nu \rightarrow \E_i\F_i \id_\nu \rightarrow \oplus_{[\beta_i - \alpha_i^\vee(\nu)]_{q_i}} \id_\nu \rightarrow 0,
\]
for all $i \in I$, and there is a natural isomorphism
\[
\F_i\E_j \cong \E_j\F_i,
\]
for all $i \neq j \in I$.
\end{cor}

\begin{prop}\label{prop:serreinRbo}
For each $i,j \in I$ there is a natural isomorphism
\[
\bigoplus^{\lfloor (d_{ij}+1)/2 \rfloor}_{a=0} \begin{bmatrix} d_{ij}+1 \\ 2a \end{bmatrix}_{q_i}  \F_i^{2a}\F_j \F_i^{d_{ij}+1-2a} \cong \bigoplus^{\lfloor d_{ij}/2 \rfloor}_{a=0}  \begin{bmatrix} d_{ij}+1 \\ 2a +1 \end{bmatrix}_{q_i} \F_i^{2a+1}\F_j \F_i^{d_{ij}-2a},
\]
 and in particular for $(\alpha_i | \alpha_j) =0$ we have $\F_i\F_j1_\nu \cong \F_j\F_i1_\nu.$
By adjunction, the same isomorphism exists for the $\E_i, \E_j$.
\end{prop}

\begin{proof}
  Similarly as in the case of the usual KLR algebras, it follows from~\cref{eq:KLRnh} and~\cref{eq:KLRR3}
  (the proof of~\cite[Proposition~6]{KL2} can be applied directly).
\end{proof}

\subsection{Proof of \cref{thm:SESRbo}}

By symmetry along the horizontal axis, we obtain a  decomposition of $R_\bo(m+1)$ as a right $R_\bo(m)$-module similar to the one of \cref{lem:Rboleftdecomp}. Note that the left and right decompositions are not compatible, and therefore we do not have a decomposition as a $R_\bo(m)$-$R_\bo(m)$-bimodule. However, the surjection $R_\bo(m+1) \twoheadrightarrow q^{2\ell} R_\bo(m)$ that projects on the summand $R_\bo(m) x_{m+1}^\ell$, given by taking $a=m+1$ in \cref{lem:Rboleftdecomp}, is a (left-invertible) map of bimodules. 

\smallskip

We define the map
\[
\pi_L^\ell : 1_{(\nu,i)} R_\bo(m+1) 1_{(\nu,i)} \twoheadrightarrow \lambda_i^2 q_i^{2\ell - 2\alpha_i^\vee(\nu)} R_\bo(\nu)[1],
\]
as the projection map on the summand $R_\bo(m) \theta_{n+1}^\ell$ in the left decomposition of $R_\bo(m+1)$ as $R_\bo(m)$-module in \cref{lem:Rboleftdecomp}. Similarly, let
\[
\pi_R^\ell : 1_{(\nu,i)} R_\bo(m+1) 1_{(\nu,i)} \twoheadrightarrow \lambda_i^2 q_i^{2\ell - 2\alpha_i^\vee(\nu)} R_\bo(\nu)[1],
\]
be the projection map on $\theta_{n+1}^\ell R_\bo(m)$ in the right decomposition. 

\begin{lem}\label{lem:Anprojbimmap}
We have
\[
\pi_L^\ell(y) = (-1)^{\deg_h(y)} \pi_R^\ell(y)
\]
for all $y \in R_\bo(m+1)$.
\end{lem}

\begin{proof}
We can suppose $y = \theta_{m+1}^\ell y'$ with $y' \in R_\bo(m)$. We want to prove that $y = (-1)^{\deg_h(y)} y' \theta_{m+1}^\ell + y_0$ for some $y_0 \notin R_\bo(m) \theta_{m+1}^\ell$. For this, it is enough to show that $y_1  \theta_{m+1} z y_2 = (-1)^{\deg_h(z)} y_1   z \theta_{n+1} y_2 + z_0$ where $y_1, y_2 \in R_\bo(m)$, $z_0  \notin R_\bo(m) \theta_{m+1}^\ell$ and $z$ is any generator of $R_\bo(m)$ (i.e. crossing, dot or tight floating dot). 

If $z = x_a$ and is on a strand labeled $j \neq i$, then it slides freely over $\theta_{m+1}$ thanks to \cref{eq:KLRdotslide}. If the strand is labeled $i$, then we compute
\begin{align*}
\tikzdiagh[xscale=.6]{0}{
	\fdot{}{.5,0}; 
	\draw (0,-.75) .. controls (0,-.375) and (1,-.375) .. (1,0) 
			   .. controls (1,.375) and (0,.375) .. (0,.75);
	\draw[xshift=1ex] (0,-.75) .. controls (0,-.45) and (1,-.45) .. (1,0) 
				            .. controls (1,.45) and (0,.45) .. (0,.75);
	\draw (2,-.75) .. controls (2,-.375) and (3,-.375) .. (3,0) 
			   .. controls (3,.375) and (2,.375) .. (2,.75);
	\draw[xshift=1ex] (2,-.75) .. controls (2,-.45) and (3,-.45) .. (3,0) 
					  .. controls (3,.45) and (2,.45) .. (2,.75);				  
  	\draw  (3,-.75)  node[below] { \small $i$}   ..  controls (3,-.375) and (0,-.375) ..  (0,0) 
  			 node [pos=1,tikzdot]{} node[pos=1, xshift=-1.5ex, yshift=.5ex] {\small $\ell$}
  			.. controls (0,.375) and (3,.375) .. (3,.75);
	\draw (1,-.75)  node[below] { \small $i$}    .. controls (1,-.375) and (2,-.375) .. (2,0) node [pos=.25,tikzdot]{} 
			.. controls (2,.375) and (1,.375) .. (1,.75); 	
}
\ & \overset{\eqref{eq:KLRnh}}{=} \ 
\tikzdiagh[xscale=.6]{0}{
	\fdot{}{.5,0}; 
	\draw (0,-.75) .. controls (0,-.375) and (1,-.375) .. (1,0) 
			   .. controls (1,.375) and (0,.375) .. (0,.75);
	\draw[xshift=1ex] (0,-.75) .. controls (0,-.45) and (1,-.45) .. (1,0) 
				            .. controls (1,.45) and (0,.45) .. (0,.75);
	\draw (2,-.75) .. controls (2,-.375) and (3,-.375) .. (3,0) 
			   .. controls (3,.375) and (2,.375) .. (2,.75);
	\draw[xshift=1ex] (2,-.75) .. controls (2,-.45) and (3,-.45) .. (3,0) 
					  .. controls (3,.45) and (2,.45) .. (2,.75);				  
  	\draw  (3,-.75) node[below] { \small $i$}   ..  controls (3,-.375) and (0,-.375) ..  (0,0) 
  			 node [pos=1,tikzdot]{} node[pos=1, xshift=-1.5ex, yshift=.5ex] {\small $\ell$}
  			.. controls (0,.375) and (3,.375) .. (3,.75);
	\draw (1,-.75)  node[below] { \small $i$}    .. controls (1,-.375) and (2,-.375) .. (2,0)
			.. controls (2,.375) and (1,.375) .. (1,.75)  node [pos=.75,tikzdot]{} ; 	
}
\ + r_i \ 
\tikzdiagh[xscale=.6]{0}{
	\fdot{}{.5,0}; 
	\draw (0,-.75) .. controls (0,-.375) and (1,-.375) .. (1,0) 
			   .. controls (1,.375) and (0,.375) .. (0,.75);
	\draw[xshift=1ex] (0,-.75) .. controls (0,-.45) and (1,-.45) .. (1,0) 
				            .. controls (1,.45) and (0,.45) .. (0,.75);
	\draw (2,-.75) .. controls (2,-.375) and (3,-.375) .. (3,0) 
			   .. controls (3,.375) and (2,.375) .. (2,.75);
	\draw[xshift=1ex] (2,-.75) .. controls (2,-.45) and (3,-.45) .. (3,0) 
					  .. controls (3,.45) and (2,.45) .. (2,.75);
	\draw  (1,-.75)  node[below] { \small $i$}  ..  controls (1,-.375) and (0,-.375) ..  (0,0)  
			node [pos=1,tikzdot]{} node[pos=1, xshift=-1.5ex, yshift=.5ex] {\small $\ell$}
			..  controls (0,.375) and (3,.375) .. (3,.75);
	\draw (3,-.75) node[below] { \small $i$}     .. controls (3,-.375)  and (1,.375) .. (1,.75);
}
\ - r_i^{-1} \ 
\tikzdiagh[xscale=.6,yscale=-1]{0}{
	\fdot{}{.5,0}; 
	\draw (0,-.75) .. controls (0,-.375) and (1,-.375) .. (1,0) 
			   .. controls (1,.375) and (0,.375) .. (0,.75);
	\draw[xshift=1ex] (0,-.75) .. controls (0,-.45) and (1,-.45) .. (1,0) 
				            .. controls (1,.45) and (0,.45) .. (0,.75);
	\draw (2,-.75) .. controls (2,-.375) and (3,-.375) .. (3,0) 
			   .. controls (3,.375) and (2,.375) .. (2,.75);
	\draw[xshift=1ex] (2,-.75) .. controls (2,-.45) and (3,-.45) .. (3,0) 
					  .. controls (3,.45) and (2,.45) .. (2,.75);
	\draw  (1,-.75)   ..  controls (1,-.375) and (0,-.375) ..  (0,0)  
			node [pos=1,tikzdot]{} node[pos=1, xshift=-1.5ex, yshift=.5ex] {\small $\ell$}
			..  controls (0,.375) and (3,.375) .. (3,.75) node[below] { \small $i$};
	\draw (3,-.75)     .. controls (3,-.375)  and (1,.375) .. (1,.75) node[below] { \small $i$};
}
\\
\ &\overset{\eqref{eq:KLRR3}}{=} \ 
\tikzdiagh[xscale=.6]{0}{
	\fdot{}{.5,0}; 
	\draw (0,-.75) .. controls (0,-.375) and (1,-.375) .. (1,0) 
			   .. controls (1,.375) and (0,.375) .. (0,.75);
	\draw[xshift=1ex] (0,-.75) .. controls (0,-.45) and (1,-.45) .. (1,0) 
				            .. controls (1,.45) and (0,.45) .. (0,.75);
	\draw (2,-.75) .. controls (2,-.375) and (3,-.375) .. (3,0) 
			   .. controls (3,.375) and (2,.375) .. (2,.75);
	\draw[xshift=1ex] (2,-.75) .. controls (2,-.45) and (3,-.45) .. (3,0) 
					  .. controls (3,.45) and (2,.45) .. (2,.75);				  
  	\draw  (3,-.75)  node[below] { \small $i$}  ..  controls (3,-.375) and (0,-.375) ..  (0,0) 
  			 node [pos=1,tikzdot]{} node[pos=1, xshift=-1.5ex, yshift=.5ex] {\small $\ell$}
  			.. controls (0,.375) and (3,.375) .. (3,.75);
	\draw (1,-.75)  node[below] { \small $i$}    .. controls (1,-.375) and (2,-.375) .. (2,0)
			.. controls (2,.375) and (1,.375) .. (1,.75)  node [pos=.75,tikzdot]{} ; 	
}
\ + r_i \ 
\tikzdiagh[xscale=.6]{0}{
	\fdot{}{.5,0}; 
	\draw (0,-.75) .. controls (0,-.375) and (.9,-.375) .. (.9,0) 
			.. controls (.9,.375) and (0,.375) .. (0,1.25);;
	\draw[xshift=1ex] (0,-.75) .. controls (0,-.45) and (.9,-.45) .. (.9,0) 
			.. controls (.9,.45) and (0,.45) .. (0,1.25);
	\draw (2,-.75) .. controls (2,.05) and (1.1,.05) .. (1.1,.5) 
			.. controls (1.1,.95) and (2,.95) .. (2,1.25);
	\draw[xshift=1ex] (2,-.75) .. controls (2,.125) and (1.1,.125) .. (1.1,.5) 
			.. controls (1.1,.875) and (2,.875) .. (2,1.25);
      	\draw  (1,-.75) node[below] { \small $i$}    ..  controls (1,-.375) and (0,-.375) ..  (0,0)  
      			node [pos=1,tikzdot]{} node[pos=1,xshift=-1.5ex,yshift=.5ex]  {\small$\ell$}
      			.. controls (0,.375) and (3,.375) .. (3,1.25);
	\draw (3,-.75)  node[below] { \small $i$}    .. controls (3,.125)  and (1,.875) .. (1,1.25);
}
\ - r_i^{-1} \ 
\tikzdiagh[xscale=.6,yscale=-1]{0}{
	\fdot{}{.5,0}; 
	\draw (0,-.75) .. controls (0,-.375) and (.9,-.375) .. (.9,0) 
			.. controls (.9,.375) and (0,.375) .. (0,1.25);;
	\draw[xshift=1ex] (0,-.75) .. controls (0,-.45) and (.9,-.45) .. (.9,0) 
			.. controls (.9,.45) and (0,.45) .. (0,1.25);
	\draw (2,-.75) .. controls (2,.05) and (1.1,.05) .. (1.1,.5) 
			.. controls (1.1,.95) and (2,.95) .. (2,1.25);
	\draw[xshift=1ex] (2,-.75) .. controls (2,.125) and (1.1,.125) .. (1.1,.5) 
			.. controls (1.1,.875) and (2,.875) .. (2,1.25);
      	\draw  (1,-.75)     ..  controls (1,-.375) and (0,-.375) ..  (0,0)  
      			node [pos=1,tikzdot]{} node[pos=1,xshift=-1.5ex,yshift=.5ex]  {\small$\ell$}
      			.. controls (0,.375) and (3,.375) .. (3,1.25) node[below] { \small $i$} ;
	\draw (3,-.75)   .. controls (3,.125)  and (1,.875) .. (1,1.25) node[below] { \small $i$};
}
\ + \ R,
\end{align*}
where the double strands represent multiple parallel strands (the number depending on $m$ and $a$), and  $R$ is a sum of terms of the following form:
\[
\tikzdiag[xscale=.6]{
	      \fdot{}{.4,-.1}; 
	\draw (0,-.75) .. controls (0,-.375) and (.8,-.375) .. (.8,0) 
			  .. controls (.8,.375) and (0,.375) .. (0,1.25);
	\draw[xshift=1ex] (0,-.75) .. controls (0,-.45) and (.8,-.45) .. (.8,0) 
				 	.. controls (.8,.45) and (0,.45) .. (0,1.25);
	\draw (2,-.75) .. controls (2,.05) and (1.3,.05) .. (1.3,.5) 
			  .. controls (1.3,.95) and (2,.95) .. (2,1.25);
	\draw[xshift=1ex] (2,-.75) .. controls (2,.125) and (1.3,.125) .. (1.3,.5) 
					  .. controls (1.3,.875) and (2,.875) .. (2,1.25);
	      \draw  (.75,-.75)     ..  controls (.75,-.375) and (0,-.375) ..  (0,0)  node [pos=1,tikzdot]{} node[pos=1,xshift=-1.5ex,yshift = .5ex]{\small$\ell$};	
	\draw (1.25,1.25) .. controls (1.25,.875) and (2,.875) ..  (2,.5) 
				.. controls (2,.25) and (0,.25) ..  (0,0)  node [pos=0,fill=black,circle,inner sep=2pt]{}; 
	\draw[xshift=1ex]  (3,-.75) .. controls (3,.05) and (3.7,.05) .. (3.7,.5) .. controls (3.7,.95) and (3,.95) .. (3,1.25);
	\draw(3,-.75) .. controls (3,.125) and (3.7,.125) .. (3.7,.5) .. controls (3.7,.875) and (3,.875) .. (3,1.25);
	\draw(2.5,-.75) -- (2.5,.5) -- (2.5,1.25)    node [pos=0,tikzdot]{};
	\draw (3.5,-.75)     .. controls (3.5,.25)  and (3,.25) .. (3,.5) .. controls (3,.875) and (3.5,.875) .. (3.5,1.25)  node [pos=0,tikzdot]{};
	 }
\]
and its mirror along the horizontal axis. 
Note that it is implicitly assumed that each of these diagrams have the element $y_1$ at the top and $y_2$ at the bottom. 
Using \cref{lem:Rboleftdecomp}, we can rewrite the composition of the last three terms in the equation above with $y_2$ as elements in $\oplus_{a=1}^n \oplus_{p \geq 0} R_\bo(m-1) \tau_{m} \tau_{m-1} \cdots \tau_a x_a^p \nsubset R_\bo(m) \theta_{m+1}^\ell$. Hence they form the term $z_0$.

If $z = \tau_i$ is a crossing, then we obtain the desired property by \cref{eq:KLRR3}, and applying a similar reasoning as above.

Finally if $z = \omega$ and is at right of a strand labeled $j \neq i$, it follows directly from \cref{eq:tightfdcommutes}. Otherwise, if the strand is labeled $i$,
 we compute 
\begin{align*}
\tikzdiagh[yscale=1.25]{0}{
	\draw (0,0) node[below] { \small $i$}  ..controls (0,.25) and (1,.25) .. (1,.5)
		 ..controls (1,.75) and (0,.75) .. (0,1) ;
	\draw (1,0) node[below] { \small $i$}  ..controls (1,.25) and (0,.25) .. (0,.5) 
		node[pos=1, tikzdot]{} node[pos=1,xshift=-1.5ex,yshift=.5ex]{\small $\ell$}
		..controls (0,.75) and (1,.75) .. (1,1);
	\fdot{}{0.5,0.5};
	\fdot{}{0.5,0};
}
\ \overset{(\ref{eq:fdmoves},\ref{eq:KLRnh})}{=} \ 
\tikzdiagh[yscale=1.25]{0}{
	\draw (0,0) node[below] { \small $i$}  ..controls (0,.25) and (1,.25) .. (1,.5)
		 ..controls (1,.75) and (0,.75) .. (0,1) ;
	\draw (1,0) node[below] { \small $i$}  ..controls (1,.25) and (0,.25) .. (0,.5) 
		node[pos=.25, tikzdot]{} node[pos=.25,xshift=1.5ex,yshift=.5ex]{\small $\ell$}
		..controls (0,.75) and (1,.75) .. (1,1);
	\fdot{}{0.5,0.5};
	\fdot{}{0.5,0};
}
\ \overset{\eqref{eq:tightfdcommutes}}{=} \ 
-\ 
\tikzdiagh[yscale=1.25]{0}{
	\draw (0,0) node[below] { \small $i$}  ..controls (0,.25) and (1,.25) .. (1,.5)
		 ..controls (1,.75) and (0,.75) .. (0,1) ;
	\draw (1,0) node[below] { \small $i$}  ..controls (1,.25) and (0,.25) .. (0,.5) 
		node[pos=.25, tikzdot]{} node[pos=.25,xshift=1.5ex,yshift=.5ex]{\small $\ell$}
		..controls (0,.75) and (1,.75) .. (1,1);
	\fdot{}{0.5,0.5};
	\fdot{}{0.5,1.0};
}
\ = \ 
-\ 
\tikzdiagh[yscale=1.25]{0}{
	\draw (0,0) node[below] { \small $i$}  ..controls (0,.25) and (1,.25) .. (1,.5)
		 ..controls (1,.75) and (0,.75) .. (0,1) ;
	\draw (1,0) node[below] { \small $i$}  ..controls (1,.25) and (0,.25) .. (0,.5) 
		node[pos=1, tikzdot]{} node[pos=1,xshift=-1.5ex,yshift=.5ex]{\small $\ell$}
		..controls (0,.75) and (1,.75) .. (1,1);
	\fdot{}{0.5,0.5};
	\fdot{}{0.5,1.0};
}
 \ + r_i^{-1} \ \sum\limits_{\substack{r+s =\\ \ell-1}} \ 
\tikzdiagh{0}{
	\draw (0,0) node[below] { \small $i$}  ..controls (0,.5) and (1,.5) .. (1,1)
		node[near start, tikzdot]{} node[near start,xshift=-1.5ex,yshift=0ex]{\small $r$};
	\draw (1,0) node[below] { \small $i$}  ..controls (1,.5) and (0,.5) .. (0,1)
		node[near start, tikzdot]{} node[near start,xshift=1.5ex,yshift=0ex]{\small $s$};
	\fdot{}{0.5,0.2};
	\fdot{}{0.5,0.8};
}
\end{align*}
Then for all $r,s \geq 0$ we compute using \cref{eq:KLRnh} again
\[
\tikzdiagh{0}{
	\draw (0,0) node[below] { \small $i$}  ..controls (0,.5) and (1,.5) .. (1,1)
		node[near start, tikzdot]{} node[near start,xshift=-1.5ex,yshift=0ex]{\small $r$};
	\draw (1,0) node[below] { \small $i$}  ..controls (1,.5) and (0,.5) .. (0,1)
		node[near start, tikzdot]{} node[near start,xshift=1.5ex,yshift=0ex]{\small $s$};
	\fdot{}{0.5,0.2};
	\fdot{}{0.5,0.8};
}
\ = \ 
\tikzdiagh{0}{
	\draw (0,0) node[below] { \small $i$}  ..controls (0,.5) and (1,.5) .. (1,1)
		node[near start, tikzdot]{} node[near start,xshift=-1.5ex,yshift=0ex]{\small $r$};
	\draw (1,0) node[below] { \small $i$}  ..controls (1,.5) and (0,.5) .. (0,1)
		node[near end, tikzdot]{} node[near end,xshift=-1.5ex,yshift=0ex]{\small $s$};
	\fdot{}{0.5,0.2};
	\fdot{}{0.5,0.8};
}
\]
Looking at these elements in the global picture yields
\[
\tikzdiagh[xscale=.6]{0}{
	\draw (0,-.75)  .. controls (0,-.375) and (1,-.375) .. (1,0) 
			   .. controls (1,.375) and (0,.375) .. (0,.75);
	\draw[xshift=1ex] (0,-.75) .. controls (0,-.45) and (1,-.45) .. (1,0) 
				            .. controls (1,.45) and (0,.45) .. (0,.75);
	\draw (-1,-.75)node[below] { \small $i$}  .. controls (-1,.25) and (2,.25) .. (2,.75)
		node[pos=.1, tikzdot]{} node[pos=.1,xshift=-1.5ex,yshift=0ex]{\small $r$};
	\draw (2,-.75) node[below] { \small $i$}  .. controls (2,-.25) and (-1,-.25) .. (-1,.75)
		node[pos=.9, tikzdot]{} node[pos=.9,xshift=-1.5ex,yshift=0ex]{\small $s$};
	\fdot{}{-0.25,0.35};
	\fdot{}{-0.25,-0.35};
}
\ = \ 
\tikzdiagh[xscale=-.6]{0}{
	\draw (0,-.75)   .. controls (0,-.375) and (1,-.375) .. (1,0) 
			   .. controls (1,.375) and (0,.375) .. (0,.75);
	\draw[xshift=1ex] (0,-.75) .. controls (0,-.45) and (1,-.45) .. (1,0) 
				            .. controls (1,.45) and (0,.45) .. (0,.75);
	\draw (-1,-.75)  node[below] { \small $i$}  .. controls (-1,.25) and (2,.25) .. (2,.75)
		node[pos=.8, tikzdot]{} node[pos=.8,xshift=-1.5ex,yshift=0ex]{\small $s$};
	\draw (2,-.75) node[below] { \small $i$}  .. controls (2,-.25) and (-1,-.25) .. (-1,.75)
		node[pos=.2, tikzdot]{} node[pos=.2,xshift=-1.5ex,yshift=0ex]{\small $r$};
	\fdot{}{1.25,0.6};
	\fdot{}{1.25,-0.6};
}
 + R
\]
which is an element not contained in $R_\bo(m) \theta_{n+1}^\ell$ for the same reasons as before. We see that together they form the element $z_0$, concluding the proof.
\end{proof}

We now have all the ingredients we need to prove \cref{thm:SESRbo}.

\begin{proof}[Proof of \cref{thm:SESRbo}]
We first construct an injective map 
\begin{equation}\label{eq:uij}
u_{ij}: q^{-(\alpha_i | \alpha_j)} R_\bo(\nu) 1_{(m-1,i)} \otimes_{m-1} 1_{(m_1,j)} R_\bo(\nu) 
\hookrightarrow 1_{(\nu,j)} R_\bo(m+1) 1_{(\nu,i)}
\end{equation}
of $R_\bo(m)$-$R_\bo(m)$-bimodules, by setting (as in~\cite[Proposition~3.3]{kashiwara})
\[
u_{ij}(x \otimes_{m-1} y) := x \tau_m y .
\]
In terms of diagrams, it consists of adding a crossing at the right
\[
\tikzdiag[xscale=.75]{
	\draw (0,-1.25) -- (0,1.25);
	\draw (.5,-1.25) -- (.5,1.25);
	\draw (1.5,-1.25) -- (1.5,1.25);
	\draw (2,-1.25) -- (2,-.5) .. controls (2,-.25) .. (2.25,-.25) node[right]{\small $j$};
	\draw (2,1.25) -- (2,.5) .. controls (2,.25) .. (2.25,.25) node[right]{\small $i$};
	\node at(1,1.2) {\small $\dots$};
	\filldraw [fill=white, draw=black] (-.25,-1) rectangle (2.25,-.5) node[midway] {\small $m$};
	\node at(1,0) {\small $\dots$}; %\filldraw [fill=white, draw=black] (-.25,-.25) rectangle (1.75,.25) node[midway] {$n{-}1$};
	\filldraw [fill=white, draw=black] (-.25,.5) rectangle (2.25,1) node[midway] {\small $m$};
	\node at(1,-1.2) {\small $\dots$};
}
\ \xmapsto{\ u_{ij} \ } \  
\tikzdiag[xscale=.75]{
	\draw (0,-1.25) -- (0,1.25);
	\draw (.5,-1.25) -- (.5,1.25);
	\draw (1.5,-1.25) -- (1.5,1.25);
	\draw (2,-1.25) -- (2,-.5) .. controls (2,0) and (2.5,0) .. (2.5,.5) -- (2.5,1.25)  node[above]{\small $j$};
	\draw (2,1.25) -- (2,.5) .. controls (2,0) and (2.5,0) .. (2.5,-.5) -- (2.5,-1.25)  node[below]{\small $i$};
	\node at(1,1.2) {\small $\dots$};
	\filldraw [fill=white, draw=black] (-.25,-1) rectangle (2.25,-.5) node[midway] {\small $m$};
	\node at(1,0) {\small $\dots$};%\filldraw [fill=white, draw=black] (-.25,-.25) rectangle (1.75,.25) node[midway] {$n{-}1$};
	\filldraw [fill=white, draw=black] (-.25,.5) rectangle (2.25,1) node[midway] {\small $m$};
	\node at(1,-1.2) {\small $\dots$};
}
\]
Then, we construct a surjective map 
\[
 1_{(\nu,i)} R_\bo(m+1) 1_{(\nu,i)} \twoheadrightarrow  R_\bo^{\xi_i}(\nu) \oplus \lambda_i^2 q_i^{ - 2\alpha_i^\vee(\nu)} R_\bo^{\xi_i}(\nu)[1],
\]
by projecting onto the direct summands $\bigoplus_{\ell \geq 0}x_{m+1}^\ell R_\bo(m)  \oplus \theta_{m+1}^\ell R_\bo(m) $ of the decomposition of $R_\bo(m+1)$ as right $R_\bo(m)$-module. By \cref{lem:Anprojbimmap} we know that this is a map of $R_\bo(m)$-$R_\bo(m)$-bimodules. 
Finally, exactness follows directly from \cref{lem:Rboleftdecomp}, since
\begin{align*}
 R_\bo(\nu) 1_{(m-1,i)}& \otimes_{m-1} 1_{(m_1,j)} R_\bo(\nu)  \\
 & \cong \begin{aligned}[t] R_\bo(\nu) 1_{(m-1,i)} \otimes_{m-1}  \bigl( \bigoplus_{a=1}^{m} \bigoplus_{\ell \geq 0} & ( R_\bo(m-1) 1_{(m,i)} \tau_{m-1} \cdots \tau_a x_a^\ell 1_\bj  \\ &\oplus  R_\bo(m-1) 1_{(m,i)} \tau_{m-1} \cdots \tau_a  \theta_a^\ell 1_\bj ) \bigr), \end{aligned}
\end{align*}
and so 
\begin{align*}
u_{ij}( R_\bo(\nu) 1_{(m-1,i)}& \otimes_{m-1} 1_{(m_1,j)} R_\bo(\nu) ) \\
& \cong \begin{aligned}[t]  \bigoplus_{a=1}^{m} \bigoplus_{\ell \geq 0} & ( R_\bo(m-1) 1_{(m,i)} \tau_m \tau_{m-1} \cdots \tau_a x_a^\ell 1_\bj  \\ &\oplus  R_\bo(m-1) 1_{(m,i)} \tau_m \tau_{m-1} \cdots \tau_a  \theta_a^\ell 1_\bj ).\end{aligned}
\end{align*}
We remark that whenever $i \neq j$, we have  
\[
\bigoplus_{\ell \geq 0} 1_{(\nu,j)}  x_{m+1}^\ell R_\bo(m) 1_{(\nu,i)}   \oplus   1_{(\nu,j)} R_\bo(m) \theta_{m+1}^\ell 1_{(\nu,i)} = 0,
\]
 and thus $u_{ij}$ is an isomorphism, concluding the proof.
\end{proof}

\subsection{Long exact sequence}

We want to lift \cref{thm:SESRbo} to the dg-world of $(R_\bo(m),d_N)$, and study the long exact sequence that it induces. Therefore we define
\[
y_N : 
\bigoplus_{\ell \geq 0}
\tikzdiag[xscale=.65]{
	\draw (0,-1) 
		-- (0,-.75)
		.. controls (0,-.5) and (1,-.5) .. (1,-.25)
		-- (1,.25)
		.. controls (1,.5) and (0,.5) .. (0,.75)
		-- (0,1);
	\draw (.5,-1) 
		-- (.5,-.75)
		.. controls (.5,-.5) and (1.5,-.5) .. (1.5,-.25)
		-- (1.5,.25)
		.. controls (1.5,.5) and (.5,.5) .. (.5,.75)
		-- (.5,1);
	\draw (1.5,-1) 
		-- (1.5,-.75)
		.. controls (1.5,-.5) and (2.5,-.5) .. (2.5,-.25)
		-- (2.5,.25)
		.. controls (2.5,.5) and (1.5,.5) .. (1.5,.75)
		-- (1.5,1);
	\draw (2,-1) 
		-- (2,-.75)
		.. controls (2,-.5) and (3,-.5) .. (3,-.25)
		-- (3,.25)
		.. controls (3,.5) and (2,.5) .. (2,.75)
		-- (2,1);
	\draw (3.25,-1) node[right]{\small $i$} .. controls (3,-1) .. 
		(3,-.75) .. controls (3,-.5) and (0,-.5) .. 
	 	(0,-.125) -- (0,.125) node[midway,tikzdot]{} node [midway, xshift=-1.5ex, yshift=.5ex] {\small $\ell$}
		.. controls (0,.5) and (3,.5) .. (3,.75)
		.. controls (3,1) .. (3.25,1);
 	\fdot{}{.375,0};
	\filldraw [fill=white, draw=black] (.75,-.25) rectangle (3.25,.25) node[midway] {\small $m$};
	\node at(1,.9) {$\dots$};
	\node at(1,-.9) {$\dots$};
}
\rightarrow
\bigoplus_{\ell \geq 0}
\tikzdiag[xscale=.65]{
	\draw (0,-1) -- (0,1);
	\draw (.5,-1) -- (.5,1);
	\draw (1.5,-1) -- (1.5,1);
	\draw (2,-1) -- (2,1);
	\draw (3,-1)  node [right] {\small $i$} .. controls (2.75,-1) ..
		(2.75,-.75) -- (2.75,.75) node[midway, tikzdot]{} node[midway,xshift=1.5ex, yshift=.5ex]{\small $\ell$}
		.. controls (2.75,1) .. (3,1);
	\filldraw [fill=white, draw=black] (-.25,-.25) rectangle (2.25,.25) node[midway] {\small $m$};
	\node at(1,.65) {\small $\dots$};
	\node at(1,-.65) {\small $\dots$};
}
\]
as the $R_\bo(m)$-$R_\bo(m)$-bimodule map given by 
\[
y_N\left(
\tikzdiagh[xscale = 1.2]{-.6ex}{
		\draw  (.25,1) .. controls (.25,.5) and (.75,.5) ..   (.75,0)   .. controls (.75,-.5) and (.25,-.5)  .. (.25,-1); 
		\draw  (.5,1) .. controls (.5,.5) and (1,.5) ..   (1,0)   .. controls (1,-.5) and (.5,-.5)  .. (.5,-1); 
		\draw  (1,1) .. controls (1,.5) and (1.5,.5) ..   (1.5,0)   .. controls (1.5,-.5) and (1,-.5)  .. (1,-1); 
		\draw  (1.25,1) .. controls (1.25,.5) and (1.75,.5) ..   (1.75,0)   .. controls (1.75,-.5) and (1.25,-.5)  .. (1.25,-1); 
	 \draw  (1.75,-1)  node[below]{$i$}.. controls (1.75,-.5) and (.25,-.5) ..    (.25,0) node[pos=1, tikzdot]{} node[pos=1,xshift=-1.5ex,yshift=.5ex]{\small $a$}
		 .. controls (.25,.5)  and (1.75,.5).. (1.75,1);
		 \node at (1.25,0) {\small $\dots$};  \node at (.8,.85) {\small $\dots$};  \node at (.8,-.85) {\tiny $\dots$};
	\fdot{}{.5,0};
	 }
\right)
:= 
\pi \left(
\tikzdiagh[xscale = 1.2]{-.6ex}{
		\draw  (.25,1) .. controls (.25,.5) and (.75,.5) ..   (.75,0)   .. controls (.75,-.5) and (.25,-.5)  .. (.25,-1); 
		\draw  (.5,1) .. controls (.5,.5) and (1,.5) ..   (1,0)   .. controls (1,-.5) and (.5,-.5)  .. (.5,-1); 
		\draw  (1,1) .. controls (1,.5) and (1.5,.5) ..   (1.5,0)   .. controls (1.5,-.5) and (1,-.5)  .. (1,-1); 
		\draw  (1.25,1) .. controls (1.25,.5) and (1.75,.5) ..   (1.75,0)   .. controls (1.75,-.5) and (1.25,-.5)  .. (1.25,-1); 
	 \draw  (1.75,-1)  node[below]{$i$}.. controls (1.75,-.5) and (.25,-.5) ..    (.25,0) node[pos=1, tikzdot]{} node[pos=1,xshift=-3.5ex,yshift=.5ex]{\small $n_i+a$}
		 .. controls (.25,.5)  and (1.75,.5).. (1.75,1);
		 \node at (1.25,0) {\small $\dots$};  \node at (.8,.85) {\small $\dots$};  \node at (.8,-.85) {\small $\dots$};
	 }
\right)
\in
\bigoplus_{\ell \geq 0}
\tikzdiag[xscale=.65]{
	\draw (0,-1) -- (0,1);
	\draw (.5,-1) -- (.5,1);
	\draw (1.5,-1) -- (1.5,1);
	\draw (2,-1) -- (2,1);
	\draw (3,-1)  node [right] {\small $i$} .. controls (2.75,-1) ..
		(2.75,-.75) -- (2.75,.75) node[midway, tikzdot]{} node[midway,xshift=1.5ex, yshift=.5ex]{\small $\ell$}
		.. controls (2.75,1) .. (3,1);
	\filldraw [fill=white, draw=black] (-.25,-.25) rectangle (2.25,.25) node[midway] {\small $m$};
	\node at(1,.65) {\small $\dots$};
	\node at(1,-.65) {\small $\dots$};
}
\]
whenever $i \in I_f$, and $y_N = 0$ for $i \notin I_f$.
Then we  define
\begin{align*}
\bigl( R_\bo^{\xi_i}(\nu) \oplus& \lambda_i^2 q_i^{ - 2\alpha_i^\vee(\nu)} R_\bo^{\xi_i}(\nu)[1], d_N \bigr) \\
&:= \cone\left( (\lambda_i^2 q_i^{ - 2\alpha_i^\vee(\nu)} R_\bo^{\xi_i}(\nu)[1], d_N) \xrightarrow{y_n}  (R_\bo^{\xi_i}(\nu), d_N) \right),
\intertext{and}
(R_\bo(\nu) &1_{(m-1,i)} \otimes_{m-1} 1_{(m-1,i)} R_\bo(\nu), d_N) \\
&:= (R_\bo(\nu) 1_{(m-1,i)}, d_N) \otimes_{(R_\bo(m-1), d_N)} (1_{(m-1,i)}R_\bo(\nu),d_N).
\end{align*}

\begin{prop}\label{prop:RbodgSES}
There is a short exact sequence of dg-bimodules
\begin{align*}
0 
&\rightarrow q_i^{-2} (R_\bo(\nu) 1_{(\nu-i,i)} \otimes_{m-1} 1_{(\nu-i,i)} R_\bo(\nu), d_N)
\rightarrow (1_{(\nu,i)} R_\bo(\nu+i) 1_{(\nu,i)},d_N) \\
&\rightarrow \bigl( R_\bo^{\xi_i}(\nu) \oplus \lambda_i^2 q_i^{ - 2\alpha_i^\vee(\nu)} R_\bo^{\xi_i}(\nu)[1], d_N \bigr)
\rightarrow 0
\end{align*}
for all $i \in I$. Moreover, there is an isomorphism
\begin{align*}
q^{-(\alpha_i | \alpha_j)}( R_\bo(\nu) 1_{(\nu-i,i)} \otimes_{m-1} 1_{(\nu'-j,j)} R_\bo(\nu'), d_N) 
\cong (1_{(\nu',j)} R_\bo(\nu+i) 1_{(\nu,i)}, d_N)
\end{align*}
for all $i \neq j \in I$ and $\nu+i = \nu'+j$.
\end{prop}

\begin{proof}
It is a straightforward consequence of \cref{thm:SESRbo}.
\end{proof}

In order to understand the consequences of this short exact sequence in homology, we need to compute the homology
\[
H\bigl( R_\bo^{\xi_i}(\nu) \oplus \lambda_i^2 q_i^{ - 2\alpha_i^\vee(\nu)} R_\bo^{\xi_i}(\nu)[1], d_N \bigr),
\]
 for all $i \in I_f$. 
 
 Therefore, we want to compute the projection of the element
\[
\bar\pi\left(
\tikzdiagh[xscale = 1.2]{-.6ex}{
		\draw  (.25,1) .. controls (.25,.5) and (.75,.5) ..   (.75,0)   .. controls (.75,-.5) and (.25,-.5)  .. (.25,-1); 
		\draw  (.5,1) .. controls (.5,.5) and (1,.5) ..   (1,0)   .. controls (1,-.5) and (.5,-.5)  .. (.5,-1); 
		\draw  (1,1) .. controls (1,.5) and (1.5,.5) ..   (1.5,0)   .. controls (1.5,-.5) and (1,-.5)  .. (1,-1); 
		\draw  (1.25,1) .. controls (1.25,.5) and (1.75,.5) ..   (1.75,0)   .. controls (1.75,-.5) and (1.25,-.5)  .. (1.25,-1); 
	 \draw  (1.75,-1)  node[below]{$i$}.. controls (1.75,-.5) and (.25,-.5) ..    (.25,0) node[pos=1, tikzdot]{} node[pos=1,xshift=-1.5ex,yshift=.5ex]{\small $p$}
		 .. controls (.25,.5)  and (1.75,.5).. (1.75,1);
		 \node at (1.25,0) {\small $\dots$};  \node at (.8,.85) {\small $\dots$};  \node at (.8,-.85) {\small $\dots$};
	 }
\right) \quad \in \quad  
\bigoplus_{\ell \geq 0}
\tikzdiag[xscale=.65]{
	\draw (0,-1) -- (0,1);
	\draw (.5,-1) -- (.5,1);
	\draw (1.5,-1) -- (1.5,1);
	\draw (2,-1) -- (2,1);
	\draw (3,-1)  node [right] {\small $i$} .. controls (2.75,-1) ..
		(2.75,-.75) -- (2.75,.75) node[midway, tikzdot]{} node[midway,xshift=1.5ex, yshift=.5ex]{\small $\ell$}
		.. controls (2.75,1) .. (3,1);
	\filldraw [fill=white, draw=black] (-.25,-.25) rectangle (2.25,.25) node[midway] {\small $H(m)$};
	\node at(1,.65) {\small $\dots$};
	\node at(1,-.65) {\small $\dots$};
}
\]
for all $p \geq n_i$. Note that we project on the homology of $(R_\bo(m),d_N)$. This will ease some of the computations we need to do.  We write $\bar\pi$ when we take the composite of $\pi$ with the projection on the homology of $(R_\bo(m),d_N)$. More precisely, $\bar\pi$ is given by 
\[
\bar\pi := 1 \otimes \pi : 
\tikzdiag[xscale=.65]{
	\draw (0,-1) -- (0,1);
	\draw (.5,-1) -- (.5,1);
	\draw (1.5,-1) -- (1.5,1);
	\draw (2,-1) -- (2,1);
	\draw (2.75,-1)  node [right] {\small $i$} .. controls (2.5,-1) ..
		(2.5,-.75) -- (2.5,.75)
		.. controls (2.5,1) .. (2.75,1) node [right] {\small $i$};
	\filldraw [fill=white, draw=black] (-.25,.25) rectangle (2.25,.75) node[midway] {\small $H(m)$};
	\filldraw [fill=white, draw=black] (-.25,-.75) rectangle (2.75,-.25) node[midway] {\small $m+1$};
	\node at(1,.9) {\small $\dots$};
	\node at(1,-0) {\small $\dots$};
	\node at(1,-.9) {\small $\dots$};
}
\twoheadrightarrow
\bigoplus_{\ell \geq 0}
\tikzdiag[xscale=.65]{
	\draw (0,-1) -- (0,1);
	\draw (.5,-1) -- (.5,1);
	\draw (1.5,-1) -- (1.5,1);
	\draw (2,-1) -- (2,1);
	\draw (3,-1)  node [right] {\small $i$} .. controls (2.75,-1) ..
		(2.75,-.75) -- (2.75,.75) node[midway, tikzdot]{} node[midway,xshift=1.5ex, yshift=.5ex]{\small $\ell$}
		.. controls (2.75,1) .. (3,1) node [right] {\small $i$};
	\filldraw [fill=white, draw=black] (-.25,.25) rectangle (2.25,.75) node[midway] {\small $H(m)$};
	\filldraw [fill=white, draw=black] (-.25,-.75) rectangle (2.25,-.25) node[midway] { \small $m$};
	\node at(1,.9) {\small $\dots$};
	\node at(1,-0) {\small $\dots$};
	\node at(1,-.9) {\small $\dots$};
}
\]
 Similarly, we write $\bar y_N$.

\begin{lem}\label{lem:bigp}
If $p \geq 2 \nu_i$, then
\[
\pi\left(
\tikzdiagh[yscale = .8]{-.4ex}{
	\draw  (.25,1) .. controls (.25,.5) and (.75,.5) ..   (.75,0)   .. controls (.75,-.5) and (.25,-.5)  .. (.25,-1); 
	\draw  (.5,1) .. controls (.5,.5) and (1,.5) ..   (1,0)   .. controls (1,-.5) and (.5,-.5)  .. (.5,-1); 
	\draw  (1,1) .. controls (1,.5) and (1.5,.5) ..   (1.5,0)   .. controls (1.5,-.5) and (1,-.5)  .. (1,-1); 
	\draw  (1.25,1) .. controls (1.25,.5) and (1.75,.5) ..   (1.75,0)   .. controls (1.75,-.5) and (1.25,-.5)  .. (1.25,-1); 
 \draw  (1.75,-1)  node[below]{$i$}.. controls (1.75,-.5) and (.25,-.5) ..    (.25,0) node[pos=1, tikzdot]{} node[pos=1,xshift=-1.5ex,yshift=.5ex]{\small $p$}
	 .. controls (.25,.5)  and (1.75,.5).. (1.75,1);
	 \node at (1.25,0) {\tiny $\dots$};  \node at (.8,.85) {\tiny $\dots$};  \node at (.8,-.85) {\tiny $\dots$};
	\draw[decoration={brace,mirror,raise=-8pt},decorate]  (.15,-1.35) -- node {$\nu$} (1.35,-1.35);
 }
 \right)
 \ \equiv \ 
 \zeta\quad
 \tikzdiag[xscale=.5,yscale=.8]{
	\draw (0,-1) -- (0,1);
	\draw (.5,-1) -- (.5,1);
	\draw (1.5,-1) -- (1.5,1);
	\draw (2,-1) -- (2,1);
	\draw (3,-1)  node [right] {\small $i$} .. controls (2.75,-1) ..
		(2.75,-.75) -- (2.75,.75) node[midway, tikzdot]{} node[midway,xshift=5.5ex, yshift=.5ex]{\small $p-\alpha_i^\vee(\nu)$}
		.. controls (2.75,1) .. (3,1);
	\node at (1,0) {\tiny $\dots$};
}
+ 
\bigoplus_{\ell = 0}^{p-\alpha_i^\vee(\nu)-1}
\tikzdiag[xscale=.5,yscale=.8]{
	\draw (0,-1) -- (0,1);
	\draw (.5,-1) -- (.5,1);
	\draw (1.5,-1) -- (1.5,1);
	\draw (2,-1) -- (2,1);
	\draw (3,-1)  node [right] {\small $i$} .. controls (2.75,-1) ..
		(2.75,-.75) -- (2.75,.75) node[midway, tikzdot]{} node[midway,xshift=1.5ex, yshift=.5ex]{\small $\ell$}
		.. controls (2.75,1) .. (3,1);
	\filldraw [fill=white, draw=black] (-.25,-.25) rectangle (2.25,.25) node[midway] {\small $m$};
	\node at (1,.65) {\tiny $\dots$}; \node at (1,-.65) {\tiny $\dots$};
}
\]
for some invertible element $\zeta \in \Bbbk^\times$. If $p < 2\nu_i$, then
\[
\pi\left(
\tikzdiagh[yscale = .8]{-.4ex}{
	\draw  (.25,1) .. controls (.25,.5) and (.75,.5) ..   (.75,0)   .. controls (.75,-.5) and (.25,-.5)  .. (.25,-1); 
	\draw  (.5,1) .. controls (.5,.5) and (1,.5) ..   (1,0)   .. controls (1,-.5) and (.5,-.5)  .. (.5,-1); 
	\draw  (1,1) .. controls (1,.5) and (1.5,.5) ..   (1.5,0)   .. controls (1.5,-.5) and (1,-.5)  .. (1,-1); 
	\draw  (1.25,1) .. controls (1.25,.5) and (1.75,.5) ..   (1.75,0)   .. controls (1.75,-.5) and (1.25,-.5)  .. (1.25,-1); 
 \draw  (1.75,-1)  node[below]{$i$}.. controls (1.75,-.5) and (.25,-.5) ..    (.25,0) node[pos=1, tikzdot]{} node[pos=1,xshift=-1.5ex,yshift=.5ex]{\small $p$}
	 .. controls (.25,.5)  and (1.75,.5).. (1.75,1);
	 \node at (1.25,0) {\tiny $\dots$};  \node at (.8,.85) {\tiny $\dots$};  \node at (.8,-.85) {\tiny $\dots$};
	\draw[decoration={brace,mirror,raise=-8pt},decorate]  (.15,-1.35) -- node {$\nu$} (1.35,-1.35);
 }
 \right)
 \  \in \ 
\bigoplus_{\ell = 0}^{p-\alpha_i^\vee(\nu)}
\tikzdiag[xscale=.5,yscale=.8]{
	\draw (0,-1) -- (0,1);
	\draw (.5,-1) -- (.5,1);
	\draw (1.5,-1) -- (1.5,1);
	\draw (2,-1) -- (2,1);
	\draw (3,-1)  node [right] {\small $i$} .. controls (2.75,-1) ..
		(2.75,-.75) -- (2.75,.75) node[midway, tikzdot]{} node[midway,xshift=1.5ex, yshift=.5ex]{\small $\ell$}
		.. controls (2.75,1) .. (3,1);
	\filldraw [fill=white, draw=black] (-.25,-.25) rectangle (2.25,.25) node[midway] {\small $m$};
	\node at (1,.65) {\tiny $\dots$}; \node at (1,-.65) {\tiny $\dots$};
}
\]
\end{lem}

\begin{proof}
The proof is an induction on $m$. If $m=0$, then it is trivial. Suppose the statement holds for $m-1$. We fix the labels of the strands as the bottom as $\bj = j_1\cdots j_m \in \Seq(\nu)$. If $j_1 = i$, then we compute
\begin{align*}
\tikzdiagh[yscale = .8]{0}{
	\draw  (.25,1) .. controls (.25,.5) and (.75,.5) ..   (.75,0)   .. controls (.75,-.5) and (.25,-.5)  .. (.25,-1) node[below]{$i$}; 
	\draw  (.5,1) .. controls (.5,.5) and (1,.5) ..   (1,0)   .. controls (1,-.5) and (.5,-.5)  .. (.5,-1); 
	\draw  (1,1) .. controls (1,.5) and (1.5,.5) ..   (1.5,0)   .. controls (1.5,-.5) and (1,-.5)  .. (1,-1); 
	\draw  (1.25,1) .. controls (1.25,.5) and (1.75,.5) ..   (1.75,0)   .. controls (1.75,-.5) and (1.25,-.5)  .. (1.25,-1); 
 \draw  (1.75,-1)  node[below]{$i$}.. controls (1.75,-.5) and (.25,-.5) ..    (.25,0) node[pos=1, tikzdot]{} node[pos=1,xshift=-1.5ex,yshift=.5ex]{\small $p$}
	 .. controls (.25,.5)  and (1.75,.5).. (1.75,1);
	 \node at (1.25,0) {\tiny $\dots$};  \node at (.8,.85) {\tiny $\dots$};  \node at (.8,-.85) {\tiny $\dots$};
 }
\ &\overset{\eqref{eq:KLRnh}}{=} \ r_i \sssum{r+s\\= p-1}
\tikzdiagh[yscale = .8]{0}{
	\draw (.25,1) .. controls (.25,-.5) and (1.75,-.5) .. (1.75,-1)  node[below]{$i$}   node[pos=.2,tikzdot]{} node[pos=.2,xshift=-1.5ex, yshift=.5ex]{\small $r$};
	\draw (.25,-1) node[below]{$i$} .. controls (.25,.5) and (1.75,.5) .. (1.75,1)  node[pos=.2,tikzdot]{} node[pos=.2,xshift=-1.5ex, yshift=.5ex]{\small $s$};
	\draw  (.5,1) .. controls (.5,.5) and (1,.5) ..   (1,0)   .. controls (1,-.5) and (.5,-.5)  .. (.5,-1); 
	\draw  (1,1) .. controls (1,.5) and (1.5,.5) ..   (1.5,0)   .. controls (1.5,-.5) and (1,-.5)  .. (1,-1); 
	\draw  (1.25,1) .. controls (1.25,.5) and (1.75,.5) ..   (1.75,0)   .. controls (1.75,-.5) and (1.25,-.5)  .. (1.25,-1); 
	 \node at (1.25,0) {\tiny $\dots$};  \node at (.8,.85) {\tiny $\dots$};  \node at (.8,-.85) {\tiny $\dots$};
 }
 \ - \ r_i^2 \sssum{r+s\\= p-2}
(r+1)\ 
\tikzdiagh[yscale = .8]{0}{
	\draw  (.25,1) -- (.25,-1) node[below]{$i$} node[midway,tikzdot]{} node[midway,xshift=-1.5ex,yshift=.5ex]{\small $r$}; 
	\draw  (.5,1) .. controls (.5,.5) and (1,.5) ..   (1,0)   .. controls (1,-.5) and (.5,-.5)  .. (.5,-1); 
	\draw  (1,1) .. controls (1,.5) and (1.5,.5) ..   (1.5,0)   .. controls (1.5,-.5) and (1,-.5)  .. (1,-1); 
	\draw  (1.25,1) .. controls (1.25,.5) and (1.75,.5) ..   (1.75,0)   .. controls (1.75,-.5) and (1.25,-.5)  .. (1.25,-1); 
 \draw  (1.75,-1)  node[below]{$i$}.. controls (1.75,-.5) and (.75,-.5) ..    (.75,0) node[pos=1, tikzdot]{} node[pos=1,xshift=-1.5ex,yshift=.5ex]{\small $s$}
	 .. controls (.75,.5)  and (1.75,.5).. (1.75,1);
	 \node at (1.25,0) {\tiny $\dots$};  \node at (.8,.85) {\tiny $\dots$};  \node at (.8,-.85) {\tiny $\dots$};
 }
\end{align*}
Then, using~\cref{eq:KLRR3} we have
\[
\tikzdiagh[yscale = .8]{0}{
	\draw (.25,1) .. controls (.25,-.5) and (1.75,-.5) .. (1.75,-1)  node[below]{$i$};
	\draw (.25,-1) node[below]{$i$} .. controls (.25,.5) and (1.75,.5) .. (1.75,1);
	\draw  (.5,1) .. controls (.5,.5) and (1,.5) ..   (1,0)   .. controls (1,-.5) and (.5,-.5)  .. (.5,-1); 
	\draw  (1,1) .. controls (1,.5) and (1.5,.5) ..   (1.5,0)   .. controls (1.5,-.5) and (1,-.5)  .. (1,-1); 
	\draw  (1.25,1) .. controls (1.25,.5) and (1.75,.5) ..   (1.75,0)   .. controls (1.75,-.5) and (1.25,-.5)  .. (1.25,-1); 
	 \node at (1.25,0) {\tiny $\dots$};  \node at (.8,.85) {\tiny $\dots$};  \node at (.8,-.85) {\tiny $\dots$};
 }
 \ = \ 
 \tikzdiagh[yscale = .8,xscale=-1]{0}{
	\draw (.25,1) .. controls (.25,-.5) and (1.75,-.5) .. (1.75,-1)  node[below]{$i$};
	\draw (.25,-1) node[below]{$i$} .. controls (.25,.5) and (1.75,.5) .. (1.75,1);
	\draw  (.5,1) .. controls (.5,.5) and (1,.5) ..   (1,0)   .. controls (1,-.5) and (.5,-.5)  .. (.5,-1); 
	\draw  (1,1) .. controls (1,.5) and (1.5,.5) ..   (1.5,0)   .. controls (1.5,-.5) and (1,-.5)  .. (1,-1); 
	\draw  (1.25,1) .. controls (1.25,.5) and (1.75,.5) ..   (1.75,0)   .. controls (1.75,-.5) and (1.25,-.5)  .. (1.25,-1); 
	 \node at (1.25,0) {\tiny $\dots$};  \node at (.8,.85) {\tiny $\dots$};  \node at (.8,-.85) {\tiny $\dots$};
 }
 \ + \sum_{j_k \neq i} s_{ij_k}^{tv} \sum_{t,v}\  \sssum{r+s=\\t-1}\ 
 \tikzdiagh[yscale = .8]{0}{
	\draw  (.5,1) .. controls (.5,.5) and (0,.5) ..   (0,0)   .. controls (0,-.5) and (.5,-.5)  .. (.5,-1); 
	\draw  (.25,1) .. controls (.25,.5) and (-.25,.5) ..   (-.25,0)   .. controls (-.25,-.5) and (.25,-.5)  .. (.25,-1); 
 	\draw  (-.25,-1)  node[below]{$i$}
 	.. controls (-.25,-.5) and (.5,-.5) ..    (.5,0) node[pos=1, tikzdot]{} node[pos=1,xshift=-1.5ex,yshift=.5ex]{\small $r$}
	 .. controls (.5,.5)  and (-.25,.5).. (-.25,1);
 	\draw[myblue] (.75,1)-- (.75,-1) node[pos=.3,tikzdot]{} node[pos=.3,xshift=1.5ex,yshift=.5ex]{\small $v$} node[below]{\small $j_k$};
	\draw  (1,1) .. controls (1,.5) and (1.5,.5) ..   (1.5,0)   .. controls (1.5,-.5) and (1,-.5)  .. (1,-1); 
	\draw  (1.25,1) .. controls (1.25,.5) and (1.75,.5) ..   (1.75,0)   .. controls (1.75,-.5) and (1.25,-.5)  .. (1.25,-1); 
 	\draw  (1.75,-1)  node[below]{$i$}
 	.. controls (1.75,-.5) and (1,-.5) ..    (1,0) node[pos=1, tikzdot]{} node[pos=1,xshift=1.5ex,yshift=.5ex]{\small $s$}
	 .. controls (1,.5)  and (1.75,.5).. (1.75,1);
 }
\]
so that, since $s < d_{ij_k}$, we obtain by the induction hypothesis 
\begin{align*}
\pi\left(
\tikzdiagh[yscale = .8]{0}{
	\draw (.25,1) .. controls (.25,-.5) and (1.75,-.5) .. (1.75,-1)  node[below]{$i$};
	\draw (.25,-1) node[below]{$i$} .. controls (.25,.5) and (1.75,.5) .. (1.75,1);
	\draw  (.5,1) .. controls (.5,.5) and (1,.5) ..   (1,0)   .. controls (1,-.5) and (.5,-.5)  .. (.5,-1); 
	\draw  (1,1) .. controls (1,.5) and (1.5,.5) ..   (1.5,0)   .. controls (1.5,-.5) and (1,-.5)  .. (1,-1); 
	\draw  (1.25,1) .. controls (1.25,.5) and (1.75,.5) ..   (1.75,0)   .. controls (1.75,-.5) and (1.25,-.5)  .. (1.25,-1); 
	 \node at (1.25,0) {\tiny $\dots$};  \node at (.8,.85) {\tiny $\dots$};  \node at (.8,-.85) {\tiny $\dots$};
 }
 \right)
 \ \in \
\bigoplus_{\ell = 0}^{p-\alpha_i^\vee(\nu)-1}
\tikzdiagh[xscale=.5,yscale=.8]{0}{
	\draw (0,-1)  node[below]{$i$} -- (0,1);
	\draw (.5,-1) -- (.5,1);
	\draw (1.5,-1) -- (1.5,1);
	\draw (2,-1) -- (2,1);
	\draw (3,-1)  node [right] {\small $i$} .. controls (2.75,-1) ..
		(2.75,-.75) -- (2.75,.75) node[midway, tikzdot]{} node[midway,xshift=1.5ex, yshift=.5ex]{\small $\ell$}
		.. controls (2.75,1) .. (3,1);
	\filldraw [fill=white, draw=black] (-.25,-.25) rectangle (2.25,.25) node[midway] {\small $m$};
	\node at (1,.65) {\tiny $\dots$}; \node at (1,-.65) {\tiny $\dots$};
}
\end{align*}
 Moreover, still by the induction hypothesis, we have
\[
\sssum{r+s\\= p-3}
(r+2)\ 
\pi\left(
\tikzdiagh[yscale = .8]{-.4ex}{
	\draw  (.25,1) -- (.25,-1) node[below]{$i$} node[midway,tikzdot]{} node[midway,xshift=-3ex,yshift=.5ex]{\small $r{+}1$}; 
	\draw  (.5,1) .. controls (.5,.5) and (1,.5) ..   (1,0)   .. controls (1,-.5) and (.5,-.5)  .. (.5,-1); 
	\draw  (1,1) .. controls (1,.5) and (1.5,.5) ..   (1.5,0)   .. controls (1.5,-.5) and (1,-.5)  .. (1,-1); 
	\draw  (1.25,1) .. controls (1.25,.5) and (1.75,.5) ..   (1.75,0)   .. controls (1.75,-.5) and (1.25,-.5)  .. (1.25,-1); 
 \draw  (1.75,-1)  node[below]{$i$}.. controls (1.75,-.5) and (.75,-.5) ..    (.75,0) node[pos=1, tikzdot]{} node[pos=1,xshift=-1.5ex,yshift=.5ex]{\small $s$}
	 .. controls (.75,.5)  and (1.75,.5).. (1.75,1);
	 \node at (1.25,0) {\tiny $\dots$};  \node at (.8,.85) {\tiny $\dots$};  \node at (.8,-.85) {\tiny $\dots$};
 }
 \right)
 \in 
 \bigoplus_{\ell = 0}^{p-\alpha_i^\vee(\nu)-1}
\tikzdiagh[xscale=.5,yscale=.8]{0}{
	\draw (0,-1)  node[below]{$i$} -- (0,1);
	\draw (.5,-1) -- (.5,1);
	\draw (1.5,-1) -- (1.5,1);
	\draw (2,-1) -- (2,1);
	\draw (3,-1)  node [right] {\small $i$} .. controls (2.75,-1) ..
		(2.75,-.75) -- (2.75,.75) node[midway, tikzdot]{} node[midway,xshift=1.5ex, yshift=.5ex]{\small $\ell$}
		.. controls (2.75,1) .. (3,1);
	\filldraw [fill=white, draw=black] (-.25,-.25) rectangle (2.25,.25) node[midway] {\small $m$};
	\node at(1,.65) {\tiny $\dots$};
	\node at(1,-.65) {\tiny $\dots$};
}
\]
Finally, if $p \geq 2\nu_i$, by the induction hypothesis we get for $s=p-2$,
\begin{align*}
\pi\left(
\tikzdiagh[yscale = .8]{-.4ex}{
	\draw  (.25,1) -- (.25,-1) node[below]{$i$}; 
	\draw  (.5,1) .. controls (.5,.5) and (1,.5) ..   (1,0)   .. controls (1,-.5) and (.5,-.5)  .. (.5,-1); 
	\draw  (1,1) .. controls (1,.5) and (1.5,.5) ..   (1.5,0)   .. controls (1.5,-.5) and (1,-.5)  .. (1,-1); 
	\draw  (1.25,1) .. controls (1.25,.5) and (1.75,.5) ..   (1.75,0)   .. controls (1.75,-.5) and (1.25,-.5)  .. (1.25,-1); 
 \draw  (1.75,-1)  node[below]{$i$}.. controls (1.75,-.5) and (.75,-.5) ..    (.75,0) node[pos=1, tikzdot]{} node[pos=1,xshift=-1.5ex,yshift=.5ex]{\small $s$}
	 .. controls (.75,.5)  and (1.75,.5).. (1.75,1);
	 \node at (1.25,0) {\tiny $\dots$};  \node at (.8,.85) {\tiny $\dots$};  \node at (.8,-.85) {\tiny $\dots$};
 }
 \right)
\  &\equiv \zeta' \ 
 \tikzdiagh[xscale=.5,yscale=.8]{0}{
	\draw (0,-1) node[below]{$i$} -- (0,1);
	\draw (.5,-1) -- (.5,1);
	\draw (1.5,-1) -- (1.5,1);
	\draw (2,-1) -- (2,1);
	\draw (3,-1)  node [right] {\small $i$} .. controls (2.75,-1) ..
		(2.75,-.75) -- (2.75,.75) node[midway, tikzdot]{} node[midway,xshift=7.5ex, yshift=.5ex]{\small $s-\alpha_i^\vee(\nu-i)$}
		.. controls (2.75,1) .. (3,1);
	\node at(1,0) {\tiny $\dots$};
}
+ 
\bigoplus_{\ell = 0}^{s-\alpha_i^\vee(\nu-i)-1}
\tikzdiagh[xscale=.5,yscale=.8]{0}{
	\draw (0,-1) node[below]{$i$} -- (0,1);
	\draw (.5,-1) -- (.5,1);
	\draw (1.5,-1) -- (1.5,1);
	\draw (2,-1) -- (2,1);
	\draw (3,-1)  node [right] {\small $i$} .. controls (2.75,-1) ..
		(2.75,-.75) -- (2.75,.75) node[midway, tikzdot]{} node[midway,xshift=1.5ex, yshift=.5ex]{\small $\ell$}
		.. controls (2.75,1) .. (3,1);
	\filldraw [fill=white, draw=black] (-.25,-.25) rectangle (2.25,.25) node[midway] {\small  $m$};
	\node at(1,.65) {\tiny $\dots$};
	\node at(1,-.65) {\tiny $\dots$};
}
\end{align*}
which concludes the case by observing that $s-\alpha_i^\vee(\nu-i) = p-\alpha_i^\vee(\nu)$, and taking $\zeta = r_i^2\zeta'$.
If $p < 2\nu_i$, the claim is immediate by the induction hypothesis.

For the case $j_1 = j \neq i$, we use~\cref{eq:KLRR2} and then the induction hypothesis to get % a degree-argument to get
\begin{align*}
\pi\left(
\tikzdiagh[yscale = .8]{0}{
	\draw[myblue]  (.25,1) .. controls (.25,.5) and (.75,.5) ..   (.75,0)   .. controls (.75,-.5) and (.25,-.5)  .. (.25,-1) node[below]{$j$}; 
	\draw  (.5,1) .. controls (.5,.5) and (1,.5) ..   (1,0)   .. controls (1,-.5) and (.5,-.5)  .. (.5,-1); 
	\draw  (1,1) .. controls (1,.5) and (1.5,.5) ..   (1.5,0)   .. controls (1.5,-.5) and (1,-.5)  .. (1,-1); 
	\draw  (1.25,1) .. controls (1.25,.5) and (1.75,.5) ..   (1.75,0)   .. controls (1.75,-.5) and (1.25,-.5)  .. (1.25,-1); 
 \draw  (1.75,-1)  node[below]{$i$}.. controls (1.75,-.5) and (.25,-.5) ..    (.25,0) node[pos=1, tikzdot]{} node[pos=1,xshift=-1.5ex,yshift=.5ex]{\small $p$}
	 .. controls (.25,.5)  and (1.75,.5).. (1.75,1);
	 \node at (1.25,0) {\tiny $\dots$};  \node at (.8,.85) {\tiny $\dots$};  \node at (.8,-.85) {\tiny $\dots$};
 }
 \right)
\ &= \ 
\sssum{t,v}
s_{ij}^{tv}\ 
\pi\left(
\tikzdiagh[yscale = .8]{0}{
	\draw[myblue]  (-.25,1) -- (-.25,-1) node[below]{$j$} node[midway,tikzdot]{} node[midway,xshift=-1.5ex,yshift=.5ex]{\small $v$}; 
	\draw  (.5,1) .. controls (.5,.5) and (1,.5) ..   (1,0)   .. controls (1,-.5) and (.5,-.5)  .. (.5,-1); 
	\draw  (1,1) .. controls (1,.5) and (1.5,.5) ..   (1.5,0)   .. controls (1.5,-.5) and (1,-.5)  .. (1,-1); 
	\draw  (1.25,1) .. controls (1.25,.5) and (1.75,.5) ..   (1.75,0)   .. controls (1.75,-.5) and (1.25,-.5)  .. (1.25,-1); 
 \draw  (1.75,-1)  node[below]{$i$}.. controls (1.75,-.5) and (.75,-.5) ..    (.75,0) node[pos=1, tikzdot]{} node[pos=1,xshift=-2.5ex,yshift=.5ex]{\small $t{+}p$}
	 .. controls (.75,.5)  and (1.75,.5).. (1.75,1);
	 \node at (1.25,0) {\tiny $\dots$};  \node at (.8,.85) {\tiny $\dots$};  \node at (.8,-.85) {\tiny $\dots$};
 }
 \right)
 \\
 \ &\equiv\ 
 t_{ij} 
 \tikzdiagh[yscale = .8]{0}{
	\draw[myblue]  (-.5,1) -- (-.5,-1) node[below]{$j$}; 
	\draw  (.5,1) .. controls (.5,.5) and (1,.5) ..   (1,0)   .. controls (1,-.5) and (.5,-.5)  .. (.5,-1); 
	\draw  (1,1) .. controls (1,.5) and (1.5,.5) ..   (1.5,0)   .. controls (1.5,-.5) and (1,-.5)  .. (1,-1); 
	\draw  (1.25,1) .. controls (1.25,.5) and (1.75,.5) ..   (1.75,0)   .. controls (1.75,-.5) and (1.25,-.5)  .. (1.25,-1); 
 \draw  (1.75,-1)  node[below]{$i$}.. controls (1.75,-.5) and (.75,-.5) ..    (.75,0) node[pos=1, tikzdot]{} node[pos=1,xshift=-3.5ex,yshift=.5ex]{\small $d_{ij}{+}p$}
	 .. controls (.75,.5)  and (1.75,.5).. (1.75,1);
	 \node at (1.25,0) {\tiny $\dots$};  \node at (.8,.85) {\tiny $\dots$};  \node at (.8,-.85) {\tiny $\dots$};
 }
+ 
\bigoplus_{\ell = 0}^{p-\alpha_i^\vee(\nu)-1}
\tikzdiagh[xscale=.5,yscale=.8]{0}{
	\draw[myblue] (0,-1) node[below]{$j$} -- (0,1);
	\draw (.5,-1) -- (.5,1);
	\draw (1.5,-1) -- (1.5,1);
	\draw (2,-1) -- (2,1);
	\draw (3,-1)  node [right] {\small $i$} .. controls (2.75,-1) ..
		(2.75,-.75) -- (2.75,.75) node[midway, tikzdot]{} node[midway,xshift=1.5ex, yshift=.5ex]{\small $\ell$}
		.. controls (2.75,1) .. (3,1);
	\filldraw [fill=white, draw=black] (-.25,-.25) rectangle (2.25,.25) node[midway] {\small $m$};
	\node at(1,.65) {\tiny $\dots$};
	\node at(1,-.65) {\tiny $\dots$};
}
\end{align*}
where we recall that $s_{ij}^{d_{ij}0} = t_{ij}$.
We conclude by applying the induction hypothesis, observing that $d_{ij} + p -\alpha_i^\vee(\nu-j) = p-\alpha_i^\vee(\nu)$.
\end{proof}

Consider also the following result, which is akin to~\cite[Lemma~5.4]{kashiwara}.

\begin{lem}\label{lem:pidotslide} 
We have for $k < k'$ and $t = k' - k$, 
\[
\bar y_N\left(
 \tikzdiagh[yscale = .8]{0}{
	\draw  (.25,1) .. controls (.25,.5) and (.75,.5) ..   (.75,0)   .. controls (.75,-.5) and (.25,-.5)  .. (.25,-1); 
	\draw  (.5,1) .. controls (.5,.5) and (1,.5) ..   (1,0)   .. controls (1,-.5) and (.5,-.5)  .. (.5,-1); 
	\draw  (1,1) .. controls (1,.5) and (1.5,.5) ..   (1.5,0)   .. controls (1.5,-.5) and (1,-.5)  .. (1,-1); 
	\draw  (1.25,1) .. controls (1.25,.5) and (1.75,.5) ..   (1.75,0)   .. controls (1.75,-.5) and (1.25,-.5)  .. (1.25,-1); 
 \draw  (1.75,-1)  node[below]{$i$}.. controls (1.75,-.5) and (.25,-.5) ..    (.25,0) node[pos=1, tikzdot]{} node[pos=1,xshift=-1.5ex,yshift=.5ex]{\small $k'$}
	 .. controls (.25,.5)  and (1.75,.5).. (1.75,1);
	 \node at (1.25,0) {\tiny $\dots$};  \node at (.8,.85) {\tiny $\dots$};  \node at (.8,-.85) {\tiny $\dots$};
	 \fdot{}{.5,0};
 }
 \right)
 \equiv
  \tikzdiagh[xscale=.5,yscale=.8]{0}{
	\draw (0,-1)  node[below,white]{$i$} -- (0,1);
	\draw (.5,-1) -- (.5,1);
	\draw (1.5,-1) -- (1.5,1);
	\draw (2,-1) -- (2,1);
	\draw (3,-1)  node [right] {\small $i$} .. controls (2.75,-1) ..
		(2.75,-.75) -- (2.75,.75) node[pos=.9, tikzdot]{} node[pos=.9,xshift=1.5ex,yshift=.5ex]{\small $t$}
		.. controls (2.75,1) .. (3,1);
	\filldraw [fill=white, draw=black,rounded corners] (-.25,-.35) rectangle (3,0.35) node[midway] {\small $\bar y_N(\xi_i^{k})$};
	\node at(1,.65) {\tiny $\dots$};
	\node at(1,-.65) {\tiny $\dots$};
}
 + 
 \sum_{\ell=0}^{t-1}
 \tikzdiagh[xscale=.5,yscale=.8]{0}{
	\draw (0,-1)  node[below,white]{$i$} -- (0,1);
	\draw (.5,-1) -- (.5,1);
	\draw (1.5,-1) -- (1.5,1);
	\draw (2,-1) -- (2,1);
	\draw (3,-1)  node [right] {\small $i$} .. controls (2.75,-1) ..
		(2.75,-.75) -- (2.75,.75) node[midway, tikzdot]{} node[midway,xshift=1.5ex,yshift=.5ex]{\small $\ell$}
		.. controls (2.75,1) .. (3,1);
	\filldraw [fill=white, draw=black] (-.25,-.35) rectangle (2.25,.35) node[midway] {\small $H(m)$};
	\node at(1,.65) {\tiny $\dots$};
	\node at(1,-.65) {\tiny $\dots$};
}
 \]
 where
 \[
   \tikzdiagh[xscale=.5,yscale=.8]{0}{
	\draw (0,-1)  node[below,white]{$i$} -- (0,1);
	\draw (.5,-1) -- (.5,1);
	\draw (1.5,-1) -- (1.5,1);
	\draw (2,-1) -- (2,1);
	\draw (3,-1)  node [right] {\small $i$} .. controls (2.75,-1) ..
		(2.75,-.75) -- (2.75,.75) 
		.. controls (2.75,1) .. (3,1);
	\filldraw [fill=white, draw=black,rounded corners] (-.25,-.35) rectangle (3,0.35) node[midway] {\small $\bar y_N(\xi_i^k)$};
	\node at(1,.65) {\tiny $\dots$};
	\node at(1,-.65) {\tiny $\dots$};
}
=
\bar y_N\left(
 \tikzdiagh[yscale = .8]{0}{
	\draw  (.25,1) .. controls (.25,.5) and (.75,.5) ..   (.75,0)   .. controls (.75,-.5) and (.25,-.5)  .. (.25,-1); 
	\draw  (.5,1) .. controls (.5,.5) and (1,.5) ..   (1,0)   .. controls (1,-.5) and (.5,-.5)  .. (.5,-1); 
	\draw  (1,1) .. controls (1,.5) and (1.5,.5) ..   (1.5,0)   .. controls (1.5,-.5) and (1,-.5)  .. (1,-1); 
	\draw  (1.25,1) .. controls (1.25,.5) and (1.75,.5) ..   (1.75,0)   .. controls (1.75,-.5) and (1.25,-.5)  .. (1.25,-1); 
 \draw  (1.75,-1)  node[below]{$i$}.. controls (1.75,-.5) and (.25,-.5) ..    (.25,0) node[pos=1, tikzdot]{} node[pos=1,xshift=-1.5ex,yshift=.5ex]{\small $k$}
	 .. controls (.25,.5)  and (1.75,.5).. (1.75,1);
	 \node at (1.25,0) {\tiny $\dots$};  \node at (.8,.85) {\tiny $\dots$};  \node at (.8,-.85) {\tiny $\dots$};
	 \fdot{}{.5,0};
 }
 \right)
 \]
\end{lem}

\begin{proof}
First we observe that 
\[
 \tikzdiagh[yscale = .8]{0}{
	\draw  (.25,1) .. controls (.25,.5) and (.75,.5) ..   (.75,0)   .. controls (.75,-.5) and (.25,-.5)  .. (.25,-1); 
	\draw  (.5,1) .. controls (.5,.5) and (1,.5) ..   (1,0)   .. controls (1,-.5) and (.5,-.5)  .. (.5,-1); 
	\draw  (1,1) .. controls (1,.5) and (1.5,.5) ..   (1.5,0)   .. controls (1.5,-.5) and (1,-.5)  .. (1,-1); 
	\draw  (1.25,1) .. controls (1.25,.5) and (1.75,.5) ..   (1.75,0)   .. controls (1.75,-.5) and (1.25,-.5)  .. (1.25,-1); 
 \draw  (1.75,-1)  node[below]{$i$}.. controls (1.75,-.5) and (.25,-.5) ..    (.25,0) node[pos=1, tikzdot]{} node[pos=1,xshift=-3.5ex,yshift=.5ex]{\small $n_i{+}k'$}
	 .. controls (.25,.5)  and (1.75,.5).. (1.75,1);
	 \node at (1.25,0) {\tiny $\dots$};  \node at (.8,.85) {\tiny $\dots$};  \node at (.8,-.85) {\tiny $\dots$};
 }
 =
  \tikzdiagh[yscale = .8]{0}{
	\draw  (.25,1) .. controls (.25,.5) and (.75,.5) ..   (.75,0)   .. controls (.75,-.5) and (.25,-.5)  .. (.25,-1); 
	\draw  (.5,1) .. controls (.5,.5) and (1,.5) ..   (1,0)   .. controls (1,-.5) and (.5,-.5)  .. (.5,-1); 
	\draw  (1,1) .. controls (1,.5) and (1.5,.5) ..   (1.5,0)   .. controls (1.5,-.5) and (1,-.5)  .. (1,-1); 
	\draw  (1.25,1) .. controls (1.25,.5) and (1.75,.5) ..   (1.75,0)   .. controls (1.75,-.5) and (1.25,-.5)  .. (1.25,-1); 
 \draw  (1.75,-1)  node[below]{$i$}.. controls (1.75,-.5) and (.25,-.5) ..    (.25,0) node[pos=1, tikzdot]{} node[pos=1,xshift=-3.5ex,yshift=.5ex]{\small $n_i{+}k$}
	 .. controls (.25,.5)  and (1.75,.5).. (1.75,1) node[pos=.8,tikzdot]{} node[pos=.8, xshift=1.5ex, yshift=0]{\small $t$};
	 \node at (1.25,0) {\tiny $\dots$};  \node at (.8,.85) {\tiny $\dots$};  \node at (.8,-.85) {\tiny $\dots$};
 }
 \in
 \tikzdiag[xscale=.5,yscale=.8]{
	\draw (0,-1) -- (0,1);
	\draw (.5,-1) -- (.5,1);
	\draw (1.5,-1) -- (1.5,1);
	\draw (2,-1) -- (2,1);
	\draw (3,-1)  node [right] {\small $i$} .. controls (2.75,-1) ..
		(2.75,-.75) -- (2.75,.75)
		.. controls (2.75,1) .. (3,1) node [right] {\small $i$};
	\filldraw [fill=white, draw=black] (-.25,.15) rectangle (2.25,.85) node[midway] {\small $H(m)$};
	\filldraw [fill=white, draw=black] (-.25,-.75) rectangle (3,-.25) node[midway] {\small $m+1$};
	\node at(1,.9) {\tiny$\dots$};
	\node at(1,-.1) {\tiny $\dots$};
	\node at(1,-.9) {\tiny$\dots$};
}
 \]
 using~\cref{eq:KLRnh} and~\cref{eq:KLRdotslide}, and the fact that $n_i$  dots on the left strand is annihilated in $H(R_\bo(m),d_N)$. 
 
 Then, using \cref{lem:Rboleftdecomp} we obtain
 \begin{equation}\label{eq:thetakdecomp}
 \tikzdiagh[yscale = .8]{0}{
	\draw  (.25,1) .. controls (.25,.5) and (.75,.5) ..   (.75,0)   .. controls (.75,-.5) and (.25,-.5)  .. (.25,-1); 
	\draw  (.5,1) .. controls (.5,.5) and (1,.5) ..   (1,0)   .. controls (1,-.5) and (.5,-.5)  .. (.5,-1); 
	\draw  (1,1) .. controls (1,.5) and (1.5,.5) ..   (1.5,0)   .. controls (1.5,-.5) and (1,-.5)  .. (1,-1); 
	\draw  (1.25,1) .. controls (1.25,.5) and (1.75,.5) ..   (1.75,0)   .. controls (1.75,-.5) and (1.25,-.5)  .. (1.25,-1); 
 \draw  (1.75,-1)  node[below]{$i$}.. controls (1.75,-.5) and (.25,-.5) ..    (.25,0) node[pos=1, tikzdot]{} node[pos=1,xshift=-3.5ex,yshift=.5ex]{\small $n_i{+}k$}
	 .. controls (.25,.5)  and (1.75,.5).. (1.75,1);
	 \node at (1.25,0) {\tiny $\dots$};  \node at (.8,.85) {\tiny $\dots$};  \node at (.8,-.85) {\tiny $\dots$};
 }
 \ = \ 
\tikzdiagh[xscale=.5,yscale=.8]{0}{
	\draw (0,-1)  node[below,white]{$i$} -- (0,1);
	\draw (.5,-1) -- (.5,1);
	\draw (1.5,-1) -- (1.5,1);
	\draw (2,-1) 
		-- (2,-.25) 
		.. controls (2,0) and (2.75,0) .. (2.75,.25) 
		-- (2.75,.75)
		.. controls (2.75,1) .. (3,1);%node[right] {\small $i$};
	\draw (2,1) 
		-- (2,.25) 
		.. controls (2,0) and (2.75,0) .. (2.75,-.25) 
		-- (2.75,-.75)
		.. controls (2.75,-1) .. (3,-1) node[right] {\small $i$};
	\filldraw [fill=white, draw=black,rounded corners] (-.25,.25) rectangle (2.25,.75) node[midway] {\small $\psi_k$};
	\filldraw [fill=white, draw=black,rounded corners] (-.25,-.75) rectangle (2.25,-.25) node[midway] {\small $\varphi_k$};
	\node at(1,.9) {\tiny$\dots$};
	\node at(1,0) {\tiny $\dots$};
	\node at(1,-.9) {\tiny$\dots$};
}
+\ 
  \tikzdiagh[xscale=.5,yscale=.8]{0}{
	\draw (0,-1)  node[below,white]{$i$} -- (0,1);
	\draw (.5,-1) -- (.5,1);
	\draw (1.5,-1) -- (1.5,1);
	\draw (2,-1) -- (2,1);
	\draw (3,-1)  node [right] {\small $i$} .. controls (2.75,-1) ..
		(2.75,-.75) -- (2.75,.75)
		.. controls (2.75,1) .. (3,1);
	\filldraw [fill=white, draw=black,rounded corners] (-.25,-.35) rectangle (3,0.35) node[midway] {\small $\bar y_N(\xi_i^k)$};
	\node at(1,.65) {\tiny $\dots$};
	\node at(1,-.65) {\tiny $\dots$};
}
 \end{equation}
 for some $\varphi_k, \psi_k \in R_\bo(m)$. We conclude by observing that
 \begin{equation}\label{eq:dotslidepsiphik}
\tikzdiagh[xscale=.5,yscale=.8]{0}{
	\draw (0,-1)  node[below,white]{$i$} -- (0,1);
	\draw (.5,-1) -- (.5,1);
	\draw (1.5,-1) -- (1.5,1);
	\draw (2,-1) 
		-- (2,-.25) 
		.. controls (2,0) and (2.75,0) .. (2.75,.25) 
		-- (2.75,.75) node[pos=.8,tikzdot]{} node[pos=.8, xshift=1.5ex, yshift=.5ex]{\small $t$}
		.. controls (2.75,1) .. (3,1);%node[right] {\small $i$};
	\draw (2,1) 
		-- (2,.25) 
		.. controls (2,0) and (2.5,0) .. (2.75,-.25) 
		-- (2.75,-.75)
		.. controls (2.75,-1) .. (3,-1) node[right] {\small $i$};
	\filldraw [fill=white, draw=black,rounded corners] (-.25,-.75) rectangle (2.25,-.25) node[midway] {\small $\varphi_k$};
	\filldraw [fill=white, draw=black,rounded corners] (-.25,.25) rectangle (2.25,.75) node[midway] {\small $\psi_k$};
	\node at(1,.9) {\tiny$\dots$};
	\node at(1,0) {\tiny $\dots$};
	\node at(1,-.9) {\tiny$\dots$};
}
=
\tikzdiagh[xscale=.5,yscale=.8]{0}{
	\draw (0,-1)  node[below,white]{$i$} -- (0,1);
	\draw (.5,-1) -- (.5,1);
	\draw (1.5,-1) -- (1.5,1);
	\draw (2,-1) 
		-- (2,-.25) 
		.. controls (2,0) and (2.75,0) .. (2.75,.25)   node[pos=.2,tikzdot]{} node[pos=.2, xshift=-1ex, yshift=.75ex]{\small $t$}
		-- (2.75,.75)
		.. controls (2.75,1) .. (3,1);%node[right] {\small $i$};
	\draw (2,1) 
		-- (2,.25) 
		.. controls (2,0) and (2.75,0) .. (2.75,-.25) 
		-- (2.75,-.75)
		.. controls (2.75,-1) .. (3,-1) node[right] {\small $i$};
	\filldraw [fill=white, draw=black,rounded corners] (-.25,-.8) rectangle (2.25,-.3) node[midway] {\small $\varphi_k$};
	\filldraw [fill=white, draw=black,rounded corners] (-.25,.3) rectangle (2.25,.8) node[midway] {\small $\psi_k$};
	\node at(1,.9) {\tiny$\dots$};
	\node at(1,0) {\tiny $\dots$};
	\node at(1,-.9) {\tiny$\dots$};
}
- r_i \sssum{r+s\\=t-1} \ 
\tikzdiagh[xscale=.5,yscale=.8]{0}{
	\draw (0,-1)  node[below,white]{$i$} -- (0,1);
	\draw (.5,-1) -- (.5,1);
	\draw (1.5,-1) -- (1.5,1);
	\draw (2,-1)  -- (2,1)  node[midway,tikzdot]{} node[midway, xshift=1.5ex, yshift=.5ex]{\small $r$};
	\draw (3.25, -1) node[right] {\small $i$}
	 .. controls (3,-1) .. (3,-.75)
	 -- (3,.75) node[midway,tikzdot]{} node[midway, xshift=1.5ex, yshift=.5ex]{\small $s$}
	 .. controls (3,1) .. (3.25,1);
	\filldraw [fill=white, draw=black,rounded corners] (-.25,-.8) rectangle (2.25,-.3) node[midway] {\small $\varphi_k$};
	\filldraw [fill=white, draw=black,rounded corners] (-.25,.3) rectangle (2.25,.8) node[midway] {\small $\psi_k$};
	\node at(1,.9) {\tiny$\dots$};
	\node at(1,0) {\tiny $\dots$};
	\node at(1,-.9) {\tiny$\dots$};
}
 \end{equation}
 thanks to \cref{eq:KLRnh}.
\end{proof}

\begin{prop}\label{prop:Rbomonicpol}
Putting $\rho_i := n_i - \alpha_i^\vee(\nu)$, we have
\[
\bar y_N\left(
 \tikzdiagh[yscale = .8]{0}{
	\draw  (.25,1) .. controls (.25,.5) and (.75,.5) ..   (.75,0)   .. controls (.75,-.5) and (.25,-.5)  .. (.25,-1); 
	\draw  (.5,1) .. controls (.5,.5) and (1,.5) ..   (1,0)   .. controls (1,-.5) and (.5,-.5)  .. (.5,-1); 
	\draw  (1,1) .. controls (1,.5) and (1.5,.5) ..   (1.5,0)   .. controls (1.5,-.5) and (1,-.5)  .. (1,-1); 
	\draw  (1.25,1) .. controls (1.25,.5) and (1.75,.5) ..   (1.75,0)   .. controls (1.75,-.5) and (1.25,-.5)  .. (1.25,-1); 
 \draw  (1.75,-1)  node[below]{$i$}.. controls (1.75,-.5) and (.25,-.5) ..    (.25,0) node[pos=1, tikzdot]{} node[pos=1,xshift=-1.5ex,yshift=.5ex]{\small $k$}
	 .. controls (.25,.5)  and (1.75,.5).. (1.75,1);
	 \node at (1.25,0) {\tiny $\dots$};  \node at (.8,.85) {\tiny $\dots$};  \node at (.8,-.85) {\tiny $\dots$};
	 \fdot{}{.5,0};
	\draw[decoration={brace,mirror,raise=-8pt},decorate]  (.15,-1.35) -- node {$\nu$} (1.35,-1.35);
 }
\right)
 \ \equiv \zeta \ 
 \tikzdiagh[xscale=.5,yscale=.8]{0}{
	\draw (0,-1)  node[below,white]{$i$} -- (0,1);
	\draw (.5,-1) -- (.5,1);
	\draw (1.5,-1) -- (1.5,1);
	\draw (2,-1) -- (2,1);
	\draw (3,-1)  node [right] {\small $i$} .. controls (2.75,-1) ..
		(2.75,-.75) -- (2.75,.75) node[midway, tikzdot]{} node[midway,xshift=3.5ex, yshift=.5ex]{\small $k{+}\rho_i $}
		.. controls (2.75,1) .. (3,1);
	\node at (1,0) {\tiny $\dots$};
}
+ 
\bigoplus_{\ell = 0}^{k+\rho_i -1}
\tikzdiagh[xscale=.5,yscale=.8]{0}{
	\draw (0,-1)  node[below,white]{$i$} -- (0,1);
	\draw (.5,-1) -- (.5,1);
	\draw (1.5,-1) -- (1.5,1);
	\draw (2,-1) -- (2,1);
	\draw (3,-1)  node [right] {\small $i$} .. controls (2.75,-1) ..
		(2.75,-.75) -- (2.75,.75) node[midway, tikzdot]{} node[midway,xshift=1.5ex, yshift=.5ex]{\small $\ell$}
		.. controls (2.75,1) .. (3,1);
	\filldraw [fill=white, draw=black] (-.25,-.35) rectangle (2.25,.35) node[midway] {\small $H(m)$};
	\node at (1,.65) {\tiny $\dots$}; \node at (1,-.65) {\tiny $\dots$};
}
\]
which is $0$ whenever $k + \rho_i < 0$, and where $\zeta \in \Bbbk^\times$.
\end{prop}

\begin{proof}
If $n_i \geq 2\nu_i$, then the result follows from  \cref{lem:bigp}. Otherwise, we take $k' = 2\nu_i - n_i$ and the result follows from \cref{lem:bigp} for $k \geq k'$. 
Suppose $k  <  k'$ and put $t = k' - k$. Then, by \cref{lem:pidotslide} we obtain
\begin{align*}
\bar y_N\left(
 \tikzdiagh[yscale = .8]{0}{
	\draw  (.25,1) .. controls (.25,.5) and (.75,.5) ..   (.75,0)   .. controls (.75,-.5) and (.25,-.5)  .. (.25,-1); 
	\draw  (.5,1) .. controls (.5,.5) and (1,.5) ..   (1,0)   .. controls (1,-.5) and (.5,-.5)  .. (.5,-1); 
	\draw  (1,1) .. controls (1,.5) and (1.5,.5) ..   (1.5,0)   .. controls (1.5,-.5) and (1,-.5)  .. (1,-1); 
	\draw  (1.25,1) .. controls (1.25,.5) and (1.75,.5) ..   (1.75,0)   .. controls (1.75,-.5) and (1.25,-.5)  .. (1.25,-1); 
 \draw  (1.75,-1)  node[below]{$i$}.. controls (1.75,-.5) and (.25,-.5) ..    (.25,0) node[pos=1, tikzdot]{} node[pos=1,xshift=-1.5ex,yshift=.5ex]{\small $k'$}
	 .. controls (.25,.5)  and (1.75,.5).. (1.75,1);
	 \node at (1.25,0) {\tiny $\dots$};  \node at (.8,.85) {\tiny $\dots$};  \node at (.8,-.85) {\tiny $\dots$};
	 \fdot{}{.5,0};
 }
 \right)
 \ &\equiv \ 
  \tikzdiagh[xscale=.5,yscale=.8]{0}{
	\draw (0,-1)  node[below,white]{$i$} -- (0,1);
	\draw (.5,-1) -- (.5,1);
	\draw (1.5,-1) -- (1.5,1);
	\draw (2,-1) -- (2,1);
	\draw (3,-1)  node [right] {\small $i$} .. controls (2.75,-1) ..
		(2.75,-.75) -- (2.75,.75) node[pos=.9, tikzdot]{} node[pos=.9,xshift=1.5ex,yshift=.5ex]{\small $t$}
		.. controls (2.75,1) .. (3,1);
	\filldraw [fill=white, draw=black,rounded corners] (-.25,-.35) rectangle (3,0.35) node[midway] {\small $\bar y_N(\xi_i^k)$};
	\node at(1,.65) {\tiny $\dots$};
	\node at(1,-.65) {\tiny $\dots$};
}
 + 
 \sum_{\ell=0}^{t-1}
 \tikzdiagh[xscale=.5,yscale=.8]{0}{
	\draw (0,-1)  node[below,white]{$i$} -- (0,1);
	\draw (.5,-1) -- (.5,1);
	\draw (1.5,-1) -- (1.5,1);
	\draw (2,-1) -- (2,1);
	\draw (3,-1)  node [right] {\small $i$} .. controls (2.75,-1) ..
		(2.75,-.75) -- (2.75,.75) node[midway, tikzdot]{} node[midway,xshift=1.5ex,yshift=.5ex]{\small $\ell$}
		.. controls (2.75,1) .. (3,1);
	\filldraw [fill=white, draw=black] (-.25,-.35) rectangle (2.25,.35) node[midway] {\small $H(m)$};
	\node at(1,.65) {\tiny $\dots$};
	\node at(1,-.65) {\tiny $\dots$};
}
\end{align*}
Therefore, we have
\[
  \tikzdiagh[xscale=.5,yscale=.8]{0}{
	\draw (0,-1)  node[below,white]{$i$} -- (0,1);
	\draw (.5,-1) -- (.5,1);
	\draw (1.5,-1) -- (1.5,1);
	\draw (2,-1) -- (2,1);
	\draw (3,-1)  node [right] {\small $i$} .. controls (2.75,-1) ..
		(2.75,-.75) -- (2.75,.75) node[pos=.9, tikzdot]{} node[pos=.9,xshift=1.5ex,yshift=.5ex]{\small $t$}
		.. controls (2.75,1) .. (3,1);
	\filldraw [fill=white, draw=black,rounded corners] (-.25,-.35) rectangle (2.75,0.35) node[midway] {\small $\bar y_N(\xi_i^k)$};
	\node at(1,.65) {\tiny $\dots$};
	\node at(1,-.65) {\tiny $\dots$};
} 
\ \equiv \zeta \ 
 \tikzdiagh[xscale=.5,yscale=.8]{0}{
	\draw (0,-1)  node[below,white]{$i$} -- (0,1);
	\draw (.5,-1) -- (.5,1);
	\draw (1.5,-1) -- (1.5,1);
	\draw (2,-1) -- (2,1);
	\draw (3,-1)  node [right] {\small $i$} .. controls (2.75,-1) ..
		(2.75,-.75) -- (2.75,.75) node[midway, tikzdot]{} node[midway,xshift=3.5ex, yshift=.5ex]{\small $k'{+}\rho_i $}
		.. controls (2.75,1) .. (3,1);
	\node at (1,0) {\tiny $\dots$};
}
+ 
\bigoplus_{\ell = 0}^{\max(k'+\rho_i , t) -1}
\tikzdiagh[xscale=.5,yscale=.8]{0}{
	\draw (0,-1)  node[below,white]{$i$} -- (0,1);
	\draw (.5,-1) -- (.5,1);
	\draw (1.5,-1) -- (1.5,1);
	\draw (2,-1) -- (2,1);
	\draw (3,-1)  node [right] {\small $i$} .. controls (2.75,-1) ..
		(2.75,-.75) -- (2.75,.75) node[midway, tikzdot]{} node[midway,xshift=1.5ex, yshift=.5ex]{\small $\ell$}
		.. controls (2.75,1) .. (3,1);
	\filldraw [fill=white, draw=black] (-.25,-.35) rectangle (2.25,.35) node[midway] {\small $H(m)$};
	\node at (1,.65) {\tiny $\dots$}; \node at (1,-.65) {\tiny $\dots$};
}
\]
From this, we deduce
\[
  \tikzdiagh[xscale=.5,yscale=.8]{0}{
	\draw (0,-1)  node[below,white]{$i$} -- (0,1);
	\draw (.5,-1) -- (.5,1);
	\draw (1.5,-1) -- (1.5,1);
	\draw (2,-1) -- (2,1);
	\draw (2.75,-1)  node [right] {\small $i$} .. controls (2.5,-1) ..
		(2.5,-.75) -- (2.5,.75) 
		.. controls (2.5,1) .. (2.75,1);
	\filldraw [fill=white, draw=black,rounded corners] (-.25,-.35) rectangle (2.75,0.35) node[midway] {\small $\bar y_N(\xi_i^k)$};
	\node at(1,.65) {\tiny $\dots$};
	\node at(1,-.65) {\tiny $\dots$};
} 
\ \equiv  \zeta \ 
\tikzdiagh[xscale=.5,yscale=.8]{0}{
	\draw (0,-1)  node[below,white]{$i$} -- (0,1);
	\draw (.5,-1) -- (.5,1);
	\draw (1.5,-1) -- (1.5,1);
	\draw (2,-1) -- (2,1);
	\draw (3,-1)  node [right] {\small $i$} .. controls (2.75,-1) ..
		(2.75,-.75) -- (2.75,.75) node[midway, tikzdot]{} node[midway,xshift=3.5ex, yshift=.5ex]{\small $k{+}\rho_i $}
		.. controls (2.75,1) .. (3,1);
	\node at (1,0) {\tiny $\dots$};
}
+ 
\bigoplus_{\ell = 0}^{k+\rho_i -1}
\tikzdiagh[xscale=.5,yscale=.8]{0}{
	\draw (0,-1)  node[below,white]{$i$} -- (0,1);
	\draw (.5,-1) -- (.5,1);
	\draw (1.5,-1) -- (1.5,1);
	\draw (2,-1) -- (2,1);
	\draw (3,-1)  node [right] {\small $i$} .. controls (2.75,-1) ..
		(2.75,-.75) -- (2.75,.75) node[midway, tikzdot]{} node[midway,xshift=1.5ex, yshift=.5ex]{\small $\ell$}
		.. controls (2.75,1) .. (3,1);
	\filldraw [fill=white, draw=black] (-.25,-.35) rectangle (2.25,.35) node[midway] {\small $H(m)$};
	\node at (1,.65) {\tiny $\dots$}; \node at (1,-.65) {\tiny $\dots$};
}
\]
which concludes the proof.
\end{proof}

We now have all the tools  we need to compute the homology of the cokernel of the short exact sequence of \cref{prop:RbodgSES}.

\begin{prop}\label{prop:HcokRbo}
There is an isomorphism of $R_\p^N(\nu)$-$R_\p^N(\nu)$-bimodules
\begin{align*}
H\bigl( R_\bo^{\xi_i}(\nu) \oplus  \lambda_i^2 q_i^{ - 2\alpha_i^\vee(\nu)} &R_\bo^{\xi_i}(\nu)[1], d_N \bigr) \\
&\cong
\begin{cases}
\bigoplus_{\ell = 0} ^{\rho_i - 1} q_i^{2 \ell} R_\p^N(\nu), & \text{if $\rho_i \geq 0$,} \\
\lambda_i^2 q_i^{-2 \alpha_i^\vee(\nu)}\bigoplus_{\ell = 0}^{-\rho_i-1} q_i^{2 \ell} R_\p^N(\nu)[1], & \text{if $\rho_i \leq 0$,}
\end{cases}
\end{align*}
where $\rho_i = n_i - \alpha_i^\vee(\nu)$.
\end{prop}

\begin{proof}
First suppose $\rho_i \geq 0$. Then, \cref{prop:Rbomonicpol} tells us that $\bar y_N(\xi_i^k)$ is a monic polynomial (up to invertible scalar) with leading terms $\xi_i^{k+\rho_i}$. This gives us the first case. If $\rho_i \leq 0$, then we have $\bar y_N(\xi_i^k) = 0$ for $k < -\rho_i$. Moreover, $\zeta^{-1} \bar y_N(\xi_i^{-\rho_i}) = 1$, and in general  $\bar y_N(\xi_i^k)$ is a monic polynomial with leading term $\xi_i^{k+\rho_i}$ for $k > -\rho_i$. This concludes the proof.
\end{proof}

\subsection{Strongly projective dg-modules}

The following notions were originally introduced by Moore~\cite{moore}. We use the presentation given in~\cite{sixdgmodels}, which is best suited for our notations. 

\begin{defn}[{\cite[Definition 8.5]{sixdgmodels}}]
Let $(R,0)$ be a ring $R$ viewed as a dg-$\bZ$-algebra concentrated in degree zero.
An $(R,0)$-module $(Q,d_Q)$ is \emph{strongly projective} if $H(Q,d_Q)$ and $\Image d_Q$ are both projective $R$-modules.
\end{defn}

\begin{lem}[{\cite[Theorem~9.3.2]{vermani}}]
Let $(P,d_P)$ be a strongly projective right $(R,0)$-module and $(N,d_N)$ any left $(R,0)$-module, then
\[
H\bigl((P,d_P) \otimes_{(R,0)} (N,d_N)\bigr) \cong H(P,d_P) \otimes_R H(N,d_N).
\]
\end{lem}

\begin{defn}[{\cite[Definition~8.17]{sixdgmodels}}]
Let $(A,d_A)$ be a dg-R-algebra. A left (resp. right) $(A,d_A)$-module $(P,d_P)$ is \emph{strongly projective} if it is a dg-direct summand of $(A,d_A) \otimes_{(R,0)} (Q,d_Q)$ (resp. $(Q,d_Q) \otimes_{(R,0)} (A,d_A)$)  for some strongly projective $(R,0)$-module $(Q,d_Q)$.
\end{defn}

\begin{prop}[{\cite[Lemma~8.23]{sixdgmodels}}] \label{prop:tensorH}
If $(P,d_P)$ is a strongly projective right $(A,d_A)$-module and $(N,d_N)$ is any left $(A,d_A)$-module, then
\[
H\bigl((P,d_P) \otimes_{(A,d_A)} (N,d_N)\bigr) \cong H(P,d_P) \otimes_{H(A,d_A)} H(N,d_N).
\]
\end{prop}

Note that if $(P,d_P)$ is a strongly projective $(A,d_A)$-module, then $H(P,d_P)$ is a projective $H(A,d_A)$-module. 
Indeed, we can assume $(P,d_P) = (A,d_A) \otimes_{(R,0)} (Q,d_Q)$, and we have
\[
H(P,d_P) \cong H(A,d_A) \otimes_{R} H(Q,d_Q).
\]
Since $H(Q,d_Q)$ is a projective $R$-module, it is a direct summand of a free $R$-module~$F$. Therefore $H(P,d_P)$ is a direct summand of $H(A,d_A) \otimes_R F$, which is a free $H(A,d_A)$-module.

\begin{rem}
This result does not hold in general. As a counterexample we can take $(A,d) = (\bQ[x],0)$ and consider the dg-module $(X,d_X) = \cone(\bQ[x] \xrightarrow{x} \bQ[x])$. In this case we have that $H(X,d_X)  \cong \bQ$ but $H((X,d_X) \otimes_{(A,d)} (X,d_X)) \cong \bQ \oplus \bQ[1]$.
\end{rem}

\subsubsection{Strong projectivity of $R_\bo(m+1)$}

Our next goal is to show the following:

\begin{prop}\label{prop:Rbosproj}
The $(R_\bo(m),d_N)$-module $( 1_{(m,i)}R_\bo(m+1), d_N)$ is strongly projective.
\end{prop}

It is obvious for $i \notin I_f$ by \cref{lem:Rboleftdecomp}, and thus we can assume $i \in I_f$.
We first construct the mapping cone
\begin{align*}
(Q,& d_Q) := \\
&\cone\bigl( 
\bigoplus_{a=1}^{m+1} \bigoplus_{\ell \geq 0}  R_\g(m) 1_{(\nu,i)} \tau_m \cdots \tau_a  \theta_a^\ell
\xrightarrow{d_Q}
\bigoplus_{a=1}^{m+1} \bigoplus_{\ell \geq 0} R_\g(m) 1_{(\nu,i)} \tau_m \cdots \tau_a x_a^\ell
 \bigr),
\end{align*}
where we think of $\tau_m \cdots \tau_a  \theta_a^\ell$ as a formal symbol that represents a degree shift corresponding to the degree of the element $1_{(\nu,i)}\tau_m \cdots \tau_a  \theta_a^\ell$ in $R_\bo(m+1)$.
The map $d_Q$ is given by first embedding $R_\g(m)$ into $R_\bo(m+1)$ through the diagrams
\[
R_\g(m) 1_{(\nu,i)} \tau_m \cdots \tau_a  \theta_a^\ell \hookrightarrow
\tikzdiag[xscale=.75]{
	\draw (0,-1.5) 
		.. controls (0,-1.25) and (.75,-1.25) ..(.75,-1)
		.. controls (.75,-.75) and (0,-.75) ..(0,-.5)
		 -- (0,.5);
	\draw (.5,-1.5) 
		.. controls (.5,-1.25) and (1.25,-1.25) ..(1.25,-1)
		.. controls (1.25,-.75) and (.5,-.75) ..(.5,-.5)
		 -- (.5,.5);
	\draw (2,-1.5) .. controls (2,-1) and (1.5,-1) .. (1.5,-.5) -- (1.5,.5);
	\draw (2.5, -1.5) .. controls (2.5,-1) and (2,-1) .. (2,-.5) -- (2,.5);
	\node at(1,.4) {\small$\dots$};
	\node at(1,-.4) {\small$\dots$};
	\filldraw [fill=white, draw=black] (-.25,-.25) rectangle (2.25,0.25) node[midway] {\small $m$};
	\draw (1.25, -1.5)  
		.. controls (1.25,-1.25) and (0,-1.25) .. (0, -1)  node[pos=1, tikzdot]{} node[pos=1, xshift=-1.5ex, yshift=.5ex]{\small $\ell$}
		.. controls (0,-.75) and (2.5,-.75) .. (2.5, -.25) 
		-- (2.5,.25)
		.. controls (2.5,.5) .. (2.75,.5)  node[right] {$i$};
	\fdot{}{.4,-1.015};
	\draw[decoration={brace,mirror,raise=-8pt},decorate]  (-.15,-1.85) -- node {$a$} (1.35,-1.85);
}
\]
then applying $d_N$ of $(R_\bo(m+1), d_N)$, then decomposing the image in the left-decomposition $\bigoplus_{a=1}^{m+1} \bigoplus_{\ell \geq 0} R_\bo(m) 1_{(m,i)} \tau_m \cdots \tau_a x_a^\ell $, and finally projecting unto the part in homogical degree zero of $R_\bo(m)$, which is trivially isomorphic to $R_\g(m)$. 
Moreover, $(R_\bo(m),d_N)$ is a (right) module over $(R_\g,0)$ which acts by gluing KLR diagrams on the bottom. 
Then, we have, as $(R_\bo(m),d_N)$-modules
\[
(R_\bo(m+1) ,d_N) \cong (R_\bo(m), d_N) \otimes_{(R_\g(m), 0)} (Q,d_Q).
\]
Therefore, we want to show that $(Q,d_Q)$ is strongly projective as $(R_\g(m), 0)$-module. We write 
\begin{align*}
Q_1[\xi_i] &:= \bigoplus_{a=1}^{m+1} \bigoplus_{\ell \geq 0}  R_\g(m) 1_{(\nu,i)} \tau_m \cdots \tau_a  \theta_a^\ell, \\
Q_0[\xi_i] &:= \bigoplus_{a=1}^{m+1} \bigoplus_{\ell \geq 0} R_\g(m) 1_{(\nu,i)} \tau_m \cdots \tau_a x_a^\ell,
\end{align*}
where we identify $\xi_i$ with $x_a$ in $Q_0$, and $\xi_i^\ell$ with $x_1^\ell$ in $\theta_a^\ell$. Note that $d_Q$ is not $\Bbbk[\xi_i]$-linear.

\begin{lem}\label{lem:RbodQinj}
The map 
\[
d_Q : Q_1[\xi_i]
\rightarrow
Q_0[\xi_i]
\]
defined above is injective.
\end{lem}

\begin{proof}
Recall the map $P$ of \cref{lem:RboMapP} given by multiplication by $\widetilde \theta_{m+1}$. Since floating dots are also annihilated in $R_\g(m)$, multiplication by $\widetilde \theta_{m+1}$ also defines a map
\begin{align}\label{eq:P'd_Q}
P' : 
Q_0[\xi_i]
\rightarrow
Q_1[\xi_i].
\end{align}
We reconsider the proof of \cref{lem:computePdN} to show that $P'\circ d_Q$ is injective. First, we introduce an order on the summands of $Q_1[\xi_i] = \bigoplus_{a=1}^{m+1} \bigoplus_{\ell \geq 0}  R_\g(m) 1_{(\nu,i)} \tau_m \cdots \tau_a  \theta_a^\ell$ by declaring that 
\begin{align*}
 R_\g(m) 1_{(\nu,i)} \tau_m \cdots \tau_a  \theta_a^\ell &\prec   R_\g(m) 1_{(\nu,i)} \tau_m \cdots \tau_a  \theta_a^{\ell'}, \\
 R_\g(m) 1_{(\nu,i)} \tau_m \cdots \tau_a  \theta_a^\ell &\prec   R_\g(m) 1_{(\nu,i)} \tau_m \cdots \tau_{a'}  \theta_{a'}^{\ell''}, 
\end{align*}
for all $a > a'$, $\ell < \ell'$, and for all $\ell''$. In other words, if there are more crossings under the floating dot, then the term is smaller. If there is the same amount of crossings, then we consider the amount of dots at the left of the floating dot, and lesser dots meaning a smaller term.  

We claim that if $Z \in R_\g(m) 1_{(\nu,i)} \tau_m \cdots \tau_a  \theta_a^\ell$ then 
\[
P' \circ d_Q(Z) =  r_i^{2 \nu_i} \sum_{p = 0}^{2\nu_i- \alpha_i^\vee(\nu)} Zx_{m+1}^{n_i+p} \varepsilon_p^i(\und x_\nu) + H,
\]
where $H \prec Z x_{m+1}^{\ell+n_i+2\nu_i-\alpha_i^\vee(\nu)}$. 
This implies that $P' \circ d_Q$ is in echelon form (with pivot being invertible scalars), and thus is injective. By consequence, so is $d_Q$.

In order to prove our claim, we need to tweak 
the proof of \cref{lem:computePdN}. 
We need to keep track of the terms that are annihilated when working over the cyclotomic quotient, and show these appear as lower terms in the order defined above. The case $j_m \neq i$ remains the same. The case $j_m = i$ and $m=1$ becomes
\begin{align*}
\tikzdiagh[scale=.75]{0}{
 	\draw (1,-.75) node[below] {\small $i$} 
 	.. controls (1,-.375) and (0,-.375) .. (0,0) 
 	.. controls (0,.375) and (1, .375) .. (1,.75)
 	.. controls (1,1.125) and (0,1.125) .. (0,1.5);
	\draw[fill=white, color=white] (.5,.375) circle (.15cm);
	\draw  (0,-.75) node[below] {\small $i$} 
	.. controls (0,-.375) and (1,-.375) .. (1,0) node [pos = .2,tikzdot]{} node [pos = .2, xshift=-1.5ex, yshift=.5ex] {\small $p$} 
	 .. controls (1,.375) and (0, .375) .. (0,.75) 
	 .. controls (0,1.125) and (1,1.125) .. (1,1.5);
	 \fdot{}{.5,.75};
}
\ &=  
r_i^2 \ 
\tikzdiagh[scale=.85]{0}{
	          \draw   (-.5,-.5) node[below]{$i$} -- (-.5,.5) node [midway,tikzdot]{}  node [midway, xshift=1.5ex, yshift=.5ex] {\small $p$}
	          		.. controls (-.5,1) and (.5,1) .. (.5,1.5);
	          \draw   (.5,-.5) node[below]{$i$} -- (.5,.5)
	          		.. controls (.5,1) and (-.5,1) .. (-.5,1.5);
		 \fdot{}{0,.5};
  	} 
\ - r_i^2 \ 
\tikzdiagh[scale=.85]{0}{
	          \draw  (.5,-1.5) node[below]{$i$} .. controls (.5,-1) and (-.5,-1) ..
	          		(-.5,-.5) -- (-.5,.5) node [midway,tikzdot]{} node [midway, xshift=1.5ex, yshift=.5ex] {\small $p$};
	          \draw  (-.5,-1.5)  node[below]{$i$} .. controls (-.5,-1) and (.5,-1) ..
	          		(.5,-.5)  -- (.5,.5);
		 \fdot{}{0,-1.4};
  	}
\ + r_i \ 
\tikzdiagh[scale=.75]{0}{
	          \draw   (-.5,-.5) node[below]{$i$} .. controls (-.5,0) and (.5,0) .. (.5,.5) 
			.. controls (.5,1) and (-.5,1) .. (-.5,1.5) node [pos = .8,tikzdot]{} node [pos = .8, xshift=1.5ex, yshift=.5ex] {\small $p$};
	          \draw   (.5,-.5) node[below]{$i$} .. controls (.5,0) and (-.5,0) .. (-.5,.5) node [pos = 1,tikzdot]{} 
			.. controls (-.5,1) and (.5,1) .. (.5,1.5);
		 \fdot{}{0,.5};
  	} 
\ -r_i \ 
\tikzdiagh[scale=.75]{0}{
	          \draw   (-.5,-.5) node[below]{$i$} .. controls (-.5,0) and (.5,0) .. (.5,.5) 
			.. controls (.5,1) and (-.5,1) .. (-.5,1.5) node [pos = .8,tikzdot]{}  node [pos = .8, xshift=2.25ex, yshift=1.5ex] {\small $p{+}1$};
	          \draw   (.5,-.5) node[below]{$i$} .. controls (.5,0) and (-.5,0) .. (-.5,.5)
			.. controls (-.5,1) and (.5,1) .. (.5,1.5);
		 \fdot{}{0,.5};
  	} 
\end{align*}
where $p = n_i + \ell$. The first term is the leading term. The second term possesses less dots on the left of the floating dot, and so it is smaller. If $a=0$, then the last two terms possess one more crossing at the bottom of the floating dot, and therefore they are smaller. If $a=1$,  then they are annihilated by~\cref{eq:KLRR2}. Finally, the two remaining cases $j_{m-1} \neq i$ and $j_{m-1}$ follow from the same arguments as in the proof of \cref{lem:computePdN}, with the lower terms in the induction hypothesis only adding lower terms because:
\begin{align*}
\tikzdiagh[xscale=.85, yscale=.5]{0}{
	\draw (0,1) 
		.. controls (0,0) and (3,0) ..
		(3,-1) node[below]{\small $i$};
	\draw (1,1)
		.. controls (1,.5) and (0,.5) .. (0,0)
		.. controls (0,-.5) and (1,-.5) .. (1,-1) node[below]{\small $i$};
	\draw (2,1)
		.. controls (2,.05) and (0,.05) .. (0,-1);
	\draw[xshift=1ex] (2,1)
		.. controls (2,-.05) and (0,-.05) ..  (0,-1);
	\draw (3,1)
		.. controls (3,.05) and (2,.05) ..  (2,-1);		
	\draw[xshift=1ex] (3,1)
		.. controls (3,-.05) and (2,-.05) ..  (2,-1);		
}
\ &= \ 
\tikzdiagh[xscale=.85, yscale=.5]{0}{
	\draw (0,1) 
		.. controls (0,-.25) and (3,-.25) ..
		(3,-1) node[below]{\small $i$};
	\draw (1,1)
		.. controls (1,.5) and (2,.5) .. (2,0)
		.. controls (2,-.5) and (1,-.5) .. (1,-1) node[below]{\small $i$};
	\draw (2,1)
		.. controls (2,.05) and (0,.05) .. (0,-1);
	\draw[xshift=1ex] (2,1)
		.. controls (2,-.05) and (0,-.05) ..  (0,-1);
	\draw (3,1)
		.. controls (3,.05) and (2,.05) ..  (2,-1);		
	\draw[xshift=1ex] (3,1)
		.. controls (3,-.05) and (2,-.05) ..  (2,-1);		
}
\intertext{ by~\eqref{eq:KLRR3}, and,}
\tikzdiagh[xscale=.85, yscale=.5]{0}{
	\draw (0,1)
		.. controls (0,.5) and (2,.5) .. (2,0)
		.. controls (2,-.5) and (1,-.5) .. (1,-1)  node[below]{\small $i$};
	\draw (1,1)
		.. controls (1,.5) and (0,.5) .. (0,0)
		.. controls (0,-.5) and (3,-.5) .. (3,-1)  node[below]{\small $i$};
	\draw (2,1)
		.. controls (2,.05) and (0,.05) .. (0,-1);
	\draw[xshift=1ex] (2,1)
		.. controls (2,-.05) and (0,-.05) .. (0,-1);
	\draw (3,1)
		.. controls (3,.05) and (2,.05) .. (2,-1);
	\draw[xshift=1ex] (3,1)
		.. controls (3,-.05) and (2,-.05) .. (2,-1);
}
\ &= 0,
\end{align*}
by~\cref{eq:KLRR3} and~\cref{eq:KLRR2}.
 This concludes the proof of the claim, and therefore of the proposition.
\end{proof}

\begin{proof}[Proof of \cref{prop:Rbosproj}]
The proof is a revisit of the proof of~\cite[Lemma~4.18]{kashiwara} that applies to our particular case.

Recall the map $P'$ from \cref{eq:P'd_Q}. We know that $P' \circ d_Q$ is given by multiplying by a monic polynomial with leading term $x_{m+1}^{n_i+2\nu_i-\alpha_i^\vee(\nu)}$ plus some remaining map giving lower terms. In particular, it is injective and we have a short exact sequence
\[
0 \rightarrow Q_1[\xi_i] \xrightarrow{P' \circ d_Q} Q_1[\xi_i] \rightarrow \cok(P' \circ d_Q) \rightarrow 0.
\]
Moreover, since $P' \circ d_Q$ is in echelon form, it means that $\cok(P' \circ d_Q)$ is a projective $R_\g(m)$-module. Thus, the sequence splits as $R_\g(m)$-modules with splitting map $\sigma : Q_1[\xi_i] \rightarrow Q_1[\xi_i]$, and we get $\sigma \circ P' \circ d_Q = \id_{Q_1[\xi_i]}$. 
Then, the short exact sequence 
\[
\begin{tikzcd}
0 
\ar{r} 
&
Q_1[\xi_i]
\ar{r}{d_Q}
&
Q_0[\xi_i]
\ar{r} 
\ar[bend left, dashed]{l}{\sigma \circ P'}
&
H(Q, d_Q), 
\ar{r} 
&
0,
\end{tikzcd}
 \]
obtained thanks to \cref{lem:RbodQinj} splits with splitting map given by $\sigma \circ P' $. Since $Q_0[\xi_i]$ is a projective $R_g(m)$-module, so is $H(Q,d_Q)$. Finally, $d_Q(Q_1[\xi_i])$ is also projective since $d_Q$ is injective and $Q_1[\xi_i]$ is projective.
\end{proof}

\subsection{Functors}\label{sec:FiNFunctors}

We define for all $i \in I$ the functors
\begin{align*}
\F^N_i(-) &:= \bigoplus_{m \geq 0} R^N_\p(m+1) 1_{(m,i)} \otimes_{R^N_\p(m)}  (-), \\
\E^N_i(-) &:= \bigoplus_{m \geq 0} \ \bigoplus_{|\nu| = m}   \lambda_i^{-1} q_i^{1+\alpha_i^\vee(\nu)} 1_{(\nu,i)} R^N_\p(m+1) \otimes_{R^N_\p(m+1)} (-),
\end{align*}
where we interpret $\lambda_i = q^{n_i}$ whenever $i \in I_f$.
Thanks to \cref{prop:Rbosproj}, these are exact. 
For $n \in \bN$, we write
\[
\oplus_{[n]_{q_i}} \id_\nu := \bigoplus_{\ell = 0}^{n-1} q_i^{1-n+2\ell} \id_{\nu},
\]
for the finite direct sum that categorifies $[n]_{q_i}$.

\begin{thm}\label{thm:sl2commutRpN}
For $i \notin I_f$ there is a natural short exact sequence
\begin{equation}\label{eq:catSeSRpN}
0 \rightarrow \F^N_i\E^N_i \id_\nu \rightarrow \E^N_i\F^N_i \id_\nu \rightarrow \oplus_{[\beta_i - \alpha_i^\vee(\nu)]_{q_i}} \id_\nu \rightarrow 0,
\end{equation}
and for $i \in I_f$ there are natural isomorphisms
\begin{equation}\label{eq:catIsoRpN}
  \begin{aligned}
 \E^N_i\F^N_i \id_\nu &\cong   \F^N_i\E^N_i \id_\nu \oplus_{[n_i - \alpha_i^\vee(\nu)]_{q_i}} \id_\nu, & \text{ if $n_i - \alpha_i^\vee(\nu) \geq 0$}, \\
 \F^N_i\E^N_i \id_\nu &\cong   \E^N_i\F^N_i \id_\nu \oplus_{[\alpha_i^\vee(\nu)-n_i]_{q_i}} \id_\nu, & \text{ if $n_i - \alpha_i^\vee(\nu) \leq 0$}.
\end{aligned}
\end{equation}
Moreover, there is a natural isomorphism
\begin{equation}\label{eq:catijcomRpN}
\F^N_i\E^N_j \cong \E^N_j\F^N_i,
\end{equation}
 for $i \neq j \in I$. 
\end{thm}

\begin{proof}
The short exact sequence~\cref{eq:catSeSRpN} and the isomorphism~\cref{eq:catijcomRpN} are immediate consequences of \cref{prop:RbodgSES} and \cref{prop:Rbosproj}. 
For the isomorphisms~\cref{eq:catIsoRpN}, \cref{prop:RbodgSES} and \cref{prop:Rbosproj} give a long exact sequence of $R_\p^N(\nu)$-$R_\p^N(\nu)$-bimodules. By  \cref{prop:HcokRbo} it truncates to a short exact sequence
\begin{align*}
0 \rightarrow  \F^N_i\E^N_i \id_\nu \rightarrow  \E^N_i\F^N_i \id_\nu \rightarrow \oplus_{[\rho_i]} \id_\nu \rightarrow 0,
\end{align*}
if $\rho_i = n_i - \alpha_i^\vee(\nu) \geq 0$, and a short exact sequence
\begin{align*}
0 \rightarrow \oplus_{[-\rho_i]} \id_\nu \rightarrow \F^N_i\E^N_i \id_\nu \rightarrow   \E^N_i\F^N_i \id_\nu \rightarrow 0,
\end{align*}
if $\rho_i = n_i - \alpha_i^\vee(\nu) \leq 0$. In the first case, we can identify 
\[
\oplus_{[\rho_i]} \id_\nu 
\cong
q_i^{1-\rho_i} \bigoplus_{\ell =0}^{\rho_i-1}
\tikzdiagh[xscale=.5,yscale=.8]{0}{
	\draw (0,-1)  node[white,below]{$i$} -- (0,1);
	\draw (.5,-1) -- (.5,1);
	\draw (1.5,-1) -- (1.5,1);
	\draw (2,-1) -- (2,1);
	\draw (3,-1)  node [right] {\small $i$} .. controls (2.75,-1) ..
		(2.75,-.75) -- (2.75,.75) node[midway, tikzdot]{} node[midway,xshift=1.5ex, yshift=.5ex]{\small $\ell$}
		.. controls (2.75,1) .. (3,1);
	\filldraw [fill=white, draw=black] (-.25,-.25) rectangle (2.25,.25) node[midway] {\small $\nu$};
	\node at (1,.65) {\tiny $\dots$}; \node at (1,-.65) {\tiny $\dots$};
}
\]
and the map $\E^N_i\F^N_i \id_\nu \rightarrow \oplus_{[\rho_i]_{q_i}} \id_\nu$  is induced by the projection $\pi$.
Thus the sequence splits with the splitting map 
$
\oplus_{[\rho_i]_{q_i}} \id_\nu \rightarrow \E^N_i\F^N_i \id_\nu,
$
given by the sum of maps $R^N_\p(\nu)\xi^\ell \rightarrow R_\p^N(\nu+i)$ that add a vertical strand labeled $i$ carrying $\ell$ dots at the right of a diagram in $R^N_\p(\nu)$.  
In the second case, we also identify
\[
\oplus_{[-\rho_i]_{q_i}} \id_\nu 
\cong
q_i^{1+\rho_i} \bigoplus_{\ell =0}^{-\rho_i-1}
\tikzdiagh[xscale=.5,yscale=.8]{0}{
	\draw (0,-1)  node[below,white]{$i$} -- (0,1);
	\draw (.5,-1) -- (.5,1);
	\draw (1.5,-1) -- (1.5,1);
	\draw (2,-1) -- (2,1);
	\draw (3,-1)  node [right] {\small $i$} .. controls (2.75,-1) ..
		(2.75,-.75) -- (2.75,.75) node[midway, tikzdot]{} node[midway,xshift=1.5ex, yshift=.5ex]{\small $\ell$}
		.. controls (2.75,1) .. (3,1);
	\filldraw [fill=white, draw=black] (-.25,-.25) rectangle (2.25,.25) node[midway] {\small $\nu$};
	\node at (1,.65) {\tiny $\dots$}; \node at (1,-.65) {\tiny $\dots$};
}
\]
Moreover the map $\oplus_{[-\rho_i]} \id_\nu \rightarrow \F^N_i\E^N_i \id_\nu$ is induced by the connecting homorphism~$\delta$. Using the notations of~\cref{eq:thetakdecomp} it takes the form
\[
\delta\left(
\tikzdiagh[xscale=.5,yscale=.8]{0}{
	\draw (0,-1) node[below,white]{$i$}  -- (0,1);
	\draw (.5,-1) -- (.5,1);
	\draw (1.5,-1) -- (1.5,1);
	\draw (2,-1) -- (2,1);
	\draw (3,-1)  node [right] {\small $i$} .. controls (2.75,-1) ..
		(2.75,-.75) -- (2.75,.75) node[midway, tikzdot]{} node[midway,xshift=1.5ex, yshift=.5ex]{\small $k$}
		.. controls (2.75,1) .. (3,1);
	\node at (1,0) {\tiny $\dots$}; 
}
\right)
=
u_{ij}^{-1}\left(
 \tikzdiagh[yscale = .8]{0}{
	\draw  (.25,1) .. controls (.25,.5) and (.75,.5) ..   (.75,0)   .. controls (.75,-.5) and (.25,-.5)  .. (.25,-1); 
	\draw  (.5,1) .. controls (.5,.5) and (1,.5) ..   (1,0)   .. controls (1,-.5) and (.5,-.5)  .. (.5,-1); 
	\draw  (1,1) .. controls (1,.5) and (1.5,.5) ..   (1.5,0)   .. controls (1.5,-.5) and (1,-.5)  .. (1,-1); 
	\draw  (1.25,1) .. controls (1.25,.5) and (1.75,.5) ..   (1.75,0)   .. controls (1.75,-.5) and (1.25,-.5)  .. (1.25,-1); 
 \draw  (1.75,-1)  node[below]{$i$}.. controls (1.75,-.5) and (.25,-.5) ..    (.25,0) node[pos=1, tikzdot]{} node[pos=1,xshift=-1.5ex,yshift=.5ex]{\small $k$}
	 .. controls (.25,.5)  and (1.75,.5).. (1.75,1);
	 \node at (1.25,0) {\tiny $\dots$};  \node at (.8,.85) {\tiny $\dots$};  \node at (.8,-.85) {\tiny $\dots$};
	 \fdot{}{.5,0};
 }
 \right)  = \ 
 \tikzdiagh[xscale=.5,yscale=.8]{0}{
	\draw (0,-1.2)  node[below,white]{$i$} -- (0,1.2);
	\draw (.5,-1.2) -- (.5,1.2);
	\draw (1.5,-1.2) -- (1.5,1.2);
	\draw (2,-1.2) 
		-- (2,-.45) 
		.. controls (2,-.2)  .. (2.25,-.2);
	\draw (2,1.2) 
		-- (2,.45) 
		.. controls (2,.2) .. (2.25,.2);
	\filldraw [fill=white, draw=black,rounded corners] (-.25,-.95) rectangle (2.25,-.45) node[midway] {\small $\varphi_k$};
	\filldraw [fill=white, draw=black,rounded corners] (-.25,.45) rectangle (2.25,.95) node[midway] {\small $\psi_k$};
	\node at(1,0) {\tiny $\dots$};
}
\]
where $u_{ij}$ is the monomorphism defined in~\cref{eq:uij}, and $0 \leq k < -\rho_i$.
We also note that~\cref{eq:dotslidepsiphik} tells us that
\begin{equation}\label{eq:psiphik+t}
 \tikzdiagh[xscale=.5,yscale=.8]{0}{
	\draw (0,-1.2)  node[below,white]{$i$} -- (0,1.2);
	\draw (.5,-1.2) -- (.5,1.2);
	\draw (1.5,-1.2) -- (1.5,1.2);
	\draw (2,-1.2) 
		-- (2,-.45) 
		.. controls (2,-.2)  .. (2.25,-.2);
	\draw (2,1.2) 
		-- (2,.45) 
		.. controls (2,.2) .. (2.25,.2);
	\filldraw [fill=white, draw=black,rounded corners] (-.25,-.95) rectangle (2.25,-.45) node[midway] {\small $\varphi_{k+t}$};
	\filldraw [fill=white, draw=black,rounded corners] (-.25,.45) rectangle (2.25,.95) node[midway] {\small $\psi_{k+t}$};
	\node at(1,0) {\tiny $\dots$};
}
\ = \ 
\tikzdiagh[xscale=.5,yscale=.8]{0}{
	\draw (0,-1.2)  node[below,white]{$i$} -- (0,1.2);
	\draw (.5,-1.2) -- (.5,1.2);
	\draw (1.5,-1.2) -- (1.5,1.2);
	\draw (2,-1.2) 
		-- (2,-.45) 
		.. controls (2,-.2)  .. (2.25,-.2) node[pos=.35,tikzdot]{} node[pos=.35,xshift=-.75ex,yshift=1ex]{\small $t$};
	\draw (2,1.2) 
		-- (2,.45) 
		.. controls (2,.2) .. (2.25,.2);
	\filldraw [fill=white, draw=black,rounded corners] (-.25,-.95) rectangle (2.25,-.45) node[midway] {\small $\varphi_k$};
	\filldraw [fill=white, draw=black,rounded corners] (-.25,.45) rectangle (2.25,.95) node[midway] {\small $\psi_k$};
	\node at(1,0) {\tiny $\dots$};
}
\end{equation}
Moreover since $\bar y_N(\xi_i^{-\rho_i}) = \zeta$ and $\bar y_N(\xi_i^{\ell}) = 0$ for $\ell <-\rho_i$, we obtain by~\cref{eq:dotslidepsiphik} again that
\begin{equation}\label{eq:Epsiphik}
 \tikzdiagh[xscale=.5,yscale=.8]{0}{
	\draw (0,-1)  node[below,white]{$i$} -- (0,1);
	\draw (.5,-1) -- (.5,1);
	\draw (1.5,-1) -- (1.5,1);
	\draw (2,-1) 
		-- (2,1); 
	\filldraw [fill=white, draw=black,rounded corners] (-.25,-.75) rectangle (2.25,-.25) node[midway] {\small $\varphi_{k}$};
	\filldraw [fill=white, draw=black,rounded corners] (-.25,.25) rectangle (2.25,.75) node[midway] {\small $\psi_{k}$};
	\node at(1,0) {\tiny $\dots$};
	\draw (3.25, -1) node[right] {\small $i$}
	 .. controls (3,-1) .. (3,-.75)
	 -- (3,.75)
	 .. controls (3,1) .. (3.25,1);
}
= 
\begin{cases}
-r_i^{-1}\zeta, &\text{ if $k = -\rho_i - 1$,} \\
0, &\text{ if $k < -\rho_i - 1$.}
\end{cases}
\end{equation}
As in~\cite[Proof of Theorem 5.2]{kashiwara}, we construct a map $\Phi \colon \F^N_i\E^N_i \id_\nu \rightarrow  \oplus_{[-\rho_i]_{q_i}} \id_\nu$ induced by the morphism of bimodules
\[
\Phi:\ 
 \tikzdiagh[xscale=.5,yscale=.8]{0}{
	\draw (0,-1.2)  node[below,white]{$i$} -- (0,1.2);
	\draw (.5,-1.2) -- (.5,1.2);
	\draw (1.5,-1.2) -- (1.5,1.2);
	\draw (2,-1.2) 
		-- (2,-.45) 
		.. controls (2,-.2)  .. (2.25,-.2);
	\draw (2,1.2) 
		-- (2,.45) 
		.. controls (2,.2) .. (2.25,.2);
	\filldraw [fill=white, draw=black,rounded corners] (-.25,-.95) rectangle (2.25,-.45) node[midway] {$x$};
	\filldraw [fill=white, draw=black,rounded corners] (-.25,.45) rectangle (2.25,.95) node[midway] {$y$};
	\node at(1,-1.1) {\tiny $\dots$};
	\node at(1,1.1) {\tiny $\dots$};
	\node at(1,0) {\tiny $\dots$};
}
\quad \mapsto \ 
\sssum{r+s=\\ -\rho_i-1} \ 
\tikzdiagh[xscale=.5,yscale=.8]{0}{
	\draw (0,-1)  node[below,white]{$i$} -- (0,1);
	\draw (.5,-1) -- (.5,1);
	\draw (1.5,-1) -- (1.5,1);
	\draw (2,-1)  -- (2,1)  node[midway,tikzdot]{} node[midway, xshift=1.5ex, yshift=.5ex]{\small $r$};
	\draw (3.25, -1) node[right] {\small $i$}
	 .. controls (3,-1) .. (3,-.75)
	 -- (3,.75) node[midway,tikzdot]{} node[midway, xshift=1.5ex, yshift=.5ex]{\small $s$}
	 .. controls (3,1) .. (3.25,1);
	\filldraw [fill=white, draw=black,rounded corners] (-.25,-.8) rectangle (2.25,-.3) node[midway] {$x$};
	\filldraw [fill=white, draw=black,rounded corners] (-.25,.3) rectangle (2.25,.8) node[midway] {$y$};
	\node at(1,.9) {\tiny$\dots$};
	\node at(1,0) {\tiny $\dots$};
	\node at(1,-.9) {\tiny$\dots$};
}
\]
for all $x,y \in R_\p^N(\nu)$.

Then we compute
\begin{align*}
\Phi \circ \delta&\left(
\tikzdiagh[xscale=.5,yscale=.8]{0}{
	\draw (0,-1) node[below,white]{$i$}  -- (0,1);
	\draw (.5,-1) -- (.5,1);
	\draw (1.5,-1) -- (1.5,1);
	\draw (2,-1) -- (2,1);
	\draw (3,-1)  node [right] {\small $i$} .. controls (2.75,-1) ..
		(2.75,-.75) -- (2.75,.75) node[midway, tikzdot]{} node[midway,xshift=1.5ex, yshift=.5ex]{\small $k$}
		.. controls (2.75,1) .. (3,1);
	%\filldraw [fill=white, draw=black] (-.25,-.25) rectangle (2.25,.25) node[midway] {$\nu$};
	\node at (1,0) {\tiny $\dots$}; %\node at (1,-.65) {\tiny $\dots$};
}
\right)
\ \overset{\eqref{eq:psiphik+t}}{=}  \ 
\sssum{r+s=\\ -\rho_i-1} \ 
\tikzdiagh[xscale=.5,yscale=.8]{0}{
	\draw (0,-1)  node[below,white]{$i$} -- (0,1);
	\draw (.5,-1) -- (.5,1);
	\draw (1.5,-1) -- (1.5,1);
	\draw (2,-1)  -- (2,1);  %node[midway,tikzdot]{} node[midway, xshift=1.5ex, yshift=.5ex]{\small $r$};
	\draw (3.25, -1) node[right] {\small $i$}
	 .. controls (3,-1) .. (3,-.75)
	 -- (3,.75) node[midway,tikzdot]{} node[midway, xshift=1.5ex, yshift=.5ex]{\small $s$}
	 .. controls (3,1) .. (3.25,1);
	\filldraw [fill=white, draw=black,rounded corners] (-.25,-.8) rectangle (2.25,-.3) node[midway] {\small $\varphi_{k+r}$};
	\filldraw [fill=white, draw=black,rounded corners] (-.25,.3) rectangle (2.25,.8) node[midway] {\small $\psi_{k+r}$};
	\node at(1,.9) {\tiny$\dots$};
	\node at(1,0) {\tiny $\dots$};
	\node at(1,-.9) {\tiny$\dots$};
}
\\
\ &\overset{\eqref{eq:Epsiphik}}{=}
-r_i^{-1}\zeta\ 
\tikzdiagh[xscale=.5,yscale=.8]{0}{
	\draw (0,-1) node[below,white]{$i$}  -- (0,1);
	\draw (.5,-1) -- (.5,1);
	\draw (1.5,-1) -- (1.5,1);
	\draw (2,-1) -- (2,1);
	\draw (3,-1)  node [right] {\small $i$} .. controls (2.75,-1) ..
		(2.75,-.75) -- (2.75,.75) node[midway, tikzdot]{} node[midway,xshift=1.5ex, yshift=.5ex]{\small $k$}
		.. controls (2.75,1) .. (3,1);
	%\filldraw [fill=white, draw=black] (-.25,-.25) rectangle (2.25,.25) node[midway] {$\nu$};
	\node at (1,0) {\tiny $\dots$}; %\node at (1,-.65) {\tiny $\dots$};
}
\ +\ 
\sum_{r=-\rho_i-k}^{-\rho_i-1} \ 
\tikzdiagh[xscale=.5,yscale=.8]{0}{
	\draw (0,-1)  node[below,white]{$i$} -- (0,1);
	\draw (.5,-1) -- (.5,1);
	\draw (1.5,-1) -- (1.5,1);
	\draw (2,-1)  -- (2,1);  %node[midway,tikzdot]{} node[midway, xshift=1.5ex, yshift=.5ex]{\small $r$};
	\draw (3.25, -1) node[right] {\small $i$}
	 .. controls (3,-1) .. (3,-.75)
	 -- (3,.75) node[midway,tikzdot]{} node[midway, xshift=6ex, yshift=.5ex]{\small $-\rho_i-1-r$}
	 .. controls (3,1) .. (3.25,1);
	\filldraw [fill=white, draw=black,rounded corners] (-.25,-.8) rectangle (2.25,-.3) node[midway] {\small $\varphi_{k+r}$};
	\filldraw [fill=white, draw=black,rounded corners] (-.25,.3) rectangle (2.25,.8) node[midway] {\small $\psi_{k+r}$};
	\node at(1,.9) {\tiny$\dots$};
	\node at(1,0) {\tiny $\dots$};
	\node at(1,-.9) {\tiny$\dots$};
}
\end{align*}
Therefore, $\Phi \circ \delta$ is given by a triangular matrix with invertible elements on the diagonal, and thus is an isomorphism. In particular, $\delta$ is left invertible, concluding the proof.
\end{proof}

\begin{cor}
For $i \in I_f$, then $1_\nu\E_i$ and $\F_i1_\nu$ are biadjoint (up to shift).
\end{cor}

\begin{proof}
By the results in~\cite{brundan}, we know the splitting map $\E_i^N\F_i^N \id_\nu \rightarrow \F_i^N \E_i^N \id_\nu$ of \cref{thm:sl2commutRpN} together with the unit and counit of the adjunction $\F_i \dashv \E_i$ allow to construct a unit and counit for the adjunction $\E_i \dashv \F_i$.
\end{proof}

\begin{prop}\label{prop:serreinRpN}
For each $i,j \in I$ there is a natural isomorphism
\begin{align*}
\bigoplus^{\lfloor (d_{ij}+1)/2 \rfloor}_{a=0} &\begin{bmatrix} d_{ij}+1 \\ 2a \end{bmatrix}_{q_i}  (\F_i^N)^{2a}\F_j^N (\F_i^N)^{d_{ij}+1-2a}  \\
&\cong 
 \bigoplus^{\lfloor d_{ij}/2 \rfloor}_{a=0}  \begin{bmatrix} d_{ij}+1 \\ 2a +1 \end{bmatrix}_{q_i} (\F_i^N)^{2a+1}\F_j^N (\F_i^N)^{d_{ij}-2a}.
\end{align*}
By adjunction, the same isomorphism exists for the $\E_i^N, \E_j^N$.
\end{prop}

\begin{proof}
This follows from \cref{prop:serreinRbo}.
\end{proof}

In particular, there is a strong 2-action of the 2-Kac--Moody algebra  of \cite{KL3,rouquier} associated to $\brak{E_i,F_i, K_i}_{i \in I_f}$ on $\oplus_{\nu \in X^+} R_\p^N(\nu)\amod$ through $\F^N_i,\E^N_i$.

\subsection{A differential on $R_\p^N$}\label{sec:dnjRpN}

We fix a subset $I_f \subset I_f' \subset I$ and consider the parabolic subalgebras $U_q(\p) \subset U_q(\p') \subset U_q(\g)$. For each $j \in I_f' \setminus I_f$ we choose a weight $n'_j \in \bN$. For $j \in I_f$ we take $n'_j := n_j \in N$, and we write $N' := \{n'_j\}_{j \in I'_f}$. Then, we equip the cyclotomic $\p$-KLR algebra $R_\p^N(m)$ with a differential $d^N_{N'}$ which is zero on dots and crossings and
\[
d_{N'}^N
 \left( 
\ 
\tikzdiagh[xscale=.75]{0}{
	\fdot{j}{0.5,0.5};
	\draw (0,0) node[below] {\plusspacing \small $j$} -- (0,1);
	\draw (1,0) node[below] {\plusspacing \small $i_1$} -- (1,1);
	\node at(2,.5) {$\dots$};
	\draw (3,0) node[below] {\plusspacing \small $i_{m-1}$} -- (3,1);
}
\ 
\right)
:= \begin{cases}
0, & \text{ if $j \notin I_f'$,}\\
(-1)^{n_j}\tikzdiagh{0}{
	     	 \draw  (0,-.5) node[below] {\small $j$}  --(0,.5)node [midway,tikzdot]{} node[midway,xshift=1.75ex,yshift=.75ex]{\small $n_j$}; 
	      	\draw  (1,-.5) node[below] {\small $i_1$} -- (1,.5);
	      	\node at(2,0) {\small $\dots$};
	      	\draw  (3,-.5) node[below] {\small $i_{m-1}$} -- (3,.5);
	 }, & \text{ if $j \in I_f' \setminus I_f$.}\\
\end{cases}
\]
As before, we extend using the graded Leibniz rule, and verifying that $d_{N'}^N$ is well-defined is straightforward.

\begin{thm}\label{thm:RpdNformal}
The dg-algebra $(R_\p^N(m), d_{N'}^N)$ is formal with homology
\[
H(R_\p^N(m), d_{N'}^N) \cong R_{\p'}^{N'}(m).
\]
\end{thm}

\begin{proof}
We have $R_\p^N(m) \cong H(R_\bo(m), d_{N})$ and $R_{\p'}^{N'}(m) \cong H(R_\bo(m), d_{N'})$ by \cref{thm:RbodNformal}. Moreover, $ d_{N'}^N$ can be lifted to $R_\bo(m)$.
We split the homological grading of $R_\bo(m)$ in three: a first one that counts the amount of floating dots with subscript in $I_f$, a second one for the floating dots with subscript in $I_f'\setminus I_f$, and a third one for $I\setminus I_f'$ that we ignore for the moment.  
Then, we have that $d_{N'}^N$ has degree $(0,-1)$ and $d_N$ has degree $(-1,0)$, and they commute with each other. Thus we have a (bounded) double complex $(R_\bo, d_N, d_{N'}^N)$ with total complex being $(R_\bo, d_{N'})$, since $d_{N'} = d_N + d_{N'}^N$. 
In particular, there is a spectral sequence from $H(R_\p^N(m), d_{N'}^N)$ to $H(R_\bo, d_{N'}) \cong R_{\p'}^{N'}(m)$.
Now, \cref{thm:RbodNformal} tells us that $H(R_\bo, d_N)$ is concentrated in homological degree zero (for the first homological grading). Thus, the spectral sequence converges at the second page, and in particular $H(R_\p^N(m), d_{N'}^N) \cong R_{\p'}^{N'}(m)$.
\end{proof}

We interpret this result as a categorical version of the fact that if there is an arrow from a parabolic Verma module $M^\p(\Lambda,N)$ to $M^{\p'}(\Lambda', N')$ (see \cref{sec:qgweightmodules}), then there is a surjection $M^\p(\Lambda,N) \twoheadrightarrow M^{\p'}(\Lambda', N')$. Indeed, in that case there is a surjective quasi-isomorphism $(R_\p^N, d_{N'}) \xrightarrow{\simeq} (R_{\p'}^{N'}, 0)$, inducing equivalences of derived categories that commute up to quasi-isomorphism with the categorical actions of $U_q(\g)$. 

%%%%%%%%%%%%%%%%	End of file	%%%%%%%%%%%%%

%% file: sections/catthm.tex
%%%%%%%%%%%%%%%%%%%%%%%%%%%%%%%%%%%%
%                 					  				  		 %
%	Categorification theorems			 					 %
%                 					  						 %
%%%%%%%%%%%%%%%%%%%%%%%%%%%%%%%%%%%%

\section{The categorification theorems}\label{sec:catthm}

Recall that the $\Bbbk$-algebra of formal Laurent series $\Bbbk\pp{x_1,\dots, x_n}$ (as constructed in~\cite{laurent}, see also \cite[\S5]{asympK0}) is given by first choosing a total additive order $\prec$ on $\bZ^n$. 
One says that a cone $C := \{\alpha_1 v_1 + \cdots + \alpha_n v_n | \alpha_i \in \bR_{\geq 0} \} \subset \bR^n$ is compatible with $\prec$ whenever $0 \prec v_i$ for all $i \in \{1,\dots,n\}$. 
 Then, we set 
\[
\Bbbk\pp{x_1,\dots,x_n} := \bigcup_{\be \in \bZ^n} x^{\be} \Bbbk_{\prec}\llbracket x_1,\dots, x_n \rrbracket,
\]
where $\Bbbk_{\prec}\llbracket x_1,\dots, x_n \rrbracket$ consists of formal Laurent series in $\Bbbk \llbracket x_1,\dots, x_n\rrbracket$ such that the terms are contained in a cone compatible with $\prec$. We will also write $\Bbbk_{\prec}^{+}\llbracket x_1,\dots, x_n \rrbracket$ for the elements in $\Bbbk_{\prec}\llbracket x_1,\dots, x_n \rrbracket$ with terms contained in a cone without the $0$ element (i.e. series with the degree zero term being zero). 
We obtain a ring by equipping $\Bbbk\pp{x_1,\dots,x_n}$ with the usual addition and multiplication of series. 
Requiring that all series are contained in cones compatible with $\prec$ ensures that the product of two elements in  $\Bbbk\pp{x_1,\dots,x_n}$  is well-defined. Indeed, under these conditions, any coefficient in the product can be determined by summing only a finite amount of terms. 

\subsection{C.b.l.f. derived category}\label{sec:cblfderived}

We fix an arbitrary additive total order $\prec$ on $\bZ^n$. 
We say that a $\bZ^n$-graded $\Bbbk$-vector space $M = \bigoplus_{ \bigoplus_{\bg \in \bZ^n}} M_\bg$ is \emph{c.b.l.f. (cone bounded, locally finite) dimensional}  if
\begin{itemize}
\item $\dim M_\bg < \infty$ for all $\bg \in \bZ^n$;
\item there exists a cone $C_M \subset \bR^n$ compatible with $\prec$ and $\be \in \bZ^n$ such that $M_\bg = 0$ whenever $\bg - \be \notin C_M$. 
\end{itemize}
In other words, $M$ is c.b.l.f. dimensional if and only if 
\[
\gdim_q M := \sum_{\bg \in \bZ^n} x^\bg \dim(M_\bg) \in x^\be \Bbbk_{\prec}\llbracket x_1,\dots, x_n \rrbracket.
\]

\smallskip

Let $(A,d)$ be a $\bZ^n$-graded dg-$\Bbbk$-algebra, where $A = \bigoplus_{(h,\bg) \in \bZ \times \bZ^n} A_\bg^h$, and $d(A_\bg^h) \subset A_\bg^{h-1}$. Suppose that $(A,d)$ is concentrated in non-negative homological degrees, that is $A_\bg^h = 0$ whenever $h < 0$. 
Let $\cD(A,d)$ be the derived category of $(A,d)$. Let $\cD^{lf}(A,d)$ be the full triangulated subcategory of $\cD(A,d)$ consisting of $(A,d)$-modules having homology being c.b.l.f. dimensional for the $\bZ^n$-grading. We call $\cD^{lf}(A,d)$ the \emph{c.b.l.f. derived category of $(A,d)$}. 
%
%As explained in~\cite[\S9]{asympK0} (see~\cite[Lemma 5.2]{kelleryang} for a proof \textcolor{red}{Add Keller-Yang ref}), $\cD(A,d)$ can be equipped with a canonical t-structure where the truncation functors are given by the usual intelligent truncations. Then, its heart $\cD(A,d)^\heartsuit$ respects
%\[
%\cD(A,d)^\heartsuit \cong H^0(A,d)\amod.
%\]
%Moreover, this t-structure restricts on $\cD^{lf}(A,d)$ so that 
%\[
%\cD^{lf}(A,d)^\heartsuit \cong H^0(A,d)\modlf,
%\]
%where $H^0(A,d)\modlf$ is the category of c.b.l.f. dimensional $H^0(A,d)$-modules. 
%
%\smallskip

\begin{defn}[\cite{asympK0}]\label{def:positivedg}
We say that $(A,d)$ is a \emph{positive c.b.l.f. dimensional dg-algebra} if 
\begin{enumerate}
\item $A$ is c.b.l.f.dimensional for the $\bZ^n$-grading;
\item $A$ is non-negative for the homological grading;
%\item $A_0^0$ is semi-simple;
\item $A_0^h = 0$ for $h >0$;  \label{eq:homologicalDegreeCdt}
\item$(A,d)$ decomposes into a direct sum of shifted copies of relatively projective modules $P_i := A e_i$ for some idempotent $e_i \in A$, such that $P_i$ is non-negative for the $\bZ^n$-grading and $A_0^0 P_i$ is semisimple. 
\end{enumerate}
\end{defn}

\begin{rem}\label{rem:removeAcyclic}
As explained in \cite[\remacyclicproj]{asympK0}, condition (\ref{eq:homologicalDegreeCdt}.) cannot be respected whenever $P_i := Ae_i$ is acyclic. However, in this case there is a quasi-isomorphism $(A,d) \xrightarrow{\simeq} (A/Ae_iA,d)$ and we can weaken hypothesis (\ref{eq:homologicalDegreeCdt}.) so that it is respected only after removing all acyclic $P_i$. This is the case of $(R_\bo,d_N)$.  
\end{rem}

\subsubsection{Asymptotic Grothendieck group}
 
 As already observed in~\cite{acharstroppel} (see also~\cite[Appendix]{naissevaz1}), one caveat of the usual definition of the Grothendieck group is that it does not allow to take into consideration infinite iterated extensions of objects. 
 %For example, the Grothendieck group of the category of c.b.l.f. dimensional $\bZ^n$-graded $\Bbbk$-vector spaces is not isomorphic to $\bZ\pp{x_1, \dots, x_n}$. 
We need to introduce new relations in the Grothendieck groups to handle such situations. One solution is to use \emph{asymptotic Grothendieck groups}, as introduced by the first author in~\cite{asympK0}. 
 
 \smallskip
 
 Let $\cC$ be a triangulated subcategory of some triangulated category $\cT$. Suppose $\cT$ admits countable products and coproducts, and these preserve distinguished triangles. Let $K_0^\Delta(\cC)$ be the triangulated Grothendieck group of $\cC$.
 
 \smallskip
 
Recall the \emph{Milnor colimit $\mcolim_{r  \geq 0} (f_r) $} of a collection of arrows $\{X_r \xrightarrow{f_r} X_{r+1}\}_{r \in \bN}$ in $\cT$ is the mapping cone fitting inside the following distinguished triangle
\[
\coprod_{r \in \bN} X_r \xrightarrow{1-f_\bullet} \coprod_{r \in \bN} X_r \rightarrow \mcolim_{r  \geq 0} (f_r) \rightarrow %\coprod_{r\in \bN} X_r[1],
\]
where the left arrow is given by the infinite matrix
\[
1-f_\bullet := 
\begin{pmatrix}
1      & 0       &  0 & 0 & \cdots \\
-f_0 & 1       & 0 & 0 & \cdots  \\
0      & -f_1  & 1 & 0  & \cdots \\
\vdots & \ddots & \ddots & \ddots & \ddots
\end{pmatrix}
\] 

There is a dual notion of Milnor limit. Consider a collection of arrows $\{X_{r+1} \xrightarrow{f_r} X_r\}_{r \geq 0}$ in $\cT$. The \emph{Milnor limit} is the object fitting inside the distinguished triangle
\[
\mlim_{r \geq 0} (f_r) \rightarrow \prod_{r \geq 0} X_r \xrightarrow{1 - f_\bullet} \prod_{r \geq 0}  X_r \rightarrow
\]

\begin{defn}\label{def:toptriangulatedK0}
The \emph{asymptotic triangulated Grothendieck group} of $\cC \subset \cT$ is given by
\[
\bKO^\Delta(\cC) := K_0^\Delta(\cC) / T(\cC),
\]
where $T(\cC)$ is generated by %elements 
\[
[Y] -  [X] = \sum_{r\geq 0} [E_r]
\]
whenever both $\bigoplus_{r \geq 0} \cone(f_r) \in \cC$ and $\bigoplus_{r \geq 0} E_r \in \cC$, and
\begin{align*}
Y \cong \mcolim\bigl(X = F_0 \xrightarrow{f_0} F_1 \xrightarrow{f_1} \cdots \bigr),
\intertext{is a Milnor colimit, or}
X \cong \mlim \bigl( \cdots \xrightarrow{f_1} F_1 \xrightarrow{f_0} F_0 = Y \bigr),
\end{align*}
is a Milnor limit, and 
\[ 
	[E_r] = [\cone(f_r)] \in K_0^\Delta(\cC),
\]
for all $r \geq 0$.
\end{defn}

In a $\bZ^n$-graded triangulated category $\cT$, we define the notion of \emph{c.b.l.f. direct sum} as follows: 
\begin{itemize}
\item take a a finite collection of objects $\{K_1,\dots, K_m\}$ in $\cT$; 
\item consider a direct sum of the form
\begin{align*}
&\bigoplus_{\bg \in \bZ^n} x^{\bg} (K_ {1,\bg} \oplus \cdots \oplus K_{m,\bg}), &
&\text{ with }&
K_{i,\bg} &= \bigoplus_{j = 1}^{k_{i,\bg}} K_i[h_{i,j,\bg}],
\end{align*}
where $k_{i,\bg} \in \bN$ and $h_{i,j,\bg} \in \bZ$ such that:
\item there exists a cone $C$ compatible with $\prec$, and $\be \in \bZ^n$ such that for all $j$ we have $k_{j,\bg} = 0$ whenever $\bg -  \be \notin C$;
\item there exists $h \in \bZ$ such that $h_{i,j,\bg} \geq h$ for all $i,j,\bg$. 
\end{itemize}

If $\cT$ admits arbitrary c.b.l.f. direct sums, then $K_0(\cT)$ has a natural structure of $\bZ\pp{x_1,\dots, x_n}$-module with
\[
\sum_{\bg \in C} a_\bg x^{\be + \bg} [X] := [\bigoplus_{\bg \in C} x^{\bg + \be} X^{\oplus a_\bg}],
\]
where $X^{\oplus a_\bg} = \bigoplus_{\ell = 1}^{|a_\bg|} X[\alpha_\bg]$ and $\alpha_\bg = 0$ if $a_\bg \geq 0$ and $\alpha_\bg =1$ if $a_\bg < 0$. 

%We  suppose that $\cC$ is c.b.l. additive (for c.b.l.f. coproducts defined as above). 

\begin{thm}[{\cite[\thmasympKO]{asympK0}}]\label{thm:triangtopK0genbyPi}
Let $(A,d)$ be a positive c.b.l.f. dg-algebra, and let $\{P_j\}_{j \in J}$ be a complete set of indecomposable cofibrant $(A,d)$-modules that are pairwise non-isomorphic (even up to degree shift). Let $\{S_j\}_{j \in J}$ be the set of corresponding simple modules. 
There is an isomorphism
\begin{align*}
\bKO^\Delta\bigl(\cD^{lf}(A,d)\bigr)  & \cong  \bigoplus_{j \in J} \bZ\pp{x_1, \dots, x_\ell} [P_j], 
\end{align*}
and $\bKO^\Delta\bigl(\cD^{lf}(A,d)\bigr)$ is also freely generated by the classes of $\{[S_j]\}_{j \in J}$.
\end{thm}

\begin{prop}[{\cite[\propcblfbim]{asympK0}}]\label{prop:cblfbiminduceKO}
Let $(A,d)$ and $(A',d')$ be two c.b.l.f. positive dg-algebras. Let $B$ be a c.b.l.f. dimensional $(A',d')$-$(A,d)$-bimodule. The derived tensor product functor
\[
F : \cD^{lf}(A,d) \rightarrow \cD^{lf}(A',d'), \quad F(X) := B \Lotimes_{(A,d)} X,
\]
induces a continuous map
\[
[F] : \bKO^\Delta(\cD^{lf}(A,d))  \rightarrow \bKO^\Delta(\cD^{lf}(A',d')).
\]
\end{prop}

We will need the following definitions in \cref{sec:2Verma}:

\begin{defn}
Let $\{K_1, \dots, K_m\}$ be a finite collection of objects in $\cC$, and let $\{E_r\}_{r \in \bN}$ be a family of direct sums of $\{K_1, \dots, K_m\}$ such that $\bigoplus_{r \in \bN} E_r$ is a c.b.l.f. direct sum of $\{K_1, \dots, K_m\}$. Let $\{M_r\}_{r \in \bN}$  be a collection of objects in $\cC$ with $M_0 = 0$, such that they fit in distinguished triangles
\[
M_r \xrightarrow{f_r} M_{r+1} \rightarrow E_r \rightarrow %M_r[1]. 
\]
Then, we say that an object $M \in \cC$ such that $M \cong_{\cT} \mcolim_{r\geq 0} (f_r)$ in $\cT$ is a \emph{c.b.l.f. iterated extension of  $\{K_1, \dots, K_m\}$}. 
%The sequence $t_{i,\bg} = \{h_{i,1,\bg}, \dots, h_{i,k_{i,\bg},\bg}\}$ of occurences of $x^\bg K_i[h_{i,j,\bg}]$ as a direct summand in $\bigoplus_{r \in \bN} E_r$ is called the \emph{degree $\bg$ multiplicity} of $K_i$.
\end{defn}

\begin{defn}
We say that $\cV$ is \emph{c.b.l.f. generated by $\{X_j\}_{j \in J}$} for some collection of elements $X_j \in \cV$ if for any object $Y$ in $\cV$ we can take a finite set $\{Y_k\}_{k \in K}$ of retracts $Y_k \subset X_{j_k}$ such that $Y$ is isomorphic to a c.b.l.f. iterated extension of $\{Y_k\}_{k \in K}$. 
\end{defn}

\subsection{Categorification}

In this section we assume that $R_\bo(\nu)$ is a $\Bbbk$-algebra over a field $\Bbbk$. We also choose an abritrary order $\prec$ for constructing $\bZ\pp{q,\Lambda}$ such that $0 \prec q \prec \lambda_i$ for all formal $\lambda_i = q^{\beta_i} \in \Lambda$. We assume that the parabolic Verma module $M^\p(\lambda,N)$ is constructed over the ground ring $R := \bQ\pp{\Lambda,q}$ (instead of $\bQ(\Lambda,q)$). 

\smallskip

Every idempotent of $R_\bo(\nu)$ is the image of an idempotent of the classical KLR algebra $R_{\g}(\nu)$ under the obvious inclusion $R_\g(\nu) \hookrightarrow R_\bo(\nu)$.  
Thanks to~\cite[Section~2.5]{KL1} we know all the idempotents of $R_\g(\nu)$. 
We define the element
\begin{equation*}
e_{i,n} := \tau_{\vartheta_n} x_1^{n-1} x_2^{n-2} \cdots x_{n-1} 1_{ii\cdots i}  \in R_\g(n),
\end{equation*}
where $\vartheta_n$ is the longest element in $S_n$. 
Let $\Seqd(\nu)$  be the set of expressions $i_1^{(m_1)}i_2^{(m_2)}\cdots i_r^{(m_r)}$ for different $r \in \bN$ and $m_\ell \in \bN$ such that $\sum_{\ell = 1}^r m_\ell \cdot \alpha_{i_\ell} = \nu$.
For each $\bi \in \Seqd(\nu)$ we define the idempotent
\[
e_{\bi} := e_{i_1,m_1}  \otimes e_{i_2,m_2} \otimes \cdots \otimes e_{i_r,m_r} \in R_\g(\nu),
\]
where $x \otimes y$ means we put the diagram of $x$ at the left of the one of $y$. 
Identifying $e_{\bi}$ with its image in $R_\bo(\nu)$, as in~\cite{KL1}, we define a projective left $R_\bo(\nu)$-module
\[
P_\bi := R_\bo(\nu) e_{\bi},
\]
Then, we put 
\[
\brak{\bi} := - \sum_{\ell = 1}^r \frac{m_\ell(m_\ell-1)}{2}  d_{i_a}.
\]
and we define $\tilde P_\bi := q^{-\brak{\bi}} P_\bi$.

\smallskip

When writing $\dots \bi \dots$ and $\dots \bj \dots$ we mean we take two sequences $\bi_1 \bi \bi_2$ and $\bi_1 \bj \bi_2$ in $\Seqd(\nu)$ that coincide everywhere except on $\bi$ and $\bj$. From the decomposition of the nilHecke algebra~\cite[\S2.2]{KL1} we get an isomorphism of $R_\p^N$-modules
\[
\tilde P_{\dots i^{m} \dots} \cong \oplus_{([m]_{q_i}!)} \tilde P_{\dots i^{(m)} \dots}.
\]
Mimicking the arguments in~\cite[Proposition~2.13]{KL1} and \cite[Proposition~6]{KL2} we have the following:

\begin{prop}\label{prop:catSerreDiv}
There are isomorphisms 
\begin{align*}
\bigoplus^{\lfloor (d_{ij}+1)/2 \rfloor}_{a=0} \tilde P_{\dots i^{(2a)}ji^{(d_{ij}+1-2a)}\dots}  \cong \bigoplus^{\lfloor d_{ij}/2 \rfloor}_{a=0} \tilde P_{\dots i^{(2a+1)}ji^{(d_{ij}-2a)} \dots}
\end{align*}
for all $i \neq j \in I$.
\end{prop}

Equipping $R_\bo(\nu)$ with $d_N$ induces a differential on $\tilde P_\bi$, and \cref{prop:catSerreDiv} holds for the dg-version $(\tilde P_\bi, d_N)$. 
We put
\begin{align*}
\cM^\p(\Lambda,N) := \bigoplus_{m \geq 0} \cD^{lf}(R_\bo(m), d_N),
\end{align*}
with the particular case $\cM(\Lambda)$ meaning $\p = \bo$ and $N = \emptyset$, and therefore $d_N = 0$.
Note that $\cD^{lf}(R_\bo(m), d_N) \cong \cD^{lf}(R_\p^N(m), 0)$. 
Let ${}_\bQ\bKO^\Delta(-) := \bKO^\Delta(-) \otimes_{\bZ\pp{q,\Lambda}} \bQ\pp{q,\Lambda}$.

\begin{prop}\label{prop:dgKLRispositive}
The $\bZ^{1+|\Lambda|}$-graded dg-algebra $(R_\bo(m), d_N)$ is a positive c.b.l.f. dimensional dg-algebra. 
%Moreover, we have
%\[
%\gdim \HOM_{R_\bo(m)}(P_\bi, P_\bi) -1 \in 
%\]
%whenever $P_\bi$ is indecomposable. 
\end{prop}

\begin{proof}
Clearly, $R_\bo(m)$ is c.b.l.f. dimensional for the $\bZ^{1+|\Lambda|}$-grading, and is non-negative for the homological grading. We can also assume we have applied \cref{rem:removeAcyclic}. 
Recall that the part in homological degree zero of $R_\bo(m)$ is isomorphic to the usual KLR algebra $R_\g(m)$. 
As explained in \cite[\S 3.3]{KL1}, for each monomial $f \in U^{-}_q(\g)$, we have a projective $R_\g(m)$-module $P_f$ (defined similarly as $P_\bi$ for $f = F_{i_1}^{(m_1)} \cdots F_{i_r}^{(m_r)}$). Moreover, by \cite[Proposition 3.22]{KL1} extended for any $\g$, $P_f$ is indecomposable if and only if $f$ is a canonical basis element. 
Also, the quadratic form in \cite[\S 14.2]{lusztig} corresponds with the graded dimension of the graded hom-spaces between these projective $R_\g(m)$-modules. The same applies for the homological degree zero part of the graded hom-spaces between our $P_\bi$'s. 
Then, by \cite[Theorem 14.2.3]{lusztig}, we obtain that 
\[
\gdim \HOM_{R_\bo(m)}(P_\bi, P_\bj) - \delta_{\bi,\bj} 
\in \bZ_{\prec}^+\llbracket q, \Lambda \rrbracket,
%\succ 0 \in \bZ_{\precc}\pp{q, \Lambda}[h],
\]
which concludes the proof. 
\end{proof}

\begin{prop}\label{prop:K0RbohalfUqg}
There is an isomorphism of $\bQ\pp{q,\Lambda}$-modules
\[
 U_q^-(\g) \otimes_{\bQ(q)} \bQ\pp{q,\Lambda} \cong {}_\bQ\bKO^\Delta(\cM(\Lambda)),
\]
and a $\bQ\pp{q,\Lambda}$-linear surjection
\[
 U_q^-(\g) \otimes_{\bQ(q)} \bQ\pp{q,\Lambda} \twoheadrightarrow {}_\bQ\bKO^\Delta(\cM^\p(\Lambda, N)),
\]
both sending $ F_{i_1}^{(m_1)} F_{i_2}^{(m_2)} \cdots F_{i_r}^{(m_r)}$ to $[(\tilde P_\bi,d_N)]$ for $\bi = i_1^{(m_1)} i_2^{(m_2)} \cdots i_r^{(m_r)}$.
\end{prop}

\begin{proof}
Since projective modules of $R_\bo(\nu)$ are in bijection with the ones of the classical KLR algebra $R_\g(\nu)$ and respect the categorified Serre relations (see \cref{prop:catSerreDiv}),
 both claims are a direct consequence of the main results in~\cite{KL1, KL2}, together with \cref{thm:triangtopK0genbyPi}.
\end{proof}

Consider the subring $P(\nu)$ of $R_{\bo}(\nu)$ consisting of dots on vertical strands (without floating dots). It admits an action of the symmetric group permuting the strands (with labels) and dots on them (not to be confused with the action of $S_m$ on $P_\nu$ from \cref{sec:bKLRbasis}).
We write $\Sym(\nu) := P(\nu)^{S_m}$ for the subring of invariants under this action.  
Clearly it lies in the center of $R_{\bo}(\nu)$ but this inclusion is strict (see~\cite{naissevaz2} or~\cite{AEHL} for a study of the center in the case of $\slt$).

\smallskip

The supercenter of $R_{\bo}(\nu)$ contains $\Sym(\nu) \otimes \bigotimes_{i \in I}  \bV^\bullet \brak{\tilde \omega_i^{0}, \dots, \tilde \omega_i^{\nu_i-1}}$ where $\tilde \omega_i^a$ is a floating dot with subscript $i$, superscript $a$ and placed in the rightmost region:
\[
\tilde \omega_i^a := \  
	\tikzdiag{
	     	 \draw  (0,-.5)   --(0,.5);
	      	\draw  (1,-.5)  -- (1,.5);
	      	\node at(2,0) {\small $\dots$};
	      	\draw  (3,-.5)  -- (3,.5);
	      	\fdot[a]{i}{3.5,0};
	 }
\]
 We conjecture that the supercenter contains no other elements. 

\begin{conj}
There is an isomorphism of rings
\[
Z(R_{\bo}(\nu)) \cong \Sym(\nu) \otimes \bigotimes_{i \in I}  \bV^\bullet \brak{\tilde \omega_i^{0}, \dots, \tilde \omega_i^{\nu_i-1}},
\]
where $Z(R_{\bo}(\nu))$ is the supercenter of $R_{\bo}(\nu)$. 
\end{conj}

In general $R_{\p,\mu}(\nu)$ is not a free module over $\Sym(\nu) \otimes \bigotimes_{i \in I}  \bV^\bullet \brak{\tilde \omega_i^{0}, \dots, \tilde \omega_i^{\nu_i-1}}$, but we have the following.

\begin{prop}
$R_{\bo}(\nu)$ is a free module over $\Sym(\nu)$ of rank $2^m (m!)^2$.
\end{prop}

\begin{proof}
It follows from \cref{thm:Rbobasis} and the fact $P(\nu)$ is a free module of rank $m!$ on $\Sym(\nu)$.
\end{proof}

Since $\Sym(\nu)$ lies in the center of $R_{\bo}(\nu)$, any simple $R_{\bo}(\nu)$-module is annihilated by $\Sym^+(\nu)$, where $\Sym^+(\nu)$ consists of the elements in $\Sym(\nu)$ with non-zero degree. In particular, a simple $R_{\bo}(\nu)$-module must be a finite dimensional $R_{\bo}(\nu)/\Sym^+(\nu)R_{\bo}(\nu)$-module. Since $R_{\bo}(\nu)/\Sym^+(\nu)R_{\bo}(\nu)$ has finite dimension over $\Bbbk$, we only have finitely many simple modules, up to shift and isomorphism. 
For each  $\bi \in \Seqd(\nu)$ such that $P_\bi$ is indecomposable, we let $S_\bi$ be the unique simple quotient of  $P_\bi$. We put $\tilde S_\bi := q^{-\brak{\bi}} S_\bi$. 
If $(P_\bi,d_N)$ is not acyclic, then it lifts automatically to a dg-version $(\tilde S_\bi, 0)$ because of \cref{prop:dgKLRispositive}. 
%We define for each $\bi \in \Seqd(\nu)$ the simple module
%\[
%S_\bi := P_\bi/ \Sym^+(\nu) P_\bi,
%\]
%which is the unique simple quotient of $P_\bi$. We put $\tilde S_\bi := q^{-\brak{\bi}} S_\bi$. 
%It lifts to a dg-version $(\tilde S_i, d_N)$. 

\smallskip

By \cref{lem:Rboleftdecomp} and \cref{prop:Rbosproj} we know that $\E_i \id_\nu$ and $\F_i\id_\nu$ are exact. Moreover, they respect the conditions of \cref{prop:cblfbiminduceKO}. Therefore, they induce maps
\begin{align*}
{[\E_i \id_\nu]} &: \bKO\bigl(\cD^{lf}(R_\bo(\nu), d_N)\bigr) \rightarrow \bKO^\Delta\bigl(\cD^{lf}(R_\bo(\nu - i), d_N)\bigr) , \\
[\F_i \id_\nu] &: \bKO\bigl(\cD^{lf}(R_\bo(\nu), d_N)\bigr) \rightarrow \bKO^\Delta\bigl(\cD^{lf}(R_\bo(\nu + i), d_N)\bigr).
\end{align*}
Then, \cref{thm:RbodNformal}, \cref{thm:sl2commutRpN} and \cref{prop:serreinRpN} tell us that ${}_\bQ\bKO^\Delta(\cM^\p(\Lambda,N))$ is an $U_q(\g)$-weight module. 
By \cref{prop:K0RbohalfUqg} we know that ${}_\bQ\bKO^\Delta(\cM^\p(\Lambda,N))$ is cyclic as $U_q(\g)$-module, with highest weight generator given by the class of $(R_\bo(0), d_N) \cong (\Bbbk,0)$. Thus ${}_\bQ\bKO^\Delta(\cM^\p(\Lambda,N))$ is a highest weight module.

\smallskip

As in \cite{KL1}, let $\psi : R_\bo(\nu) \rightarrow \opalg{R_\bo(\nu)}$ be the map that takes the mirror image of diagrams along the horizontal axis.
Given a left $(R_\bo(\nu),d_N)$-module $M$, we obtain a right $(R_\bo(\nu),d_N)$-module $M^\psi$ with action given by $m^\psi \cdot r := (-1)^{\deg_h(r) \deg_h(m)} \psi(r) \cdot m$ for $m \in M$ and $r \in R_\bo(\nu)$. 
Then, we define the bifunctor
\begin{align*}
(-,-) &: \cM^\p(\Lambda,N) \times \cM^\p(\Lambda,N)  \rightarrow \cD^{lf}(\Bbbk, 0), 
&
(M,M') := M^\psi \Lotimes_{(R_\bo,d_N)} M',
\end{align*}
where $\Lotimes$ is the derived tensor product. 

\begin{prop}\label{prop:catshap}
The bifunctor defined above respects:
\begin{itemize}
\item $((R_\bo(0),d_N),(R_\bo(0),d_N)) \cong (\Bbbk,0)$;
\item $(\Ind_m^{m+i} M,M') \cong (M, \Res_m^{m+i} M')$ for all $M,M' \in \cM^\p(\Lambda,N)$; 
\item $(\oplus_f M,M') \cong (M, \oplus_f M') \cong \oplus_f (M,M')$ for all $f \in \bZ\pp{q,\Lambda}$.
\end{itemize}
%\[
%(\Ind_m^{m+i} M,M') \cong (M, \Res_m^{m+i} M'),
%\]
\end{prop}

\begin{proof}
Straightforward. 
\end{proof}

Comparing \cref{prop:catshap} with \cref{def:shap}, we deduce 
%that
%Note that \cref{prop:catshap} implies 
that $(-,-)$ is a categorification of the Shapovalov form on $\bKO^\Delta(\cM^\p(\Lambda,N))$. Moreover, it turns $\tilde S_\bi$ into the dual of $\tilde P_\bi$ for each $\bi \in \Seqd(\nu)$ such that $\tilde P_\bi$ is indecomposable. Recall $ M^\p(\Lambda,N)$ is the parabolic Verma module, and we assume $\Lambda = \{q^{\beta_i} | i \in I_r\}$ contains only formal weights.

\begin{thm}\label{thm:catallverma}
The asymptotic Grothendieck group ${}_\bQ\bKO^\Delta(\cM^\p(\Lambda,N))$ is a $U_q(\g)$-weight module, with action of $E_i, F_i$ given by $[\E_i], [\F_i]$. Moreover, there is an isomorphism of $U_q(\g)$-modules
\[
{}_\bQ\bKO^\Delta(\cM^\p(\Lambda,N)) \cong M^\p(\Lambda,N).
\]
%where $\Lambda = \{q^{\beta_i} | i \in I_r\}$ contains only formal weights. 
\end{thm}

\begin{proof}
We already proved the first claim above.
Because of \cref{prop:acyclic}, for $i \in I_f$, both $[\F_i]$ and $[\E_i]$ act as locally nilpotent operators. In particular, the $U_q(\mathfrak l)$-submodule of ${}_\bQ\bKO^\Delta(\cM^\p(\Lambda,N))$ given by 
\[
U_q(\mathfrak l) \otimes_{U_q(\g)} [(R_\bo(0), d_N)],
\]
is an integrable module for the Levi factor $U_q(\mathfrak l)$. 
Since it is an integrable cyclic weight module, it must be isomorphic to $V(\Lambda,N)$ (see~\cite{lusztig}). 
Therefore, there is a surjective $U_q(\g)$-module morphism
\[
\gamma : M^\p(\Lambda, N) \twoheadrightarrow {}_\bQ\bKO^\Delta(\cM^\p(\Lambda,N)).
\]
Since $M^\p(\Lambda, N)$ is irreducible and $\gamma$ is non-zero, it must be an isomorphism.
\end{proof}

Let $m_r = F_\bi v_{\Lambda,N}$ be an induced basis element of $M^{\p}(\Lambda,N)$ with $\bi \in \Seq(\nu)$. Then, the isomorphism of Theorem~\ref{thm:catallverma} identifies $m_r$ with the class $[(R_\bo(\nu),d_N)1_\bi]$. 
Similarly, let $m_s' = F_\bj v_{\Lambda,N}$ for $j \in \Seqd(\nu)$ be a canonical basis element, and let $m^s$ be its dual in the dual canonical basis. 
Then, the isomorphism identifies $m_s'$ with $[(\tilde P_\bj,d_N)]$ and $m^s$ with $[(\tilde S_\bj,d_N)]$. 
Moreover, computing the c.b.l.f. composition series of $\tilde P_\bi$ (see~\cite[\S7]{asympK0}) or taking a certain cofibrant replacement of $\tilde S_\bi$ (see~\cite[\S9]{asympK0}) gives a categorical version of the change of basis between canonical and dual canonical basis elements.

%%%%%%%%%%%%%%%%	End of file	%%%%%%%%%%%%%

%% file: sections/2verma.tex
%%%%%%%%%%%%%%%%%%%%%%%%%%%%%%%%%%%%
%                 					  				  		 %
%	2-Verma modules				 					 %
%                 					  						 %
%%%%%%%%%%%%%%%%%%%%%%%%%%%%%%%%%%%%

\section{2-Verma modules}\label{sec:2Verma}

Let $\Bbbk$ be a field of characteristic $0$. 
Let $\cV \in \dgcat_\Bbbk$ be a $\bZ$-graded  pretriangulated dg-category (see \cref{def:pretrdg}). Let $\cEnd_{\Hqe}(\cV) := \cRHom_{\Hqe}(\cV,\cV)$ be the dg-category of quasi-endofunctors on $\cV$  (see \cref{app:closedmondg}).

\begin{rem}
For example, $\cV$ could be the dg-category $\cD_{dg}(R,d)$ of cofibrant dg-modules over a dg-algebra $(R,d)$ (see \cref{def:dgenhderived}). Then, the subcategory of $\cEnd_{\Hqe}(\cV)$ consisting of coproduct preserving quasi-functors would be given by the dg-enhanced derived category of dg-bimodules $\cD_{dg}(\opalg{(R,d)} \otimes (R,d))$ (see \cref{thm:repdg}).
\end{rem}

%\smallskip 

Let $\Q_i := \bigoplus_{\ell \geq 0} q_i^{1+2\ell}\id$. It is a categorification of $\frac{q_i}{1-q_i^{2}} = \frac{1}{q_i^{-1}-q_i}$. We start by introducing a notion of dg-categorical action and dg-2-representation. 

\begin{defn}\label{def:dgcat}
A \emph{weak dg-categorical $U_q(\g)$-action} on $\cV$ is a collection of quasi-endofunctors $\F_i,\E_i,\K_\gamma \in Z^0(\cEnd_{\Hqe}(\cV))$ for all $i \in I$ and $\gamma \in Y^{\vee}$ such that
\begin{itemize}
\item there are isomorphisms
\begin{align*} 
\K_0 &\cong \id, & \K_\gamma \K_{\gamma'} &\cong \K_{\gamma+\gamma'}, \\
\K_\gamma \E_i &\cong q^{\gamma(\alpha_i)} \E_i\K_\gamma,  & \K_\gamma\F_i &\cong q^{-\gamma(\alpha_i)} \F_i \K_\gamma,
\end{align*}
where $q$ denotes the shift in the $q$-grading;
\item there is a quasi-isomorphism
\begin{equation}\label{eq:isoconedgcat}
\cone\bigl(\F_i\E_j \xrightarrow{u_{ij}} \E_j\F_i \bigr) \xrightarrow{\simeq} \delta_{ij} \cone\bigl(\Q_i\K_i \xrightarrow{h_i} \Q_i\K_i^{-1} \bigr),
\end{equation}
where $\K_i := \K_{\alpha_i^\vee}$;
\item there are isomorphisms
\begin{align*}
\bigoplus^{\lfloor (d_{ij}+1)/2 \rfloor}_{a=0} \begin{bmatrix} d_{ij}+1 \\ 2a \end{bmatrix}_{q_i}  \F_i^{2a}\F_j \F_i^{d_{ij}+1-2a} &\cong \bigoplus^{\lfloor d_{ij}/2 \rfloor}_{a=0}  \begin{bmatrix} d_{ij}+1 \\ 2a +1 \end{bmatrix}_{q_i} \F_i^{2a+1}\F_j \F_i^{d_{ij}-2a}, \\
\bigoplus^{\lfloor (d_{ij}+1)/2 \rfloor}_{a=0} \begin{bmatrix} d_{ij}+1 \\ 2a \end{bmatrix}_{q_i}  \E_i^{2a}\E_j \E_i^{d_{ij}+1-2a} &\cong \bigoplus^{\lfloor d_{ij}/2 \rfloor}_{a=0}  \begin{bmatrix} d_{ij}+1 \\ 2a +1 \end{bmatrix}_{q_i} \E_i^{2a+1}\E_j \E_i^{d_{ij}-2a},
\end{align*}
for all $i \neq j \in I$.
\end{itemize}
We say a weak dg-categorical $U_q(\g)$-action on $\cV$ is a \emph{dg-categorical action}
if in addition
\begin{itemize}
\item $\F_i$ is left adjoint to $q_i^{-1}\K_i \E_i$ in $Z^0(\cEnd_{\Hqe}(\cV))$;
\item there is a 
 map of algebras
\[
R_\g(\bi) \rightarrow Z^0(\END(\F_\bi)) := \bigoplus_{z \in \bZ} Z^0(\Hom(\F_\bi, q^z \F_\bi)),
\]
 with $\F_\bi := \F_{i_1} \cdots \F_{i_m}$, for all $\bi \in \Seq(m)$, inducing a surjection 
 \[
 R_\g(\bi) \otimes_\Bbbk Z^0(\END_\cV(M)) \twoheadrightarrow  Z^0(\END_\cV(\F_\bi M)),
 \]
 for all $M \in \cV$;
 \item $\cV$ is dg-triangulated (i.e. $H^0(\cV)$ is idempotent complete).
\end{itemize}
Such a $\cV$ carrying a dg-categorical action is called a \emph{dg-2-representation} of $U_q(\g)$. 
\end{defn}

The following notions are dg-2-categorical lifts of the notions of weight module and integrable module.
\begin{defn}
We say that a dg-2-representation $\cV$ is a \emph{weight dg-2-representation} if there is a map
\[
\lambda : Y^\vee \rightarrow \cEnd_{\Hqe}(\cV),
\]
where $\lambda(\gamma)$ commutes with the grading shift $q$ for all $\gamma \in Y^\vee$ and $\lambda(\gamma)\circ\lambda(\gamma') \cong \lambda(\gamma+\gamma')$, such that
\begin{align*}
\cV &\cong \bigoplus_{y \in Y} \cV_{\lambda,y}, & \K_\gamma|_{\cV_{\lambda,y}} (-) = \lambda(\gamma) q^{\gamma(y)} (-).
\end{align*}
\end{defn}

\begin{defn}
We say that a weight dg-2-representation  $\cV$ is \emph{$i$-integrable} if 
\begin{itemize}
\item $\lambda(\alpha_i^\vee) = q^{n_i}$ for some $n_i \in \bN$;
\item there is a quasi-isomorphism 
\begin{equation}\label{eq:intcone}
\cone\bigl(\Q_i\K_i \cV_{\lambda,y} \xrightarrow{h_i} \Q_i\K_i^{-1} \cV_{\lambda,y} \bigr) \xrightarrow{\simeq} \oplus_{[n_i -  \alpha_i^\vee(y)]_{q_i}} \id,
\end{equation}
where $\oplus_{[m]_{q_i}} \id := \oplus_{[-m]_{q_i}} \id [1]$ whenever $m < 0$;
\item  $\F_i$ and $\E_i$ are locally nilpotent.
\end{itemize}
\end{defn}

Under some mild hypothesis, this definition recovers the notion of integrable 2-re\-pre\-sen\-ta\-tion from~\cite{rouquier} and~\cite{laudaimplicit}. 

\begin{prop}
Suppose $\cV$ is $i$-integrable for all $i \in I$. 
Also suppose that there is some $M \in \cV_{\lambda,0}$ such that $\E_iM = 0$ for all $i \in I$ and $\End_\cV(M) \cong (\Bbbk, 0)$, and $H^0(\cV)$ is c.b.l.f. generated by $\{\F_\bi M\}_{\bi \in \Seq(I)}$. 
Then, $H^0(\cV)$ carries an integrable categorical $U_q(\g)$-action in the sense of~\cite{rouquier}.
\end{prop}

\begin{proof}
First, by adjunction, \cref{eq:isoconedgcat} and \cref{eq:intcone}, we have
\[
\gdim_q H^0(\END_{\cV}(\F_\bi M)) \cong \gdim_q R_\g^N(\bi),
\]
for all $\bi \in \Seq(I)$. In particular, we have that $x_1^{n_i} 1_i$ acts by $0$ on $H^0(\End_{\cV}(\F_i M))$ for all $i \in I$, and $x_1 1_i$ acts non-trivially whenever $n_i > 1$. Thus, there is a map
\[
\gamma :
R_\g^N(\bi) \rightarrow H^0(\END_{\cV}(\F_\bi M)).
\]
Since $\gamma$ is surjective, we obtain $R_\g^N(\bi) \cong  H^0(\END_{\cV}(\F_\bi M))$, and the result follows from \cref{thm:sl2commutRpN}.
\end{proof}

For a $\bZ^n$-graded dg-algebra $(A,d)$, we put $\cD^{lf}_{dg}(A,d)$ for the dg-category having as objects the one in $\cD^{lf}(A,d) \cap \cD_{dg}(A,d)$ and the hom-spaces inherited from $\cD_{dg}(A,d)$. It is a dg-enhancement of the c.b.l.f. derived category of $(A,d)$. 

\begin{defn}\label{def:2verma}
A \emph{parabolic 2-Verma module} for $\p$ is a $\bZ \times \bZ^{|I_r|}$-graded weight dg-2-representation $\cV$ such that
\begin{itemize}
\item the highest weight space $\cV_\lambda :=\cV_{\lambda,0} \cong \cD_{dg}^{lf}(\Bbbk,0)$;
\item there exists $M \cong (\Bbbk,0) \in \cV_\lambda$ such that $\cV_{\lambda,y}$ is c.b.l.f. generated by $\{F_\bi M\}_{\bi \in \Seq(y)}$ for all $-y \in X^+$, and $\cV_{\lambda,y} = 0$ otherwise ;
\item $\cV$ is $i$-integrable for all $i \in I_f$;
\item $h_j = 0$ and $\lambda_{\alpha_j^\vee} = \lambda_j$ (the degree shift) for all $j \notin I_f$; 
\item  for each $j \notin I_f$, $n_j \in \bN$ and $\bi \in \Seq(I)$,  after specializing $\lambda_j = q^{n_j}$, there exists a differential $d_{n_j}$ anticommuting with the differential $d$ of $\bigl(\End_\cV(F_\bi M),d\bigr)$ such that the  triangulated dg-category generated by c.b.l.f. iterated extension of the representable modules of $\cV^{n_j} := \bigoplus_{\bi \in \Seq(I)} \bigl(\End_\cV(F_\bi M), d+d_{n_j}\bigr)$ is $j$-integrable with $\lambda(\alpha_j^\vee)  = q^{n_j}$.
\end{itemize}
\end{defn}

\begin{prop}\label{prop:isoend2verma}
Let $\cV$ be a parabolic 2-Verma module. 
There is an isomorphism
\[
(R_\bo(\bi),d_N) \cong \END_\cV(\F_\bi M),
\]
in $\cD(\Bbbk,0)$ for $M \cong (\Bbbk, 0) \in \cV_{\lambda}$. 
\end{prop}

\begin{proof}
First, by adjunction together with \cref{eq:isoconedgcat} and \cref{eq:intcone} we have
\begin{equation}\label{eq:isoENDFiRpN}
\END_\cV(\F_i M) \cong \HOM_\cV(M, q_i^{-1} \K_i \E_i \F_i M) \cong R_\p^N(i),
\end{equation}
in $\cD(\Bbbk,0)$ for all $i \in I$. Also,  
\begin{equation}\label{eq:dimH2Verma}
\gdim_q H^*(\END_{\cV}(\F_\bi M)) = \gdim_q R_\p^N(\bi), 
\end{equation}
for all $\bi \in \Seq(I)$. 
In particular, there is a relation up to homotopy
\begin{align}\label{eq:tightfdcommutesuptohomotopy}
\alpha \ 
\tikzdiagh[yscale=1.5]{0}{
	\draw (0,0) node[below]{\small $j$} ..controls (0,.15) and (1,.15) .. (1,.5)
		 ..controls (1,.85) and (0,.85) .. (0,1) -- (0,1.5);
	\draw (1,0) node[below]{\small $i$} ..controls (1,.15) and (0,.15) .. (0,.5)
		..controls (0,.85) and (1,.85) .. (1,1) -- (1,1.5);
	\fdot{i}{.4,.5};
	\fdot{j}{.5,1.25};
}
\ + \beta \ 
\tikzdiagh[yscale=1.5]{0}{
	\draw (0,-.5) node[below]{\small $j$} -- (0,0)  ..controls (0,.15) and (1,.15) .. (1,.5)
		 ..controls (1,.85) and (0,.85) .. (0,1);
	\draw (1,-.5) node[below]{\small $i$}  -- (1,0)..controls (1,.15) and (0,.15) .. (0,.5)
		..controls (0,.85) and (1,.85) .. (1,1);
	\fdot{i}{.4,.5};
	\fdot{j}{.5,-.25};
}
\ = \ 0,
\end{align}
in $\END_\cV(\F_i\F_jM)$ for all $i,j \in I_r$, identifying the diagrams with the image of the KLR elements under the surjection $R_\g(ij) \twoheadrightarrow Z^0(\END_\cV(F_iF_jM))$, and the floating dot coming from the isomorphism \cref{eq:isoENDFiRpN}. 
Then, the existence of $d_{n_i}$ and $d_{n_j}$ forces to have $\alpha = \beta$. Thus, by \cref{cor:Rbmindesc}, there is an $A_\infty$-map
\[
(R_\bo(\bi), d_N) \rightarrow \END_\cV(\F_\bi M).
\]
By \cref{eq:dimH2Verma}, we conclude it is a quasi-isomorphism. Thus, there exists an isomorphism $(R_\bo(\bi), d_N) \cong \END_\cV(\F_\bi M)$ in $\cD(\Bbbk,0)$.
\end{proof}

Using \cref{thm:repdg} we can think of $\F_i^N$ and $\E_i^N$ from \cref{sec:FiNFunctors} as quasi-endofunctors of  $\cD_{dg}(R_\bo, d_n)$. 
By \cref{prop:RbodgSES} we obtain immediately the following:

\begin{cor}\label{cor:quasiisocones}
For all $i \in I$ there is a quasi-isomorphism of cones
\[
\cone\bigl(\F_i^N\E_i^N \id_\nu \rightarrow \E_i^N\F_i^N \id_\nu \bigr) \xrightarrow{\simeq}  \cone\bigl(\Q_i \lambda_i q_i^{-\alpha_i^\vee(\nu)}  \id_\nu \rightarrow \Q_i \lambda_i^{-1} q^{\alpha_i^\vee(\nu)} \id_\nu\bigr), 
\]
in $\cEnd_{\Hqe}(\cD_{dg}(R_\bo,d_N))$. 
\end{cor}

Together with \cref{prop:serreinRpN}, it means that the dg-enhancement $\cM^\p_{dg}(\Lambda,N)$ of $\cM^\p(\Lambda,N)$ (obtained by replacing $\cD^{lf}(R_\bo(m), d_N)$ with $\cD^{lf}_{dg}(R_\bo(m), d_N)$) is a weight dg-2-representation of $U_q(\g)$, where 
\[
\lambda(\alpha_i^\vee) := \begin{cases}
\lambda_i, & \text{ whenever $i \in I_r$,} \\
q^{n_i}, & \text{ whenever $i \in I_f$.}
\end{cases}
\] 
Then, by \cref{thm:sl2commutRpN}, we obtain that $\cM^\p_{dg}(\Lambda, N)$ is a parabolic 2-Verma module. 

\begin{cor}
Let $\cV$ be a parabolic 2-Verma module. 
There is a quasi-equivalence 
\[
\cM^\p_{dg}(\Lambda, N) \xrightarrow{\simeq} \cV.
\]
\end{cor}

\begin{proof}
Since $\cV_{\lambda, y}$ is c.b.l.f. generated by $\bigoplus_{ \bi \in \Seq(y) } F_\bi M$, we have that $\cV_{\lambda, y}$ is completely determined as dg-category by $\END_\cV( F_\bi M)$. 
Thus, we conclude by using \cref{prop:isoend2verma}. 
\end{proof}

\begin{rem}
A parabolic 2-Verma module can also be given a `2-categorical' interpretation as an $(\infty,2)$-category where the hom-spaces are stable $(\infty,1)$-categories. For this, it is enough to see $\cD_{dg}(R_\bo(\nu), d_N)$ as $0$-cells in the $(\infty,2)$-category of $A_\infty$-categories constructed by Faonte~\cite{faonte}, and replace $\cHom_{\Hqe}$ by the dg-nerve of Lurie~\cite{lurie}. Thanks to~\cite{faonte2}, we know that this is a stable $(\infty,1)$-category. 
\end{rem}

%%%%%%%%%%%%%%%%	End of file	%%%%%%%%%%%%%

%% file: sections/appendix_dgcat.tex
%%%%%%%%%%%%%%%%%%%%%%%%%%%%%%%%%%%%
%                 					  				  		 %
%	Appendix : Dg-categories			 					 %
%                 					  						 %
%%%%%%%%%%%%%%%%%%%%%%%%%%%%%%%%%%%%

\appendix

\section{Summary on the homotopy category of dg-categories and pretriangulated dg-categories}\label{sec:appendixA}
We gather some useful results on the homotopy category of dg-categories. References for this section are \cite{keller} and \cite{toen}. We also suggest \cite{kellersurvey} and \cite{toenlectures} for nice surveys on the subject. 

\smallskip

Our goal is to recall how to construct a category of dg-categories up to quasi-equivalence, so that the space of functors between two `triangulated categories' is `triangulated'. 

\subsection{Dg-categories}

Recall the definition of a dg-category:

\begin{defn}
A \emph{dg-category} $\cA$ is a $\Bbbk$-linear category such that:
\begin{itemize}
\item $\Hom_\cA(X,Y)$ is a $\bZ$-graded $\Bbbk$-vector space ;
\item the composition
\[
\Hom_{\cA}(Y,Z) \otimes_\Bbbk \Hom_\cA(X,Y) \xrightarrow{\ - \circ - \ }  \Hom_\cA(X,Z),
\]
preserves the $\bZ$-degree ;
\item there is a differential $d : \Hom _\cA(X,Y)^i \rightarrow \Hom_\cA(X,Y)^{i-1}$ such that
\begin{align*}
d^2 &= 0, & d(f \circ g) = df \circ g + (-1)^{|f|} f \circ dg.
\end{align*}
\end{itemize}
\end{defn}

\begin{rem}
We use a differential of degree $-1$ to match the conventions used in the rest of the paper.
\end{rem}

\begin{exe}
Any dg-algebra $(A,d)$ can be seen as a dg-category $\boldsymbol{BA}$ with a single abstract object $\star$ and $\Hom_{\boldsymbol{BA}}(\star,\star) := (A,d)$. 
\end{exe}

\begin{exe}\label{ex:dgcomplexes}
Let $\cC$ be an abelian, Grothendieck, $\Bbbk$-linear category. Consider the category $C(\cC)$ of complexes in $\cC$, and define $C_{dg}(\cC)$ as the category where
\begin{itemize}
\item objects are complexes in $\cC$;
\item hom-spaces are homogeneous maps of $\bZ$-graded modules;
\item the differential $d : \Hom_{C_{dg}(\cC)}(X^\bullet,Y^\bullet)^i \rightarrow \Hom_{C_{dg}(\cC)}(X^\bullet,Y^\bullet)^{i-1}$ is given by
\[
df := d_Y \circ f - (-1)^{|f|} f \circ d_X.
\]
\end{itemize}
This data forms a dg-category. 
\end{exe}

Given a dg-category $\cA$, one defines   
\begin{enumerate}
\item the \emph{underlying category $Z^0(\cA)$} as
\begin{itemize}
\item having the same objects as $\cA$;
\item $\Hom_{Z^0(\cA)}(X,Y) := \ker\bigl( \Hom_{\cA}(X,Y)^0 \xrightarrow{d} \Hom_{\cA}(X,Y)^{-1} \bigr)$;
\end{itemize}
\item the \emph{homotopy category $H^0(\cA)$} (or $[\cA]$) as
\begin{itemize}
\item having the same objects as $\cA$;
\item $\Hom_{H^0(\cA)}(X,Y) := H^0(\Hom_\cA(X,Y),d)$.
\end{itemize}
\end{enumerate}

\begin{exe}
For $\cC$ as in \cref{ex:dgcomplexes}, we have $Z^0(C_{dg}(\cC)) \cong C(\cC)$ and $H^0(C_{dg}(\cC)) \cong Kom(\cC)$ the homotopy category of complexes in $\cC$. 
\end{exe}

\subsection{Category of dg-categories}

\begin{defn}
A \emph{dg-functor $F : \cA \rightarrow \cB$} is a functor between two dg-categories such that $F(d_\cA f) = d_\cB(Ff)$. 
We write $[F] : H^0(\cA) \rightarrow H^0(\cB)$ for the induced functor. 
\end{defn}

We write $\dgcat$ for the \emph{category of dg-categories}, where objects are dg-categories and hom-spaces are given by dg-functors. 

\smallskip

Let $F,G : \cA \rightarrow \cB$ be a pair of dg-functors between dg-categories. One defines $\cHom(F,G)$ as the $\bZ$-graded $\Bbbk$-module of homogeneous natural transformations equipped with the differential induced by $d \in \Hom_\cB(FX,GX)$ for all $X \in \cA$. Then, we put $\Hom(F,G) := Z^0(\cHom(F,G))$. 

\begin{defn}
A dg-functor $\cA \rightarrow \cB$ is a \emph{quasi-equivalence} if
\begin{itemize}
\item $F : \Hom_{\cA}(X,Y) \xrightarrow{\simeq} \Hom_{\cB}(FX,FY)$ is a quasi-isomorphism for all $X,Y \in \cA$;
\item $[F] : H^0(\cA) \rightarrow H^0(\cB)$ is essentially surjective (thus an equivalence).
\end{itemize}
\end{defn}

One defines the dg-category $\cHom(\cA,\cB)$ of dg-functors between $\cA$ and $\cB$ as
\begin{itemize}
\item objects are dg-functors $\cA \rightarrow \cB$;
\item hom-spaces are $\Hom_{\cHom(\cA,\cB)}(F,G) := \cHom(F,G)$.
\end{itemize}
There is also a notion of tensor product of dg-categories $\cA \otimes \cB$ defined as
\begin{itemize}
\item objects are pairs $X \otimes Y$ for all $X \in \cA$ and $Y \in \cB$;
\item hom-spaces are $\Hom_{\cA \otimes \cB}(X \otimes Y, X' \otimes Y') := \Hom_{\cA}(X,X') \otimes_\Bbbk \Hom_{\cB}(Y,Y')$ with composition
 \[
 (f' \otimes g') \circ (f \otimes g) := (-1)^{|g'||f|} (f' \circ f) \otimes (g' \circ g);
 \]
\item the differential is $d(f \otimes g) := df \otimes g + (-1)^{|f|} f \otimes dg$.
\end{itemize}
Then, there is a bijection
\[
\Hom_{\dgcat}(\cA \otimes \cB, \cC) \cong \Hom_{\dgcat}(\cA, \cHom(\cB,\cC)).
\]
This defines a symmetric closed monoidal structure on $\dgcat$. 
 However, the tensor product of dg-categories does not preserve quasi-equivalences. 

\subsection{Dg-modules}

Let $\cA$ be a dg-category. The opposite dg-category $\opalg{\cA}$ is given by 
\begin{itemize}
\item same objects as in $\cA$;
\item $\Hom_{\opalg{\cA}}(X,Y) := \Hom_{\cA}(Y,X)$;
\item composition $g \circ_{\opalg{\cA}} f := (-1)^{|f||g|} f \circ_\cA g$. 
\end{itemize}
 A \emph{left (resp. right) dg-module} $M$ over $\cA$ is a dg-functor 
\[
M : \cA \rightarrow C_{dg}(\Bbbk), \quad \text{(resp. } N : \opalg{\cA} \rightarrow C_{dg} \text{),}
\]
where $C_{dg}(\Bbbk)$ is the dg-category of $\Bbbk$-complexes. 
The  \emph{dg-category of (right) dg-modules} is $\opalg{\cA}\amod := \cHom(\opalg{\cA}, C_{dg}(\Bbbk))$. 
The \emph{category of (right) dg-modules} is $C(\cA) := Z^0(\cA\amod)$, and it is an abelian category. 
The \emph{derived category $\cD(\cA)$} is the localization of $Z^0(\opalg{\cA}\amod)$ along quasi-isomorphisms.

Moreover, for any $X \in \cA$ there is a right dg-module 
\[
X^{\wedge} := \Hom_\cA(-,X).
\]
One calls such dg-module \emph{representable}. Any dg-module quasi-isomorphic to a representable dg-module is called \emph{quasi-representable}. It yields a dg-enriched Yoneda embedding
\[
\cA \rightarrow \opalg{\cA}\amod.
\]

\begin{exe}
Let $(A,d)$ be a dg-algebra. Then $Z^0(\boldsymbol{BA})\amod \cong (A,d)\amod$ and $\cD(\boldsymbol{BA}) \cong \cD(A,d)$. 
The unique representable dg-module $\Hom_{\boldsymbol{BA}}(-,\star)$ is equivalent to the free module $(A,d)$. 
\end{exe}

\subsection{Model categories}

We recall the basics of model category theory from~\cite{modelcat}. 
Model category theory is a powerful tool to study localization of categories. 
For example, we can use it to compute hom-spaces in a derived category. 
We will mainly use it to describe the homotopy category of dg-categories up to quasi-equivalence. 

\smallskip

Let $M$ be a category with limits and colimits. 
\begin{defn}
A \emph{model category} on $M$ is the data of three classes of morphisms
\begin{itemize}
\item the \emph{weak equivalences $W$};
\item the \emph{fibrations $Fib$};
\item the \emph{cofibrations $Cof$};
\end{itemize}
satisfying
\begin{itemize}
\item for $X \xrightarrow{f} Y \xrightarrow{g} Z \in M$, if two out of three terms in $\{f,g,g\circ f\}$ are in $W$, then so is the third;
\item \emph{stability along retracts}: $W,Fib$ and $Cof$ are stable along retracts, that is if we have a commutative diagram
\[
\begin{tikzcd}
X \ar{r} \ar[bend left=30]{rr}{\id_X} \ar{d}{g} & Y  \ar{d}{f} \ar{r} & X  \ar{d}{g} \\
X' \ar{r} \ar[swap,bend right=30]{rr}{\id_{X'}} & Y'  \ar{r}& X'
\end{tikzcd}
\]
and $f  \in W,Fib$ or $Cof$ then so is $g$. 
\item \emph{factorization}: any $X \xrightarrow{f} Y$ factorizes as $p \circ i$ where $p \in Fib$ and $i \in Cof \cap W$ or $p \in Fib \cap W$ and $i \in Cof$, and the factorization is functorial in $f$;
\item \emph{lifting property}: given a commutative square diagram
\[
\begin{tikzcd}
A \ar[swap]{d}{Cof \ni i }\ar{r} & X \ar{d}{p \in Fib} \\
B \ar{r} \ar[dashed]{ur}{\exists h} & Y
\end{tikzcd}
\]
with $i \in Cof$ and $p \in Fib$, if either $i \in W$ or $p \in W$, then there exists $h : B \rightarrow X$ making the diagram commute.
\end{itemize}
\end{defn}

We tend to think about fibrations as `nicely behaved surjections', and cofibrations as `nicely behaved injections'. 

\smallskip

The localization $Ho(M) := W^{-1}M$ of $M$ along weak equivalences is called the \emph{homotopy category of $M$}. It has a nice description in terms of \emph{homotopy classes} of maps between \emph{fibrant} and \emph{cofibrant} objects. 

\begin{defn}
If $\emptyset \rightarrow X \in Cof$, then we say $X$ is \emph{cofibrant}. 
If $Y \rightarrow * \in Fib$, then $Y$ is \emph{fibrant}.
\end{defn}

One says that $f \sim g$, that is $f : X \rightarrow Y$ is \emph{homotopy equivalent} to $g : X \rightarrow Y$, if there is a commutative diagram
\[
\begin{tikzcd}
X \ar[loop above]{}{\id_X} \ar{d}{i} \ar{dr}{f} & {} \\
C(X) \ar{r}{h} \ar[bend left=60,dashed]{u}{Fib \cap W \ni p}  \ar[bend right=60,dashed,swap]{d}{p} & Y \\
X \ar[loop below]{}{\id_X} \ar[swap]{u}{j} \ar[swap]{ur}{g} & {}
\end{tikzcd}
\]
where $i \sqcup j : X \sqcup X \rightarrow C(X) \in Cof$. 
One calls $C(X)$ the \emph{cylinder object of $X$}. 
When $X$ is cofibrant and $Y$ fibrant, then $\sim$ is an equivalence relation on $\Hom_M(X,Y)$. Moreover, we have
\[
\Hom_{Ho(M)}(X,Y) \cong \Hom_{M}(X,Y)/\sim
\]
whenever $X$ is cofibrant and $Y$ fibrant. Note that any $X \in M$ admits a cofibrant replacement $QX$ since we have a commutative diagram
\[
\begin{tikzcd}
\emptyset \ar[swap]{dr}{Cof \ni i} \ar{rr} && X \\
&QX \ar[swap]{ur}{p \in Fib \cap W} &
\end{tikzcd}
\]
Similarly, any $Y \in M$ admits a fibrant replacement $RY$. 

\smallskip

Let $M^{cf}$ be the full subcategory of $M$ given by objects that are both fibrant and cofibrant. Let $M^{cf}/\sim$ be the quotient of $M^{cf}$ by identifying maps that are homotopy equivalent. 
Then, the localization functor $M \rightarrow Ho(M)$ restricts to $M^{cf}$, inducing an equivalence of categories
\[
M^{cf}/\sim \ \xrightarrow{\ \simeq\ }\  Ho(M).
\]

\begin{exe}
Let $C(\Bbbk)$ be the category of complexes of $\Bbbk$-modules. It comes with a model category structure where $W$ is the quasi-isomorphisms, $Fib$ is the surjective maps, and $Cof$ is given by the maps respecting the lifting property. All objects are fibrant and the cofibrant objects are essentially the complexes of projective $\Bbbk$-modules. Then $Ho(C(\Bbbk)) \cong \cD(\Bbbk)$. 
\end{exe}

A model category on $M$ is a \emph{$C(\Bbbk)$-model category} if it is (strongly) enriched over $C(\Bbbk)$, and the models are compatible (see~\cite[\S3.1]{toenlectures} for a precise definition). This definition means that we have:
\begin{itemize}
\item a tensor product $ - \otimes - : C(\Bbbk) \times M \rightarrow M$;
\item an enriched dg-hom-space $\cHom_M(X,Y) \in C(\Bbbk)$ for any $X,Y \in M$ compatible with the tensor product:
\[
\Hom_{M}(E \otimes X, Y) \cong \Hom_{C(\Bbbk)}(E, \cHom_M(X,Y)) \text{; }
\] 
\item $Ho(M)$ is enriched over $\cD(\Bbbk) \cong Ho(C(\Bbbk))$;
\item a derived hom-functor 
\[
\cRHom_M(X,Y) := \cHom_M(QX, RY) \in \cD(\Bbbk),
\]
where $QX$ is a cofibrant replacement of $X$, and $RY$ a fibrant replacement of $Y$;
\item $\Hom_{Ho(M)}(X,Y) \cong H^0(\cRHom_M(X,Y))$.
\end{itemize}
Note that in particular for $X,Y \in M^{cf}$ we have $\Hom_{Ho(M)}(X,Y) \cong  H^0(\cHom(X,Y))$. 

\begin{exe}
Let $\cA$ be  a dg-category. There is a $C(\Bbbk)$-model category on $\cA\amod$ where $W$ is given by the quasi-isomorphisms, $Fib$ are the surjective morphisms, and $Cof$ is given by the maps respecting the lifting property. Then $Ho(\cA\amod) \cong \cD(\cA)$. 
\end{exe}

\begin{rem}
In the $C(\Bbbk)$-model category $\cA\amod$, all objects are fibrant. Moreover, $P$ is cofibrant if and only if for all surjective quasi-isomorphism $f : L \xrightarrow{\simeq} X$ (i.e. map in $W \cap Fib$) then there exists $h : P \rightarrow L$ such that the following diagram commutes:
\[
\begin{tikzcd}
\emptyset \ar{d} \ar{r} & L \ar[twoheadrightarrow]{d}{\vsimeq} \\
P  \ar[dashed]{ur}{\exists h} \ar{r} & X
\end{tikzcd}
\]
In a practical way, cofibrant dg-modules are quasi-isomorphic to direct summand of dg-modules admitting a (possibly infinite) exhaustive filtration where all the quotients are free dg-modules.
\end{rem}

\begin{defn}
For $M$ a $C(\Bbbk)$-model category, let $\und{M}$ (resp. $Int(M)$) be the dg-category with
\begin{itemize}
\item the same objects as $M$ (resp. $M^{cf}$);
\item $\Hom_{\und M}(X,Y) := \cHom_M(X,Y)$.
\end{itemize}
\end{defn}

Then, we have $H^0(Int(M)) \cong Ho(M)$, and we say that $Int(M)$ is a \emph{dg-enhancement} of $Ho(M)$.

\begin{defn}\label{def:dgenhderived}
We write 
\[
\cD_{dg}(\cA) := Int(\cA\amod)
\]
for the \emph{dg-enhanced derived category of $\cA$}.
\end{defn}
Note that $\cD_{dg}(\cA)$ is a dg-enhancement of $\cD(\cA)$ since we have $H^0(\cD_{dg}(\cA)) \cong \cD(\cA)$. 

\begin{exe}
Let $R$ be a $\Bbbk$-algebra viewed as a dg-category with trivial differential. Then we have that $\cD_{dg}(R)$ is the dg-category of complexes of projective $R$-modules. 
\end{exe}

\subsection{The model category of dg-categories}

Let $W$ be the collection of quasi-equivalences in $\dgcat$. Let $Fib$ be the collection of dg-functors $F : \cA \rightarrow \cB$ in $\dgcat$ such that
\begin{enumerate}
\item $F_{X,Y} : \Hom_{\cA}(X,Y) \twoheadrightarrow \Hom_{\cB}(FX,FY)$ is surjective;
\item for all isomorphism $v : F(X) \xrightarrow{\simeq} Y \in H^0(\cB)$ there exists an isomorphism $u : X \xrightarrow{\simeq} Y_0 \in H^0(\cA)$ such that $[F](u) = v$. 
\end{enumerate}
This defines a model structure on $\dgcat$ where everything is fibrant. One calls 
\[
\Hqe := Ho(\dgcat)
\]
 the \emph{homotopy category of dg-categories} (up to quasi-equivalence). 

\smallskip

How can we compute $\Hom_{\Hqe}(\cA,\cB)$ ? It appears that constructing a cofibrant replacement for $\cA$ is in general a difficult problem. However, we can do the following:
\begin{enumerate}
\item replace $\cA$ by a \emph{$\Bbbk$-flat} quasi-equivalent dg-category $\cA'$: meaning it is such that 
\[
\Hom_{\cA'}(X,Y) \otimes_\Bbbk -
\]
 preserves quasi-isomorphisms  (e.g. when $\Hom_{\cA'}(X,Y)$ is cofibrant in $C(\Bbbk)$, i.e. a complex of projective $\Bbbk$-modules);
\item define $\Rep(\cA,\cB)$ as the subcategory of $\cD(\opalg{\cA} \otimes \cB )$ 
with $F \in \Rep(\cA,\cB)$ if and only if for all $X \in \cA$ there exists $Y \in \cB$ such that 
\[
X \Lotimes F \cong_{\cD(\cB)} Y^\vee, 
\]
(in other words, $F$ is a dg-bimodule sending representable $\cA$-modules to quasi-representable $\cB$-modules);
\item then
\[
\Hom_{\Hqe}(\cA,\cB) \cong Iso(\Rep(\cA,\cB)), 
\]
where $Iso$ means the set of objects up to isomorphism. 
\end{enumerate}
 
\begin{rem}
Note that whenever $\Bbbk$ is a field, all dg-categories are $\Bbbk$-flat.
\end{rem}
 
We refer to elements in $\Rep(\cA,\cB)$ as \emph{quasi-functors}. Note that a quasi-functor $F : \cA \rightarrow \cB$ induces a functor 
\[
 [F] : H^0(\cA) \rightarrow H^0(\cB).
\]
 Thus, we can think of $\Rep(\cA,\cB)$ as the category of `representations up to homotopy' of $\cA$ in $\cB$. 
 
\subsubsection{Closed monoidal structure}\label{app:closedmondg}

If $\cA$ is cofibrant, then $- \otimes \cA$ preserves quasi-equivalences and one can define the bifunctor
\[
- \Lotimes - : \Hqe \times \Hqe \rightarrow \Hqe, \quad \cA \Lotimes \cB := Q\cA \otimes Q\cB, 
\]
where $Q\cA$ and $Q\cB$ are cofibrant replacements. Then, as proven by Toen~\cite{toen}, there exists an internal hom-functor $\cRHom_{\Hqe}(-,-)$ such that
\[
\Hom_{\Hqe}(\cA \Lotimes \cB, \cC) \cong \Hom_{\Hqe}(\cA, \cRHom_{\Hqe}(\cB,\cC)).
\]
Therefore, $\Hqe$ is a symmetric closed monoidal category. 

\begin{rem}
Note that the internal hom can not simply be the derived hom functor (because tensor product of cofibrant dg-categories is not cofibrant in general). 
\end{rem}

Define the \emph{dg-category of quasi-functors} $\Rep_{dg}(\cA,\cB)$ as 
\begin{itemize}
\item the objects in $\Rep(\cA,\cB) \cap (\opalg{\cA} \otimes \cB\amod)^{cf}$;
\item the dg-homs $\cHom(X,Y)$ of $Int(\opalg{\cA} \otimes \cB\amod)$.
\end{itemize} 
In other words, $\Rep_{dg}(\cA,\cB)$ is the full subcategory of quasi-functors in $\cD_{dg}(\opalg{\cA} \otimes \cB)$, thus of cofibrant dg-bimodules that preserves quasi-representable modules.  It is a dg-enhancement of $\Rep(\cA,\cB)$. 

\smallskip

If $\cA$ is $\Bbbk$-flat, then 
\[
\cRHom_{\Hqe}(\cA,\cB) \cong_{\Hqe} \Rep_{dg}(\cA,\cB).
\]
Thus $H^0(\cRHom_{\Hqe}(\cA,\cB) ) \cong \Hom_{\Hqe}(\cA,\cB)$. 

\begin{rem}
If $\Bbbk$ is a field of characteristic $0$, then the dg-category $\cRHom_{\Hqe}(\cA,\cB)$ is equivalent to the $A_\infty$-category of strictly unital $A_\infty$-functors \cite{faonte}.
\end{rem}

\begin{exe}
We have $\Rep_{dg}(\cA, Int(C(\Bbbk))) \cong Int(\opalg{\cA}\amod) \cong \cD_{dg}(\cA)$.
\end{exe}

Recall that classical Morita theory says that for $A$ and $B$ being $\Bbbk$-algebras, there is an equivalence
\[
\Hom^{cop}(A\amod, B\amod) \cong \opalg{A}\otimes_\Bbbk B \amod,
\]
where $\Hom^{cop}$ is given by the functors that preserve coproducts. 

\smallskip

Similarly, we put $\Rep_{dg}^{cop}(\cD_{dg}(\cA), \cD_{dg}(\cB))$ for the subcategory of $\Rep_{dg}(\cD_{dg}(\cA), \cD_{dg}(\cB))$ where $F \in \Rep_{dg}^{cop}(\cD_{dg}(\cA), \cD_{dg}(\cB))$ if and only if $[F]    : \cD(\cA) \rightarrow \cD(\cB)$ preserves coproducts. 

\begin{thm}\label{thm:repdg}
If $\cA$ is $\Bbbk$-flat, then we have
\[
\cRHom^{cop}_{\Hqe}(\cD_{dg}(\cA), \cD_{dg}(\cB)) := \Rep_{dg}^{cop}(\cD_{dg}(\cA), \cD_{dg}(\cB)) \cong_{\Hqe} \cD_{dg}(\opalg{\cA} \otimes \cB).
\]
\end{thm}

Under the hypothesis of \cref{thm:repdg}, the internal composition of dg-quasifunctors preserving coproducts is given by taking a cofibrant replacement of the derived tensor product over $\cA$. 

\subsection{Pretriangulated dg-categories}\label{app:triangulateddg}

Basically, a triangulated dg-cate\-go\-ry is a dg-ca\-te\-go\-ry such that its homotopy category is canonically triangulated. But before being able to give a precise definition, we need to do a detour through Quillen exact categories, Frobenius categories and stable categories.

\subsubsection{Frobenius structure on $C(\cA)$}

Recall a Quillen exact category \cite{quillenexact} is an additive category with a class %$\cE$ 
of short exact sequences
\[
0 \rightarrow X \xrightarrow{f} Y \xrightarrow{g} Z \rightarrow 0,
\]
called \emph{conflations}, which are pairs of ker-coker, where $f$ is called an \emph{inflation} and $g$ a \emph{deflation}, respecting some axioms:
\begin{itemize}
\item the identity is a deflation;
\item the composition of deflations is a deflation;
\item deflations (resp. inflations) are stable under base (resp. cobase) change.
\end{itemize}
A \emph{Frobenius} category is a Quillen exact category having enough injectives and projectives, and where injectives coincide with projectives. The \emph{stable category $\und \cC$} of  a Frobenius category $\cC$ is given by modding out the maps that factor through an injective/projective object. 
It carries a canonical triangulated structure where:
\begin{itemize}
\item the suspension functor $S$ is obtained by taking the target of a conflation
\[
0 \rightarrow X \rightarrow IX \rightarrow SX \rightarrow 0,
\]
where $IA$ is an injective hull of $X$, for all $X \in \cC$;
\item the distinguished triangles are equivalent to standard triangles
\[
X \xrightarrow{f} Y \xrightarrow{g} Z \xrightarrow{h} SX,
\]
 obtained from conflations by the following commutative diagram:
 \[
 \begin{tikzcd}
 0 \ar{r} & X \ar{r}{f}  \ar{d}{\id}  & Y  \ar{r}{g}  \ar{d} & Z  \ar{r}  \ar{d}{h} & 0 \\
 0 \ar{r} & X \ar{r} & IX \ar{r} & SX \ar{r} & 0.
 \end{tikzcd}
 \]
\end{itemize}

\begin{exe}\label{ex:quillenexact}
Let $\cA$ be a small dg-category. One can put a Frobenius structure on $C(\cA) (:= Z^0(\opalg{\cA}\amod))$ by using split short exact sequences as class of conflations. Then there is an equivalence $\und{C(\cA)} \cong H^0(\opalg{\cA}\amod)$, and the suspension functor coincides with the usual homological shift. 
Moreover, $\cD(\cA)$ inherits the triangulated structure from $H^0(\cA\amod)$, where distinguished triangles are equivalent to distinguished triangles obtained from all short exact sequences in $C(\cA)$.
\end{exe}

\subsubsection{Pretriangulated dg-categories}

Remark for any dg-category $\cA$ there is a Yoneda functor
\[
Z^0(\cA) \rightarrow C(\cA), \quad X \mapsto \Hom_{\cA}(-,X).
\]
\begin{defn}\label{def:pretrdg}
A dg-category $\cT$ is \emph{pretriangulated} if the image of the Yoneda functor is stable under translations and extensions (for the Quillen exact structure on $C(\cT)$ described in \cref{ex:quillenexact}).
\end{defn}

This definition implies that 
\begin{itemize}
\item $Z^0(\cT)$ is a Frobenius subcategory of $C(\cT)$;
\item $H^0(\cT)$ inherits a triangulated structure, called \emph{canonical triangulated structure}, from $H^0(\cT\amod)$.
\end{itemize}

\begin{exe}
Let $\cA$ be a dg-category. We have that $\cD_{dg}(\cA)$ is pretriangulated with $Z^0(\cD_{dg}(\cA)) \cong C(\cA)^{cf}$. 
Moreover, the canonical triangulated structure of $H^0(\cD_{dg}(\cA))$ coindices with the usual on $\cD(\cA)$. 
\end{exe}

Then, it is possible to show that
\begin{itemize}
\item any dg-category $\cA$ admits a pretriangulated hull $pretr(\cA)$ such that
\[
\cRHom_{\Hqe}(\cA,\cT) \xrightarrow{\simeq} \cRHom_{\Hqe}(pretr(\cA), \cT),
\]
for all pretriangulated dg-category $\cT$;
\item $\cRHom_{\Hqe}(\cA,\cT)$ is pretriangulated whenever $\cT$ is pretriangulated;
\item any dg-functor $F: \cT \rightarrow \cT'$ between pretriangulated dg-categories induces a triangulated functor $[F] : H^0(\cT) \rightarrow H^0(\cT')$. 
\end{itemize}

Note that for $\cA$ being $\Bbbk$-flat, the pretriangulated structure of $\cRHom_{\Hqe}(\cD_{dg}(\cA), \cD_{dg}(\cB))$ restricts to the one of $\cD_{dg}( \opalg{\cA} \otimes \cB)$ (viewed as sub-dg-category). In particular, we obtain distinguished triangles of quasi-functors from short exact sequences of dg-bimodules. 

\begin{defn}
For a morphism $f : X \rightarrow Y \in Z^0(\cT)$ in the underlying category of pretriangulated dg-category $\cT$, one calls \emph{mapping cone} an object $\cone(f) \in \cT$ such that
\[
\cone(f)^\wedge \cong \cone(X^\wedge \xrightarrow{\bar f} Y^\wedge) \in H^0(\cT\amod).
\]
\end{defn}

\subsubsection{Dg-Morita equivalences}

\begin{defn}
A dg-functor $F : \cA \rightarrow \cB$ is a dg-Morita equivalence if it induces an equivalence
\[
\boldsymbol{L}F : \cD(\cA) \xrightarrow{\simeq} \cD(\cB) : X \mapsto F(QX),
\]
where $QX$ is a cofibrant replacement of $X$. 
\end{defn}

\begin{exe}
In particular, a quasi-equivalence is a dg-Morita equivalence and the functor that sends dg-categories to their pretriangulated hull $\cA \mapsto pretr(\cA)$ is a dg-Morita equivalence.
\end{exe}

\begin{thm}[\cite{tabuada}]
There is a model structure $\dgcat_{mor}$ on $\dgcat$ where the weak-equivalences are the dg-Morita equivalences and the fibrations are the same as before. 
\end{thm}

\begin{defn}
We say that $\cT$ is triangulated if it is fibrant in $\dgcat_{mor}$.
\end{defn}

Equivalently, $\cT$ is triangulated if and only if the Yoneda functor induces an equivalence $H^0(\cT\amod) \xrightarrow{\simeq} \cD^c(\cT)$ (i.e. every compact object is quasi-representable). 
Also equivalently, $\cT$ is triangulated if and only if $\cT$ is pretriangulated and $H^0(\cT\amod)$ is idempotent complete. 

\smallskip

In particular, any category admits a triangulated hull $tr(\cA)$ (i.e. fibrant replacement). 
It is given by $tr(\cA) := \cD_{dg}^c(\cA)$, the dg-category of compact objects in $\cD_{dg}(\cA)$. 

\begin{exe}
Let $R$ be a $\Bbbk$-algebra viewed as a dg-category. Then $\cD_{dg}^c(R)$ is the dg-category of perfect complexes, i.e. bounded complexes of finitely generated projective $R$-modules. 
\end{exe}

%%%%%%%%%%%%%%%%	End of file	%%%%%%%%%%%%%

%% file: bibliography/bibliography.tex
%%%%%%%%%%%%%%%%%%%%%%%%%%%%%%%%%%%%
%                 					  				  		 %
%	Bibliography   				 					 %
%                 					  						 %
%%%%%%%%%%%%%%%%%%%%%%%%%%%%%%%%%%%%

\bibliographystyle{bibliography/habbrv}
%\bibliography{bibliography/2verma}	

\input{bibliography/2verma.bbl}
%%%%%%%%%%%%%%%%	End of file	%%%%%%%%%%%%%